\documentclass[a4paper]{article}

\usepackage{hyperref}
\usepackage[margin=2.5cm]{geometry}
\usepackage{mathtools}
\usepackage{amssymb}
\usepackage{amsthm}
\usepackage{empheq}

\usepackage{enumitem}
\usepackage[utf8]{inputenc}

\usepackage[dvipsnames]{xcolor}
\usepackage{amsmath}
\usepackage{cleveref}
\usepackage{mathtools}

\usepackage{thmtools,thm-restate}
\usepackage{tikz}
\usetikzlibrary{arrows}
\usetikzlibrary{calc}

\usepackage{esvect}
\usepackage{setspace}
\usepackage{enumitem}
\usepackage[linesnumbered,commentsnumbered,ruled,vlined]{algorithm2e}

\newcommand{\ie}{i.e.}
\newcommand{\eg}{e.g.}

\newcommand{\st}{\text{ s.t. }}

\newcommand{\clN}{\overline{\mathcal{N}}^2}
\newcommand{\clNw}{\overline{\mathcal{N}}^{-\infty}}

\newcommand{\norm}[1]{\left\lVert#1\right\rVert}
\newcommand{\as}{\mathrm{a}}
\newcommand{\cs}{\mathrm{c}}
\newcommand{\T}{\top}
\newcommand{\R}{\mathbb{R}}
\newcommand{\CP}{\mathrm{CP}}
\newcommand{\MCP}{\mathrm{MCP}}
\newcommand{\mcp}{^{\mathfrak{m}}}
\newcommand{\pr}[2]{\left\langle #1, #2 \right\rangle}
\def\<#1>{\mathinner{\langle#1\rangle}}

\newcommand{\SLC}{\operatorname{SLC}}

\newcommand{\Primal}{\mathcal{P}}
\newcommand{\Primalp}{\mathcal{P}_{++}}
\newcommand{\Dual}{\mathcal{D}}
\newcommand{\Dualp}{\mathcal{D}_{++}}
\newcommand{\eps}{\varepsilon}
\newcommand{\proj}{\Pi}
\newcommand{\poly}{\operatorname{poly}}
\renewcommand{\cal}{\mathcal}

\newcommand{\mnbp}{{\cal N}^{\mathfrak{mp}}}
\newcommand{\mnbd}{{\cal N}^{\mathfrak{md}}}
\newcommand{\MP}{\mathbf{M}^{\rm P}}
\newcommand{\MD}{\mathbf{M}^{\rm D}}
\newcommand{\MPalg}{\wt{\mathbf{M}}^{\rm P}}
\newcommand{\MDalg}{\wt{\mathbf{M}}^{\rm D}}
\newcommand{\Cit}{C_{\rm it}}
\newcommand{\Dit}{D_{\rm it}}

\newcommand{\im}{\operatorname{im}}

\newcommand{\diag}{\operatorname{diag}}

\renewcommand{\rho}{\varrho}

\newcommand{\1}{\mathbf 1}
\newcommand{\0}{\mathbf 0}
\newcommand{\wt}{\widetilde}
\newcommand{\wh}{\widehat}

\newcommand{\polarized}{\gamma}

\newcommand{\mdot}{\bullet}

\newcommand{\nx}{\xerr}
\newcommand{\ns}{\serr}
\newcommand{\ndx}{\Delta \nx}
\newcommand{\nds}{\Delta \ns}

\newcommand{\nW}{\widehat{W}}

\newcommand{\scale}{\Xi}
\newcommand{\src}{X}
\newcommand{\targ}{Y}

\newcommand{\Nalg}{\wt{N}}
\newcommand{\Balg}{\wt{B}}

\newcommand{\lift}{\ell}

\newcommand{\error}{\hat{\xi}}

\newcommand{\xerr}{\hat x}
\newcommand{\Xerr}{\xerr}
\newcommand{\Werr}{\hat W}
\newcommand{\Wperr}{{\hat W}^\perp}

\newcommand{\serr}{\hat s}
\newcommand{\Serr}{\serr}

\newcommand{\slls}{\mathrm{\ell}}

\newcommand{\thresh}{\tau}
\newcommand{\cnt}[2]{{\cal C}_{\sigma}\left(#1,#2\right)}

\newcommand{\xy}{y}

\newcommand{\adj}{\operatorname{ad}}
\newcommand{\cp}{^{\mathrm{cp}}}

\newcommand{\sccp}{^{\mathrm{cp}(f)}}
\newcommand{\pz}{z}
\newcommand{\px}{x}
\newcommand{\ps}{s}
\newcommand{\bx}{x'} \newcommand{\bs}{s'} 

\newcommand{\mz}{z\mcp}
\newcommand{\mx}{x\mcp}
\newcommand{\ms}{s\mcp}
\newcommand{\gap}{\overline{\mu}}
\newcommand{\rank}{\mathrm{rank}}

\newcommand{\SVDA}{\rho}

\newcommand{\maxexpidx}{\zeta}
\newcommand{\genexpidx}{\bar{\zeta}}
\newcommand{\stepparam}{\nu}

\newcommand{\alg}{\wt}
\newcommand{\Valg}{\alg{V}}
\newcommand{\Ualg}{\alg{U}}

\newcommand{\supp}{\operatorname{supp}}

\newcommand{\idealdirection}{ideal direction\xspace}
\newcommand{\ideal}{\mathrm{ideal}}
\newcommand{\prj}{\mathrm{proj}}

\newcommand\restr[2]{{\left.\kern-\nulldelimiterspace #1 \vphantom{\big|} \right|_{#2} }}

\newcommand\restro[3]{{\left.\kern-\nulldelimiterspace #1 \vphantom{\big|} \right|^{#3}_{#2} }}

\newcommand{\lspan}{\operatorname{span}}

\makeatletter
\let\original@algocf@latexcaption\algocf@latexcaption
\long\def\algocf@latexcaption#1[#2]{\@ifundefined{NR@gettitle}{\def\@currentlabelname{#2}}{\NR@gettitle{#2}}\original@algocf@latexcaption{#1}[{#2}]}
\makeatother

\newlength{\leftstackrelawd}
\newlength{\leftstackrelbwd}
\def\leftstackrel#1#2{\settowidth{\leftstackrelawd}{${{}^{#1}}$}\settowidth{\leftstackrelbwd}{$#2$}\addtolength{\leftstackrelawd}{-\leftstackrelbwd}\leavevmode\ifthenelse{\lengthtest{\leftstackrelawd>0pt}}{\kern-.5\leftstackrelawd}{}\mathrel{\mathop{#2}\limits^{#1}}}

\newtheorem{theorem}{Theorem}[section]
\newtheorem{lemma}[theorem]{Lemma}

\newtheorem{proposition}[theorem]{Proposition}

\newtheorem{corollary}[theorem]{Corollary}
\theoremstyle{definition}
\newtheorem{definition}[theorem]{Definition}

\newtheorem{remark}[theorem]{Remark}

\makeatletter
\providecommand*{\cupdot}{\mathbin{\mathpalette\@cupdot{}}}
\newcommand*{\@cupdot}[2]{\ooalign{$\m@th#1\cup$\cr
    \hidewidth$\m@th#1\cdot$\hidewidth
  }}
\makeatother

\foreach \x in {A,...,Z}{\expandafter\xdef\csname m\x\endcsname{\noexpand\mathbf{\x}}
}

\DeclareMathOperator*{\argmax}{arg\,max}
\DeclareMathOperator*{\argmin}{arg\,min}

\newcommand{\Pcal}{\mathcal{P}}

\newcounter{mycomment}

\newlist{mycases}{enumerate}{1}
\setlist[mycases,1]{wide=0pt, widest=99,leftmargin=\parindent,label=\textbf{Case \Roman*.}, ref=Case \Roman*}

\newlist{myfirstcases}{enumerate}{1}
\setlist[myfirstcases,1]{wide=0pt, widest=99,leftmargin=\parindent,label=\textbf{Case I.\arabic*.}, ref=Case I.\arabic*}

\newcommand\myshade{100}
\colorlet{mylinkcolor}{NavyBlue}
\colorlet{mycitecolor}{YellowOrange}
\colorlet{myurlcolor}{Aquamarine}

\hypersetup{
  linkcolor  = mylinkcolor!\myshade!black,
  citecolor  = mycitecolor!\myshade!black,
  urlcolor   = myurlcolor!\myshade!black,
  colorlinks = true,
}

\newcommand\blfootnote[1]{\begingroup
  \renewcommand\thefootnote{}\footnote{#1}\addtocounter{footnote}{-1}\endgroup
}

\title{Interior point methods are not worse than Simplex \blfootnote{This project has received funding from the European Research Council (ERC) under the European Union's Horizon 2020 research and innovation programme: DD from grant agreement no.  805241-QIP,  BN and LAV from grant agreements no. 757481-ScaleOpt. We thank the 2021 Hausdorff Research Institute for Mathematics Program Discrete Optimization during which part of this work was developed.}}
\date{}

\author{
Xavier Allamigeon\thanks{Inria, CMAP, CNRS, Ecole Polytechnique, IP Paris, France}
\and
Daniel Dadush\thanks{Centrum Wiskunde \& Informatica, Amsterdam, The Netherlands}
\and
Georg Loho\thanks{University of Twente, The Netherlands}
\and
Bento Natura\thanks{Georgia Institute of Technology, USA} 
\and
L{\'{a}}szl{\'{o}} A. V{\'{e}}gh\thanks{London School of Economics and Political Science, UK} 
}

\begin{document}

\maketitle

\begin{abstract}

We develop a new `subspace layered least squares' interior point method (IPM) for solving linear programs. Applied to an $n$-variable linear program in standard form, the iteration complexity of our IPM is up to an $O(n^{1.5} \log n)$ factor upper bounded by the \emph{straight line complexity} (SLC) of the linear program. This term refers to the minimum number of segments of any
piecewise linear curve that traverses the \emph{wide neighborhood} of the
central path, a lower bound on the iteration complexity of
any IPM that follows a piecewise linear trajectory along a path induced by a self-concordant barrier. In particular, our algorithm matches the number of iterations of any such IPM up to the same factor $O(n^{1.5}\log n)$.

As our second contribution, we show that the SLC of any linear program is
upper bounded by $2^{n(1 + o(1))}$, which implies that our IPM's iteration complexity
is at most exponential. This is in contrast to existing iteration complexity bounds that depend on either bit-complexity or
condition measures;  these can be unbounded in the problem dimension.
We achieve our upper bound by showing that the central path is
well-approximated by a combinatorial proxy we call the \emph{max central path},
which consists of $2n$ shadow vertex simplex paths. Our upper bound
complements the lower bounds of Allamigeon, Benchimol, Gaubert, and Joswig
(SIAGA 2018), and Allamigeon, Gaubert, and Vandame (STOC 2022), who
constructed linear programs with exponential SLC. 

Finally, we show that each iteration of our IPM can be implemented in strongly polynomial time. Along the way, we develop a deterministic algorithm that approximates the singular value decomposition of a matrix in strongly polynomial time to high accuracy, which may be of independent interest.

\end{abstract}
 \newpage
\tableofcontents

\section{Introduction}\label{sec:intro}

In this paper, we develop a new interior point method for finding exact solutions to  linear programming. Our method is `universal' in the sense that it matches the iteration complexity of any other interior point method up to a small polynomial factor. 
Our analysis also reveals connections between  interior point methods and the simplex method.We consider linear programming (LP) in the following primal-dual form:
\begin{equation}
\label{LP_primal_dual} \tag{LP}
\begin{aligned}
\min \; & \pr{c}{x} \\
\mA x & = b  \\
x & \geq \0 \, ,
\end{aligned}
\qquad \qquad 
\begin{aligned}
\max \; & \pr{b}{y}\\
\mA^\top y + s & = c \\
s & \geq \0 \, ,
\end{aligned}
\end{equation}
where $\mA \in \R^{m \times n}$, $b \in \R^m$ and $c \in \R^n$, and $\mathrm{rk}(\mA)=m \leq n$. We let
 \[
 \Primal \coloneqq  \{x \in \R^n \colon  \mA x=b, x \geq \0\}\, , \quad
 \Dual \coloneqq  \{s \in \R^n \colon \exists y \st \mA^\top y+ s=c, s\geq \0\}\, 
 \]
 denote the primal and dual feasible regions and furthermore 
 \[
 \Primalp \coloneqq \{x \in \Primal:\, x > \0\}\, , \quad
 \Dualp \coloneqq \{s \in \Dual:\, s > \0\}\, 
 \]
 denote the primal and dual strictly feasbible regions.
 We summarize the notation used in the paper in Section~\ref{sec:prelims}.Throughout, we assume that \eqref{LP_primal_dual} is feasible and bounded; consequently, $\Primal,\Dual\neq\emptyset$.  We let $v^\star$ denote the optimum value of~\eqref{LP_primal_dual}. 
 Our focus is on LP algorithms that find exact primal and dual optimal solutions.

 The \emph{simplex method} traverses a path formed by vertices and edges of $\Primal$ according to a certain \emph{pivot rule}.
 Albeit efficient in practice, there is no polynomial-time variant known, and there are
exponential worst case examples for several pivot rules. The first such construction was given by 
 Klee and Minty~\cite{KleeMinty} for Dantzig's pivot rule.   

Breakthrough developments in the seventies and eighties led to the
 first poly\-nomial-time algorithms for linear programming: the \emph{ellipsoid method} by Kha\-chiyan \cite{Khachiyan79}, and \emph{interior point methods} introduced by Karmarkar \cite{Karmarkar84}. These algorithm run in time $\mathrm{poly}(n,L)$, where $L$ denotes the
encoding-length $L$ of the rational input $(\mA,b,c)$ of \eqref{LP_primal_dual}.

While the simplex method may require exponentially many iterations, it is never worse: for any non-cycling pivot rule, the  number of pivot steps can be bounded by the number of bases, at most $\binom{n}{m}< 2^n$.  
Whereas the bound  $\mathrm{poly}(n,L)$ is typically much better, the encoding length $L$
may be arbitrarily large. To the extent of our knowledge, no variant of the ellipsoid or interior point methods have been shown to admit a bound $f(n)$ on the number of iterations for any function $f \colon \mathbb{N}\to\mathbb{N}$ prior to our work. 

Even though LPs with exponential encoding length do not frequently appear in practice, there are  examples when the binary encoding is exponential yet one could efficiently implement arithmetic operations using a different encoding, see Megiddo \cite{Megiddo82binary}. The net present value problem in project scheduling is a particular example of a
 natural optimization problem that can be reformulated as an LP of exponential encoding length, see Grinold \cite{G72}.
From a theoretical perspective, finding an interior point method with an absolute bound $f(n)$ on the number of iterations connects to the fundamental open question on finding a strongly polynomial algorithm for linear programming. In such an algorithm, the  number of arithmetic operations is bounded as $\text{poly}(n)$, and the algorithm uses polynomial space. This question takes its roots in the development of the simplex method, and appears in Smale's list of open problems for the $21^\mathrm{st}$ century~\cite{Smale98}. 

\paragraph{Interior point methods and the central path}
Whereas the simplex method moves on the boundary of the feasible region $\Primal$,  interior point methods (IPM) reach an optimal solution by iterating through the strict interior of $\Primal$.
    Path-following interior point methods are driven to an optimal point by following a smooth trajectory called the \emph{central path}. In the most standard setting~\cite{Renegar1988}, the latter is defined as the parametric curve $\mu \in (0, \infty) \mapsto z\cp(\mu) \coloneqq (x\cp(\mu), s\cp(\mu))$, where $x\cp(\mu)$ and $(y\cp(\mu), s\cp(\mu))$ are the unique solutions to the system
\begin{equation}
\begin{aligned}\label{eq:CP-equations}
\mA x\cp(\mu) & = b \, , \enspace x\cp(\mu) > \0 \\
\mA^\top y\cp(\mu) + s\cp(\mu) & = c \, , \enspace s\cp(\mu) > \0 \\
x\cp(\mu)_i s\cp(\mu)_i & = \mu \, \quad \text{for all} \enspace i \in [n] \, .\\
\end{aligned}
\end{equation}
This system arises from the optimality conditions of convex problems obtained by penalizing the original linear programs with the logarithmic barrier, \ie, respectively adding terms of the form $-\mu \sum_{i = 1}^n \log x_i$ and $\mu \sum_{i = 1}^n \log s_i$ to the objective functions of the primal and dual~\eqref{LP_primal_dual}. The weight of the penalty is given by the parameter $\mu > 0$.
When $\mu \searrow 0$, the central path $z\cp(\mu)$ converges to  a pair of optimal solutions $(x^\star, s^\star)$ of~\eqref{LP_primal_dual}, which can be easily deduced from the fact that the duality gap of $z\cp(\mu)$ is given by $\pr{c}{x\cp(\mu)} - \pr{b}{y\cp (\mu)}= \pr{x\cp(\mu)}{s\cp(\mu)} = n \mu$. 
Accordingly, we define the quantity $\gap(z) \coloneqq \pr{x}{s}/n$ for any feasible point $z=(x,s)\in \Primal\times \Dual$, which we refer to as the \emph{normalized duality gap} of $z$.

Interior point methods iteratively compute approximations of the points on the central path associated with successive values of $\mu$ that decrease geometrically; at most $O(\sqrt n \log (\mu / \mu'))$ iterations are needed to decrease the normalized duality gap from $\mu$ to $\mu'$. 
The iterations  follow an improvement direction, \eg, a Newton step, while remaining in a certain neighborhood of the central path, and can be implemented in polynomial time. The classical analysis yields a bound $O(\sqrt{n} L)$ on the number of iterations  for solving \eqref{LP_primal_dual} for a rational input $(\mA,b,c)$ of total encoding length $L$. There have been significant improvements in recent years both for general LP as well as for special classes, see Section~\ref{sec:related_work}.

A running time bound dependent on $L$ requires a rational input; in contrast, the simplex method can be implemented in $2^{n}\mathrm{poly}(n)$ even in the \emph{real model of computation}.
Whereas standard IPMs use bit-complexity arguments to terminate, they have also been extended to the real model of computation, e.g., by Vavasis and Ye~\cite{VavasisYeRealNumberData}. The running time of such algorithms is parametrized by \emph{condition numbers} that capture geometric properties of the input. In a remarkable paper, Vavasis and Ye \cite{Vavasis1996} introduced a \emph{layered least squares (LLS)} interior point method that runs in $O(n^{3.5} \log (\bar\chi_{\mA}+n))$ iterations, where $\bar\chi_{\mA}$ is the Dikin--Stuart--Todd condition number associated with the kernel of $\mA$ (but independent of $b$ and $c$). As a consequence, they also derive a structural characterization of the central path: there are at most $\binom{n}{2}$ `short and curved' segments, possibly separated by `long and straight' segments.  
The LLS directions are refined Newton steps that can traverse the latter segments.

Lan, Monteiro and Tsuchiya \cite{LMT09} gave a scaling invariant \emph{trust region IPM} taking $O(n^{3.5} \log (\bar\chi^*_{\mA}+n))$ iterations. Here, $\bar\chi^*_{\mA}$ is the minimum value of $\bar\chi_{\mA}$ that can be achieved by any column rescaling. However, computing the step directions in this algorithm has a weakly polynomial dependence on $b$ and $c$. In recent work, Dadush, Huiberts, Natura, and V\'egh \cite{DHNV20} gave a scaling invariant LLS algorithm with iteration bound $O(n^{2.5}\log(n) \log (\bar\chi^*_{\mA}+n))$, where the step directions can be computed by solving linear systems. We discuss the literature on these IPMs in more detail in Section~\ref{sec:related_work}.

\paragraph{Lower bounds on interior point methods}
LLS methods provide strongly polynomial LP algorithms whenever $\bar\chi^*_{\mA}\in 2^{\mathrm{poly}(n)}$; this is always the case if the
encoding-length of $\mA$ is polynomially bounded. One may wonder if some variant of IPM could be strongly polynomial for all LPs. 
A negative answer to this question was given in recent work by
Allamigeon, Benchimol, Gaubert, and Joswig: they  used tropical geometry to build pathological linear programs on which the number of iterations of IPM has to be exponential (in $m,n$)~\cite{ABGJ18,ABGJ21}. Their construction shows that, when the entries of $\mA$, $b$, and $c$ are of very different orders of magnitude, the central path can be significantly deformed to the boundary of the feasible set. Allamigeon, Gaubert and Vandame later extended this result to the broad class of path-following IPMs using
 any self-concordant barrier function~\cite{AGV22}; concurrently, Zong, Lee, and Yue \cite{Zong2023} obtained impossibility for short step methods. The paper \cite{AGV22}  exhibits a counterexample where the feasible set is an $n$-dimensional combinatorial cube and the shape of the central path is analogous to the simplex paths on pathological instances of LP for the simplex method, akin to the  Klee--Minty cube~\cite{KleeMinty}.

\subsection{Contributions} 

\paragraph{A `near-optimal' interior point method}
The papers \cite{ABGJ18,ABGJ21} implicitly rely on the following lower bound: the trajectory of an IPM performing $T$ iterations in the wide neighborhood of the central path  defines a piecewise linear curve. Hence, the minimum number of pieces of \emph{any} piecewise linear curve in the same neighborhood provides a lower bound on the number of iterations.
We introduce a new interior point method based on \emph{subspace layered least squared steps} (see~\nameref{alg:subspace_ipm}) and show that the number of iterations of our IPM can be upper bounded in such terms. 

The algorithm  navigates through the  $\ell_2$-neighborhood of the central path:
\begin{equation} \label{eq:l2_neighborhood}
 \cal N^2(\beta) \coloneqq \left\{z = (x,s) \in \Primal_{++} \times \Dual_{++} : \norm{\frac{xs}{\gap(z)} - \1} \leq \beta \right\}\, , \qquad (0 < \beta< 1/4)\, ,
\end{equation}
where $xs\in\R^n$ denotes the coordinate-wise  product and $\1\in\R^n$ is the $n$-dimensional all-ones vector.

We also define the \emph{wide neighborhood} as follows:
\begin{equation} \label{eq:wide_neighborhood}
\cal N^{-\infty}(\theta) \coloneqq \left\{z = (x,s) \in \Primal_{++} \times \Dual_{++}:\, x s \geq (1-\theta) \gap(z) \1\right\} \, , \qquad (0 < \theta < 1)\, .
\end{equation}
We let $\clN(\beta)\coloneqq \operatorname{cl}(\cal N^2(\beta))$  and $\clNw(\beta)\coloneqq \operatorname{cl}(\cal N^{-\infty}(\beta))$ denote 
the closures of these neighborhoods. These also include points $z=(x,s)\in  \Primal_{+} \times \Dual_{+}$ with $\gap(z)=0$, i.e., optimal solutions. Our algorithm will terminate with an optimal solution in $\clN(\beta)$. We show the following iteration bound. 

\begin{theorem}\label{th:main_upper_bound-curve}
Let $\beta\in (0,1/6]$, $\theta\in(0,1)$ and $\mu_0> \mu_1\ge 0$.
  Let $\Gamma: (\mu_1,\mu_0) \to \clNw(\theta)$
  be any piecewise linear curve
satisfying $\gap\left(\Gamma(\mu)\right) = \mu$, $\forall \mu \in (\mu_1,\mu_0)$ with $T$ linear segments. 

Starting from any point $z^0 \in \cal N^2(\beta)$ such that $\gap(z^0) \leq \mu_0$, the algorithm~\nameref{alg:subspace_ipm} finds a solution $z^1\in \clN(\beta)$ with $\gap(z^1) \le \mu_1$ within 
\begin{equation*}
  O\left(\frac{n^{1.5}}{\beta} \log \Big(\frac{n}{\beta(1-\theta)}\Big) \mathinner T\right)
\end{equation*} 
iterations. \end{theorem}

At a high level, our strategy is to show that any `somewhat straight'
segment of the central path, corresponding to a single straight segment in
the wide neighborhood $\cal N^{-\infty}(\theta)$, can be decomposed into at
most $n$ short segments of length $\poly(n/(1-\theta))$ (as measured by the
ratio of the start and end parameter), where consecutive short segments are
possibly separated by `long and straight' segments. To traverse the long and
straight segments we develop a novel \emph{subspace} LLS step, which
generalizes prior LLS steps from coordinate subspaces to general ones.   
Before describing this in more detail, we present a stronger form of Theorem~\ref{th:main_upper_bound-curve}, and two applications.

\paragraph{The max central path and the straight-line complexity}
We next formulate a slight\-ly stronger form of \Cref{th:main_upper_bound-curve}.
The piecewise linear curve $\Gamma$ in the statement above lives in a $2n$-dimensional space; our next statement argues in terms of $2n$ separate objects in $2$-dimensional space.

Recall that  $v^\star$ denotes the optimum value of~\eqref{LP_primal_dual}. The \emph{max central path} is defined as the parametric curve $g \mapsto \mz(g) \coloneqq (\mx(g), \ms(g)) \in \R^{2n}_{+}$, where $\mx_i(g)$ and $\ms_i(g)$ are the optimal values of the following parametric LPs, respectively:
\begin{equation}\label{eq:max_cp}
\begin{aligned}
\max \; & x_i \\
\mA x & = b \, , \; x \geq \0\\
\pr{c}{x} & \leq v^\star + g \, ,
\end{aligned}
\qquad \qquad 
\begin{aligned}
\max \; & s_i \\
\mA^\top y + s & = c \, , \; s \geq \0\\
\pr{b}{y} & \geq v^\star - g \, .
\end{aligned}
\end{equation}
As we show in Section~\ref{sec:max-cp}, the maps $\mx_i(g)$ and $\ms_i(g)$ are piecewise linear concave, and the number of pieces can be related to the complexity of the simplex method with the shadow vertex rule.

The {max central path} can be seen as a combinatorial proxy to the central path.
In Section~\ref{sec:max-cp}, we show the following relationship. 
The upper bounds are immediate by noting that the duality gap for $(x\cp(\mu),s\cp(\mu))$ is $n\mu$.
\begin{restatable}{siamlemma}{maxvscentralpath}\label{lem:cp-max-cp}
For every $\mu>0$ and the central path point $z\cp(\mu)=(x\cp(\mu),s\cp(\mu))$, 
\[
\frac{\mz(n\mu)}{2n}\le z\cp(\mu)\le \mz(n\mu)\, .
\]
\end{restatable}

\begin{figure}[htb!]
  \begin{center}
\begin{tikzpicture}[scale=1]
\draw[gray!30,very thin] (-0.5,-0.5) grid (6.5,6.5);
\draw[gray!50,->,>=stealth'] (-0.5,0) -- (6.5,0) node[color=black!80,right] {$x_1$};
\draw[gray!50,->,>=stealth'] (0,-0.5) -- (0,6.5) node[color=black!80,above] {$x_2$};
\clip (-0.5,-0.5) rectangle (6.5,6.5);

\fill[gray!30, opacity=0.6] (0,0) -- (5,1) -- (6,3) -- (2,5.5) -- cycle;
\draw[black,very thick] (0,0) -- (5,1) -- (6,3) -- (2,5.5) -- cycle;

\coordinate (p1) at (6,5.5);
\coordinate (p2) at (5.5,5.5);
\coordinate (p3) at (5,4.4);
\draw[red!90!black,ultra thick] (p1) -- (p2);
\draw[orange,ultra thick] (p2) -- (p3);
\draw[yellow!90!black, ultra thick] (p3) -- (0,0);

\draw[very thick,dashed,red!90!black] (-.5,9.5) -- (9.5,-.5);
\fill[red!90!black] (p1) circle (3pt);
\draw[very thick,dashed,orange] (-.5,8) -- (8,-.5);
\fill[orange] (p2) circle (3pt);
\draw[very thick,dashed,yellow!90!black] (-.5,6.5) -- (6.5,-.5);
\fill[yellow!90!black] (p3) circle (3pt);

\fill[black] (0,0) circle (3pt);

\end{tikzpicture}
\end{center}
   \caption{The projection of the max central path on the primal coordinates $(x_1,x_2)$ for the cost function $x_1 + x_2$. Dashed lines correspond to level sets at breakpoints.
  Note that the max central path does not lie in the feasible region $\Primal$.}
\label{fig:max-cp-figure}
\end{figure}
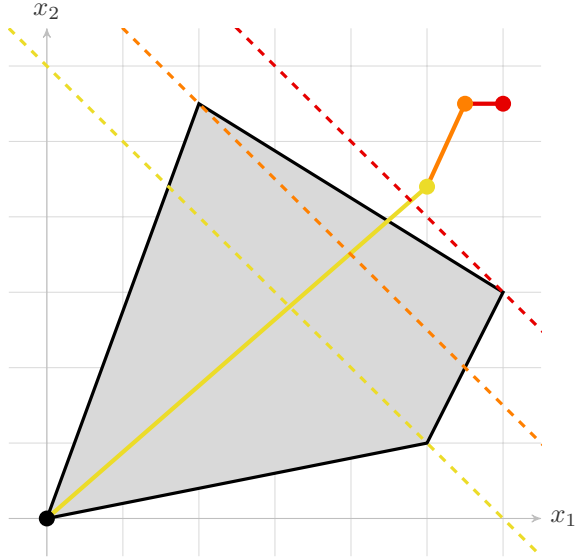

For each $i\in [n]$ and $\theta\in [0,1)$, we define the  \emph{primal and dual multiplicative neighborhoods of the max central path} as
\[
\begin{aligned}
\mnbp_{i}(\theta) &\coloneqq \{ (g,z) \in \R^2_+ \colon (1-\theta) \mx_i(g) \leq z \leq \mx_i(g) \} \, ,\\
\mnbd_{i}(\theta) &\coloneqq \{ (g,z) \in \R^2_+ \colon (1-\theta) \ms_i(g) \leq z \leq \ms_i(g) \}\, .
\end{aligned}
\]
\begin{definition}[Straight-line complexity]
\label{def:slc}
For $i\in [n]$, $\theta\in [0,1)$, and $0\le\underline g\le \overline g$, we define the \emph{primal straight-line complexity w.r.t.~coordinate $i$} as the minimum number of linear segments of any piecewise linear curve traversing the neighborhood $\mnbp_{i}(\theta)$ between parameter values $\overline g$ and $\underline g$, that is,
\begin{equation}\label{eq:slc-def}
\SLC^{\mathrm{p}}_{i,\theta}(\underline g, \overline g) \coloneqq 
\min 
\begin{multlined}[t]
\big\{ p \geq 1 : \exists (g_k, z_k)_{ k \in [p+1]} \in \R^2_+ \, , \enspace \overline g = g_1, \enspace \underline g = g_{p+1} \, , \\
\forall k \in [p] \, , \enspace [(g_k, z_k), (g_{k+1}, z_{k+1})] \subset {\mnbp_{i}}(\theta) \big\} \, .
\end{multlined}
\end{equation}
The \emph{dual straight-line complexity w.r.t.~coordinate $i$} is defined analogously for the dual neighborhood 
$\mnbd_{i}(\theta)$.
\end{definition}

We note that Lemma~\ref{lemma:slc2} shows that one would get an equivalent definition by allowing only breakpoints of the form $(g_k,z_k)=(g_k,\mx_i(g_k))$, and requiring $g_1 > g_2 > \dots > g_{p+1}$; this is simple consequence of the concavity of $g\mapsto \mx_i(g)$ shown in Lemma~\ref{lem:simplex-main}.

The stronger form of \Cref{th:main_upper_bound-curve} is as follows:
\begin{theorem}\label{th:main_upper_bound-slc}
Let $\beta\in (0,1/6]$, $\theta \in [0,1)$ and $\mu_0> \mu_1\ge 0$.
Given a starting point $z^0 \in \cal N^2(\beta)$ such that $\gap(z^0) \leq \mu_0$, the algorithm~\nameref{alg:subspace_ipm} finds
a solution $z^1\in \clN(\beta)$ with   $\gap(z^1) \leq \mu_1$ in 
\[
O\left(\frac{\sqrt{n}}{\beta}\log\left(\frac{n}{\beta(1-\theta)}\right)\min\left\{\sum_{i=1}^n \SLC^{\mathrm{p}}_{\theta,i}(n\mu_1,n\mu_0),\sum_{i=1}^n \SLC^{\mathrm{d}}_{\theta,i}(n\mu_1,n\mu_0)\right\}\right)
\]many iterations. If $\mu_1 = 0$, letting $z^1 = (x^1,s^1) \in \Primal \times
\Dual$, the supports $B :=\supp(x^1)$ and $N:=\supp(s^1)$ form a partition of $[n]$  and the algorithm additionally outputs $(v^1,w^1) \in \R^{2n}$
satisfying 
\begin{enumerate}
\item $v^1 \in \im(A^\top), v^1_B > 0$ and $Aw^1 = 0$, $w^1_N > 0$.
\item $\norm{(x^1_B v^1_B,s^1_N w^1_N)-\1_n} \leq \beta$.
\end{enumerate}
The algorithm can be implemented in the real RAM model, moreover, each  iteration runs in strongly polynomial time in the Turing model. \end{theorem}

Some remarks are in order. The computational models and the meaning of strong polynomiality in this context are explained in Section~\ref{sec:comp-models}. The condition for $\mu_1=0$ means that the final output is near the analytic centers of the primal and dual optimal faces, along with a certificate of this fact. This is discussed in Section~\ref{sec:cent-limit}. With respect to the iteration bound, the minimum
of the primal and dual straight line complexities is just to make the
statement symmetric; however, it can be shown that the two terms in the
minimum are within a constant factor of each other (see \Cref{sec:decompose}). 

From the above statement, Theorem~\ref{th:main_upper_bound-curve} follows
directly; the proof is in Section~\ref{sec:running-time-analysis} but we already give the intuition: 
according to Lemma~\ref{lem:cp-max-cp},  if  $\Gamma:
(\mu_1,\mu_0) \to \clNw(\theta)$, $\theta \in (0,1)$ is a piecewise linear
curve satisfying $\gap\left(\Gamma(\mu)\right) = \mu$, $\forall \mu \in
(\mu_1,\mu_0)$ with $T$ linear segments, then for each $i\in [n]$, the
projection of $\Gamma$ to $x_i$ gives a piecewise linear curve in
${\mnbp_{i}}(1-(1-\theta)/(2n))$ for the interval  $[n\mu_1,n\mu_0]$, and
analogously for the $s_i$'s.

\paragraph{An exponential upper bound on the number of iterations}
The number of piecewise linear segments of the curves $\mx_i(g)$ and $\ms_i(g)$ yield
trivial upper bounds on the straight-line complexities in Theorem~\ref{th:main_upper_bound-slc}.

These can be naturally interpreted in the context of the 
\emph{shadow vertex simplex rule}.
 Originally dubbed `parametric simplex' by Gass and
Saaty~\cite{jour/nrlq/GS55}, this is one of the most extensively analyzed simplex
rules from a theoretical perspective. The shadow vertex rule was used in Borgwardt's average case analysis \cite{Borgwardt2012} and in Spielman and Teng's smoothed analysis \cite{Spielman2004}. The interested reader may refer to the recent survey for a detailed exposition by Dadush and Huiberts~\cite{bwcachapter}.

Given a pointed polyhedron $\Primal \subseteq \R^n$ and two objectives $c^{(1)},c^{(2)} \in
\R^n$, the shadow vertex rule consists in iterating over the vertices of $\Primal$ successively maximizing the objectives $(1-\lambda) c^{(1)} + \lambda c^{(2)}$ as $\lambda$ goes from $0$ to $1$. Under non-degeneracy assumptions, the vertices of the path correspond to those vertices of the two-dimensional projection $\{ (\pr{c^{(1)}}{x}, \pr{c^{(2)}}{x}) \colon x \in \Primal \}$ that maximize some open interval of objectives $(1-\lambda) e^1 + \lambda e^2$, $\lambda \in [0,1]$ (where $e^1$ and $e^2$ are the standard basis for $\R^2$).
We denote by $S_{\Primal}(c^{(1)}, c^{(2)})$ the number of vertices of the projection of the simplex path in this two-dimensional projection; this corresponds to the number of non-degenerate pivots.

Recall that $(x^\star, s^\star)$ is the optimal solution of~\eqref{LP_primal_dual} at the central path limit point. In Section~\ref{sec:max-cp}, we show that
\begin{restatable}{siamlemma}{descriptionpiecesmaxcentral}\label{lem:simplex-main}
 The following hold:
\begin{enumerate}[label=(\roman*)]
\item $\forall i \in [n]$, $g \mapsto \mx_i(g)$ is a piecewise linear concave
non-dec\-reasing function with $S_\Primal(-s^{\star},e^{i})$ pieces. That is,
$\SLC^{\mathrm{p}}_{0,i}(0,\infty)= S_\Primal(-s^\star,e^{i})$.
\item $\forall i \in [n]$, $g \mapsto \ms_i(g)$ is a piecewise linear concave non-decreasing function with
$S_\Dual(-x^{\star},e^{i})$ pieces.
That is,
$\SLC^{\mathrm{d}}_{0,i}(0, \infty)= S_\Dual(-x^\star,e^{i})$.
\end{enumerate}
\end{restatable}

 As a consequence, we obtain the following bound:
\begin{theorem}\label{th:coro-max-central-path}
Given a starting point $z^0 \in \cal N^2(1/6)$, algorithm~\nameref{alg:subspace_ipm} finds an optimal solution of~\eqref{LP_primal_dual} in a number of iterations bounded by
\[
O\left(\sqrt{n} \log(n) \mathinner{}
\min \left\{ \sum_{i=1}^{n} S_\Primal(-s^\star,e^i), \sum_{i=1}^n S_\Dual(-x^\star,e^i)\right\}\right) \le O\left(\binom{n}{m} n^{1.5} \log(n)\right)\, .
\]
\end{theorem}
 Theorem~\ref{th:coro-max-central-path} thus complements the results of \cite{ABGJ18,ABGJ21} by giving a singly exponential upper bound. We note that the max central path also plays an important if implicit role in the papers~\cite{ABGJ18,ABGJ21,AGV22}, as it can be directly related to the tropical central path by the log-limit, see discussion in Section~\ref{sec:related_work}.

 Theorem~\ref{th:coro-max-central-path} assumes that a feasible starting point $z^0\in\cal N^2(\beta)$ is given. This assumption can be removed \eg~by using the standard homogeneous self-dual embedding \cite[Section 5.3.1]{Ye-book}. Then, the bounds in the theorem will refer to the shadow vertex paths and the number of vertices in the self-dual program.
\paragraph{Matching the complexity of any path-following method}
The second implication 
of Theorems~\ref{th:main_upper_bound-curve} and~\ref{th:main_upper_bound-slc} shows
that the number of iterations of~\nameref{alg:subspace_ipm} lower bounds, up to a factor $n^{1.5} \log n$, the running time of essentially \emph{any} interior point method.

Let $f$ be a self-concordant barrier over the polyhedron~$\Primal$ with complexity value $\vartheta_f$; we introduce these concepts in \Cref{sec:self-concordant}. This defines a corresponding central path, with  $x\sccp(g)$ denoting the unique point of the central path with gap $g > 0$. 
 Recall that $v^\star$ denotes the optimum value of \eqref{LP_primal_dual}.
 We define the \emph{wide neighborhood w.r.t. the barrier $f$ for a parameter $\theta\in(0,1)$} as
\begin{equation}
\cal N\sccp(\theta) \coloneqq \big\{ x \in \Pcal : x \geq (1-\theta) x\sccp(g) \enspace\text{where}\enspace \pr{c}{x} = v^\star + g \big\} \, . \label{eq:self_concordant_neighborhood}
\end{equation}
Note that the $\ell_2$-neighborhood
$\cal N^2(\theta)$ and the wide neighborhood $\cal N^{-\infty}(\theta)$ are defined in the primal-dual space $\R^{2n}$, whereas $\cal N\sccp(\theta)\subseteq \R^{n}$ is in the primal space. Projecting  the neighborhoods $\cal N^{-\infty}(\theta)$ to primal variables can be shown to be equivalent to the neighborhoods of the form $\cal N\sccp(\theta)$ for the logarithmic barrier $f=-\sum_{i\in[n]} \log (x_i)$. We refer to \Cref{sec:self-concordant} 
for a discussion on the generality of the neighborhoods $\cal N\sccp(\theta)$.

\begin{theorem}\label{th:coro-lower-bound}
Let $f$ be a self-concordant barrier over the polyhedron~$\Primal$ with complexity value $\vartheta_f$. Let $\beta\in(0,1/6]$,  
$\theta\in (0,1)$, and $g_0> g_1\ge 0$.
Assume an interior point method proceeds through $T$ straight-line steps from $x^0$ to $x^1$ inside the wide neighborhood $\cal N\sccp(\theta)$ with $g_0={\pr{c}{x^0} - v^\star}$ and $g_1={\pr{c}{x^1} - v^\star}$.

Given any $z^0\in \cal N^2(\beta)$ with  $n\gap(z^0)\le g_0$,
let $T'$ be the number of iterations of ~\nameref{alg:subspace_ipm}  to reach the first iterate $z^1\in\clN(\beta)$ with $n\gap(z^1)\le g_1$. Then,
\[
T'=O\left(T n^{1.5}  \log \left(\frac{n \vartheta_f}{1-\theta}\right)\right)\, .
\]
\end{theorem}

Thus---up to a polynomial factor---our algorithm matches the running time of any IPM for any self-concordant barrier function staying in the extremely wide neighborhood $\cal N\sccp\left(1-1/2^{\mathrm{poly}(n)}\right)$. In particular, we obtain  polynomial-time bounds in case the bit-complexity or a condition number such as $\bar\chi^*_\mA$ is bounded.

\paragraph{Comparison to the Trust Region IPM}
\nameref{alg:subspace_ipm} also has an interesting relation to the Trust Region IPM algorithm by Lan, Monteiro, and Tsuchiya \cite{LMT09}. The trust region steps are obtained as optimal solutions to primal and dual quadratic programs (see~\eqref{eq:trust-region} below). These programs in essence capture the longest possible step achievable at the current point (up to a certain factor). However, it is currently not known how to solve these programs to sufficient accuracy in strongly polynomial time (though this can be done in weakly polynomial time). Lan, Monteiro, and Tsuchiya show in \cite{LMT09} that the number of iterations of the trust region algorithm can be bounded as $O(n^{3.5} \log (\bar\chi^*_{\mA}+n))$, by adapting the analysis of the LLS methods \cite{Monteiro2003,Vavasis1996}. 

The step directions used by our algorithm are feasible solutions to
\eqref{eq:trust-region} for a suitable parameter. This implies that the steps
of the Trust Region algorithm are always at least as long as the steps in our
algorithm; as a consequence, the iteration bounds of our
algorithm are also applicable to the Trust Region algorithm. Whereas any
individual step of our algorithm could be arbitrarily worse than the one
using the trust region step, \Cref{th:coro-lower-bound} implies that overall
we may only take $O(n^{1.5} \log {n})$ more iterations. We emphasize that
\cite{LMT09} only provides a $\bar\chi^*_{\mA}$ dependent iteration bound,
and we do not see a way to obtain even an $f(n)$ bound for their algorithm
without using the majority of the analysis of \nameref{alg:subspace_ipm}.

\subsection{Techniques}
\label{sec:techniques}

We now explain the key ideas of the algorithm \nameref{alg:subspace_ipm} and the analysis.

\subsubsection{Polarization of the Central Path}

The first key idea behind the proof of Theorems~\ref{th:main_upper_bound-curve} and \ref{th:main_upper_bound-slc} is the following: every linear segment
in the wide neighborhood gives rise to a \emph{polarized segment} of the
central path.
A  segment of the central path $\CP[\mu_1,\mu_0] \coloneqq  \{z\cp(\mu): \mu
\in [\mu_1,\mu_0]\}$, $0 \leq \mu_1 < \mu_0$,
is polarized, if it admits a partition $B \cup N =
[n]$ such that the primal variables in $B$ are barely changing while those in
$N$ are scaling down linearly with the parameter $\mu$ (vice versa for the
dual variables). More precisely, $\forall \mu \in [\mu_1,\mu_0]$, we require
\begin{align}
\polarized x_i\cp(\mu_0)  &\leq x_i\cp(\mu) \leq n x_i\cp(\mu_0)\, ,\quad \forall i \in B\, , \nonumber \\
\frac{\mu}{n \mu_0} x_i\cp(\mu_0) &\leq x_i\cp(\mu) \leq \frac{\mu}{\polarized \mu_0} x_i\cp(\mu_0)\, ,\quad  \forall i \in N\, , \label{eq:polar-req}
\end{align}
where $\polarized \in (0,1]$ is a \emph{polarization parameter} (see \Cref{def:polarized} and \Cref{cor:polarization}). By definition of the central path, the same relation holds for dual variables
$s\cp(\mu)$, $\mu \in [\mu_1,\mu_0]$, with the roles of $N$ and $B$ swapped. We note that the upper bounds on $x_i\cp(\mu)$ for $i\in B$ and the lower bounds on $x_i\cp(\mu)$ for $i\in N$ hold  by the near-monotonicity property  of the central path (see Lemma~\ref{lem: central_path_bounded_l1_norm}); the important parts of the definition are the other two bounds.

For simplicity of notation, let us restrict to line segments between two
points on the central path. To relate polarization to the wide neighborhood,
we show that if the line segment $[z\cp(\mu_1),z\cp(\mu_0)]$ between central path
points is contained in the wide neighborhod $\cal N^{-\infty}(\theta)$, then
the corresponding segment of the central path is polarized with $\polarized =
\frac{(1-\theta)^2}{16n^3}$ with respect to some partition $B \cup N = [n]$
(see Lemma~\ref{lem:wide-polar} for the general statement). 

One should read this last statement as saying that a segment of the central path is
`approximately linear' if and only if it is polarized (in fact, one can show that segment is
$1$-polarized if and only if it is linear).
The link between polarization and linearity
is surprisingly elementary; it follows from the analysis of
the inequalities of the wide neighborhood~\eqref{eq:wide_neighborhood}: 
\begin{multline*}
\big((1-\alpha) x_i\cp(\mu_0) + \alpha x_i\cp(\mu_1)\big)~\big((1-\alpha) s_i\cp(\mu_0) + \alpha
s_i\cp(\mu_1)\big) \\ \geq (1-\theta)((1-\alpha)\mu_0 + \alpha \mu_1)\, ,
 \,  \forall \alpha
\in [0,1]\, ,\, i \in [n]\, ,
\end{multline*}
where we recall that $z\cp(\mu_0) = (x\cp(\mu_0),s\cp(\mu_0)), z\cp(\mu_1) =
(x\cp(\mu_1),s\cp(\mu_1))$. For example, if $\theta = 0$, it is not hard to check
that for each $i \in [n]$, one must have either $x_i\cp(\mu_0) = x_i\cp(\mu_1)$ and
$s_i\cp(\mu_1) = \frac{\mu_1}{\mu_0} s_i\cp(\mu_0)$ (i.e., $i \in B$) or $x_i\cp(\mu_0) = 
\frac{\mu_1}{\mu_0} x_i\cp(\mu_1)$ and $s_i\cp(\mu_1) = s_i\cp(\mu_0)$ (i.e., $i \in N$).    

Given the above, the main task in proving \Cref{th:main_upper_bound-curve}, namely
traversing linear segments in the wide-neighborhood, can be reduced to
traversing $\polarized$-polarized segments of the central path. The main guarantee of
our algorithm~\nameref{alg:subspace_ipm} is in fact that it can
traverse any $\polarized$-polarized segment of the path in
$O(n^{1.5}\log(n/\polarized))$ iterations (see Theorem~\ref{thm:polar-ipm})

To derive the stronger bound in \Cref{th:main_upper_bound-slc}, a key step is
to use the max central path to guide the decomposition of the central path into
polarized segments. We will show that one can decompose the central path into
polarized segments where the polarization partitions do not change ``too
quickly'' from segment to segment. Specifically, the sum of partition changes
will be bounded by the sum of straight line complexities of either the primal
or dual. The formal statement is given below: 

\begin{restatable}{theorem}{deltaslc}\label{thm:delta-slc}
Let $\theta \in [0,1)$ and $\mu_0 > \mu_1\ge 0$. The segment $\CP[\mu_1,\mu_0]$ can be decomposed into a sequence of $\frac{1-\theta}{4n}$-polarized segments with partitions $(B^{(k)}, N^{(k)})$, $k\in [T]$, such that
\[
\max \left\{ \sum_{k=1}^T | N^{(k)} \Delta N^{(k-1)} |, T \right\} \leq 2\min\left\{ \sum_{i=1}^n \SLC^{\mathrm{p}}_{\theta, i}(n\mu_1, n\mu_0),  \sum_{i=1}^n \SLC^{\mathrm{d}}_{\theta, i}(n\mu_1, n\mu_0)\right\} 
\]
where  $N^{(0)}\coloneqq\emptyset$. 
\end{restatable}

For the sake of symmetry, we state the upper bound above terms of the minimum
of either primal or dual straight line complexities. However, it can be shown
that both are equivalent up to a constant factor (see \Cref{lem:dual-slc}).

We note that polarization plays an important if implicit role in prior
layered least squares analyses~\cite{DHNV20,MonteiroT05,Vavasis1996}. In
particular, the `long and straight' segments in these works are all
polarized. What was unclear in these works, however, is whether polarization
\emph{by itself} was sufficient to make a segment easy to traverse. Indeed,
these works all crucially rely upon numerical condition numbers of the
instance which can be effectively unbounded in the present context. Beyond
the LLS context, we are further unaware of central path analyses exploiting
the tight connection between approximate linearity and polarization, and we
hope this will encourage future study. 

As is clear from the definition, polarization provides us extremely useful
`long-range' control over the evolution of variables on a segment. Note
that $\polarized$-polarization is mostly interesting when the segment itself
is \emph{long}, namely, when $\mu_0/\mu_1 \gg 1/\polarized$. We now explain how
to leverage this control to traverse any $\polarized$-polarized segment using
\emph{subspace LLS steps}. 

\subsubsection{Traversing a Polarized Segment}

Let $\CP[\mu_1,\mu_0]$, $0 \leq \mu_1 < \mu_0$,  be a $\polarized$-polarized
segment with partition $B \cup N = [n]$. 

For simplicity of presentation, let us assume that given any iterate $(x,s)$
in the narrow neighborhood $\cal N^2(1/6)$ used in our algorithm,  we can jump to the exact central path point $z\cp(\mu) \in \CP =
\cal N^2(0)$ with $\mu = \gap(x,s)$ for free. Let us further assume
that the algorithm knows the partition $B,N$ (we discuss how to effectively
compute it at the end) and that we are given the starting point $z^{(0)} \coloneqq 
z\cp(\mu_0)$.    

Our abstract algorithm will thus compute iterates $z^{(0)},z^{(1)},\dots$ on
the central path $\CP$ with $\gap(z^{(0)}) > \gap(z^{(1)}) > \dots$. To move
from $z^{(t)}$ to $z^{(t+1)}$, we first compute a \emph{movement direction} 
\begin{align*}
\Delta z^{(t)} = (\Delta x^{(t)},\Delta s^{(t)}) \in \ker(\mA) \times \im(\mA^\T) \eqqcolon W \times W^\perp,
\end{align*}
together with a step-length $\alpha^{(t)} \in [0,1]$,  chosen such that  $z^{(t)} + \alpha \Delta z^{(t)} \in \clN(1/6)$, $0 \leq
\alpha \leq \alpha^{(t)}$. Lastly, assume we can jump for free to $z^{(t+1)} \in \CP$
satisfying $\gap(z^{(t+1)}) = \gap(z^{(t)} + \alpha^{(t)} \Delta z^{(t)})$.

Given this setup, our goal is to compute movement directions and step-lengths, such that after
$k = O(n^{1.5} \log(n/\polarized))$ iterations, we have $\gap(z^{(k)}) \leq
\mu_1$, i.e., that we have traversed the segment.
We would like to emphasize that our algorithm will in fact compute the
movement direction $\Delta z^{(t)}$ using only \emph{local information} at
$z^{(t)}$, without any explicit knowledge of the polarized segment.

A natural movement direction is \emph{affine scaling} used in predictor-corrector methods, see Section~\ref{sec:aff-lls}. This direction guarantees multiplicative $1-\Omega(1/\sqrt{n})$
decrease in normalized gap per step.  Hence, if 
$\mu_0/\mu_1 \leq
\poly(n,1/\polarized)$, then simply using $\sqrt{n}
\log(\mu_0/\mu_1)$ affine scaling iterations is sufficient for our purposes. 

Thus, we may assume that $\mu_0/\mu_1 \gg
\poly(n,1/\polarized)$. In this case, we will show that the affine scaling direction $(\Delta x^\as,\Delta s^\as)$ at the current iterate $(x^{(t)},s^{(t)})$ reveals the correct partition $B \cup N =
[n]$ whenever a sufficiently long step exists . This is because the standard affine scaling step
itself exhibits a polarized behaviour: we can simply select $B$ as the set of coordinates $i$ where $|\Delta x^\as_i /x^{(t)}_i|<|\Delta s^\as_i/s^{(t)}_i|$, i.e., the relative primal movement is smaller than the relative dual movement (see \Cref{def:associated_partition}).

\paragraph{Trust Region Programs and Subspace LLS}

The trust region programs introduced by Lan, Monteiro, Tsuchiya~\cite{LMT09} provide a good starting point for defining our movement direction  $\Delta z^{(t)} = (\Delta x^{(t)},$ $\Delta s^{(t)}) \in W \times W^\perp$ from an iterate $z^{(t)}=(x^{(t)},s^{(t)}) \in \CP[\mu_1,\mu_0]$ and a
given a partition $[n]=B \cup N$: 
\begin{equation}\label{eq:trust-region}
\begin{aligned}
\min_{\Delta x \in W} &\left\{\norm{(x^{(t)}_N+\Delta x_N)/x^{(t)}_N}: \norm{\Delta x_B/x^{(t)}_B} \leq \nu\right\}\, ,  \\
\min_{\Delta s \in W^\perp} &\left\{\norm{(s^{(t)}_B+\Delta s_B)/s^{(t)}_B}: \norm{\Delta s_N/s^{(t)}_N} \leq \nu\right\}\, ,
\end{aligned}
\end{equation}
where $\nu = O(\beta)$ is sufficient for the induced step to stay inside the $\cal N^2(\beta)$ neighborhood. We use the notation 
$\Delta x/x^{(t)} \coloneqq  (\Delta x_1/x^{(t)}_1, \dots, \Delta x_n/x^{(t)}_n)$ and
similarly for $\Delta s/s^{(i)}$. The norms
$\norm{x/x^{(t)}}$ and $\norm{s/s^{(t)}}$ 
 are the so-called primal and dual
\emph{local norms} at $x^{(t)}$ and $s^{(t)}$.\footnote{Recall the assumption that $(x^{(t)},s^{(t)})$ is on the central path.} By definition, the optimal
primal trust region direction $\Delta x^{*}$ achieves a maximal
multiplicative decrease on the coordinates in $N$ while `barely moving' the
coordinates in $B$ as measured in the local norm. The optimal dual
direction $\Delta s^{*}$ achieves the same on the dual side with the role of
$N$ and $B$ swapped. 

Note that these directions mesh well with polarization of the segment
$\CP[\mu_1,\mu_0]$. In particular, they reflect the idea that the coordinates
of $x\cp(\mu)$ in $N$ should be linearly scaling down while those in $B$ are
staying mostly fixed, and vice versa for $s\cp(\mu)$.  As shown in~\cite{LMT09} (see also \Cref{prop:trust-to-step-size}), moving in any
direction
$\Delta z^{(t)} = (\Delta x^{(t)}, \Delta
s^{(t)})$ corresponding to feasible solutions to~\eqref{eq:trust-region}, 
the normalized gap can be reduced as
\begin{equation}\label{eq:trust-region-progress}
\gap(z^{(t+1)})/ \gap(z^{(t)}) = O\left(\norm{(x^{(t)}_N+\Delta x^{(t)}_N)/x^{(t)}_N} + \norm{(s^{(t)}_B+\Delta s^{(t)}_B)/s^{(t)}_B}\right).  
\end{equation}
That is, we can achieve a drop that corresponds to the sum of primal and dual objective values.

In many ways, the trust region direction can be seen as the `optimal'
movement direction. However, \cite{LMT09}  solves the quadratic convex
programs in \eqref{eq:trust-region} in weakly polynomial time with dependence
on the vectors $b$ and $c$ in \eqref{LP_primal_dual}. It is not known whether
a strongly polynomial algorithm (with dependence only on $n$) exists.
Further, the analysis in \cite{LMT09} relies on combinatorial progress
measures adapted from the LLS analyses, which are to coarse to directly
measure progress on a polarized segment (in these analyses, combinatorial
progress is only guaranteed every $\Omega(\sqrt{n} \log(n+\bar{\chi}^*_\mA))$
iterations).   

\medskip

Instead of optimally solving \eqref{eq:trust-region}, we introduce what we
call \emph{subspace LLS steps} that yield `good enough' approximate solutions
to make rapid progress along a polarized segment. We restrict the set of primal and dual directions $\Delta x \in V^{(t)}$ and $\Delta
s \in U^{(t)}$ for subspaces $V^{(t)}\subseteq W$ and
$U^{(t)}\subseteq W^\perp$ satisfying:
\[
\begin{aligned}
\forall\ \Delta x \in V^{(t)}, \norm{\Delta x_B/x_B^{(t)}} &\leq \thresh \norm{\Delta x_N/x_N^{(t)}}\, ,\\
\forall\ \Delta s \in U^{(t)}, \norm{\Delta s_N/s_N^{(t)}} &\leq \thresh \norm{\Delta s_B/s_B^{(t)}}\, ,
\end{aligned}
\]
where we set $\thresh = O(\nu/\sqrt{n})$. We call any such subspaces
$V^{(t)}$ and $U^{(t)}$ \emph{cheap lift subspaces} with lifting cost
$\thresh$. Note that for any partial vector $\Delta x_N \in \pi_N(V^{(t)})
\subseteq \pi_N(W)$, the subspace $V^{(t)}$ provides a way to ``lift''
$\Delta x_N$ to a full vector $(\Delta x_B, \Delta x_N) \in W$ at a ``cost''
of $\norm{\Delta x_B/x_B} \leq \tau \norm{\Delta x_N/x_N}$. Similarly, $U^{(t)}$
provides a way to lift partial vectors in $\pi_B(U^{(t)}) \subseteq
\pi_B(W^\perp)$ cheaply into $W^\perp$. We formally define the associated
`lifting operator' in Section~\ref{sec:lifting_maps}. With a cheap lift
subspace at hand, it is not hard to show that any optimal solution to the
subspace constrained trust-region program automatically satisfies the norm
constraints in program~\eqref{eq:trust-region}. Therefore, by restricting
ourselves to search directions $\Delta x \in V^{(t)}$ and $\Delta s \in
U^{(t)}$ as above, we can solve simple unconstrained minimum-norm point
problems in the local norms while still guaranteeing that the computed
search directions are feasible for the trust-region program.

There is a lot of flexibility to choose these subspaces. A canonical choice
is a lifting of the subspace spanned by the singular vectors of the `lifting operator'
(see Definition~\ref{def:lifting-map}) whose corresponding singular values
are at most $\thresh$. While singular values and their corresponding singular
subspaces cannot be computed exactly, they can be closely approximated in
strongly polynomial time (see \Cref{sec:sing-val}). Furthermore, as we show
in \Cref{sec:slls-alg}, even very coarse approximations of the singular values and subspaces will suffice.

\paragraph{Analyzing Subspace LLS} 
At each iteration, our algorithm computes the affine scaling steps and the
subspace LLS steps as above, and uses the one that enables more progress
along the central path. For simplicity of exposition, we use the canonical
cheap lift subspaces $U^{(t)}, V^{(t)}$ to compute the subspace LLS
direction $\Delta z^{(t)}$ as above. 

Let us now explain the key idea in showing that subspace LLS steps can reach
the end of the current $\polarized$-polarized segment $\CP[\mu_1,\mu_0]$ in
$O(n^{1.5} \log(n/\polarized))$ iterations.  Let $k = \Omega(\sqrt{n}
\log(n/\polarized))$.  Given any iterate $z^{(t)} \in \CP[\mu_1,\mu_0]$, if
$\gap(z^{(t+k)})>\mu_1$---i.e., we have not reached the end of the
segment---then we show that  both $\dim(U^{(t+k)}) > \dim(U^{(t)})$ and
$\dim(V^{(t+k)}) > \dim(V^{(t)})$. The overall bound follows since this can
occur at most $n$ times. 

To get this result, we analyze the evolution of what we call the
`\idealdirection' at $z^{(t)}$, which we define to be $z\cp(\mu_1)-z^{(t)}$,
i.e., the difference between the current iterate and the end of the segment.
A crucial observation is that if $z\cp(\mu_1)-z^{(t)}$ were a feasible
solution to~\eqref{eq:trust-region}, then following this direction would get
to within a $\poly(n/\polarized)$ factor for the end of the segment in one
step (though we do not know how to compute it). Furthermore, the
\idealdirection is never far from being feasible, in particular, it is
feasible if the bound of $\nu$ is replaced by $O(n)$. Subspace LLS steps will
allow us to leverage the \idealdirection via the following dichotomy. Given
an iterate $z^{(t)}$, either the \idealdirection $z\cp(\mu_1)-z^{(t)}$ is
mostly ``aligned'' with the LLS subspaces $U^{(t)} \times V^{(t)}$, in which
case the LLS step brings us close to the end of the segment, or if not, it
brings us close to the time where the cheap lift subspaces increase in dimension. In the latter case, we crucially use
the polarization property to analyze the evolution of the singular values of
lifting operators. In both cases, the notion of close will mean that decreasing
the gap by an additional $\poly(n/\polarized)$ will be sufficient enter a new
part of the segment. In particular, $O(\sqrt{n}\log(n/\polarized))$ additional iterations will suffice to make the desired progress. 

This concludes our overview of the proof of \Cref{th:main_upper_bound-curve};
the detailed argument is presented in \Cref{sec:running-time-analysis}. In
\Cref{sec:amortized}, we present an amortized analysis that yields the
stronger bound in \Cref{th:main_upper_bound-slc} in terms of the
straight-line complexities. In particular, given a piecewise linear curve in
the central path neighborhood where the subsequent pieces are polarized with
partitions $(B^{(1)},N^{(1)})$, $(B^{(2)},N^{(2)}),\ldots,(B^{(T)},N^{(T)})$,
we show that the number of iterations can be bounded by 
\begin{equation}
  O\Big(\sqrt{n} \log(n/\polarized)\sum_{t=1}^{T} (|N^{(t)}\Delta N^{(t-1)}|+1)\Big), \label{eq:amortized-iteration-bound}
\end{equation}
with the convention that $N^{(0)} =
\emptyset$ (see \Cref{thm:amortized_main}). This can be better than the
previous bound $O(n^{1.5} \log(n/\polarized)T)$ if the number of indices
changing between polarizing partitions in subsequent polarized segments is
small compared to $n$. The proof of \Cref{th:main_upper_bound-slc} now
follows by combining~\eqref{eq:amortized-iteration-bound} together with
\Cref{thm:delta-slc}.

\subsection{Related Work}\label{sec:related_work}

Interior points methods have been a tremendously active and fruitful research area since the seminal works of Karmarkar~\cite{Karmarkar84} and
Renegar~\cite{Renegar1988} in the 80's. Remarkable advances have been made both in speed as well as applicability of IPMs.
We first briefly review works that---unlike the present paper---aim for $\eps$-approximate solutions. A key ingredient has been the use of different, self-concordant barrier functions. Like the logarithmic barrier, every such function gives rise to a notion of central path. 
In the general setting, the iteration complexity to get an $\eps$-approximation of the optimal value is bounded by $O(\vartheta_f^{1/2} \log \eps^{-1})$ for the complexity parameter $\vartheta_f$. General bounds on self-concordant barriers were given by Nesterov and Nemirovski \cite{Nesterov1994}, improved recently by Lee and Yue~\cite{LeeYue18} and Chewi~\cite{Chewi:2023}. 
Specific barrier functions include  Vaidya's volumetric barrier~\cite{Vaidya1989}, the entropic barrier by Bubeck and Eldan~\cite{BE14}, and the weighted log-barrier by Lee and Sidford~\cite{LS14, LS19}. 

Recent improvements make use of efficient data structures to amortize the cost of the iterative updates, and 
work with approximate computations, see Cohen, Lee and
Song~\cite{CLS19}, van den Brand~\cite{vdb20}, and van den Brand, Lee,  Sidford, and  Song~\cite{vdb20-tall-dense}.
For special classes of LP such as network flow and matching problems, even faster algorithms have been obtained using, among other techniques, fast
Laplacian solvers~\cite{ST04}, see e.g.~\cite{Axiotis2022,Daitch2008,Gao2022,Madry2013,Brand2020,Brand2021}, culminating in the very recent almost-linear time minimum-cost flow algorithm~\cite{Chen2022maximum}.

\medskip

Layered least squares IPMs, initiated by Vavasis and Ye \cite{Vavasis1996} find exact optimal solutions and their running time bound is independent of $b$ and $c$. Improved LLS algorithms were given by Megiddo, Mizuno, and Tsuchiya  \cite{MMT98} and Monteiro and Tsuchiya  \cite{Monteiro2003,MonteiroT05}. 
As discussed previously, scaling invariant algorithms with a $\bar\chi^*_{\mA}$ dependence are the Trust Region algorithm by
Lan, Monteiro, and Tsuchiya \cite{LMT09}, and the LLS algorithm \cite{DHNV20} that relies on approximating circuit imbalances. 

\medskip

There is an interesting connection between IPMs and differential geometry.
Sonnevend, Stoer, and Zhao~\cite{Sonnevend1991} introduced a 
primal-dual curvature concept for the central path, and related the curvature integral to the iteration complexity of IPMs.
Monteiro and
Tsuchiya~\cite{Monteiro2008} showed that a curvature integral is bounded by
$O(n^{3.5} \log(\bar{\chi}^*_{\mA}+n))$. This has been extended to SDP and symmetric cone programming~\cite{kakihara2014}, and it was also studied in the context of information geometry \cite{kakihara2013}.

\medskip

Relating central paths and simplex paths has been mainly been explored in the context of building LPs with pathological properties. On top of the construction of~\cite{AGV22} that we already discussed, Deza, Nematollahi and Terlaky~\cite{DezaNT} built a Klee--Minty cube with exponentially many redundant inequalities where the central path is distorted into the neighborhood of a simplex path that visits all $2^n$ vertices. 

\medskip

The max central path studied in this paper is related to the tropical central path in~\cite{ABGJ18,ABGJ21,AGV22}. The latter arises when studying parametric families of LPs where the input $(\mA, b, c)$ depends on a parameter~$t > 1$. The tropical central path is defined as the log-limit, \ie, the limit as $t \to \infty$ of the image under the map $z \mapsto \log_t z = \frac{\log z}{\log t}$, of the central path of these LPs. In~\cite{ABGJ18,ABGJ21,AGV22}, it was shown that the tropical central path corresponds to the greatest point (entrywise) of the log-limit of the feasible sets of~\eqref{eq:max_cp}. This turns out to be precisely the log-limit of the max central path.

\medskip

As stated earlier in the introduction, there is no known polynomial time
variant of the simplex method which traverses the edges of a base polyhedron.
Deviating from this model, Kelner and Spielman~\cite{kelner2006randomized}
gave a weakly-polynomial time LP algorithm which uses the simplex method on
random perturbations of a polyhedron. Specifically, their algorithm
determines unboundedness for $\cal P \coloneqq \{x \in \R^n: \mA x \leq \1\}$, which is strongly polynomially
equivalent to general LP. They apply a shadow vertex simplex method on $\cal P$
with randomly perturbed right hand sides as part of a subroutine which either
computes a point of large norm -- used to make $\cal P$ ``round'' -- or computes
a suitable certificate of boundedness.  While based on the shadow vertex
simplex method, their algorithm is inherently weakly polynomial and does not
admit a running time that only depends on $n$ and $m$.    

\subsection{Organization of the Paper}

In \Cref{sec:prelims}, we introduce our notation, the basic tools in linear
algebra we require (\Cref{sec:lin-alg-prelim}), as well as the important
properties of the central path and its neighborhoods
(\Cref{sec:ipm-prelims}). We also discuss the affine scaling steps used in
predictor-correct methods (\Cref{sec:aff-lls}). \Cref{sec:polar} deals with the polarized segments of the
central path and their connection with linear segments in the wide
neighborhood. \Cref{sec:max-cp} studies the
max central path and shows how to use it to decompose the central into polarized segments, proving \Cref{thm:delta-slc}. \Cref{sec:trust-region}, proves important 
properties of trust-region steps, and in particular, how to identify optimal
trust-region partitions. \Cref{sec:lifting-slls} develops the theory of
subspace layered least squares directions and cheap lift subspaces, and also gives
algorithms for computing them. From the theory side, the fundamental concept
of a lifting operator is introduced in \Cref{sec:lifting_maps}, and the relationship between cheap
lift subspaces and approximate singular subspaces of the lifting operator is
given in \Cref{sec:cl}. The algorithm~\nameref{alg:cheap_lift}, used to compute these subspaces, is presented analyzed in \Cref{sec:comp-cl}. The algorithm~\nameref{alg:subspace_ipm} is introduced in
\Cref{sec:slls-alg}, where the analysis for polarized segments is presented
in \Cref{sec:running-time-analysis}, proving
Theorem~\ref{th:main_upper_bound-curve}. Section~\ref{sec:amortized} presents
an amortized analysis, leading to the proof of
Theorem~\ref{th:main_upper_bound-slc}. In \Cref{sec:sing-val}, we present a
deterministic strongly polynomial algorithm for computing approximate
singular value decompositions, which is needed to compute the cheap lift
subspaces used by the IPM. Finally, \Cref{sec:self-concordant} discusses
interior point methods with self-concordant barrier functions, and proves
Theorem~\ref{th:coro-lower-bound}.
 Omitted proofs are deferred to the Appendix.
 
\section{Preliminaries} \label{sec:prelims}

\subsection{Notation} We let $\R_{++}$ denote the set of positive reals, $\R_{+}$ the set of
nonnegative reals, and $\mathbb{N} = \{1,2,\dots\}$ denote the natural numbers. For $n \in \mathbb{N}$, we let $[n]\coloneqq\{1,2,\ldots,n\}$. $B,N \subseteq [n]$ form a partition of $[n]$
if $B \cup N = [n]$ and $B \cap N = \emptyset$. We say that $(B,N)$ is a
\emph{non-trivial} partition of $[n]$ if additionally $B \neq \emptyset$ and
$N \neq \emptyset$, and is a \emph{trivial} partition otherwise.  For $a \in \R$, $\lceil a \rceil \in \mathbb{Z}$ is the
smallest integer greater than or equal to $a$, and $\lfloor a \rfloor \in \mathbb{Z}$ is the largest integer less than or equal to $a$. For $a \in \R_{++}$, we let $\log(a)$ denote the natural logarithm of $a$. For $a \in \R$ and $S \subseteq \R$, we define $1[a \in S] = 1$ if $a \in S$ and $0$ otherwise.  Throughout, we consider inequalities of vectors coordinate-wise. Let $e^i\in \R^n$ denote the
$i$th standard basis vector, and $\1_n,\0_n \in \R^n$ denote the vector of
all ones and all zeros respectively. For $x \in \R^n$, we let $\supp(x) \coloneqq \{i \in [n]: x_i \neq 0\}$ denote the support of $x$. For two points $x,y\in\R^n$, we let
$[x,y]\coloneqq\{\lambda x+(1-\lambda)y:\, 0\le\lambda\le1\}$ denote the
line-segment connecting $x$ and $y$. For two sets $S, T$ we let $S \Delta T
\coloneqq (S \setminus T) \cup (T \setminus S)$ be the symmetric difference
between $S$ and $T$. If $S,T \subseteq \R^n$, we define their Minkowski sum
$S+T \coloneqq \{s+t: s \in S, t \in T\}$. For a vector $t \in \R^n$, we let
$S+t \coloneqq S + \{t\}$ for notational simplicity. For a function $f \colon S
\rightarrow \R$ and $T \subseteq S$, we let $\argmin_{x \in T} f(x) \coloneqq
\{y \in T: f(y) = \min_{x \in T} f(x)\}$ denote the set of minimizers of $f$
with respect to $T$. By convention, $\argmin_{x \in T} f(x) = \emptyset$ if
the minimum value of $f$ inside $T$ is not attained. 

The standard inner product between two vectors is denoted by $\pr{x}{y}=x^\top
y$, for $x,y \in \R^n$, and the Euclidean norm by $\norm{x} \coloneqq
\sqrt{\sum_{i=1}^n x_i^2}$. We further let $\norm{x}_\infty \coloneqq \max_{i
\in [n]} |x_i|$ denote the $\ell_\infty$ norm, and $\norm{x}_1 \coloneq \sum_{i=1}^n |x_i|$ denote the $\ell_1$ norm. 

For a vector $x\in \R^n$, we let $\diag(x)\in \R^{n\times n}$ denote the
diagonal matrix with $x$ on the diagonal. For $x,y\in \R^n$, we use the
notation $xy\in \R^n$ for the coordinate-wise (Hadamard) product $xy=\diag(x)y=(x_iy_i)_{i\in
[n]}$. For $\xi \in \R^n_{++}$ and a linear subspace $W
\subseteq \R^n$, we use the notation $\xi W \coloneqq \{\xi w: w\in W\}$. For
$p\in \mathbb{Q}$ and $x \in \R^n$, we also use the notation $x^{p}\in\R^n$
to denote the vector $(x_i^{p})_{i\in [n]}$, where we will always ensure that
the corresponding coordinates are well-defined. Similarly, for $x \in \R^n$
and $y \in \R^n_{++}$, we let $x/y \in\R^n$ denote the vector
$(x_i/y_i)_{i\in [n]}$. For vectors $x^{(1)},\dots,x^{(k)} \subseteq \R^n$,
we define their linear span as $\lspan(x^{(1)},\dots,x^{(k)}) \coloneqq
\{\sum_{i=1}^k \lambda_i x^{(i)}: \lambda \in \R^k\}$. For non-empty set $I
\subseteq [n]$, we define $\pi_I \colon \R^n \rightarrow \R^I$ to be the
coordinate projection onto $I$, that is $\pi_I(x) \coloneqq x_I$, $\forall x \in \R^n$. By convention, we define $\R^\emptyset \coloneqq \{0\}$ to be the
trivial subspace, and we let $\pi_\emptyset(x) \coloneqq x_{\emptyset}
\coloneqq 0$, $\forall x \in \R^n$.  We also define $\R^n_I \coloneqq  \{x
\in \R^n: \supp(x) \subseteq I\}$ to be the set of vectors in $\R^n$ with
support contained in $I$. 

For a matrix $\mM \in \R^{m \times n}$, we let $\mM^\T \in \R^{n \times m}$
denote the matrix transpose satisfying $(\mM^\T)_{ji} \coloneqq \mM_{ij}$,
for $i \in [m], j \in [n]$. For subsets $S \subseteq [m]$, $T \subseteq [n]$,
we define $\mM_{S,T}$ to be the matrix with input space $\R^T$ and output
space $\R^S$ induced by the columns in $T$ and rows in $S$ of $\mM$.  We
further use the notation $\mM_T \coloneqq \mM_{\mdot,T} \coloneqq
\mM_{[m],T}$ to index the corresponding columns of $\mM$, and
$\mM_{S,\mdot} \coloneqq \mM_{S,[n]}$. By convention, we let $\mM_{S,\emptyset} 
\coloneqq [\0_S]$, $\mM_{\emptyset,T} \coloneqq [\0_{T}]^\T$ and
$\mM_{\emptyset,\emptyset} \coloneqq [0]$ (recall that $\R^\emptyset=\{0\}$).

\paragraph{Subspace Formulation of Linear Programming}
 It will be more convenient for our algorithm and analysis to represent
\eqref{LP_primal_dual} in an equivalent subspace language.  Throughout the
paper, we let $W=\ker(\mA) \subseteq \R^n$ denote the kernel of $\mA$ and
$W^\perp = \im(\mA^\T)$ denote the image of the transpose of $\mA$ (see
\Cref{sec:lin-alg-prelim} for formal definitions). Using this notation,
\eqref{LP_primal_dual} can be written in the form

\begin{align} \label{LP-subspace}
\begin{aligned}
\min \; & \pr{c}{x} \\
x &\in W + d \\
x &\geq \0_n,
\end{aligned}
\quad\quad
\begin{aligned}
\max \; &  \pr{d}{c-s} \\
s  &\in W^\perp+c \\
s &\geq \0_n,
\end{aligned}
\end{align}
where $d \in \R^n$ is any solution $\mA d = b$. A natural choice of $d$ is the minimum norm
solution, namely, $d = \argmin\{\norm{x}: \mA x=b\}$. 

Note that $s \in W^\perp+c$ is equivalent to $\exists y \in \R^m$ such that
$\mA^\T y + c = s$. Hence, the 
 original variable $y$ is implicit.
The feasible regions can be written as
 \[
 \Primal = \{x \in \R^n: x \in W+d, x \geq \0\}\, , \quad
 \Dual = \{s \in \R^n: s \in W^\perp+c, s \geq \0\}\, .
 \]

 \subsection{Models of Computation}\label{sec:comp-models}
As is standard in the interior point literature, our algorithms use the \emph{real RAM model of computation}. In this model, the input is given by $K$ real numbers. The algorithm may perform a sequence of basic arithmetic operations ($+$, $-$, $\times$, $/$), and comparisons on real numbers.
The algorithm is \emph{polynomial in the real RAM model} if the total number of such operations is polynomially bounded in $K$. In the case of \eqref{LP_primal_dual}, the input numbers are the entries of $(\mA,b,c)$ and $K=n\times m+n+m$.
We note that while square roots appear in the paper, they are only used in the analysis and not in the actual algorithm.

Consider now a problem in the Turing model where the input is given by $K$ integers.
An algorithm is \emph{strongly polynomial in the Turing model} if it performs $\mathrm{poly}(K)$ arithmetic operations and comparisons as in the real RAM model, and additionally, the algorithm is in PSPACE: the bit-complexity of all numbers during the computations remains bounded polynomially in the input bit-complexity.

The algorithm in \Cref{th:main_upper_bound-slc} is polynomial in the real RAM model as long as the straight-line complexity of the underlying linear program is polynomial. Moreover, every single iteration is strongly polynomial in the Turing model in the following sense. The input of an iteration is given by the current iterate $z=(x,s)$ and the matrix $\mA$. If $(x,s)$ is a rational vector and $\mA$ is a rational matrix, then all computations during a single iteration are polynomially bounded in the bit-length of $(\mA,x,s)$. 

However, even if the total number of iterations is strongly polynomial, this does not suffice for the entire algorithm to be strongly polynomial. This is because the input of each iterate is the output of the previous one, and the bit complexity of the current iterate can increase polynomially in every iteration. Obtaining stronger guarantees in the Turing model requires additional rounding that goes beyond the scope of this paper.

\paragraph{Strong Polynomial Linear Algebra} Throughout the paper, we
will strongly make use of the fact that many basic linear algebraic operations on
matrices, such as computing matrix inverses, powers, products and determinants can be performed in strongly polynomial time. The foundational
result from this perspective is that of Edmonds~\cite{edmonds1967systems},
who showed that Gaussian elimination can be implemented in strongly
polynomial time, where special care is required to maintain reduced
representations of rational numbers without performing greatest common
divisor computations (as these are not strongly polynomial). Further results by
Strassen~\cite{strassen1973vermeidung} and then
Berkowitz~\cite{berkowitz1984computing} showed that one can compute the
determinant of an integer matrix in strongly polynomial time without
division, which in particular allows one to implement Gaussian elimination on
a matrix without the need to maintain reduced representations (the entries of
each iterate can be expressed directly as as a ratio of integer matrix
determinants). For an overview of strongly polynomial linear algebra, we
refer the interested reader to \cite[Section 1.4]{gls}. 

\subsection{Linear Algebra Preliminaries}
\label{sec:lin-alg-prelim}

In this section, we review fundamental concepts in linear algebra from an
operator theoretic perspective, including the notion of adjoint and
pseudoinverse operators, orthogonal projections, and the singular value
decomposition. Throughout the exposition, we restrict to operators between
linear subspaces of $\R^n$ where we fix the standard inner product. The
concepts developed here will be needed for the definition of lifting operators
and their duality properties (covered in \Cref{sec:lifting_maps}), for the
computation of pseudoinverses, and to make the relationship between
operators and their associated matrices precise. For a more thorough background, the
interested reader may consult the following reference
textbooks~\cite{young1988introduction,campbell2009generalized,axler2023linear,Bhatia1997}.   

For a linear subspace $U \subseteq \R^n$, a basis $u_1,\dots,u_d$ of $U$ is
any maximal subset of linearly independent vectors in $U$. The dimension
$\dim(U) \coloneqq d$ is the number of vectors in any basis. For any two
linear subspaces $V,W \subseteq \R^n$, we have that $\dim(V) + \dim(W) =
\dim(V \cap W) + \dim(V+W)$, that is, dimension is a modular function
over subspaces.

\begin{definition}[Linear Operator] 
$T: U \rightarrow V$ is a linear operator between linear subspaces $U
\subseteq \R^n, V \subseteq \R^m$ if $T(a x + b y) = a T(x) + b T(y)$, for
all $x,y \in U$ and $a,b \in \R$. We will only consider linear operators
defined on linear subspaces of some $\R^n$, on which the standard inner
product is always defined. For $x \in U$, we will often write $Tx \coloneqq
T(x)$ for simplicity of notation. 

For a matrix $\mM \in \R^{m \times n}$, the linear operator $\cal T(\mM):
\R^n \rightarrow \R^m$ induced by $\mM$ is defined by $\cal T(\mM)(x) = \mM
x$, $\forall x \in \R^n$. Similarly, for an operator $T: \R^n \rightarrow
\R^m$, there is a unique matrix $\cal M(T) \in \R^{m \times n}$, defined by
$\cal M(T)_{ij} = (e^i)^\T T(e^j)$, $\forall i \in [m], j \in [n]$,
satisfying $\cal T(\cal M(T)) = T$. More generally, for $T: U \rightarrow V$,
we define $\cal M(T) \in \R^{m \times n}$ to be the unique matrix satisfying
$T(u) = \cal M(T) u$, $\forall u \in U$, and $\cal M(T)^\T e^i \in U$,
$\forall i \in [m]$.
\label{def:operator}
\end{definition}

\begin{remark}
\label{rem:operator-matrix}
Given the correspondence above, in the remainder of the paper, we will
identify an $m \times n$ matrix $\mM$ with its associated operator $\cal T(\mM):
\R^n \rightarrow \R^m$ (noting that $\cal M(\cal T(\mM)) = \mM$). All operator theoretic definitions will transfer directly to matrices via this
identification. Specifically, we will treat the matrix $\mM$ as the operator $\cal T(\mM)$ whenever we need to apply an operator theoretic concept to $\mM$. 

It is useful to note that the input space $U = \R^n$ and output space $V =
\R^m$ in the definition of the $\cal T(\mM)$ are chosen maximally, that is,
they consist of the entire ambient input and output spaces. For a general
operator $T\colon U \rightarrow V$, $U \subseteq \R^n$, $V \subseteq \R^m$,  note that we may have $\cal T(\cal M(T))\neq T$ (e.g., $T$ may
be invertible while $\cal T(\cal M(T))\colon \R^n \rightarrow \R^m$ need not be).      
\end{remark}

We let $\mI_U \colon U \rightarrow U$ denote the identity operator on $U$, where we
use the shorthand $\mI_n$ for the identity on $\R^n$. We let $\ker(T) = \{x \in U: Tx = \0_m\}$ and $\im(T) = \{Tx: x \in U\}$ denote
the kernel and image of $T$ respectively, and let $\rank(T) = \dim(\im(T))$.
A fundamental identity is $\dim(U)=\rank(T)+\dim(\ker(T))$. $T$ is invertible
if there exists a $T^{-1}\colon V \rightarrow U$ satisfying $T^{-1} \circ T =
\mI_U$, $T \circ T^{-1} = \mI_V$. $T$ is
invertible if and only if $\ker(T) = \{\0\}$ and $\im(T) = V$. For a linear
subspace $S \subseteq U$, we define the restricted operator $\restr{T}{S}\colon S
\rightarrow V$ to be the operator $T$ restricted to the subspace $S$. For $X
\supseteq \im(\restr{T}{S})$, we also define 
$\restro{T}{S}{X}: S \rightarrow X$, which modifies both the input and
output space of $T$ (note that the condition on $X$ ensures that
$\restro{T}{S}{X}$ is well-defined).

For two linear subspaces $V,W \subseteq \R^n$, we write $V \perp W$ to
indicate that $V$ and $W$ are orthogonal, that is, $\pr{v}{w} = 0$, $\forall
v \in V, w \in W$. We define $W^\perp \coloneqq  \{x \in \R^n: \pr{x}{y} = 0,
\forall y \in W\}$ as the orthogonal complement of $W$. The orthogonal
complement satisfies $W + W^\perp = \R^n$, $W \cap W^\perp = \{\0_n\}$, and
$(W^\perp)^\perp = W$. We say that $w_1,\dots,w_d$ are orthonormal vectors in $W$, if
$\pr{w_i}{w_j} = 1$ whenever $i=j$ and $0$ otherwise, for $i,j \in [d]$.
Additionally, $w_1,\dots,w_d$ form an orthonormal basis of $W$ if $d =
\dim(W)$. 

The following identities for orthogonal complements will be used
throughout the paper. We state them without proof.

\begin{proposition}\label{prop:sum-complement-identity}
For linear subspaces $V,W \subseteq \R^n$, $(V+W)^\perp = V^\perp \cap
W^\perp$.
\end{proposition}

\begin{proposition}\label{prop:subspace-complement-identity}
For a partition $I \cup J = [n]$ and a linear subspace $W \subseteq \R^n$,
$\pi_I(W)^\perp = \pi_I((W + \R^n_J)^\perp) = \pi_I(W^\perp \cap \R^n_I)$
holds.  
\end{proposition}

For every linear operator, there is a corresponding adjoint operator, whose
properties we will use heavily. 

\begin{definition}[Adjoint Linear Operator]
\label{def:adjoint}
For a linear operator $T\colon U \rightarrow V$ between linear subspaces $U
\subseteq \R^n, V \subseteq \R^m$, we define the adjoint operator $\adj(T)\colon V
\rightarrow U$, to be the unique linear map satisfying $\pr{v}{T(u)} =
\pr{\adj(T)(v)}{u}$, $\forall u \in U, v \in V$. By uniqueness
of the adjoint, note $T=\adj(\adj(T))$. The inner products we use on $U$ and $V$ are the standard inner product on $\R^n$ and $\R^m$ respectively. 
\end{definition}

\begin{remark}
\label{rem:adjoint}
Letting $\cal M(T) \in \R^{m \times n}$ be the associated matrix as in
\Cref{def:operator} for $T \colon U \rightarrow V$, one has
$\cal M(\adj(T)) = \cal M(T)^\T$. That is, the adjoint operator
corresponds to the transpose of the associated matrix. This equality follows
from $\langle v, T(u) \rangle$ $= \langle v, \cal M(T) u \rangle = \langle \cal M(T)^\T
v, u \rangle$ and $\cal M(T)^\T v \in U$ by our assumptions on $\cal M(T)$ (in
particular, $\im(\cal M(T)^\T) \subseteq U$), and the uniqueness of the adjoint.  
\end{remark}

The following proposition collects relevant properties of the adjoint
that we will need. The proof is given in \Cref{sec:appendix-la}.

\begin{proposition}
\label{prop:adjoint}
Let $T \colon U \rightarrow V$ be a linear operator between $U \subseteq \R^n, V
\subseteq \R^m$. Then, the following holds:
\begin{enumerate}
\item \label{prop:adjoint-decomposition-v} $\ker(\adj(T)) \perp \im(T)$ and $\ker(\adj(T))+\im(T) = V$.
\item \label{prop:adjoint-decomposition-u} $\ker(T) \perp \im(\adj(T))$ and $\ker(T)+\im(\adj(T)) = U$.
\item \label{prop:adjoint-image} $\im(T) = T(\im(\adj(T)))$ and $\im(\adj(T))=\adj(T)(\im(T))$. \newline In particular, $\rank(T) = \rank(\adj(T))$. 
\end{enumerate}
\end{proposition}

\begin{definition}[Orthogonal Projection]
For a linear subspace $W \subseteq \R^n$, we define $\Pi_W\colon \R^n \rightarrow
\R^n$, the orthogonal projection onto $W$, to be unique linear operator
satisfying that $\Pi_W(x) \in W$ and $x-\Pi_W(x) \in W^\perp$, $\forall x \in
\R^n$. \label{def:orthogonal-projection}
\end{definition}

\begin{remark}
\label{rem:orthogonality}
For $x \in \R^n$ and $W \subseteq \R^n$, since $\pr{\Pi_W(x)}{x-\pi_W(x)}=0$, we have that $\norm{x}^2 = \pr{x}{x} = \norm{\Pi_W(x)}^2 + \norm{x-\Pi_W(x)}^2$. In particular, $\norm{x} \geq \norm{\Pi_W(x)}$.
\end{remark}

The following proposition states the important properties of orthogonal
projections that will be used many times throughout the paper. We state it
without proof. For a reference proving some of these properties, see for
example \cite[Theorem 6.57 and Theorem 6.61]{shores2018applied}. 

\begin{proposition}
Let $W \subseteq \R^n$ be a linear subspace. Then, the orthogonal projection
$\Pi_W: \R^n \rightarrow \R^n$ satisfies the following:
\begin{enumerate}
\item $\Pi_W$ is self-adjoint, that is, $\adj(\Pi_W) = \Pi_W$. In particular,
when interpreted as an $n \times n$ matrix, $\Pi_W$ is symmetric, that is $\Pi_W^\T = \Pi_W$.
\item $\mI_n - \Pi_W = \Pi_{W^\perp}$. 
\item $\Pi_W(x) = \argmin_{z \in W} \norm{x-z} = \argmin_{y \in x+W^\perp}
\norm{y}$, $\forall x \in \R^n$. In particular, $\Pi_W(x)$ is the unique
point in $W \cap (W^\perp + x)$.
\end{enumerate}
\label{prop:orth-proj-props}
\end{proposition}

The following proposition gives an explicit formula for associated matrices
of general operators in terms of projections, and provides some of their
basic properties. The proof is given in \Cref{sec:appendix-la}. 

\begin{proposition}
\label{prop:associated-matrix}
Let $T \colon U \rightarrow V$, $U \subseteq \R^n$, $V \subseteq \R^m$ be a linear
operator. Then, $\cal M(T)_{ij} = (e^i)^\T T(\Pi_U e^j)$, $\forall i \in [m],
j \in [n]$, and $\cal T(\cal M(T)) = \restro{T}{U}{\R^m} \circ \restro{\Pi_U}{\R^n}{U}$, where $\Pi_U \colon \R^n \rightarrow \R^n$ is the orthogonal projection
onto $U$. Furthermore, $\im(\cal M(T)) = \im(T)$ and $\im(\cal M(T)^\T) = \im(\adj(T))$.
\end{proposition}

The next proposition shows that restricting the input to any subspace containing the image of the adjoint does not change the associated matrix. 

\begin{proposition} 
\label{prop:ass-mat-eq}
Let $T \colon U \rightarrow V$, $U \subseteq \R^n$, $V \subseteq \R^m$ be a linear
operator. Let $\wh U$ be a linear subspace satisfying $\im(\adj(T)) \subseteq \wh U \subseteq U$ and let $\wh T = \restr{T}{\wh U}$.  Then, $\cal M(T) = \cal M(\wh T)$.
\end{proposition}
\begin{proof}
By definition, $\cal M(T)(u) = T(u) =
\wh T(u)$, $\forall u \in \wh U$. Then by \Cref{prop:associated-matrix}, $\im(\cal M(T)^\T) = \im(\adj(T)) \subseteq \wh U$, and thus
$\cal M(T) = \cal M(\wh T)$ by uniqueness.  
\end{proof}

For the task of computing projection matrices, we will require the pseudoinverse operator.

\begin{definition}[Moore-Penrose Pseudoinverse]
\label{def:pseudoinverse}
Let $T \colon U \rightarrow V$, $U \subseteq \R^n$, $V \subseteq \R^m$ be a linear
operator. Then, the Moore-Penrose pseudoinverse $T^+\colon V \rightarrow U$ of $T$
is the unique linear operator satisfying $T^+ T(u) = u$, $\forall u \in
\im(\adj(T))$ and $\ker(T^+) = \ker(\adj(T))$.  \end{definition}

The following proposition gives the relation between pseudoinverses and
projection operators and the connection to least-squares problems. For a
proof, see \cite[Statements 6.69, 6.70]{axler2023linear}.
 
\begin{proposition}
\label{prop:pseudoinverse-projection}
Let $\mM \in \R^{m \times n}$. Then, the following hold:
\begin{enumerate}
\item \label{eq:pseudo-proj} $\mM^+ \mM = \Pi_{\im(\adj(\mM))}$ and $\mM \mM^+ = \Pi_{\im(\mM)}$.
\item \label{eq:pseudoinverse-regression}
  $\mM^+ v = \argmin_{u \in U : \mM u = \proj_{\im(\mM)}v} \norm{u}_2$.
\end{enumerate}
\end{proposition}

The next proposition shows that pseudoinverses and orthogonal
projections can be computed in strongly polynomial time. For a proof, see for
example \cite[Theorem 1.3.2 and Algorithm 1.3.1]{campbell2009generalized}.

\begin{proposition}
\label{prop:pseudo-compute}
Let $\mM \in \R^{m \times n}$ be a matrix and let  $r = \rank(\mM)$.  Then
the following holds:
\begin{enumerate}
\item If $r = m$, we have $\mM^+ = \mM^\T (\mM \mM^\T)^{-1}$.
\item If $r = n$, we have $\mM^+ = (\mM^\T \mM)^{-1} \mM^ \T$.
\item \label{it:lower-rank-formula-pseudo-inverse} If $r < \min\{m,n\}$, then for any rank factorization 
$\mM = \mA \mB$ satisfying $\mA \in \R^{m \times r}, \mB \in \R^{r \times n}$
with $r=\rank(\mA)=\rank(\mB)$, we have 
\[
\mM^+ = \mB^+ \mA^+ = \mB^\T (\mB \mB^ \T)^{-1} (\mA^\T \mA)^{-1} \mA^\T.
\]
\end{enumerate} 
In particular, on input $\mM$, the linear operators $\mM^+$, $\Pi_{\im(\mM)}
= \mM \mM^+$ and $\Pi_{\im(\mM^\T)} = \mM^+ \mM$ can be computed in strongly
polynomial time. 
\end{proposition}

We will need the notions of singular values and singular value decompositions
defined below. 

\begin{definition}[Singular Value Decomposition] A linear operator $T \colon U
\rightarrow V$, where $U \subseteq \R^n$ and $V \subseteq \R^m$ are linear
subspaces, admits a singular value decomposition (SVD) 
\begin{equation}
\label{eq:svd}
\cal M(T) = \sum_{i=1}^{\rank(T)} \sigma_i(T) v_i u_i^\T 
\end{equation}
where $v_1,\dots,v_{\rank(T)} \in V$ and $u_1,\dots,u_{\rank(T)} \in U$ are
orthonormal vectors in their respective subspaces and $\sigma_1(T) \geq
\cdots \geq \sigma_{\rank(T)}(T) > 0$. We define 
\begin{equation}
\label{eq:sing-vector}
\sigma(T) \coloneqq (\sigma_1(T),\dots,\sigma_{\dim(U)}(T)),
\end{equation}
where $\sigma_i(T) \coloneqq 0$ for
$i \in [\dim(U)] \setminus [\rank(T)]$, the complete vector of singular
values listed in non-increasing order, and 
\begin{equation}
\label{eq:pos-sing-vector}
\sigma^+(T) \coloneqq (\sigma_1(T),\dots,\sigma_{\rank(T)}(T)),
\end{equation}
the subvector of positive singular
values. We use the shorthand $\sigma_{\max}(T) \coloneqq \sigma_1(T)$ and $\sigma_{\min}(T) \coloneqq \sigma_{\dim(U)}(T)$. By convention, we let $\sigma_i(T) = 0$ for $i > \dim(U)$, though we
do no include this as an entry of $\sigma(T)$. 
We will often need to count the number of singular values inside an interval $I \subseteq \R_+$. For this purpose, we use the following notation:
\[
\cnt{T}{I} \coloneqq |\{i \in [\dim(U)]: \sigma_i(T) \in I\}|.
\]
We will also use the shorthand $\cnt{T}{t} \coloneqq \cnt{T}{[0,t]}$ for $t \geq
0$. Note that if $U = \{\0\}$, then $\rank(T) = 0$ and $\cnt{T}{\R_+} = 0$.

\label{def:svd}
\end{definition}

\begin{remark}
\label{rem:svd-uniqueness}
We remark that while the singular value decomposition $\cal M(T) =
\sum_{i=1}^{\rank(T)} \sigma_i(T) v_i u_i^\T$ above need not be unique (\ie,
if some of the non-zero singular values are equal), the vector of
singular values $\sigma(T)$ does not depend on the choice of singular value
decomposition.   
\end{remark}

We will require the standard relation between the adjoint and 
the singular value decomposition, which we state without proof.

\begin{proposition} 
Let $T \colon U \rightarrow V$ be a linear operator with singular value decomposition $\cal M(T) = \sum_{i=1}^{\rank(T)} \sigma_i v_i u_i^\T$ as in \Cref{def:svd}. Then,  $\adj(T) \colon V \rightarrow U$ has
singular value decomposition $\cal M(\adj(T))$ $= \sum_{i=1}^{\rank(T)} \sigma_i u_i v_i^\T$ and
$\sigma^+(T) = \sigma^+(\adj(T))$.  
\label{prop:sing-adjoint}
\end{proposition}

We will often need to relate the singular values of an operator with the
singular values of its restrictions and its associated matrix.
The following direct corollary of \Cref{prop:ass-mat-eq} gives the precise
relations. 

\begin{proposition}
Let $T \colon U \rightarrow V$, $U \subseteq \R^n, V \subseteq \R^m$ a linear operator, and let $\wh U$ be a linear subspace satisfying $\im(\adj(T)) \subseteq \wh U \subseteq U$. Then $\sigma^+(\cal M(T)) = \sigma^+(T) = \sigma^+(\restr{T}{\wh U})$.
\label{prop:restrict-sing}
\end{proposition}
\begin{proof}
By \Cref{prop:ass-mat-eq}, we have that $\cal M(T) = \cal M(\wh T)$, and
hence both operators have the same SVD in \Cref{def:svd}. Therefore, the
positive singular value vector is identical. The same holds for $\sigma^+(\cal M(T)) \coloneqq \sigma^+(\cal T(\cal M(T)))$. \end{proof}

One of the most useful ways to characterize singular values is via the
Courant-Fischer variational characterization. See for example \cite[Corollary
III.1.2]{Bhatia1997}, which gives the variational characterization for the
eigenvalues of $\adj(T) T$ (equal to squared singular values of $T$).

\begin{proposition}[Max-Min Principle for Singular Values]
Let $T \colon U \rightarrow V$ be a linear operator. Then, for $1 \leq i \leq
\dim(U)$, we have that
\begin{align}
  \sigma_i(T) &= \min_{\substack{\dim(S) \geq \dim(U)-i+1 \\ S \subseteq U}}\, \max_{x \in S\setminus\{\0\}} \frac{\norm{Tx}}{\norm{x}} \label{eq:min-max-sing} \\ &= \max_{\substack{\dim(S) \geq i \\ S \subseteq U}}\, \min_{x \in S \setminus \{\0\}} \frac{\norm{Tx}}{\norm{x}}\,. \label{eq:max-min-sing}
\end{align}
where $S$ ranges over the linear subspaces of $U$.
\label{prop:min-max-sing}
\end{proposition}

\begin{remark}
\label{rem:sing-dimension}
In~\Cref{prop:min-max-sing}, one can replace the
conditions $\dim(S) \geq$ \newline $\dim(U)-i+1$ and $\dim(S) \geq i$ with $\dim(S) =
\dim(U)-i+1$ and $\dim(S)=i$ respectively. In~\eqref{eq:min-max-sing}, any
subspace $S'$ of $S$ of dimension exactly $\dim(U)-i+1$ satisfies
\[
\max_{x' \in S' \setminus \{\0\}} \frac{\norm{T(x')}}{\norm{x'}} \leq \max_{x
\in S \setminus \{\0\}} \frac{\norm{T(x)}}{\norm{x}},
\]
and in~\eqref{eq:max-min-sing},
any subpace $S'$ of $S$ of dimension exactly $i$
satisfies 
\[
\min_{x' \in S' \setminus \{\0\}} \frac{\norm{T(x')}}{\norm{x'}} \geq 
\min_{x \in S \setminus \{\0\}} \frac{\norm{T(x)}}{\norm{x}}.
\]
The flexibility to use different dimensions will be useful in the sequel,
however.   
\end{remark}

\begin{remark}
\label{rem:sing-identity}
By \Cref{prop:min-max-sing} above, for any linear subspace $S \subseteq U$,
$1 \leq \dim(S) \leq \dim(U)$, we have the useful identities
\begin{align}
\sigma_{\max}(\restr{T}{S}) &\coloneqq \sigma_1(\restr{T}{S}) = \max_{x \in S \setminus \{\0\}} \frac{\norm{Tx}}{\norm{x}}, \nonumber \\
\sigma_{\min}(\restr{T}{S}) &\coloneqq \sigma_{\dim(S)}(\restr{T}{S}) = \min_{x \in S \setminus \{\0\}} \frac{\norm{Tx}}{\norm{x}}, \label{eq:sing-identity}
\end{align}
where $\restr{T}{S}$ is the operator $T$ restricted to $S$. If $\dim(S) = 0$,
recall that $\sigma_1(\restr{T}{S}) = 0$ by convention. 
\end{remark}

\begin{remark}
\label{rem:canonical-subspaces}
Let $\cal M(T) = \sum_{i=1}^r \sigma_i v_i u_i^\T$, $r = \rank(T)$, be the
SVD of $T$ and let $\mU = [u_1,\dots,u_r]$. For $i \in [\dim(U)]$, a
canonical choice for the subspace attaining the minimum
in~\eqref{eq:min-max-sing} is $\im(\mU_{\geq i})+\ker(T)$ for $i \in
[\rank(T)]$, and $\ker(T)$ for $i > \rank(T)$. For $i \in [\rank(T)]$, using
that $\im(\mU_{\geq i}) \subseteq \im(\adj(T))$ and $\im(\adj(T)) \perp \ker(T)$, this can be
verified as follows:
\begin{align*}
\sigma_1(\restr{T}{\im(\mU_{\geq i})+\ker(T)}) & = \max_{\substack{w \in \im(\mU_{\geq i}), z \in \ker(T) \\ w+z \neq \0_n}}
\frac{\norm{T(w+z)}}{\norm{w+z}} \\ 
&= \max_{\substack{w \in \im(\mU_{\geq i}), z \in \ker(T) \\ w+z \neq \0_n}}
\frac{\norm{Tw}}{\sqrt{\norm{w}^2+\norm{z}^2}} \\
&= \max_{\substack{w \in \im(\mU_{\geq i}) \\ w \neq \0_n}}
\frac{\norm{Tw}}{\norm{w}} 
= \max_{\substack{x \in \R^{r-i+1} \\ x \neq \0_{r-i+1}}}
\frac{\norm{T(\mU_{\geq i} x)}}{\norm{\mU_{\geq i} x}} \\
&= \max_{\substack{x \in \R^{r-i+1} \\ x \neq \0_{r-i+1}}}
\frac{\sqrt{\sum_{j=r-i+1}^r \sigma_j^2 x_j^2}}{\norm{x}} = \sigma_{r-i+1}, 
\end{align*}
where the second to last equality uses the orthonormality of $v_1,\dots,v_r
\in V$ and $u_1,\dots,u_r \in U$, and the last equality uses the
non-increasing order of $\sigma$. Note further that $\dim(\im(\mU_{\geq
i})+\ker(T)) = \dim(\im(\mU_{\geq i})) + \dim(\ker(T)) = (\rank(T)-i+1) +
\dim(U)-\rank(T) = \dim(U)-i+1$, and hence the subspace has the correct
dimension. For $i \in [\dim(U)] \setminus [\rank(T)]$, one trivially has 
$\sigma_1(\restr{T}{\ker(T)}) = 0$ and $\dim(\ker(T)) = \dim(U)-\rank(T) \geq
\dim(U)-i+1$, as needed. 

Similarly, a canonical choice for the subspace attaining the maximum
in~\eqref{eq:max-min-sing} is $\im(\mU_{\leq i})$ if $i \in [\rank(T)]$, and
$U$  if $i > \rank(T)$ (recalling that $\sigma_i(T) = 0$ for $i > \rank(T)$).
\end{remark}

\subsubsection{Approximate Singular Subspaces}
\label{sec:approx-sing-subspace}
In this subsection, we define approximate singular subspaces for general
linear operators and collect their main properties, which will be crucial for
the analysis and implementation of our IPM.  

\begin{definition}[Approximate Singular Subspace]
\label{def:apx-sing-sub}
Let $T \colon \src \rightarrow \targ$ be a linear operator. A linear subspace $S \subseteq
\src$ is a $\SVDA$-approximate singular subspace for $T$, for $\SVDA \geq 1$,
if $\sigma_1(\restr{T}{S}) \leq \SVDA \sigma_{\dim(\src)-\dim(S)+1}(T)$.
\end{definition}

\begin{remark}[Trivial Subspace] 
\label{rem:trivial-subspace}
The trivial subspace $S \coloneqq \{\0\} \subseteq \src$ is always
$1$-approximate singular subspace for $T$ as $\sigma_1(\restr{T}{S}) = 0 =
\sigma_{\dim(\src)+1}(T) = \sigma_{\dim(\src)-\dim(S)+1}(T)$ by convention.
\end{remark}

When computing approximate singular subspaces, we will use the following
lemma to relate approximate singular subspaces of an operator to those of its
associated matrix. The proof is given in \Cref{sec:appendix-apx-ss}.

\begin{lemma} 
Let $T: \src \rightarrow \targ$, $\src \subseteq \R^n, \targ \subseteq \R^m$ be a linear operator and let $\bar{T} = \cal M(T)$ be its associated matrix. Then, for a linear subspace $S$ satisfying $\ker(\bar{T}) \subseteq S \subseteq \R^n$, the following holds:
\begin{enumerate}
\item \label{lem:ass-op-mat-1} $\ker(T) \subseteq S \cap \src = \Pi_\src(S)$ and $\sigma_1(\restr{T}{S \cap \src}) = \sigma_1(\restr{\bar T}{S})$.
\item \label{lem:ass-op-mat-2} $S$ is a $\SVDA$-approximate singular subspace for $\bar{T}$ of
dimension $d \geq 0$ if and only if $S \cap \src$ is a $\SVDA$-approximate
singular subspace for $T$ of dimension $d-\dim(\src^\perp)$. 
\end{enumerate}
Moreover, for any $\tau \geq 0$, $\cnt{T}{\tau} + \dim(X^\perp) = \cnt{\bar{T}}{\tau}$.
\label{lem:ass-op-mat}
\end{lemma}

Our definition of an approximate singular subspace $S$ for $T$ of dimension
$d$ corresponds to $S$ be an approximate minimizer
of~\eqref{eq:min-max-sing}, that is 
\[
\sigma_1(\restr{T}{S}) \leq \SVDA \min_{C \subseteq \src, \dim(C)=d} \sigma_1(\restr{T}{C}) = \SVDA
\sigma_{\dim(\src)-d+1}(T).
\]
Given this, one may ask whether there is a
relationship between $S$ and maximizers for the complementary
program~\eqref{eq:max-min-sing}, that is, 
\[
\max_{D \subseteq
\dim(D)=\dim(\src)-d} \sigma_{\min}(\restr{T}{D}) = \sigma_{\dim(\src)-d}.
\]
The following lemma shows that in fact the orthogonal complement of $S$
inside $\src$ is indeed an approximate maximizer of~\eqref{eq:max-min-sing}, under the condition that
there is a gap between the singular values $\sigma_{\dim(\src)-d}(T)$ and
$\sigma_{\dim(\src)-d+1}(T)$. This will be important for the error analysis
of subspace LLS steps defined in \Cref{sec:slls}. The proof is given in \Cref{sec:appendix-apx-ss}.

\begin{lemma}
\label{lem:approx-complement}
Let $T \colon \src \rightarrow \targ$ be a linear operator. Let $S \subseteq \src$
be a $\SVDA$-approximate subspace with $0 \leq \dim(S) < \dim(\src)$ for $T$. Then, for $\bar{S} = \src
\cap S^\perp$, we have that
\[
\sigma_{\min}(\restr{T}{\bar{S}})^2 \geq \sigma_{\dim(\src)-\dim(S)}(T)^2 - \SVDA^2 \sigma_{\dim(\src)-\dim(S)+1}(T)^2.
\]
\end{lemma}

\subsection{Preliminaries on Interior-Point Methods}
\label{sec:ipm-prelims}
 
In this section, we recall standard properties of the central path and IPM that will be required for
our algorithm.
To ensure that the central path is well-defined, we assume that $\Primal$ and $\Dual$ admit strictly feasible solutions, \ie, the sets
$\Primalp$ and $\Dualp$ are both nonempty. We recall the notation $\pz\cp(\mu) = (\px\cp(\mu),\ps\cp(\mu))$ to denote the central path point at $\mu$ as in definition~\eqref{eq:CP-equations}.

Given $z = (x,s) \in \Primal \times \Dual$, we recall that the normalized duality gap is defined as $\gap(z) \coloneqq \frac{\pr{x}{s}}{n}$. The following identity is useful in  comparing duality gaps.

\begin{proposition}\label{prop:gap-formula} Given $x,x' \in W+d$, $s,s' \in W^\perp+c$, we have that
\[
\pr{x}{s} + \pr{x'}{s'} = \pr{x}{s'} + \pr{x'}{s}.
\] 
In particular, if $\pr{x'}{s'}=0$, then
\[
\pr{x}{s} = \pr{x}{s'} + \pr{x'}{s}.
\]
\end{proposition}
\begin{proof}
Since $x-x' \in W$ and $s-s' \in W^\perp$, we have that
\[
0 = \pr{x-x'}{s-s'} \Leftrightarrow 
\pr{x}{s} + \pr{x'}{s'} = \pr{x}{s'} + \pr{x'}{s}. 
\]
\end{proof}

The next proposition shows that the normalized duality gap is a linear
function for convex combinations of points.

\begin{proposition}[Linearity duality gap] \label{prop:gap-linear}
  For $x^{(1)},\dots,x^{(k)} \in W+d$, $s^{(1)},\dots,$ $s^{(k)} \in W^\perp+c$ forming the sequence $z^{(1)} = (x^{(1)},s^{(1)}), \dots, z^{(k)} = (x^{(k)},s^{(k)})$ and $\lambda \in \R^k$ such that $\sum_{i=1}^k \lambda_i = 1$, we have that
  \begin{equation*}
  \gap\left(\sum_{i=1}^k \lambda_i z^{(i)}\right) = \sum_{i=1}^k \lambda_i \gap(z^{(i)}) \,.
  \end{equation*}
\end{proposition}
\begin{proof}
Using that $\sum_{i=1}^k \lambda_i = 1$ and the orthogonality of $x^{(i)}-d \in W$ and $s^{(i)}-c \in W^{\perp}$ for all $i \in [k]$ we first get
\begin{align*}
\pr{\sum_{i=1}^k \lambda_i x^{(i)}}{\sum_{j=1}^k \lambda_i s^{(i)}} &= 
\pr{\sum_{i=1}^k \lambda_i (x^{(i)}-d) + d}{\sum_{j=1}^k \lambda_i (s^{(i)}-c) + c} \\
&= \pr{d}{c} + \sum_{i=1}^k \lambda_i(\pr{x^{(i)}-d}{c} + \pr{d}{s^{(i)}-c}) \\ &= \sum_{i=1}^k \lambda_i(\pr{d}{c} + \pr{x^{(i)}-d}{c} + \pr{d}{s^{(i)}-c}) \\
&= \sum_{i=1}^k \lambda_i\pr{x^{(i)}}{s^{(i)}}.
\end{align*}
Division by $n$ yields the respective normalized duality gap. 
\end{proof}

 A key property of the central path is \emph{`near monotonicity'},
 formulated in the following lemma, see \cite[Lemma 16]{Vavasis1996}.
 \begin{lemma}\label{lem: central_path_bounded_l1_norm}
 For the central path points at $0 < \mu' \leq \mu$, 
 we have
 \[
 \norm{\frac{x\cp(\mu')}{x\cp(\mu)} + \frac{s\cp(\mu')}{s\cp(\mu)}}_\infty \leq n\, .
 \]
 \end{lemma}
 In \Cref{lem:local-norm-monotone}, we also present an $\ell_1$-variant of this bound.

\paragraph{Central Path Neighborhoods}

The neighborhoods $\cal N^2(\beta)$ and $\cal N^{-\infty}(\theta)$ introduced in~\eqref{eq:l2_neighborhood} and~ \eqref{eq:wide_neighborhood} comprise the points $z = (x,s) \in \Primal \times \Dual$ such that the \emph{centrality error}, \ie, the norm of the vector $\frac{x s}{\gap(z)} - \1$, is bounded. They use of the $\ell_2$-norm and the $\ell_\infty$-seminorm $\norm{u}_{{-\infty}} \coloneqq \max_{1 \leq i \leq n}\max(0,-u_i)$, respectively. 

 We will often use the following proposition which is immediate from the definition of~${\cal N}^2(\beta)$.
 \begin{proposition}\label{prop:x_i-s_i}
 Let $z = (x, s) \in
   {\cal N}^2(\beta)$ for $\beta\in (0,1)$, and $\mu=\gap(z)$. Then for each $i \in [n]$
 \[
 (1-\beta)\mu\le x_i s_i \le (1+\beta)\mu\, .
 \]
 \end{proposition}
\begin{proof}
 By definition of $\cal N(\beta)$ we have for all $i \in [n]$ that
 $|\frac{\px_i\ps_i}{\mu} - 1| \le \norm{\frac{\px \ps}{\mu} - \1} \le \beta$ and so $(1-\beta) \mu \le x_i s_i \le (1+\beta) \mu$.
 \end{proof}

The following proposition gives a bound on the distance between a point $z\in {\cal N}^2(\beta)$ in the $\beta$-neighborhood and the corresponding central path point with the same normalized duality gap $\pz(\mu)$ for $\mu = \gap(z)$. See e.g., \cite[Lemma 5.4]{Gonzaga92} and \cite[Proposition 2.1]{Monteiro2003}.
\begin{proposition}\label{prop:near-central}
  Let $z = (x, s) \in
   {\cal N}^2(\beta)$ for $\beta\in (0,1/4]$ and  $\mu=\gap(z)$, and
   consider the central path point
 $\pz\cp(\mu)=(\px\cp(\mu),\ps\cp(\mu))$. For each $i\in[n]$,
 \[
 \begin{aligned}
 \frac{\px_i}{1+2\beta}\le \frac{1-2\beta}{1-\beta}\mathinner{} \px_i&\le \px\cp_i(\mu)\le
 \frac{\px_i}{1-\beta}\, ,\quad \mbox{and}\\
 \frac{\ps_i}{1+2\beta}\le  \frac{1-2\beta}{1-\beta}\mathinner{} \ps_i&\le \ps\cp_i(\mu)\le
 \frac{\ps_i}{1-\beta}\, .
 \end{aligned}
 \]
 \end{proposition}
 
We will need the following lemma regarding the near-optimality of the choice
$\gap(z)$ as $\pr{x}{s}/n$ for a point $z = (x,s)$ with respect to minimizing centrality error.

\begin{lemma}[{\cite[Lemma 4.4]{Monteiro2003}}] \label{lem:choice-of-mu}
  For $\beta\in(0,1/4]$, let 
$z = (x,s) \in \Primalp \times \Dualp$ and $\mu' > 0$ satisfy
$\norm{xs/\mu' - \1} \leq \beta$. Then, 
    \[
    (1-\beta/\sqrt{n})\mu'\leq \gap(z) \leq (1+\beta/\sqrt{n})\mu\, \quad \mbox{and}\quad z \in \cal N^2(\beta/(1-\beta))\, .
\]  
\end{lemma}

The next lemma relates a point in the wide neighborhood to the corresponding central path point.
\begin{lemma}
\label{lem:wide-to-cp}
Let $z=(\bx,\bs) \in \cal N^{-\infty}(\theta)$, $\theta \in [0,1)$. Then for $\mu
= \gap(z)$, and the corresponding central path point $z\cp(\mu) = (x\cp(\mu),s\cp(\mu))$, we have that
\[
\frac{1}{2n}\bx \leq x\cp(\mu) \leq \frac{2n}{1-\theta} \bx\quad \mbox{and}\quad
\frac{1}{2n}\bs \leq s\cp(\mu) \leq \frac{2n}{1-\theta} \bs\, .
\]  
\end{lemma}
\begin{proof}
We only prove the inequalities on $x'$; the proof of the inequalities on $s'$ is symmetric. Let $(x,s)\coloneqq(x\cp(\mu),s\cp(\mu))$. Using Proposition~\ref{prop:gap-formula}, for $i \in [n]$ we have that\[
\frac{\bx_i}{x_i} = \frac{\bx_i s_i}{\mu} \leq \frac{1}{\mu}(\pr{\bx}{s} + \pr{x}{\bs}) = \frac{1}{\mu}(\pr{\bx}{\bs} + \pr{x}{s}) = 2n.
\]
This proves the first inequality; note that this part does not use $z\in\cal N^{-\infty}(\theta)$, but only that $z\in\Primalp\times\Dualp$.
For the second inequality, $z\in\cal N^{-\infty}(\theta)$ by definition implies
\[
\frac{x_i}{\bx_i} \leq \frac{x_i\bs_i}{\mu(1-\theta)} \leq \frac{1}{\mu(1-\theta)}(\pr{\bx}{\bs}+\pr{x}{s} ) = \frac{2n}{1-\theta},
\]
as needed.
\end{proof}

\subsection{The Central Path Limit}\label{sec:cent-limit}
Assume that the polytope
$\Primal = \{x \in \R^n: x \in W+d, x \geq \0\}$ has a strictly positive solution and is bounded. Then, the \emph{analytic center} of $\Primal$ is the point
\[
\tilde x\coloneqq \arg\max_{x\in \Primal}\sum_{i=1}^n \log x_i\, .
\]
By Lagrangian duality, there exists a vector $\tilde v\in W^\perp$ such that $\tilde x \tilde v=\1_n$.
The limit point $z^\star=(x^\star,s^\star)$ of the central path corresponds to the analytic centers of the primal and dual optimal faces. Namely, assume $\Primalp,\Dualp\neq\emptyset$, and let $(B^\star,N^\star)$ denote the optimal partition, i.e., $x\in\Primal$ is optimal if and only if $\supp(x)\subseteq B^\star$ and $s\in\Dual$ is optimal if and only if $\supp(y)\subseteq N^\star$. Let
\[
\Primal^\star\coloneqq \Primal\cap \R^n_{B^\star}\quad\mbox{and}\quad
\Dual^\star\coloneqq \Dual\cap \R^n_{N^\star} 
\]
denote the set of primal and dual optimal solutions. Then, $\pi_{B^\star}(\Primal^\star)\subseteq\R^{B^\star}$ and $\pi_{N^\star}(\Dual^\star)\subseteq\R^{N^\star}$ are the projections of the optimal sets to the respective coordinate sets, and  the following holds; see e.g., \cite[Theorem I.30]{roos2005interior}
\begin{theorem}
Assume $\Primalp,\Dualp\neq\emptyset$. Then, the optimal partition $(B^\star,N^\star)$ and the limit $(x^\star,s^\star)=\lim_{\mu\searrow 0} (x(\mu),s(\mu))$ exist, and $x^\star_{B^\star}$ is the analytic center of $\pi_{B^\star}(\Primal^\star)$, and $s^\star_{N^\star}$ is the analytic center of $\pi_{N^\star}(\Dual^\star)$.
\end{theorem}
Consider now the output in  \Cref{th:main_upper_bound-slc} in the case $\mu_1=0$. Since $B=\supp(x^1)$ and $N=\supp(s^1)$ form a partition of $[n]$, it follows that the solutions $(x^1,s^1)$ are primal and dual optimal, and $B=B^\star$, $N=N^\star$. The additional output $(v^1,w^1)$ in \Cref{th:main_upper_bound-slc} satisfying 
$\norm{(x^1_B v^1_B,s^1_N w^1_N)-\1} \leq \beta$
provides Lagrange certificates that the solution $(x^1,s^1)$ is close to $(x^\star,s^\star)$. Namely, $v^1\subseteq \R^B$ certifies that $x^1_B$ is multiplicatively near the analytic center of $\pi_B(\Primal^\star)$, and $w^1\subseteq \R^N$ is analogously a certificate for $s^1_N$ and $\pi_N(\Dual^\star)\subseteq\R^N$ (see \Cref{prop:near-central}).

 \subsection{Predictor-Corrector Methods}\label{sec:aff-lls}
 Given $z = (x,s) \in \Primalp \times \Dualp$,
 the search directions commonly used in interior-point methods are
 obtained as the solution $(\Delta x,\Delta s)$ to the
 following linear system for some $\stepparam\in [0,1]$.
 \begin{align}
 \Delta x &\in W \label{aff:x} \\
 \Delta s &\in W^\perp \label{aff:s}\\
 s\Delta x + x \Delta s &=\stepparam \mu \1 -xs \label{aff:sum}
 \end{align}
 Predictor-corrector methods, such as the Mizuno--Todd--Ye
 Predictor-Corrector algorithm \cite{MTY}, alternate
 between two types of steps. In \emph{corrector steps}, we use $\stepparam=1$. This gives the \emph{centrality
 direction}, denoted as $\Delta z^\cs=(\Delta x^\cs, \Delta s^\cs)$. In \emph{predictor steps}, we use
 $\stepparam=0$. This direction is also called the \emph{affine scaling
   direction}, and will be denoted as $\Delta z^\as=(\Delta x^\as, \Delta s^\as)$
 throughout.

Let $z \coloneqq  (x,s) \in \cal N^2(\beta)$ be our current iterate. In our
algorithm, we will first apply a corrector step to get $z^\cs \coloneqq  z + \Delta
z^\cs$, which will reduce our centrality error by a factor $2$, that is,
$z^\cs \in \cal N^2(\beta/2)$, without changing the gap $\gap(z)$. Following
this, we apply a predictor step to get $z^+ \coloneqq  z^\cs + \alpha^\as \Delta
z^\as$, for $\alpha^\as \in (0,1]$, which will make progress along the
central path while maintaining that $z^+ \in \clN(\beta)$. Here we slightly
abuse notation, by letting $\Delta z^\as \coloneqq  \Delta (z^\cs)^\as$, that is the
predictor direction computed from the recentered iterate $z^\cs$.
The step-length $\alpha^\as > 0$ will be chosen such that   
\begin{align*}
   \alpha^\as \le \sup\{\alpha \in [0,1] : \forall \alpha' \in [0 ,\alpha]:  z^\cs + \alpha' \Delta z^\as \in
   \cal{N}^2(\beta)\}.
\end{align*}
Thus, we conclude $z^+=z^\cs+\alpha^\as \Delta z^\as \in \clN(\beta)$. We
remark that the closure allows us to take a step that goes all the way to an
optimal solution. If $z^+ \in \cal N^2(\beta)$, i.e., if we have not
arrived at an optimal solution, then $z^+$ remains a valid iterate for the
next step.  

\begin{remark}  In contrast to predictor-corrector methods such as  \cite{MTY}, ours is a `corrector-predictor' method, first performing corrector steps followed by predictor steps in each iteration. While the two descriptions are equivalent, it is more convenient for the description of the final iterate, i.e., achieve the termination guarantee in Theorem~\ref{th:main_upper_bound-slc} for the same $\beta$ neighbourhood. 
\end{remark}

The next proposition summarizes well-known properties of predictor and corrector steps, see
\eg{}~\cite[Section~4.5.1]{Ye-book}.

\begin{proposition}\label{prop:predictor-corrector} Let $z = (x,s)
   \in \mathcal{N}^2(\beta)$ for $\beta\in (0,1/6]$.
 \begin{enumerate}[label=(\roman*)]
 \item For $z \in \cal N^2(\beta)$, let $\Delta z^\cs$ be the corrector direction at $z$. Then for $z^\cs=z +\Delta z^\cs$, we have $\gap(z^\cs)=\gap(z)$ and $z^\cs \in \cal N^2(\beta/2)$.\label{i:corrector}  
\item For the affine scaling step, we have $\gap(z^+)=(1-\alpha^\as)\gap(z)$ and $z^+ \in \clN(\beta)$.
\label{i:gap}
 \item The affine scaling step-length $\alpha^\as$ can be chosen in the range 
 \[0 \leq \alpha^\as \leq \max\left\{\frac{\beta}{2\sqrt{n}},1-\frac{2\norm{\Delta
     x^\as\Delta s^\as}}{\beta\gap(z)}\right\}\, .
 \]\label{i:stepsize}
 \item After a sequence of $\left\lceil \tfrac{3\sqrt{n}}{\beta}\log(1/\eps)
\right\rceil$, $\eps \in (0,1]$, corrector and predictor steps from $z$,
assuming the affine scaling step-lengths are all at least
$\tfrac{\beta}{3\sqrt{n}}$, we obtain an iterate $z'=(x',s')\in {\cal
N}^2(\beta)$ such that $\gap(z') \le \eps \gap(z)$. \label{i:gap-decrease}
\end{enumerate}
\end{proposition}

\begin{remark}
By~\ref{i:stepsize}, an affine scaling step-length of
$\tfrac{\beta}{2\sqrt{n}}$ is always valid. Therefore, the assumption
in~\ref{i:gap-decrease} that the step-lengths are at least
$\tfrac{\beta}{3\sqrt{n}}$ is conservative. We use this conservative estimate
for purely computational reasons as $\tfrac{\beta}{2\sqrt{n}}$ may be
irrational. In particular, this choice is designed to combine with the
step-length computation given by \Cref{prop:compute-standard-step} below.
\end{remark}

The following proposition explicitly states that the predictor and corrector steps can be computed in strongly polynomial time, and that we can select appropriate step-lengths. The reason is simply that computing the steps amounts to solving linear systems  based on the data
$x,s,\mu$ and $\stepparam$.

\begin{proposition}[Step Formulas]
\label{prop:compute-standard-step}
Let $\mA \in \R^{m \times n}$, $\rank(\mA)=m$, $W = \ker(\mA)$ and $W^\perp = \im(\mA^\T)$. Let $z \coloneqq  (x,s)$, where $x \in W, x > \0$, $s \in W^\perp, s > \0$ and $t \in \R^n$. Then, the solution to $s \Delta x + x \Delta s = t$, $\Delta x \in W$, $\Delta s \in W^\perp$ can be expressed as
\begin{align*}
\Delta s &= \mA^{\T}\left(\mA \diag\left(x/s \right) \mA^{\T}\right)^{-1} \mA \frac{t}{s}  \\
\Delta x &= \frac{t}{s} - \diag\left(x/s\right) \mA^{\T}\left(\mA \diag\left(x/s\right) \mA^{\T}\right)^{-1} \mA \frac{t}{s}. 
\end{align*}
Moreover, both the affine
scaling step $(\Delta x^\as, \Delta s^\as)$ and corrector step $(\Delta x^\cs, \Delta s^\cs)$, corresponding to $t = \stepparam
\gap(z) \1_n - xs$ for $\stepparam \in \{0,1\}$ respectively, can be
computed in strongly polynomial time. Furthermore, given $\Delta x, \Delta s \in \R^n$ and $\beta \in (0,1/6]$, one can in strongly polynomial time compute an affine scaling step-length $\alpha^\as \in (0,1]$ satisfying 
\[
\max \left\{ \frac{\beta}{3 \sqrt{n}}, 1- \frac{3\norm{\Delta x^\as \Delta s^\as}}{\beta \gap(z)} \right\} \leq \alpha^\as \leq \max \left\{ \frac{\beta}{2\sqrt{n}}, 1-\frac{2\norm{\Delta x^\as \Delta s^\as}}{\beta \gap(z)} \right\}\, .
\]  
\end{proposition}
\begin{proof}
It is directly verified by inspection that $s \Delta x + x \Delta s = t$ and
that $\mA \Delta x = \0$ and $\Delta s \in \im(\mA^\T)$. Strongly polynomial computability follows since matrix inversion can be done in strongly polynomial time.

For the affine scaling step-length, note that $\lceil a\rceil$ can be computed strongly polynomially, using  $O(\log n)$ comparisons as long as $a=\mathrm{poly}(n)$. Thus, we can compute  $a = \frac{\beta}{\lceil
2\sqrt{n} \rceil}$, $r = \norm{\Delta x^{\as} \Delta s^\as}_\infty$ and $b = 1 - \lceil 2 \norm{\Delta x^\as \Delta s^\as}/r \rceil r/(\beta \gap(z))$.
We then return $\alpha^\as \coloneqq  \max
\{a,b\}$. Since $\lceil 2 \sqrt{n} \rceil \in [2\sqrt{n},3\sqrt{n}]$ and $
\lceil 2\norm{\Delta x^\as \Delta s^\as}/r \rceil r \in [2\norm{\Delta x^\as \Delta s^\as}, 3\norm{\Delta x^\as \Delta s^\as}]$, we have that
$\frac{\beta}{3\sqrt{n}} \leq a \leq \frac{\beta}{2\sqrt{n}}$ and
$1-\frac{3\norm{\Delta x^\as \Delta s^\as}}{\beta \gap(z)} \leq b \leq
1-\frac{2\norm{\Delta x^\as \Delta s^\as}}{\beta \gap(z)}$. Thus, $\alpha^\as$
satisfies the requirement.     
\end{proof}

\paragraph{Minimum-Norm Viewpoint}
We introduce some useful notation for the algorithm, and derive the minimum-norm interpretation of the affine scaling steps.
\begin{definition}[Normalized Iterates, Gap Vector and Subspaces]\label{def:normalized}
For $z=(x,s)\in\Primalp\times \Dualp$, we let 
\begin{equation}\label{eq:error}
\begin{aligned}
\error(z)&\coloneqq \sqrt{\frac{xs}{\gap(z)}}\in\R^n \, , \\ 
\xerr &\coloneqq x\error(z)^{-1}= \sqrt{\frac{x\gap(z)}{s}} \in\R^n \, , \\
\serr &\coloneqq s\error(z)^{-1}= \sqrt{\frac{s \gap(z)}{x}} \in\R^n\, .
\end{aligned}
\end{equation}
We call $\error(z)$ the \emph{normalized gap vector} and simply use $\error$ when clear from the context. We call $\xerr$ and $\serr$ the \emph{normalized primal and dual iterates}, respectively. We also define the \emph{normalized subspaces} 
\[
\hat{W}\coloneqq \xerr^{-1}W\, \quad\mbox{and }\quad \hat{W}^\perp \coloneqq   \serr^{-1}W^\perp\, .\]
\end{definition}

 If $z=(x,s)$ falls on the central path, that is, $xs=\gap(z)\1$, then $\error(z)=\1$, $\xerr=x$ and $\serr=s$. The variables $\xerr$ and $\serr$ represent natural adjustments for points off the central path. The next statement is immediate from the definitions, using $\Xerr \Serr = \gap(z)\1$. 
\begin{proposition}\label{prop:xerr-orth}
The subspaces $\xerr^{-1} W$ and $\serr^{-1}W^\perp$ are orthogonal.
\end{proposition}

The following is a simple corollary of Proposition~\ref{prop:x_i-s_i}.

\begin{proposition}\label{prop:xi}
For $z=(x,s)\in \cal N^2(\beta)$ for $\beta\in(0,1)$, we have $\norm{\xi}=\sqrt{n}$. Moreover,
\[
\begin{aligned} 
\sqrt{1-\beta} \1 &\le \error \le \sqrt{1+\beta} \1\, , \\
\sqrt{1-\beta}\nx \le x \le \sqrt{1+\beta} \nx \,& \quad\mbox{and}\quad
\sqrt{1-\beta}\ns \le s \le \sqrt{1+\beta}\ns \, .
\end{aligned}
\]
\end{proposition}

We will frequently use the rescaled subspaces $\xerr^{-1} W$ and $\serr^{-1}W^\perp$ that correspond to using the local geometry at the point $z=(x,s)$. 
Throughout, we will refer to $\norm{\xerr^{-1} w}$ and $\norm{\serr^{-1}w}$ as the \emph{primal and dual local norms} of the vector $w\in\R^n$ at the point $z=(x,s)\in\Primalp\times \Dualp$.

\Cref{aff:sum} for the predictor step ($\stepparam=0$) with update direction $(\Delta x^\as, \Delta s^\as)$ can be written as 
\begin{equation}
    x^{-1}\Delta x^\as + x^{-1}\Delta s^\as = -\1\, ,
\end{equation}
or equivalently, \begin{equation}
  \label{eq:relate-as-direction-to-err}
    \xerr^{-1} \Delta x^\as + \serr^{-1} \Delta s^\as = -\error,
\end{equation}
which serves the purpose that now $\Xerr^{-1}\Delta x^\as \in \Werr$ and $ \Serr^{-1} \Delta s^\as \in \Wperr$ are orthogonal vectors (Proposition~\ref{prop:xerr-orth}). Thus, $\Xerr^{-1}\Delta x^\as$ and $\Serr^{-1} \Delta s^\as$ give an  orthogonal decomposition of $-\error$.  This leads to the following formulas:
\begin{equation}\label{eq:aff-compute}
    \begin{aligned}
        \Delta x^\as &= - \Xerr \proj_{\Werr}(\error)\, , \\
        \Delta s^\as &= - \Serr \proj_{\Wperr}(\error)\, .
    \end{aligned}
\end{equation}
 Equivalently, we can see  $\Delta z^\as=(\Delta x^\as,\Delta s^\as)$ as the 
 optimal solutions of the following minimum-norm problems:
 \begin{equation}\label{eq:aff-minnorm-rescaled}
 \begin{aligned}
 \Delta x^\as &= \xerr \argmin_{\delta \in\Werr}\norm{\error+\delta} \, , \\
  \Delta s^\as &= \serr \argmin_{\delta \in\Wperr}\norm{\error+\delta} \, . \\
 \end{aligned}
 \end{equation}
We can rewrite these equivalently as projections in $W$ and $W^\perp$, noting that $\xerr^{-1} x=\serr^{-1}s=\error$. 
 \begin{equation}\label{eq:aff-minnorm}
 \begin{aligned}
 \Delta x^\as &= \argmin_{\Delta x \in W}\norm{\xerr^{-1}(x+\Delta x)} \, , \\
 \Delta s^\as &= \argmin_{\Delta s \in W^\perp} \norm{\serr^{-1}(s+\Delta s)}\, .
 \end{aligned}
 \end{equation}

We will require the following generic monotonicity estimate in terms of local
norms. The next lemma is an $\ell_1$-variant of \Cref{lem: central_path_bounded_l1_norm}, and the proof is implicit in the proof of \cite[Lemma
16]{Vavasis1996}. \begin{lemma}
\label{lem:local-norm-monotone}
Let $z = (x,s) \in \cal N(\beta)$, $\beta \in (0,1)$, and $z' = (x',s') \in \Primal \times \Dual$. Then, we have that
\begin{equation}
\norm{(\nx^{-1} x',\ns^{-1} s')}_1 \leq \frac{n}{\sqrt{1-\beta}} \left(1+\frac{\gap(z')}{\gap(z)}\right). \label{eq:local-norm-monotone-bnd}
\end{equation}
\end{lemma}
\begin{proof}
\begin{align*}
\norm{(\nx^{-1}x',\ns^{-1}s')}_1 &= \frac{\norm{(\ns x', \nx s')}_1}{\gap(z)}
\leq \frac{1}{\sqrt{1-\beta}} \frac{\norm{(s x', x s')}_1}{\gap(z)} 
 = \frac{1}{\sqrt{1-\beta}} \frac{\pr{s}{x'} + \pr{x}{s'}}{\gap(z)} \\
&= \frac{1}{\sqrt{1-\beta}} \frac{\pr{s}{x}+\pr{x'}{s'}}{\gap(z)} = \frac{1}{\sqrt{1-\beta}}\frac{n\gap(z)+n\gap(z')}{\gap(z)} \\ &= \frac{n}{\sqrt{1-\beta}}\left(1+\frac{\gap(z')}{\gap(z)}\right), 
\end{align*}
where the first inequality follows from \Cref{prop:xi}.
\end{proof}

\paragraph{Step-Length Estimates}
We will also need good estimates on the size on predictor steps beyond affine
scaling. Our main estimate in this regard is given below. We use the latter
part to compute the step-length associated with layered least squares steps.
Note that the algorithm outputs a step-length of $0$ in case the requirements
of the step are not satisfied.

 \begin{proposition}[Step-length Estimate for General Directions]
 \label{prop:step-size} Let $z=(x,s) \in \cal N^2(\beta/2)$, $\mu \coloneqq  \gap(z)$, $\beta \in (0,1/6]$. Consider directions $\Delta x \in W$, $\Delta s \in W^\perp$ that satisfy 
\begin{equation}
\delta \coloneqq \norm{\frac{\Delta x \Delta s}{\mu}} \leq \frac{\beta}{9}, \quad \quad
\eps \coloneqq \norm{\frac{(x+\Delta x) (s+\Delta s)}{\mu}} \leq \frac{\beta}{9}. \label{eq:step-size-cond}
\end{equation}Then $(x+ \alpha \Delta x, s + \alpha \Delta s) \in \clN(\beta)$ and $\gap(x+ \alpha \Delta x, s + \alpha \Delta s) \in [1 \pm 1/8](1-\alpha) \mu$, for all $0 \leq \alpha \leq 1-\frac{8\eps}{\beta}$. Furthermore, given $\Delta x \in \R^n, \Delta s \in \R^n, \mu > 0,\beta > 0$, one can in strongly polynomial time output a step-length $\alpha^\slls \in [0,1]$ satisfying $1-\frac{9\eps}{\beta} \leq \alpha^\slls \leq 1-\frac{8\eps}{\beta}$ if $\max \{\norm{\frac{\Delta x \Delta s}{\mu}},\norm{\frac{(x+\Delta x)(s+\Delta s)}{\mu}}\} \leq \beta/9$ and $\alpha^\slls = 0$ otherwise.
 \end{proposition}
 \begin{proof}
Let $z_\alpha \coloneqq  (x+ \alpha \Delta x, s + \alpha \Delta s)$ for $0 \leq \alpha \leq 1-\frac{8\eps}{\beta}$, $\alpha < 1$. (Note that $\alpha<1$ only becomes relevant in case $\eps=0$; we will discuss this case later.) We first bound
the centrality error using the estimate $(1-\alpha) \mu$ for $\mu(z_\alpha)$
as follows:
\begin{equation*}
 \begin{aligned}
 &\norm{\frac{(x+ \alpha \Delta x) (s+\alpha\Delta s)}{(1-\alpha) \mu} - \1}\\
 &= \norm{\frac{(1-\alpha) x s + \alpha (x + \Delta x)(s+\Delta s) - \alpha(1-\alpha) \Delta x \Delta s}{(1-\alpha)\mu} - \1}  \\
 &\leq \norm{\frac{x s}{\mu}-\1} + \frac{\alpha}{1-\alpha} \norm{\frac{(x+\Delta x)(s+\Delta s)}{\mu}} + \alpha \norm{\frac{\Delta x \Delta s}{\mu}}  \\
 &\leq \beta/2 + \frac{\alpha}{1-\alpha} \eps + \alpha \delta \leq  \beta/2 + \beta/8 + \beta/9 < \frac{3}{4}\beta\, ,
 \end{aligned}
 \end{equation*}
where the last inequality follows since $\frac{\alpha}{1-\alpha} \eps \leq
\beta/8$ for $0 \leq \alpha \leq 1-8\eps/\beta$ and $\alpha < 1$ (needed to
ensure the denominator $(1-\alpha)$ is positive). 

By \Cref{lem:choice-of-mu} and the above bound,  we get that $\gap(z_\alpha)
\in [1\pm \frac{3}{4}\beta/\sqrt{n}](1-\alpha)\mu \subseteq [1 \pm
\frac{1}{8}](1-\alpha)\mu$, and $z_\alpha \in \cal N^2\left(\frac{\frac{3}{4}
\beta}{1-\frac{3}{4}\beta}\right) \subseteq \cal N^2(\beta)$, for $\beta \in
(0,1/6]$. If $\eps = 0$, letting $\alpha \rightarrow 1^-$, we conclude by
continuity that $z_1 \coloneqq  (x+\Delta x,s+\Delta s) \in \clN(\beta)$ and
$\gap(z_1) = 0$, as needed.

For the last part, note that the condition $\max \{\frac{\norm{\Delta x \Delta s}}{\mu},\frac{\norm{(x+\Delta x)(s+\Delta s)}}{\mu}\} \leq \beta/9$ can be checked in strongly polynomial time by squaring both sides. If this check fails, output $\alpha^\slls = 0$. Otherwise, compute $r = \norm{(x+\Delta x)(s + \Delta s)}_\infty$ and
compute $\nu = \lceil 8\norm{(x+\Delta x)(s + \Delta s)}/r \rceil \in [8,\lceil
8\sqrt{n} \rceil]$ via binary search in $O(\log n)$ time, and return
$\alpha^\slls = 1-\frac{\nu r}{\mu \beta}$. For correctness, note that 
$\norm{(x+\Delta x)(s + \Delta s)} \leq \nu (r/8) \leq \frac{9}{8} \norm{(x+\Delta
x)(s + \Delta s)}$, and thus the desired inequalities follow recalling that $\eps \coloneqq  \norm{(x+\Delta x)(s + \Delta s)}/\mu$. 
\end{proof}   
 
\section{Polarization of the Central Path} \label{sec:polar}

We now introduce the notion of polarized segments of the central path. For $0\le\mu_1\le\mu_0$, the central path segment between these values is denoted by
\begin{equation}\label{eq:CP}
\CP[\mu_1,\mu_0] \coloneqq \{\pz\cp(\mu) : \mu_1 \leq \mu \leq \mu_0\}\, .
\end{equation}

\begin{definition}[Polarization]
\label{def:polarized}
For $\polarized \in (0,1]$ and $\mu_0 > \mu_1 \geq 0$, we say that the segment
$\CP[\mu_1,\mu_0]$ is
\emph{$\polarized$-polarized} if there exists a partition $B \cup N = [n]$ such that for all $\mu \in [\mu_1,\mu_0]$:
\begin{align*}
x_i\cp(\mu) \geq \polarized x_i\cp(\mu_0)\, ,\quad \forall i \in B\, , \\ 
s_i\cp(\mu) \geq \polarized s_i\cp(\mu_0)\, ,\quad \forall i \in N\, .
\end{align*}
\end{definition} 

\begin{remark}
\label{rem:pol-eq} By continuity of the central path and the condition $\mu_0
> \mu_1 \geq 0$, we may restrict the polarization check above to $\mu \in
[\mu_1,\mu_0]$, $\mu > 0$. For $i \in N$, since $s_i\cp(\mu) =
\mu/x_i\cp(\mu)$ for $\mu > 0$, the condition $s_i\cp(\mu) \geq \polarized
s_i\cp(\mu_0)$ is equivalent to 
\begin{equation}
\mu/x_i \cp(\mu) \geq \polarized \mu_0/x_i \cp(\mu_0) \quad \Leftrightarrow \quad 
x_i(\mu) \leq \frac{\mu}{\polarized \mu_0} x_i\cp(\mu_0).    
\end{equation}
Thus, the polarization condition can be stated only in terms of the primal
central path. Similarly, switching the roles of $B$ and $N$, the polarization
condition can also be stated only in terms of the dual central path.  
\end{remark}

\begin{remark}
As stated, the notion of polarization requires an inequality to hold for all
$\mu \in [\mu_1,\mu_0]$. At the cost of losing a factor $n$ however, it is in
fact sufficient to check 
the polarization condition only at $\mu = \mu_1$. This
follows by the near-monotonicity of the central path (\Cref{lem:
central_path_bounded_l1_norm}):
\[
\frac{x_i\cp(\mu)}{x_i\cp(\mu_0)} =  \frac{x_i\cp(\mu)}{x_i\cp(\mu_1)} \cdot \frac{x_i\cp(\mu_1)}{x_i\cp(\mu_0)} \geq \frac{1}{n} \cdot \frac{x_i\cp(\mu_1)}{x_i\cp(\mu_0)}\, ,\quad \forall i \in [n]\, ,
\] 
The same is true for $s(\mu)$ by a symmetric
argument. 
\end{remark}

As a direct consequence of the definition together with near-monotonicity, we
deduce the following crucial corollary:

\begin{corollary}\label{cor:polarization} 
Let $\CP[\mu_1,\mu_0]$, $0 \leq \mu_1 \leq \mu_0$, be $\polarized$-polarized
with respect to the partition $B \cup N = [n]$. Then, for all $\mu \in
[\mu_1,\mu_0]$, the following holds:
\begin{enumerate}[label=(\arabic*)]
\item $\polarized x_i(\mu_0) \leq x_i(\mu) \leq n x_i(\mu_0)$, $i \in B$. 
\item $\polarized s_i(\mu_0) \leq s_i(\mu) \leq n s_i(\mu_0)$, $i \in N$. 
\item $\frac{\mu}{n \mu_0} x_i(\mu_0) \leq x_i(\mu) \leq \frac{\mu}{\polarized \mu_0} x_i(\mu_0)$, $i \in N$.
\item $\frac{\mu}{n \mu_0} s_i(\mu_0) \leq s_i(\mu) \leq \frac{\mu}{\polarized \mu_0} s_i(\mu_0)$, $i \in B$.
\end{enumerate}
\end{corollary}
\begin{proof}
The first inequalities in (1) and (2) are the definition of
$\polarized$-polarization and the second inequalities are from \Cref{lem:
central_path_bounded_l1_norm}.  (3)
and (4) are equivalent to (1) and (2) using the central path relations
$x(\mu_0) s(\mu_0) = \mu_0 \1$ and $x(\mu) s(\mu) = \mu \1$. 
\end{proof}

Section~\ref{sec:slls-alg} introduces the algorithm
~\nameref{alg:subspace_ipm} that can traverse
$\polarized$-polarized segments  in
$O(n^{1.5}$ $\log(n/\polarized))$ iterations. 
Theorem~\ref{th:main_upper_bound-curve} follows by combining this algorithm with the following decomposition result that is the main result of this section; the proof can be found in Section~\ref{sec:running-time-analysis}.
The stronger variant 
Theorem~\ref{th:main_upper_bound-slc} is proved in Section~\ref{sec:amortized} using an additional amortization argument.
\begin{theorem}
\label{thm:wide-polar}
Let $\Gamma: (\mu_1,\mu_0) \rightarrow \clNw(\theta)$, $\theta \in
(0,1)$, $0 \leq \mu_1 < \mu_0 \leq \infty$, be a piecewise linear curve
satisfying $\gap\left(\Gamma(\mu)\right) = \mu$, $\forall \mu \in (\mu_1,\mu_0)$ consisting of $T$ linear segments. Then, $\CP[\mu_1,\mu_0]$ can be
decomposed into $T$ segments that are
$\frac{(1-\theta)^2}{16n^3}$-polarized. 
\end{theorem}

\Cref{thm:wide-polar} is a direct consequence of the following key lemma.

\begin{lemma}
For $\theta \in (0,1)$, let $[z^{(0)},z^{(1)}] \subseteq \clNw(\theta)$, 
$\gap(z^{(0)}) > \gap(z^{(1)})$. Then, $\CP[\gap(z^{(1)}),\gap(z^{(0)})]$ is
$\frac{(1-\theta)^2}{16n^3}$-polarized.
\label{lem:wide-polar}
\end{lemma}

\begin{proof}[Proof of \Cref{thm:wide-polar}] 
By assumption, the curve $\Gamma([\mu_1,\mu_0]) = \cup_{i=1}^T
[z^{(i-1)},z^{(i)}] \subseteq \clNw(\theta)$, where $\mu_0 = \gap(z^{(0)}) >
\gap(z^{(1)}) > \cdots > \gap(z^{(T)}) = \mu_1$. By \Cref{lem:wide-polar},
each segment $\CP[\gap(z^{(i)}),\gap(z^{(i-1)})]$, $i \in [T]$, is therefore
$\frac{(1-\theta)^2}{16n^3}$-polarized with respect to some polarization
partition $(B^{(i)},N^{(i)})$. This proves the theorem.   
\end{proof}

It remains to prove \Cref{lem:wide-polar}. The proof requires the following simple technical lemma that
allows us to relate approximate centrality along lines to polarization. 

\begin{lemma}
\label{lem:ratio}
For any $u,v>0$,
\begin{equation}
\label{eq:ratio-test}
\min_{\alpha \in [0,1]} \frac{(1-\alpha + \alpha u)(1-\alpha + \alpha v)}{1-\alpha + \alpha uv} =\min\left\{1,\left(\frac{\sqrt{u} +\sqrt{v}}{1 + \sqrt{uv}}\right)^2\right\}\leq 2(u+v) \, .
\end{equation}
\end{lemma}
\begin{proof}
To show the equality, let $\mu \coloneqq  uv$. 
  Note that
\begin{equation}\label{eq:min-ratio}
\begin{aligned}
\min_{\alpha \in [0,1]} \frac{(1-\alpha + \alpha u)(1-\alpha + \alpha v)}{1-\alpha + \alpha uv}
&= \min_{\alpha \in [0,1]} \frac{(1-\alpha)^2 + \alpha^2\mu + \alpha(1-\alpha)(u+v)}{1-\alpha + \alpha \mu} \\
&= 1 +  \min_{\alpha \in [0,1]}\left(u+v-(1+\mu)\right)  \frac{\alpha(1-\alpha)}{1-\alpha + \alpha \mu} \, . 
\end{aligned}
\end{equation}
\noindent{\bf Case I: $\left(\frac{\sqrt{u} +\sqrt{v}}{1 + \sqrt{uv}}\right)^2\geq 1$.} In this case, we need to show that the minimum of the expression is 1. It is easy to see that the condition equivalent to $u+v \geq 1+uv=1+\mu$. Thus, the minimum value of~\eqref{eq:min-ratio} is clearly $1$,
attained at $\alpha \in \{0,1\}$. 

\noindent{\bf Case II: $\left(\frac{\sqrt{u} +\sqrt{v}}{1 + \sqrt{uv}}\right)^2< 1$,} or equivalently, $u+v < 1+\mu$. In this case, the minimizer
of~\eqref{eq:min-ratio} corresponds to the maximizer of
$\frac{\alpha(1-\alpha)}{1-\alpha + \alpha \mu}$. This function takes value
$0$ at $\alpha \in \{0,1\}$ and is strictly positive for $0 < \alpha < 1$.
Furthermore, the unique critical point in the interval $[0,1]$ occurs at
$\alpha^* = \frac{1}{1+\sqrt{\mu}}$, which is thus the maximizer. The minimum
value of~\eqref{eq:min-ratio} is therefore
\[
1 + (u+v-(1+\mu)) \frac{\alpha^*(1-\alpha^*)}{1-\alpha^* + \alpha^* \mu} = \frac{u+v+2\sqrt{\mu}}{(1+\sqrt{\mu})^2}\, ,\]
as required.
The inequality in the statement follows easily as
\[
\left(\frac{\sqrt{u} +\sqrt{v}}{1 + \sqrt{uv}}\right)^2\le \left({\sqrt{u} +\sqrt{v}}\right)^2=2(u+v)-(\sqrt{u}-\sqrt{v})^2\, .
\]
\end{proof}

\begin{proof}[Proof of \Cref{lem:wide-polar}]
For $\alpha \in [0,1]$, let $z^{(\alpha)} \coloneqq  (x^{(\alpha)}, s^{(\alpha)}) \coloneqq  (1-\alpha)
z^{(0)} + \alpha z^{(1)}$. By \Cref{prop:gap-linear}, we first note that the
normalized gap function $\gap(z)$ is in fact linear on $[z^{(0)},z^{(1)}]$. That is,
\[
\mu_\alpha \coloneqq \gap(z^{(\alpha)}) = (1-\alpha) \gap(z^{(0)}) + \alpha \gap(z^{(1)})\, .
\] 

For any $i\in [n]$, $z^{(\alpha)}\in\cal N^{-\infty}(\theta)$ implies \[
\begin{aligned}
\frac{x^{(\alpha)}_i s^{(\alpha)}_i }{(1 - \alpha) x^{(0)}_{i}s^{(0)}_{i} + \alpha x^{(1)}_{i} s^{(1)}_{i}} & \geq \frac{(1-\theta) \mu_\alpha}{(1-\alpha) x^{(0)}_{i}s^{(0)}_{i} + \alpha x^{(1)}_{i} s^{(1)}_{i}} \\
& \geq \frac{(1-\theta) \mu_\alpha}{n((1 - \alpha) \gap(z^{(0)}) + \alpha \gap(z^{(1)}))} = \frac{1-\theta}{n}\, .
\end{aligned}
\]
Note that the above expression is the same as in \Cref{lem:ratio} for $u=x^{(1)}_{i}/x^{(0)}_{i}$,  $v=s^{(1)}_{i}/s^{(0)}_{i}$. 
Since the bound is true for any $\alpha\in[0,1]$, the Lemma implies
\[
\frac{1-\theta}{n}\le 2\left(\frac{x^{(1)}_{i}}{x^{(0)}_{i}}+\frac{s^{(1)}_{i}}{s^{(0)}_{i}}\right)\, . 
\]
Let 
\[
B \coloneqq  \left\{i \in [n]: \frac{x^{(1)}_{i}}{x^{(0)}_{i}} \geq \frac{s^{(1)}_{i}}{s^{(0)}_{i}}\right\}\, ,\quad N \coloneqq  [n]
\setminus B\, .
\]
 Then, $x^{(1)}_{i}/x^{(0)}_{i}\geq \frac{1-\theta}{4n}$,
for all $i \in B$, and $s^{(1)}_{i}/s^{(0)}_{i}  \geq \frac{1-\theta}{4n}$ for all $i\in N$.

For any $\alpha \in [0,1]$ and $i \in B$, 
\[
\frac{x^{(\alpha)}_{i}}{x^{(0)}_{i}} = (1-\alpha) + \alpha \frac{x^{(1)}_{i}}{x^{(0)}_{i}} \geq \min \left\{1, \frac{1-\theta}{4n}\right\} = \frac{1-\theta}{4n}.  
\]
Similarly, for $i \in N$, $s^{(\alpha)}_{i}/s^{(0)}_{i}\geq \frac{1-\theta}{4n}$.

For the central path point $z\cp(\mu_\alpha) =
(x\cp(\mu_\alpha),s\cp(\mu_\alpha))$ at $\mu_\alpha$,
the bounds in \Cref{lem:wide-to-cp} relating points in a neighborhood with central path points give
\[
\frac{x_i\cp(\mu_\alpha)}{x_i\cp(\mu_0)} \geq
\frac{x^{(\alpha)}_i/(2n)}{\frac{2n}{1-\theta} x^{(0)}_i} \geq
\frac{(1-\theta)^2}{16n^3}\, ,\quad\forall i\in B\, .
\]  
By a similar argument, we also have $s_i\cp(\mu_\alpha)/s_i\cp(\mu_0) \geq
\frac{(1-\theta)^2}{16n^3}$, $\forall i \in N$. Thus, $\CP[\mu_1,\mu_0]$ is
$\frac{(1-\theta)^2}{16n^3}$-polarized.
\end{proof}
 
\section{The Max Central Path}\label{sec:max-cp}

In this section, we derive key properties of the max central path and how to
use the max central path to decompose the central path into polarized
segments. In particular, we prove \Cref{lem:cp-max-cp},
\Cref{lem:simplex-main} and \Cref{thm:delta-slc}.  Given $g \geq 0$, we
denote by 
\[
\begin{aligned}
\Primal_g & \coloneqq \{ x \in \R^n\, :\, \mA x = b \, , \; x \geq \0 \, , \;\pr{c}{x} \leq v^\star + g \} \, , \\
\Dual_g & \coloneqq \{ s \in \R^n\, :\, \exists y\in\R^m \; \mA^\top y + s = c \, , \; s \geq \0 \, , \; \pr{b}{y} \geq v^\star - g \}
\end{aligned}\, .
\]
We will also use the subspace formulation \eqref{LP-subspace} with $d\in\R^n$ such $\mA d=b$; in these terms, we can write $\Primal_g =\{x\in\R^n\, :\,x\in W+d\, , \;x\geq 0\,,\; \pr{c}{x} \leq v^\star + g \}$ and $\Dual_g =\{s\in\R^n\, :\,s\in W^\top+c\, , \;s\geq 0\,,\; \pr{d}{c-s} \geq v^\star - g \}$.
the feasible sets of the linear programs in~\eqref{eq:max_cp}. They correspond to the sets of the primal and dual feasible points $(x, s) \in \Primal \times \Dual$ with objective value within $g$ from the optimum $v^\star$, respectively.  By the assumption that \eqref{LP_primal_dual} is feasible and bounded, it follows that $\Primal_g$ and $\Dual_g$ are both non-empty for any $g\ge 0$.

We recall that the duality gap of any pair $(x, (y, s))$ of primal-dual feasible points of~\eqref{LP_primal_dual} fulfills $\pr{c}{x} - \pr{b}{y} = \pr{x}{s}$. In particular, we have $\pr{x}{s^\star} = \pr{c}{x} - v^\star$ and $\pr{x^\star}{s} = v^\star - \pr{b}{y}$. Thus, the two sets $\Primal_g$ and $\Dual_g$ are equivalently given by
\[
\Primal_g = \{x \in \Primal : \pr{x}{s^\star} \leq g\}\, , \quad \Dual_g = \{s \in \Dual : \pr{x^\star}{s} \leq g\} \, .
\]
These expressions are in fact independent of the choice of optimal solutions $(x^\star, s^\star)$. The following claim is immediate by our assumption that $\Primalp$ and $\Dualp$ are non-empty.

\begin{proposition}\label{prop:mcp_well_defined}
For all $g \geq 0$, the sets $\Primal_g$ and $\Dual_g$ are bounded.
\end{proposition}
\begin{proof}
We restrict to the proof of the boundedness of $\Primal_g$, since the proof is analogous for $\Dual_g$. Let $s^\circ \in \Dual_{++}$ be a strictly feasible point of the dual, and $x \in \Primal_g$. By \Cref{prop:gap-formula}, we have 
\[
\pr{x}{s^\star} + \pr{x^\star}{s^\circ} = \pr{x}{s^\circ} + \pr{x^\star}{s^\star} \, . 
\]
Since $\pr{x^\star}{s^\star} = 0$, we deduce that $\pr{x}{s^\circ} \leq  g +\pr{x^\star}{s^\circ}$. As $s^\circ > 0$, this implies that $x_i \leq (g +\pr{x^\star}{s^\circ})/s^\circ_i$ for all $i \in [n]$.
\end{proof}

We denote by $\MCP = \{\mz(g) : g \geq 0\}$ the whole max central path. 
The max central path point $\mz (g) = (\mx(g), \ms(g))$ is the entry-wise maximum of the set $\Primal_g \times \Dual_g$. 

While the points of the max central path are not feasible in general, the following theorem shows that the max central path shares important similarities with the central path: 
\begin{theorem}[Centrality of the max central path]
\label{thm:mcp-central}
For all $g \geq 0$, we have that
\[
g \leq \mx_i(g) \ms_i(g) \leq 2g\, \quad \forall i \in [n]\, .
\]
\end{theorem}
\begin{proof}
We first prove the upper bound. For $i \in [n]$, let $x^{(i)} \in
\argmax \{x_i: x \in \Primal_g\}$ and $s^{(i)} \in
\argmax \{s_i: s \in \Dual_g\}$. Note
that $x^{(i)},s^{(i)}$ exist by \Cref{prop:mcp_well_defined}.
Then, 
\[
\mx_i(g)\ms_i(g) = x^{(i)}_i s^{(i)}_i \leq \pr{x^{(i)}}{s^{(i)}} = \pr{x^{(i)}}{s^{\star}} + \pr{x^{\star}}{s^{(i)}} \leq 2g\, , 
\]
where the last equality follows from \Cref{prop:gap-formula}.
We now prove the lower bound. We assume $g > 0$, since the statement is
trivial otherwise. 

Note that the dual program of $\max
\{x_i: x \in W+d, x \geq 0, \pr{x}{s^{\star}} \leq g\}$ can be expressed as
\[
\min \left\{\alpha g + \pr{u}{x^\star}: \alpha s^{\star} + u \geq e^i, u \in W^\perp, \alpha \geq 0\right\}\, ,
\]
using that $\pr{u}{x^\star}=\pr{u}{d}$ since $d-x^\star\in W$, $u\in W^\perp$.
Similarly, the dual program of  $\max
\{s_i: s \in W^\perp+c, s \geq 0, \pr{s}{x^{\star}} \leq g\}$ can be expressed as
\[
\min \left\{\beta g + \pr{v}{s^\star}: \beta x^{\star} + v \geq e^i, v \in W, \beta \geq 0\right\}\, .
\]
Let us pick optimal $(\alpha,u)$ and $(\beta,v)$ to these two programs. The product of the objective values is thus equal to $\mx_i(g)\ms_i(g)$; the proof is complete by showing a lower bound $g$.

We first claim that
\begin{equation}\label{eq:z-u}
\pr{u}{x^\star}\ge0\quad\mbox{ and }\quad \pr{v}{s^\star}\ge 0\, .
\end{equation}
By symmetry, it suffices to prove the first claim. Recall $x^\star\ge 0$; we show that whenever $x^\star_j>0$ then we must have $u_j\ge 0$. This follows since by complementary, $s^\star_j=0$, and that we have the constraint $\alpha s^\star_j+u_j\ge e^i_j$.

Next, note that the constraints in the two programs imply
\begin{equation}\label{eq:e-i-e-i}
1=\pr{e^i}{e^i}\le \pr{\alpha s^\star+u}{\beta x^\star +v}=\alpha\pr{v}{s^\star}+\beta\pr{u}{x^\star}\, .
\end{equation}
Now, the product of the objective values can be written as 
\[
\begin{aligned}
\mx_i(g)\ms_i(g)&=(\alpha g + \pr{u}{x^\star})(\beta g + \pr{v}{s^\star})\\
&=\alpha\beta g^2 + g \left(\alpha\pr{v}{s^\star}+\beta\pr{u}{x^\star}\right)+\pr{u}{x^\star}\cdot\pr{v}{s^\star}\ge g\, ,
\end{aligned}
\]
using~\eqref{eq:z-u} and~\eqref{eq:e-i-e-i}. This concludes the proof.
\end{proof}

Given the above, we are ready to show

\maxvscentralpath*

\begin{proof}
Recall that $z\cp(\mu) = (x\cp(\mu),s\cp(\mu))$ with 
\[\pr{x\cp(\mu)}{s\cp(\mu)} =
\pr{x\cp(\mu)}{s^\star} + \pr{x^\star}{s\cp(\mu)} = n\mu\]
using \Cref{prop:gap-formula}. Therefore,  $x\cp(\mu) \in
\Primal_{n\mu}$ and $s\cp(\mu) \in \Dual_{n\mu}$. By definition of the max central path, $z\cp(\mu) =
(x\cp(\mu),s\cp(\mu)) \leq (\mx(n\mu),\ms(n\mu)) = \mz(n\mu)$. For the second inequality,
note that
\begin{align*}
x\cp(\mu) &= \frac{\mu}{s\cp(\mu)} \geq \frac{\mu}{\ms(n \mu)} = \frac{\mu}{\mx(n\mu)
\ms(n \mu)} \mx(n\mu) \\
&\stackrel{\text{Thm. }\ref{thm:mcp-central}}{\geq} \frac{\mu}{2n\mu} \1 \mx(n\mu) = \frac{\mx(n\mu)}{2n}\, .
\end{align*}
By a symmetric argument, $s\cp(\mu) \geq \ms(n\mu)/2n$.
\end{proof}

\subsection{The Shadow Vertex Simplex Rule} 

Given a pointed polyhedron $P \subseteq \R^n$ (that means that it has at least one vertex) and two objectives $c^{(1)},c^{(2)} \in
\R^n$ such that $\max_{x \in P} \pr{c^{(1)}}{x},\ \max_{x \in P} \pr{c^{(2)}}{x}$ $<
\infty$, we recall that the shadow vertex simplex rule consists in pivoting over vertices of $P$ maximizing the objectives $(1-\lambda) c^{(1)} + \lambda c^{(2)}$ as $\lambda$ goes from $0$ to $1$. More formally, a sequence of vertices $v^{(1)},\dots,v^{(k)}
\in P$ is a \emph{$(c^{(1)},c^{(2)})$-shadow vertex path on $P$} if 
\begin{itemize}
\item $[v^{(i)},v^{({i+1})}]$ is an edge
of $P$, $\forall i \in [k-1]$, 
\item $\pr{c^{(2)}}{v^{(i)}} < \pr{c^{(2)}}{v^{({i+1})}}$, $\forall
i \in [k-1]$, and 
\item there exists $0=\lambda_0 < \lambda_1 \leq \cdots \leq
\lambda_{k-1} < \lambda_k=1$ such that $\forall i \in [k]$,
$\pr{v^{(i)}}{(1-\alpha)c^{(1)} + \alpha c^{(2)}} = \max_{x \in P} \pr{x}{(1-\alpha)c^{(1)} +
\alpha c^{(2)}}$, $\forall \alpha \in [\lambda_{i-1},\lambda_i]$. 
\end{itemize}  

To analyze shadow vertex paths further, we define the two-dimensional projection
\[
P[c^{(1)},c^{(2)}] \coloneqq \left\{\left(\pr{c^{(1)}}{x},\pr{c^{(2)}}{x}\right) : x \in P \right\} = \left(c^{(1)},c^{(2)}\right)^{\top}\cdot P \, .
\]

The vertices of $P[c^{(1)},c^{(2)}]$ maximizing an open interval of objectives in $(1-\lambda) e^{1} + \lambda e^{2}$, $\lambda \in [0,1]$ are
precisely the projections of vertices $v^{(i)}$, $i \in [k]$, on the shadow path
such that $\lambda_{i-1} < \lambda_i$.

We define $S_{P}(c^{(1)},c^{(2)})$ as the
number of vertices of $P[c^{(1)},c^{(2)}]$ maximizing an open interval of objectives
in $(1-\lambda) e^{1} + \lambda e^{2}$, $\lambda \in [0,1]$. By the preceding
observations, we have that $S_P(c^{(1)},c^{(2)})$ is a lower bound on the number of
vertices of any $(c^{(1)},c^{(2)})$-shadow vertex path.

In the above, we restricted both starting and ending objectives $c^{(1)},c^{(2)}$ to
have finite objective value on $P$. It will be useful in the sequel to extend
to the case where $P$ might be unbounded in direction $c^{(2)}$. 
In this case, we define the
shadow vertex path as above, with the only modification being that we let
$\lambda_k \coloneqq \max \{\lambda \in [0,1]: \max_{x \in P} \pr{x}{(1-\lambda)c^{(1)}
+ \lambda c^{(2)}} < \infty\}$, that is, the simplex path stops just before
reaching an unbounded ray for $c^{(2)}$. In this setting, note that
$S_P(c^{(1)},c^{(2)})$ is still well-defined and continues to be a lower bound on the
number of vertices on any $c^{(1)},c^{(2)}$ shadow vertex path.     

We are ready to prove Lemma~\ref{lem:simplex-main}. 
\descriptionpiecesmaxcentral*

\begin{proof}
We only prove part {\em (i)}; part {\em (ii)} follows analogously. 
For $i \in [n]$, let $Q_i =
P[s^{\star},e^{i}]$. We note that $\mx_i(g) = \max \{v_2: (v_1,v_2) \in Q_i,
v_1 \leq g\}$. 
In particular, the map $\mx_i(g)$ is a non-decreasing function of $g$. Moreover, it is easy to verify that $\mx_i(g)$ is concave. 

Again by definition, $S_\Primal(-s^{\star},e^{i})$ equals the number of vertices of
$Q_i$ maximizing an open interval of objectives in ${\mathcal O} \coloneqq \{-(1-\lambda) e^{1} + \lambda e^{2}: \lambda \in [0,1]\} \subset \R^2$.

Define $\bar{u}_i(g) \coloneqq \sup \{v_2: (g,v_2) \in Q_i\}$, which is defined to
equal $-\infty$ if $\{(g,v_2) \in Q_i\} = \emptyset$. By \Cref{prop:mcp_well_defined}, note that $\bar{u}_i(g) < \infty$ for all $g \geq 0$. By
convexity of $Q_i$, $\bar{u}_i$ is concave function on $\R_+$. Let $h = \sup
\{\bar{u}_i(g): g \geq 0\}$. 

Assume $h = \infty$. By concavity, $\bar{u}_i$ must be a
strictly increasing function on $\R_+$. In particular, $\bar{u}_i(g) =
\mx_i(g)$. Given this, we see that the linear pieces of $\mx_i(g)$ are in one to
one correspondence with the edges of $Q_i$ on the upper convex hull whose
projection onto the $e_1$-axis have positive length (i.e., excluding the
potential edge $\{(0,v_2) \in Q_i\}$). Since $Q_i \subseteq \R^2_+$, every
such edge can be uniquely associated with its left endpoint (which is always
a vertex of $Q_i$). It is now easy to check geometrically that the set of
such endpoints exactly corresponds to the set of vertices that are maximizers
of the objectives in an open interval of ${\mathcal O}$.    

Assume $h < \infty$. Let $g_h \coloneqq \min \{g \geq 0 : \mx_i(g) = h\}$. It is direct
to see that $\mx_i(g) = \bar{u}_i(g)$ if $g \leq g_h$ and that $\mx_i(g) = h$ for
$g \geq g_h$. Furthermore, $\mx_i(g)$ is strictly increasing on $[0,g_h]$.
From this, it is easy to see geometrically that the number of linear pieces
of $\mx_i$ is one plus the number of edges of $Q_i$ on the upper convex hull
lying in the band $\{(v_1,v_2) : 0 \leq v_1 \leq g\}$, where the extra linear
segment corresponds to constant segment between $g_h$ and $\infty$. As in the
previous case, these linear segments can be uniquely identified with their
left endpoints, which correspond to vertices of $Q_i$. Furthermore, it is
easy to check that these correspond to vertices of $Q_i$ maximizing an open
interval of objectives in $\mathcal O$. 
\end{proof}

\subsection{Decomposition of Straight Line Segments into Polarized Segments}\label{sec:decompose}
In this section we prove \Cref{thm:delta-slc}. To begin, we provide some
simple consequences of the concavity of the max central path coordinates. 

\begin{lemma}\label{lemma:mcp_subhomogeneous}
Let $g \geq 0$ and $\alpha \geq 1$. Then for all $i \in [n]$ we have $\mx_i(\alpha g) \leq \alpha \mx_i(g)$ and $\mx_i(\frac{1}{\alpha}g) \ge \frac{1}{\alpha}\mx_i(g)$.
\end{lemma}

\begin{proof}
We use the concavity of $g \mapsto \mx_i(g)$ (\Cref{lem:simplex-main}) and deduce that $(1 - \frac{1}{\alpha}) \mx_i(0) + \frac{1}{\alpha} \mx_i(\alpha g) \leq \mx_i(g)$. The first inequality of the lemma follows from $\mx_i(0) \geq 0$. The second inequality follows from the first applied to $\tilde g = \frac{g}{\alpha}$.
\end{proof}

The next lemma shows that the definition of the primal straight-line complexity (Definition~\ref{def:slc}) can be restricted to breakpoints $(g_k,z_k)=(g_k,\mx_i(g_k))$. By symmetry, the analogous statement for the dual straight-line complexity also holds.

\begin{lemma}\label{lemma:slc2}
Let $0 \leq \underline g < \overline g$. We have
\begin{equation}\label{eq:slc2}
\SLC^{\mathrm{p}}_{\theta, i}(\underline g, \overline g) = 
\min 
\begin{multlined}[t]
\big\{ p \geq 1 \colon \exists (g_k)_{ k \in [p+1]} \in \R \, , \enspace \overline g = g_1 > g_2 > \dots > g_{p+1} = \underline g \, , \\
\forall k \in [p] \, , \; \forall \lambda \in [0,1] \, , (1 - \lambda) \mx_i(g_k) + \lambda \mx_i(g_{k+1}) \\ \geq (1-\theta) \mx_i((1-\lambda) g_k + \lambda g_{k+1}) 
\big\}
\end{multlined}
\end{equation}
\end{lemma}

\begin{proof}
Let $\overline g = g_1 > g_2 > \dots > g_{p+1} = \underline g$. By the concavity of $g \mapsto \mx_i(g)$, the condition $(1 - \lambda) \mx_i(g_k) + \lambda \mx_i(g_{k+1}) \leq \mx_i((1-\lambda) g_k + \lambda g_{k+1})$ is satisfied for every $k \in [p]$ and $\lambda \in [0,1]$. We deduce that $\SLC_{\theta, i}(\underline g, \overline g)$ is less than or equal to the right-hand side of~\eqref{eq:slc2}.

We now prove the converse inequality. Let $(g_1, z_1), \dots, (g_{p+1}, z_{p+1})$ as in the right-hand side of~\eqref{eq:slc-def}. Up to removing or shortening segments in the sequence $\big([(g_k, z_k), (g_{k+1}, z_{k+1})]\big)_{k \in [p]}$, we can assume that $\overline g = g_1 \geq g_k > g_{p+1} = \underline g$ for every $k \in [p]$. As $z_k \leq \mx_i(g_k)$ for all $k \in [p+1]$, we have $(1 - \lambda) \mx_i(g_k) + \lambda \mx_i(g_{k+1}) \geq (1 - \lambda) z_k + \lambda z_{k+1} \geq (1-\theta) \mx_i((1-\lambda) g_k + \lambda g_{k+1})$ for all $k \in [p]$. It remains us to show that, up to removing some $g_k$, we can ensure $g_1 > g_2 > \dots > g_{p+1}$. Suppose that there is $k \in [p]$ such that $g_1 > \dots > g_k$ and  $g_{k+1} \geq g_k$. We introduce the smallest integer $l \geq k$ such that $g_{l+1} < g_k$ ($l$ is well-defined as $g_{p+1} = \underline{g} < g_k$). By concavity of $\mx_i$, we have $\mx_i(g_k) \geq (1 - \bar \lambda) \mx_i(g_l) + \bar \lambda \mx_i(g_{l+1})$ for $\bar \lambda = \frac{g_k - g_l}{g_{l+1} - g_l} \in [0,1]$. Then, for all $\lambda \in [0,1]$, \begin{align*}
(1 - \lambda) \mx_i(g_k) + \lambda \mx_i(g_{l+1}) & \geq (1 - \lambda') \mx_i(g_l) + \lambda' \mx_i(g_{l+1}) \\
& \geq (1-\theta) \mx_i((1 - \lambda') g_l + \lambda' g_{l+1}) \\
& = (1-\theta) \mx_i((1 - \lambda) g_k + \lambda g_{l+1}) 
\end{align*}
where $\lambda' = \lambda + \bar \lambda (1 - \lambda) \in [0,1]$. By removing the $g_{k+1}, \dots, g_l$ and recursively applying the same argument on the remaining sequence, we end up with a sequence $\overline g = g_1 > \dots > g_{q+1} = \underline g$ as in the right-hand side of~\eqref{eq:slc2}, with $q \leq p$. This shows that $\SLC^{\mathrm{p}}_{\theta, i}(\underline g, \overline g)$ is greater than or equal to the right-hand side of~\eqref{eq:slc2}.
\end{proof}

In the next lemma, we show that the primal and dual straight line complexity
for a single component of the maximum central path are related up to a factor
$2$, modulo a slight widening of the neighborhood. This shows that  that
primal and dual straight line complexities are essentially equivalent. The
proof proceeds by decomposing straight line segments into polarized segments.  
\begin{lemma}
Let $L \coloneqq \SLC^{\mathrm{p}}_{1-\alpha,i}(\underline g, \overline g)$, $\alpha \in (0,1]$, $i \in [n]$, and let
$\overline g = g_1 > g_2 > \dots > g_{L+1} = \underline g$ be the minimizer
of~\eqref{eq:slc2}. Define $\phi_i: [\underline g, \overline g] \rightarrow
\R_+$ so that $\phi_i(g) = \max \{\mx_i(g_{j+1}), \frac{g}{g_{j}}
\mx_i(g_{j})\}$ for $g \in [g_j,g_{j+1}]$, $j \in [L]$. Then, $\phi_i(g)$ and $\frac{g}{\phi_i(g)}$ are continuous piecewise linear functions on $(\underline g, \overline g]$ with at most $2L$ pieces, and
\begin{align}
\mx_i(g) &\geq \phi_i(g) \geq \frac{\alpha}{2} \mx_i(g), \label{eq:primal-polar} \\
\ms_i(g) &\geq \frac{\alpha g}{2\phi_i(g)} \geq \frac{\alpha}{4} \ms_i(g), \quad \forall g \in (\underline{g}, \overline{g}]. \label{eq:dual-polar}
\end{align}
Moreover, $\SLC^{\mathrm{d}}_{1-\alpha/4,i}(\underline g, \overline g) \leq 2 \SLC^{\mathrm{p}}_{1-\alpha,i}(\underline g, \overline g)$.
\label{lem:dual-slc}
\end{lemma}

\begin{remark}
For $\alpha \in (0,1/2]$, it can be shown that straight line complexity satisfies 
$\SLC^{\mathrm{d}}_{1-\alpha,i}(\underline g, \overline g) \leq O(1)
\SLC^{\mathrm{d}}_{1-\alpha/4,i}(\underline g, \overline g)$. That is,
shrinking the neighborhood size by a constant factor can only change the
straight line complexity of a single coordinate of the max central path by a
constant factor. In particular, this implies that
$\SLC^{\mathrm{d}}_{1-\alpha,i}(\underline g, \overline g) \leq O(1)
\SLC^{\mathrm{p}}_{1-\alpha,i}(\underline g, \overline g)$ for $\alpha \in
(0,1/2]$. As this does not have a significant effect on our iteration bounds,
we do not prove this here. 
\end{remark}

\begin{proof}[Proof of \Cref{lem:dual-slc}]
Noting that $\mx_i(g_{j}) \geq \tfrac{g_{j}}{g_{j-1}} \mx_i(g_{j-1})$, $j \in
\{2,\dots,L+1\}$, by \Cref{lemma:mcp_subhomogeneous} and $\mx_i(g_j) \geq
\mx_i(g_{j+1})$, $j \in [L]$, by monotonicity, we have that $\phi(g_j) =
\mx_i(g_j)$, $\forall j \in [L+1]$. In particular, $\phi_i$ is
well-defined. Since each function $\max \{\mx_i(g_{j+1}), \frac{g}{g_{j}} \mx_i(g_j)\}$, $i \in [L]$, is continuous piecewise linear
with at most $2$ pieces, $\phi_i$ is also continuous with at most $2L$ pieces.
By our assumption that $\Primal$ is strictly feasible, we have that
$\mx_i(g_j) > 0$ for $j \in [L]$ since $g_j > g_{L+1} \geq 0$.  In
particular, for $g \in (g_{j+1},g_j]$, $i \in [L]$, we have that
$\tfrac{g}{\phi_i(g)} = \min \{ \tfrac{g}{\mx_i(g_{j+1})},
\tfrac{g_{j}}{\mx_i(g_{j})} \}$ if $\mx_i(g_{j+1}) > 0$ and
$\tfrac{g}{\phi_i(g)} = \tfrac{g_j}{\mx_i(g_j)}$ otherwise. Thus,
$\frac{g}{\phi_i(g)}$ is also continous piecewise linear with at most $2L$
pieces on $(\underline g, \overline g]$.

We now prove~\eqref{eq:primal-polar},~\eqref{eq:dual-polar}. For $j \in [L]$,
$g \in [\underline g, \overline g]$, let 
\[
l^j_i(g) \coloneqq \left(\mx_i(g_{j+1})-g_{j+1}\frac{\mx_i(g_j)-\mx_i(g_{j+1})}{g_j-g_{j+1}}\right) +
\frac{\mx_i(g_j)-\mx_i(g_{j+1})}{g_j-g_{j+1}}g := a_j + b_j g,
\]
denote the linear interpolation of $\mx_i$ through $g_j$ and $g_{j+1}$. By
concavity, monotonicity of $\mx_i$ and $0 \leq g_{j+1} < g_j$, the
intercept $a_j$ and slope $b_j$ of $l^j_i$ satisfy
\begin{align}
0 \leq \mx_i(0) \leq l^j_i(0) &= a_j = \mx_i(g_{j+1})-g_{j+1}\frac{\mx_i(g_j)-\mx_i(g_{j+1})}{g_j-g_{j+1}} \leq \mx_i(g_{j+1}), \label{eq:a-ineq} \\
0 \leq \frac{\mx_i(g_j)-\mx_i(g_{j+1})}{g_j-g_{j+1}} &= b_j \leq \frac{\mx_i(g_j)-\mx_i(0)}{g_j-0} \leq \frac{\mx_i(g_j)}{g_j}\, . \label{eq:b-ineq}
\end{align}
For $g \in (g_{j+1},g_j]$, then by~\eqref{eq:a-ineq} and~\eqref{eq:b-ineq} we
get that
\begin{align}
\phi_i(g)  
&= \max \left\{l^j_i(g_{j+1}), g\frac{l^j_i(g_j)}{g_j} \right\} \\
&\leq \max \left\{l^j_i(g_{j+1})+b_j(g_j-g_{j+1}), \left(1-\frac{g}{g_j}\right) l^j_i(0) + \frac{g}{g_j} l^j_i(g_j)\right\} = l^j_i(g) \nonumber \\
&= a_j + g b_j \leq \mx_i(g_{j+1}) + g \frac{\mx_i(g_j)}{g_j}
\leq 2 \phi_i(g). \label{eq:st-pol}
\end{align}
Since $[(g_j, \mx_i(g_j)), (g_{j+1}, \mx_i(g_{j+1}))] \subset
{\mnbp_{i}}(1-\alpha)$ by construction, for $g \in
(g_{j+1},g_j]$ we have that  
\[
\mx_i(g) \geq l^j_i(g) \geq \phi_i(g) \geq \frac{1}{2} l^j_i(g) \geq \frac{\alpha}{2} \mx_i(g).
\]
Since the above holds for all $j \in [L]$, this
proves~\eqref{eq:primal-polar}. To prove~\eqref{eq:dual-polar}, we
combine~\eqref{eq:primal-polar} with \Cref{thm:mcp-central} as follows: for
$g \in (\underline g, \overline g]$,
\[
\ms_i(g) = \frac{\mx_i(g) \ms_i(g)}{\mx_i(g)} \geq \frac{g}{\mx_i(g)} \geq \frac{\alpha g}{2 \phi_i(g)} \geq \frac{\alpha g}{2 \mx_i(g)} = \frac{\alpha \ms_i(g)}{2} \frac{g}{\mx_i(g) \ms_i(g)} \geq \frac{\alpha \ms_i(g)}{4}.
\]
From here, the moreover $\SLC^{\mathrm{d}}_{1-\alpha/4,i}(\underline g,
\overline g) \leq 2 \SLC^{\mathrm{p}}_{1-\alpha,i}(\underline g, \overline
g)$ follows immediately from~\eqref{eq:dual-polar} and the fact that
$\frac{\alpha g}{2 \phi_i(g)}$ is continuous piecewise linear with a number of pieces bounded by
$2\SLC^{\mathrm{p}}_{1-\alpha,i}(\underline g, \overline g)$ on
$(\underline g, \overline g]$ (note that we can extend $\frac{\alpha g}{2
\phi_g(g)}$ to $[\underline g, \overline g]$ by continuity).
\end{proof}

We are now ready to prove \Cref{thm:delta-slc}; we restate it here for convenience.
\deltaslc*

\begin{proof}
We only prove the statement for the primal straight-line complexity bound;
the proof for the dual bound follows analogously.

Let $\underline g\coloneqq n\mu_1$ and $\overline g\coloneqq n\mu_0$ denote
the gaps at $\mu_1$ and $\mu_0$. For $i \in [n]$, apply \Cref{lem:dual-slc}
on $\mx_i$ on $[\overline g, \underline g]$ to get $\overline g = g_{i,1}
> \dots > g_{i,L_i+1} = \underline g$, $L_i = \SLC^{\mathrm{p}}_{\theta,
i}(\underline g, \overline g)$, and the corresponding piecewise linear
function $\phi_i(\cdot)$ satisfying $\mx_i(g) \geq \phi_i(g) \geq
\tfrac{1-\theta}{2} \mx_i(g)$, $\forall g \in [\underline g, \overline g]$.
For $j \in [L_i]$, define $g_{i,j}^+ \coloneqq \tfrac{g_{i,j}
\mx_i(g_{i,j+1})}{\mx_i(g_{i,j})} \in [g_{i,j+1},g_{i,j}]$, which we note satisfies $\phi_i(g_{i,j}^+) = \max \left\{\mx_i(g_{i,j+1}), \frac{g_{i,j}^+
\mx_i(g_{i,j})}{g_{i,j}}\right\} = \frac{g_{i,j}^+ \mx_i(g_{i,j})}{g_{i,j}} = \mx_i(g_{i,j+1})$. From here, let $G_i \coloneqq \{g_{i,j}: i \in [L_i+1]\} \cup \{g_{i,j}^+: i \in [L_i]\}$
denote the breakpoints of the function $\phi_i$. For a closed interval $[a,b]
\subseteq [\underline g, \overline g]$ with $a <
b$, we say that $[a,b]$ is type (B) for $\phi_i$ if $\phi_i$ is constant on
$[a,b]$ and is type (N) for $\phi_i$ if $\phi_i(g') = \frac{g'}{g}
\phi_i(g)$, $\forall ~ g',g \in [a,b], g' \leq g$ and $g > 0$.  By construction of $\phi_i$
and $G_i$, if $G_i \cap (a,b) = \emptyset$, then $[a,b]$ must be either type
(B) or type (N) for $\phi_i$ (note that these options are mutually exclusive).

Define $G = \cup_{i=1}^n G_i$ to be the union of the breakpoints of the
functions $\phi_i, i \in [n]$. Let $\underline g = g_T < \dots < g_0 =
\overline g$ satisfy $G = \{g_i: i \in \{0,\dots,T\}\}$. Since $\underline g
\in G_i$, $\forall i \in [n]$, we have that 
\begin{equation}
T = |G|-1 \leq \sum_{i=1}^n (|G_i|-1) \leq 2 \sum_{i=1}^n L_i
\coloneqq 2 \sum_{i=1}^n  \SLC^{\mathrm{p}}_{\theta, i}(\underline g, \overline g). \label{eq:it-amort-bnd-1}
\end{equation}

Define $(B^{(0)},N^{(0)}) \coloneqq ([n],\emptyset)$, and for
$k \in [T]$, define $(B^{(k)},N^{(k)})$ by 
\begin{align*}
B^{(k)} &= \{i \in [n]: [g_k,g_{k-1}] \text{ is type (B) for } \phi_i \}, \\
N^{(k)} &= \{i \in [n]: [g_k,g_{k-1}] \text{ is type (N) for } \phi_i \}. 
\end{align*}
By construction of $G$, we have $(g_k,g_{k-1}) \cap G_i = \emptyset$, $i \in
[n]$, and hence $(B^{(k)},N^{(k)})$ partitions $[n]$. The sum of partition differences then satisfies the bound
\begin{align}
\sum_{k=1}^T |N^{(k)} \Delta N^{(k-1)}| &= \sum_{i=1}^n \sum_{k=1}^T |1[i
\in N^{(k)}]-1[i \in N^{(k-1)}]| \leq \sum_{i=1}^n \sum_{k=1}^T 1[g_{k-1}
\in G_i] \nonumber \\
&= \sum_{i=1}^n (|G_i|-1) \leq 2 \sum_{i=1}^n  \SLC^{\mathrm{p}}_{\theta, i}(\underline g, \overline g). \label{eq:it-amort-bnd-2}
\end{align}

Define $h_i \coloneqq g_i/n$, $i \in \{0,\dots,T\}$, where we have $h_0
= g_0/n = \overline g/n = \mu_0$ and $h_T = g_T/n = \underline g/n = \mu_1$.
For $k \in [T]$, we claim that the central path segment
$\CP[h_k,h_{k-1}]$ is $\gamma = \frac{1-\theta}{4n}$-polarized with
polarization partition $(B^{(k)},N^{(k)})$. Together with the claim, the
bounds~\eqref{eq:it-amort-bnd-1} and~\eqref{eq:it-amort-bnd-2} directly imply theorem. 

We now prove the claim for $k \in [T]$. For $\mu \in [h_k,h_{k-1}] =
[g_k/n,g_{k-1}/n]$ and $i \in B^{(k)}$, by \Cref{lem:cp-max-cp} and
the guarantees on $\phi_i$, we have that  
\begin{equation}
x\cp_i(\mu) \geq \frac{1}{2n} \mx_i(n \mu)
\geq \frac{1}{2n} \phi_i(n\mu) = \frac{1}{2n} \phi_i(g_{k-1}) \geq \frac{1-\theta}{4n} \mx_i(g_{k-1}) \geq \frac{1-\theta}{4n} x_i\cp(h_{k-1}),
\label{eq:primal-polarization}
\end{equation}
where in the first equality we used that $[g_k,g_{k-1}]$ is type (B) for
$\phi_i$. Similarly, for $i \in N^{(k)}$, we have that
\begin{align}
x\cp_i(\mu) &\leq  \mx_i(n \mu) \leq \frac{2}{1-\theta} \phi_i(n \mu) = \frac{2}{1-\theta} \frac{\mu}{h_{k-1}} \phi_i(n h_{k-1}) \nonumber \\ & \leq \frac{2}{1-\theta} \frac{\mu}{h_{k-1}} \mx_i(n h_{k-1}) \leq \frac{4 n}{1-\theta} \frac{\mu}{h_{k-1}} x_i\cp(h_{k-1}), \label{eq:dual-polarization}
\end{align}
where in the first equality we used that $[g_k,g_{k-1}]$ is type (N) for
$\phi_i$. By \Cref{rem:pol-eq}, the
inequalities~\eqref{eq:primal-polarization} and~\eqref{eq:dual-polarization}
imply that $\CP[h_k,h_{k-1}]$ is $\frac{1-\theta}{4n}$-polarized with
polarization partition $(B^{(k)},N^{(k)})$ as desired. 
\end{proof}
  \section{The Trust Region Step}
\label{sec:trust-region}

In this section, we prove important properties for the trust-region step of
Lan, Monteiro and Tsuchiya~\cite{LMT09}, as defined in the
program~\eqref{eq:trust-region} in the introduction. The estimates proved
here will play a crucial role in the analysis of the IPM, and correspond to
refinements of existing estimates in~\cite{LMT09} and~\cite{MonteiroT05} that
are more adapted to our IPM.  

Our first estimate, given in \Cref{sec:step-tr} is a general relationship
between the parameters of trust-region directions and the achievable
step-length guarantees one can derive from them. The second main result,
given in \Cref{sec:ass-part-as}, regards the computation of the optimal
partition $(B,N)$ to use when computing trust-region directions. We show that
if a suitably good trust-region direction exists with respect to some
partition $(B,N)$, then this partition can be read off from the coordinates
of the affine scaling direction (which implies that it is unique). We prove
this by showing that the affine scaling direction must be close to any good
enough trust region direction. This is analoguous to a result
of~\cite{MonteiroT05}, who showed that the same statement holds if we replace
the trust-region direction by the layered least squares direction. 

We dub the partition induced by the affine scaling direction the
\emph{associated partition} (in~\cite{LMT09} this was called the AS
bipartition). With respect to our IPM, the associated partition will be used
to identify the current polarization partition. 

\subsection{Step-Length Estimates for Trust Region Directions}
\label{sec:step-tr}

We now formally link the properties of trust-region program solutions to the
achievable step-lengths in \Cref{prop:step-size}. A similar estimate is
proved in~\cite[Lemma 2.7]{LMT09}. 

\begin{proposition}
\label{prop:trust-to-step-size}
Let $x,s \in \R^n_{++}$ satisfy $\norm{x s/\mu-\1_n} \leq 1/6$, where $\mu
\coloneqq \gap(x,s)$, and let $B \cup N = [n]$ be a partition. For $\Delta
x, \Delta s \in \R^n$, assume that
\begin{equation}
\norm{(\ndx_B,\nds_N)} \leq \delta \leq 1/6\, , \quad \quad
\norm{(\error_B+\nds_B,\error_N+\ndx_N)} \leq \eps \leq 1/6, \label{eq:trust-bnd-ass}
\end{equation}
where $\error \coloneqq \sqrt{xs/\mu}$, $(\ndx,\nds) \coloneqq (\nx^{-1} \Delta x, \ns^{-1} \Delta s)$ and $(\nx,\ns) \coloneqq (x \error^{-1}, s \error^{-1})$. Then,
\begin{align}
\norm{\frac{\Delta x \Delta s}{\mu}} &\leq \sqrt{2} \delta, \label{eq:dx-ds-bnd} \\  
\norm{\frac{(x+\Delta x)(s+\Delta s)}{\mu}} & \leq \sqrt{2} \eps. \label{eq:rx-rs-bnd}    
\end{align}
\end{proposition}
\begin{proof}
To derive~\eqref{eq:dx-ds-bnd} and~\eqref{eq:rx-rs-bnd}, we prove the following more general statements:
\ifdefined\SIAMversion
\begin{align}
\norm{\frac{\Delta x \Delta s}{\mu}}^2 \leq &\norm{\ndx_B}^2\left(1+\tfrac{1}{12}+\norm{\error_B+\nds_B}\right)^2 + \nonumber \\ & \norm{\nds_N}^2\left(1+\tfrac{1}{12}+\norm{\error_N+\ndx_N}\right)^2, \label{eq:dx-ds-gen-bnd} \\  
\norm{\frac{(x+\Delta x)(s+\Delta s)}{\mu}}^2 \leq &\norm{\error_B+\nds_B}^2\left(1+\tfrac{1}{12}+\norm{\ndx_B}\right)^2 + \nonumber \\ &\norm{\error_N+\ndx_N}^2\left(1+\tfrac{1}{12}+\norm{\nds_N}\right)^2. \label{eq:rx-rs-gen-bnd}    
\end{align}
\else
\begin{align}
\norm{\frac{\Delta x \Delta s}{\mu}}^2 \leq &\norm{\ndx_B}^2\left(1+\tfrac{1}{12}+\norm{\error_B+\nds_B}\right)^2 + \norm{\nds_N}^2\left(1+\tfrac{1}{12}+\norm{\error_N+\ndx_N}\right)^2, \label{eq:dx-ds-gen-bnd} \\  
\norm{\frac{(x+\Delta x)(s+\Delta s)}{\mu}}^2 \leq &\norm{\error_B+\nds_B}^2\left(1+\tfrac{1}{12}+\norm{\ndx_B}\right)^2 + \norm{\error_N+\ndx_N}^2\left(1+\tfrac{1}{12}+\norm{\nds_N}\right)^2. \label{eq:rx-rs-gen-bnd}    
\end{align}

\fi
To derive~\eqref{eq:dx-ds-bnd}, we combine~\eqref{eq:dx-ds-gen-bnd} and the assumed bounds~\eqref{eq:trust-bnd-ass} as follows:
\begin{align*}
\norm{\frac{\Delta x \Delta s}{\mu}}^2 &\leq
\norm{\ndx_B}^2\left(1+\tfrac{1}{12}+\norm{\error_B +\nds_B}\right)^2 +
\norm{\nds_N}^2\left(1+\tfrac{1}{12}+\norm{\error_N + \ndx_N}\right)^2 \\
&\leq \norm{(\ndx_B,\nds_N)}^2(1+\tfrac{1}{12}+\eps)^2 
\leq \delta^2(1+\tfrac{1}{12}+\tfrac{1}{6})^2 \leq 
\delta^2(\tfrac{5}{4})^2 \leq 2 \delta^2. 
\end{align*}
Similarly,~\eqref{eq:rx-rs-bnd} is derived by combining~\eqref{eq:rx-rs-gen-bnd} and~\eqref{eq:trust-bnd-ass} as 
\ifdefined\SIAMversion
\begin{align*}
\norm{\frac{(x+\Delta x)(s+\Delta s)}{\mu}}^2 &\leq \norm{\error_B+\nds_B
}^2(1+\tfrac{1}{12}+\norm{\ndx_B})^2 + \\ &\quad \, \norm{\error_N+
\ndx_N}^2(1+\tfrac{1}{12}+\norm{\nds_N})^2 \\ 
&\leq \eps^2(1+\tfrac{1}{12}+\delta)^2 \leq 2 \eps^2.
\end{align*}
\else
\begin{align*}
\norm{\frac{(x+\Delta x)(s+\Delta s)}{\mu}}^2 &\leq \norm{\error_B+\nds_B
}^2(1+\tfrac{1}{12}+\norm{\ndx_B})^2 + \norm{\error_N+
\ndx_N}^2(1+\tfrac{1}{12}+\norm{\nds_N})^2 \\ 
&\leq \eps^2(1+\tfrac{1}{12}+\delta)^2 \leq 2 \eps^2.
\end{align*}
\fi

We now focus on the proofs of~\eqref{eq:dx-ds-gen-bnd}
and~\eqref{eq:rx-rs-gen-bnd}.

\paragraph{Proof of inequality~\eqref{eq:dx-ds-gen-bnd}} To begin, note that
\[
\norm{\frac{\Delta x \Delta s}{\mu}}^2 = \norm{\ndx \nds}^2 = \norm{\ndx_B
\nds_B}^2 + \norm{\ndx_N \nds_N}^2.
\]
For the $B$ term on the right hand side, we have that
\begin{align*}
\norm{\ndx_B \nds_B}^2 &\leq \norm{\ndx_B}^2 \norm{\nds_B}^2_\infty \\
&\leq \norm{\ndx_B}^2 \left(\norm{\error_B}_\infty+\norm{\error_B+\nds_B}_\infty\right)^2 \\
&\leq \norm{\ndx_B}^2\left(1 + \tfrac{1}{12} + \norm{\error_B+\nds_B}\right)^2,
\end{align*}
where the last inequality follows since $\norm{\error_B+ \nds_B}_\infty \leq
\norm{\error_B+ \nds_B}$ and 
\[
\norm{\error_B}_\infty \leq \norm{\error}_\infty = \max_{i \in [n]}
\sqrt{\tfrac{x_i s_i}{\mu}} \leq \sqrt{1+\tfrac{1}{6}} \leq 1+\tfrac{1}{12}.
\]
By a symmetric argument, swapping the roles of ($x$,$s$) and ($B$,$N$), we
also get that
\[
\norm{\ndx_N \nds_N}^2 \leq \norm{\nds_N}^2\left(1 + \tfrac{1}{12} + \norm{\error_N+ \ndx_N}\right)^2.
\]
Inequality~\eqref{eq:dx-ds-gen-bnd} now follows by combining for the $B$ and
$N$ bounds above.

\paragraph{Proof of inequality~\eqref{eq:rx-rs-gen-bnd}}
Similar to the above, we have that  
\begin{align*}
\norm{\frac{(x+\Delta x)(s+\Delta s)}{\mu}}^2 &= \norm{(\error+\ndx)(\error+\nds)}^2 \\ &=
\norm{(\error_B + \ndx_B)(\error_B+\nds_B)}^2 + \norm{(\error_N + \ndx_N)(\error_N+\nds_N)}^2.
\end{align*}
As before, we bound the $B$ and $N$ parts separately. For the $B$ part, we have that 
\begin{align*}
\norm{(\error_B+\ndx_B)(\error_B+\nds_B)}^2 &\leq \norm{\error_B+\nds_B}^2 \norm{\error_B+\ndx_B}^2_\infty \\
&\leq \norm{\error_B+\nds_B}^2 \left(\norm{\error_B}_\infty+\norm{\ndx_B}_\infty\right)^2 \\
&\leq \norm{\error_B+\nds_B}^2(1 + \tfrac{1}{12} + \norm{\ndx_B})^2,
\end{align*}
where the last inequality from identical to the first part. Again, by a
symmetric argument,
\[
\norm{(\error_N+\ndx_N)(\error_N+\nds_N)}^2 \leq \norm{\error_N+\ndx_N}^2(1 + \tfrac{1}{12} + \norm{\nds_N})^2.
\] 
Inequality~\eqref{eq:rx-rs-gen-bnd} now follows by combining for the $B$ and
$N$ bounds above.
\end{proof}

\subsection{The Associated Partition}
\label{sec:ass-part-as}

The trust-region step is applicable for any non-trivial partition $B\cup
N=[n]$ and $z \in\cal N^2(\beta)$. Following \cite{LMT09}, we choose a
natural partition derived from the size of normalized coordinates of the
affine scaling direction:

\begin{definition}[Associated partition] \label{def:associated_partition}
For $z=(x,s)\in\cal N^2(\beta)$, let $(\Delta x^\as,\Delta s^\as)$ be the affine scaling step as in \eqref{eq:aff-compute}. Let us define the associated partition $\wt B_z\cup \wt N_z=[n]$ as
\[
\wt B_z\coloneqq \left\{i : \left|\frac{\Delta x^\as_i}{x_i}\right| < \left|\frac{\Delta s^\as_i}{s_i}\right|\right\}\, , \quad  \wt N_z \coloneqq  [n] \setminus \wt B_z\, .
\]
\end{definition}
The affine scaling step is the canonical candidate for an improving direction. Namely, for each $i\in \wt B_z$ the variable $s_i$ decreases at a faster rate than $x_i$, and vice versa for $i\in \wt N_z$. 

As we show below, as long as the trust-region program admits a sufficiently
good solution with respect to the polarization partition $(B,N)$, then the
affine scaling step is close to this trust-region direction and has
associated partition $(B,N)$. This implies that the optimal choice of
partition is in fact unique under the assumption that a long steps exists.

\begin{lemma}
Let $z = (x,s) \in \cal N^2(\beta)$, $\beta \in (0,1/6]$, $\mu \coloneqq \gap(z)$, $B \cup N = [n]$ be a partition. Assume that there exists $(\Delta x,\Delta s) \in W \times W^\perp$ satisfying
\begin{equation}
\norm{(\ndx_B,\nds_N)} \leq \delta \leq 1/30,
\quad \quad \norm{(\error_N+\ndx_N,\error_B+\nds_B)} \leq \epsilon \leq 1/30,
\label{eq:trust-ass-bnd-2}
\end{equation}
where $\error \coloneqq \sqrt{xs/\mu}$, $(\ndx,\nds) \coloneqq (\nx^{-1}
\Delta x, \ns^{-1} \Delta s)$, $(\nx,\ns) \coloneqq (x \error^{-1}, s
\error^{-1})$.
Then, the affine scale direction $(\Delta x^\as, \Delta s^\as)$ at $z$ satisfies
\begin{equation}
\norm{\frac{\Delta x^\as \Delta s^\as}{\mu}} = \norm{\frac{(x+\Delta x^\as)(s+\Delta s^\as)}{\mu}} \leq 3.5(\delta + \eps), \label{eq:as-bnd}
\end{equation}
and the associated partition at $z$ satisfies $(\Balg_z,\Nalg_z) = (B,N)$.
\label{lem:as-vs-tr}
\end{lemma}
\begin{proof}
Recall that the affine scaling step $(\Delta x^\as,\Delta s^\as)$ is defined by
\[
\ndx^\as + \nds^\as = -\error,
\]
where $(\ndx^\as,\nds^\as) \coloneqq (\nx^{-1} \Delta x^\as, \ns^{-1} \Delta
s^\as) \in \nx^{-1} W \times \ns^{-1} W^\perp$ form an orthogonal
decomposition of $-\error$. By orthogonality, we therefore have that
\[
\norm{\ndx^\as-\ndx}^2 + \norm{\nds^\as-\nds}^2 = \norm{\error+\ndx+\nds}^2.
\]
By the triangle inequality,
\begin{align*}
\norm{\error_B+\ndx_B+\nds_B}^2 &\leq 2\norm{\ndx_B}^2+2\norm{\error_B+\nds_B}^2\, , \\
\norm{\error_N+\ndx_N+\nds_N}^2 &\leq 2\norm{\nds_N}^2+2\norm{\error_N+\ndx_N}^2\, .
\end{align*}
Therefore, by~\eqref{eq:trust-ass-bnd-2},
\begin{align*}
\norm{\error+\ndx+\nds}^2 &= \norm{\error_B+\ndx_B+\nds_B}^2+\norm{\error_N+\ndx_N+\nds_N}^2 \\ &\leq 2(\norm{\ndx_B}^2+\norm{\nds_N}^2+\norm{\error_N+\ndx_N}^2+\norm{\error_B+\nds_B}^2) \\ &\leq 2(\delta^2+\eps^2).  
\end{align*}
In particular,
\begin{equation}
\begin{aligned}
\norm{(\ndx^\as_B,\nds^\as_N)} &\leq \sqrt{2(\delta^2+\eps^2)} + \norm{(\ndx_B,\nds_N)} \leq \sqrt{2}(\delta+\eps)+\delta \nonumber \\ & \leq (\sqrt{2}+1)(\delta+\eps) \leq \frac16\, , \label{eq:transfer-dx-ds} \\
\norm{(\error_N+\ndx^\as_N,\error_B+\nds^\as_B)} &\leq \sqrt{2(\delta^2+\eps^2)} + \norm{(\error_N+\ndx_N,\error_B+\nds_B)} \nonumber \\ &\leq (\sqrt{2}+1)(\delta+\eps) \leq \frac16\,.
\end{aligned}
\end{equation}
Therefore by \Cref{prop:trust-to-step-size}, since affine scaling satisfies $\Delta x^\as \Delta s^\as = (x+\Delta x^\as)(s+\Delta s^\as)$, we get that
\[
\norm{\frac{\Delta x^\as \Delta s^\as}{\mu}} = \norm{\frac{(x+\Delta x^\as)(s+\Delta s^\as)}{\mu}} \leq \sqrt{2}(\sqrt{2}+1)(\delta + \eps) \leq 3.5(\delta+\eps).
\]
For the last part, notice that $|\Delta x^\as_i/x_i| < |\Delta s^\as_i/s_i|
\Leftrightarrow |\ndx^\as_i| < |\nds^\as_i|$, $\forall i \in [n]$. By
\Cref{prop:xi}, since $\beta \in (0,1/6]$, we have that $\error
\geq \sqrt{1-\beta}\1_n \geq (1-1/6)\1_n$. For $i \in B$, using that
$(\sqrt{2}+1)(\delta+\eps) \leq 3(\delta+\eps) \leq 1/3$ together with~\eqref{eq:transfer-dx-ds} we see that
\[
|\ndx^\as_i| \leq 3(\delta+\eps) < (1-1/6)-3(\delta+\eps) \leq \error_i-|\ndx^\as_i| \leq |\error_i+\ndx^\as_i| = |\nds^\as_i|.
\]
By a symmetric argument, we also have $|\nds^\as_i| < |\ndx^\as_i|$, for $i
\in N$. Therefore, $(\Balg_z,\Nalg_z) = (B,N)$.
\end{proof}

\begin{remark}
\label{rem:as-vs-tr}
Given that the affine scaling step is always close to a trust-region direction, as long as the direction yields a sufficiently long step,
one may wonder why the trust-region
direction, or its approximate version the subspace LLS direction (defined in the next section), is even
needed. At a quantitative level, this comes from the fact that the
trust-region direction assumed to exist in~\Cref{lem:as-vs-tr} would in fact
allow us to decrease the gap by a factor $\Theta(\eps)$ (see
\Cref{prop:step-size}), whereas the ``nearby'' affine scaling would only
achieve $\Theta(\delta+\eps)$, which can be arbitrarily worse. Indeed, recall
that in~\eqref{eq:trust-region} we only enforce $\delta = O(\beta)$, that is,
proportional to the neighborhood size, while trying to minimize $\eps$. 
\end{remark}

\section{The Subspace LLS Direction and Cheap Lift Subspaces}~\label{sec:lifting-slls}

In this section, we formally define the subspace layered least squares (SLLS)
steps and show how to compute the cheap lift subspaces as described in the
Introduction. In \Cref{sec:slls}, we formally define the SLLS step and overview
its basic computational aspects. Then, in \Cref{sec:lifting_maps}, we introduce the
lifting map and operator and overview their basic properties and give an
algorithm to compute them. In \Cref{sec:cl}, we define cheap lift subspaces,
that will be used for taking subspace LLS steps, as well as the lifting
operators, and show how to construct cheap lift subspaces from approximate
singular subspaces (as defined in \Cref{def:apx-sing-sub}). In
\Cref{sec:comp-cl}, we show how to compute cheap lift subspaces of
approximately maximum dimension using an approximate singular value
decomposition (see the $\SVDA$-\text{SVD} problem). 

\subsection{The Subspace Layered Least Squares Direction}
\label{sec:slls}

We now introduce a new update direction, called the \emph{subspace layered
least squares} (SLLS) update direction, which will allow us to accelerate our
IPM over long straight parts of the central path. As discussed in the
previous sections, such straight segments of the central path are in
fact \emph{polarized} according to a partition $B \cup N = [n]$
(\Cref{def:polarized}). Within this segment, the primal variables indexed
by $N$ scale down linearly with respect to the parameter $\mu$, while the primal
variables in $B$ will remain roughly constant. For the dual, the situation is
reversed, the variables in $B$ scale down while the variables in $N$ remain
roughly constant.

As discussed, in the Introduction, the ``optimal'' step direction that mimics
the polarization pattern is the \emph{trust region step} of Lan, Monteiro and
Tsuchiya~\eqref{eq:trust-region}. To be able to suitably approximate this
step in strongly polynomial time, we require the SLLS update direction, which
we formally define below.

\begin{definition}[Subspace Layered Least Squares direction]\label{def:subspace_lls}
Let $z \coloneqq (x,s) \in \cal N^2(\beta)$, $\mu = \gap(z)$, $B \cup N = [n]$ be a non-trivial partition. Let $V \subseteq W$, $U \subseteq W^\perp$ be linear subspaces satisfying 
$\dim(\pi_N(V)) = \dim(V)$, and $\dim(\pi_B(U)) = \dim(U)$. The \emph{Subspace LLS (SLLS)} update direction $(\Delta x^\slls, \Delta s^\slls) \in W \times W^\perp$ at $z$ with respect to $(B,N,V,U)$ is defined as
\begin{align}
\Delta x^\slls \coloneqq \argmin_{\delta \in V} \norm{\frac{x_N+\delta_N}{\Xerr_N}}^2, \label{eq:primal-slls} \\
\Delta s^\slls \coloneqq \argmin_{\delta \in U} \norm{\frac{s_B+\delta_B}{\Serr_B}}^2, \label{eq:dual-slls} 
\end{align}
where $\Xerr \coloneqq \sqrt{x\mu/s}$ and $\Serr \coloneqq \sqrt{s\mu/x}$. 
\end{definition}
\begin{remark}
\label{rem:subspace_lls}
Note that $\Delta x^\slls, \Delta s^\slls$ are indeed well-defined, as our
assumption that $\dim(\pi_N(V)) = \dim(V)$ and $\dim(\pi_B(U)) = \dim(U)$
allows  to uniquely determine $\Delta x^\slls, \Delta s^\slls$ from their
coordinates in $N$ and $B$ respectively. In particular, if $N = \emptyset$,
then $V = \{\0_n\}$ and hence $\Delta x^\slls = \0_n$. Similarly, if $B =
\emptyset$, then $U = \{\0_n\}$ and $\Delta s^\slls = \0_n$. 

As with affine scaling, one can interpret the SLLS step directions in
terms of orthogonal projections. Using that $\error \coloneqq \sqrt{xs/\mu} =
x/\Xerr = s/\Serr$, in direct analogy with~\eqref{eq:aff-compute} in
\Cref{sec:aff-lls}, $(\nx^{-1} \Delta x)_N$ and $(\ns^{-1} \Delta s)_B$ are
respectively the orthogonal projections of $-\error_N$ and $-\error_B$ onto
$\pi_N(\nx^{-1} V)$ and $\pi_B(\ns^{-1} U)$. In contrast, whereas the affine
scaling direction minimizes the norm of the primal-dual residual vector
$(\error+\ndx,\error+\nds)$, the subspace LLS direction instead tries to
minimize the norm of $(\error_N+\ndx_N,\error_B+\nds_B)$. 
\end{remark}

Similarly to the affine scaling and corrector direction, the SLLS direction can be computed in
strongly polynomial time, given an appropriate representation of the
subspaces. The formulas for the subspace LLS step directions are given in the next
proposition. These are computed by solving the linear system which sets the
gradient of the corresponding quadratic optimization problems to zero.

\begin{proposition}[Subspace LLS Step Formulas]
\label{prop:lls-step-formula}
Let $z=(x,s) \in \R^{2n}_{++}$ be an iterate, $B \cup N = [n]$ be a non-trivial partition, and $\MP \in \R^{n \times m_p},
\MD \in \R^{n \times m_d}$ be matrices where $V \coloneqq \im(\MP)$, $U
\coloneqq \im(\MD)$ satisfy $\dim(V) = \dim(\pi_N(V))$ and $\dim(U)=\dim(\pi_B(U)) $. Then, the SLLS direction $(\Delta x^\slls, \Delta s^\slls)$ at $z$
with respect to $(B,N,U,V)$ can be computed in strongly polynomial time as
follows: 
\begin{equation*}
\begin{aligned}
\Delta x^\slls &= -\MP\left((\MP_{N,\mdot})^\T \diag(s_N/x_N) \MP_{N,\mdot}\right)^+ \left(\MP_{N,\mdot}\right)^\T s_N, \\ 
\Delta s^\slls &= -\MD \left((\MD_{B,\mdot})^\T \diag(x_B/s_B) \MD_{B,\mdot}\right)^+ \left(\MD_{B,\mdot}\right)^\T x_B.
\end{aligned}
\end{equation*}
In the above, by convention, $\Delta x^\slls = \0_n$ if $N = \emptyset$ and
$\Delta s^\slls = \0_n$ if $B = \emptyset$.
\end{proposition}
\begin{proof}
The strong polynomiality follows directly from the fact that computing
pseudoinverses is strongly polynomial. We thus focus on correctness. We
prove correctness of the formula for $\Delta x^\slls$ as the analysis for
$\Delta s^\slls$ is symmetric. The formula is clearly correct if $N = \emptyset$ by convention, so assume $N \neq \emptyset$. Let 
\[
z_x \coloneqq -\left((\MP_{N,\mdot})^\T \diag(s_N/x_N) \MP_{N,\mdot}\right)^+ \left(\MP_{N,\mdot}\right)^\T s_N \in \R^{m_p}.
\]
Recalling that $V \coloneqq \im(\MP)$, we have that $\Delta x^\slls = \MP z_x \in V$. Thus, letting $g_x(z) = \norm{\frac{x_N + \MP_{N,\mdot} z}{\nx_N}}^2$, it suffices to check that
\begin{align}
z_x \in \argmin_{z \in \R^{m_p}} g_x(z) &\Leftrightarrow \nabla g_x(z_x) = \0_{m_p}
\Leftrightarrow 2 \left(\MP_{N,\mdot}\right)^\T\left(\frac{x_N + \MP_{N,\mdot} z_x}{\nx_N^2}\right) = \0_{m_p} \nonumber \\ & \Leftrightarrow \left(\MP_{N,\mdot}\right)^\T\left(s_N + \diag(s_N/x_N) \MP_{N,\mdot} z_x\right) = \0_{m_p}, \label{eq:optimality-condition}
\end{align}
where we have used convexity of $g_x$ and $\nx \coloneqq \sqrt{x\mu/s}$ where
$\mu \coloneqq \gap(x,s) > 0$. Note that while $g_x$ need not have a unique
minimizer (\ie, the columns of $\MP$  may be linearly dependent), the
condition $\dim(V) = \dim(\pi_N(V)) = \rank(\MP_{N,\mdot})$ indeed ensures
that $\Delta x^\slls = \MP z_x$ is the unique minimizer
to~\eqref{eq:primal-slls} assuming $z_x$ minimizes $g_x$. From here, it suffices to check that $z_x$ satisfies~\eqref{eq:optimality-condition}, and hence minimizes $g_x$. For this purpose, letting $\mP = \left(\MP_{N,\mdot}\right)^\T\diag(s_N/x_N)\MP_{N,\mdot}$, by \Cref{prop:pseudoinverse-projection} part~\eqref{eq:pseudo-proj} and $x,s > \0_n$
\begin{align*}
\left(\MP_{N,\mdot}\right)^\T s_N &= \Pi_{\im\left(\left(\MP_{N,\mdot}\right)^\T \right)} \left(\MP_{N,\mdot}\right)^\T s_N = \Pi_{\im\left(\mP\right)} \left(\MP_{N,\mdot}\right)^\T s_N \\
&= \mP \mP^+ \left(\MP_{N,\mdot}\right)^\T s_N,
\end{align*}
where we have used that $\im\left(\left(\MP_{N,\mdot}\right)^\T \right) = \im\left(\left(\MP_{N,\mdot}\right)^\T\diag(s_N/x_N)\MP_{N,\mdot}\right)$.
\end{proof}

\subsection{Lifting Maps and Operators}\label{sec:lifting_maps}
The algorithm in \Cref{sec:slls-alg} and the analysis in \Cref{sec:running-time-analysis} will rely crucially
on the properties of \emph{lifting maps and operators}. These will be used to compute cheap
lift subspaces for the purpose of computing subspace layered least squares
directions. These maps have appeared in many prior works on layered least
squares algorithms~\cite{Vavasis1996,Monteiro2003,DHNV-23-MP}. In this
section, we give a self-contained overview of all their main properties.
Importantly, we show their duality properties and that they can be computed
in strongly polynomial time. 

\begin{definition}\label{def:lifting-map}
Given a non-trivial partition $I \cup J = [n]$ and a subspace $W \subseteq
\R^n$, we define the lifting map $L_I^W \colon \R^I \rightarrow \R^n$ as
follows:
\begin{equation}
\label{eq:lift1}
L^W_I(x) \coloneqq \argmin \{\norm{w}: w \in W, w_I = \proj_{\pi_I(W)}(x)\}.
\end{equation}
We further define the lifting operator $\lift^W_I \colon \pi_I(W) \rightarrow \pi_J(W^\perp)$ by 
\begin{equation}
\label{eq:lift2}
\lift_I^W(x) \coloneqq (L^W_I(x))_J = \pi_J\left(L^W_I(x)\right), \forall x \in \pi_I(W) \,.
\end{equation}
Note that if $x\in \pi_I(W)$, then $w=L_I^W(x) = (\lift_I^W(x),x)$ is the minimum-norm point in $W$ with $w_I=x$.  

By convention, if $(I,J)$ is a trivial partition, we define $L_I^W := \Pi_W$
if $I=[n]$, and define $L_I^W$ to be the linear operator from
$\R^\emptyset \coloneqq \{0\}$ to $\R^n$ if $I=\emptyset$. We furthermore 
define $\lift_I^W$ to be the linear operator from $W$ to
$\R^\emptyset$ if $I=[n]$ and from $\R^\emptyset$ to $W^\perp$ if
$I=\emptyset$. 
\end{definition}

The computation of lifting operators will be crucial to our IPM. The
following proposition, whose proof is deferred to later in the section,
gives an explicit formula for the associated matrices of the lifting map and
operator and shows how to compute them in strongly polynomial time.

\begin{proposition}[Computing Lifting Maps]
Let $\mA \in \R^{m \times n}$, $\rank(\mA) = m$, and $I \cup J = [n]$ be a
partition. Then, for $W = \ker(\mA)$, one can in strongly
polynomial time compute the associated matrices (as in \Cref{def:operator}) for the lifting map
and operator as follows:
\[
\cal M(L_I^W) = \mP_{\mdot, I} (\mP_{I,I})^+, \quad \quad \cal M(\lift_I^W) = \mP_{J,I} (\mP_{I,I})^+,
\]
where $\mP \coloneqq\mI_n - \mA^\T(\mA \mA^\T)^{-1}\mA$ is the orthogonal projection
onto $W$, i.e., the matrix associated with $\Pi_W$. \label{prop:lift-compute}
\end{proposition}

Towards proving the above, we first show that $L_I^W$ and $\lift_I^W$ are
both well-defined linear operators, which is not directly obvious from the
definitions. Linearity of these operators, as well as other key properties,
is proven in \Cref{lem:lift-operator} and \Cref{lem:lift-map} below.

\begin{lemma}
\label{lem:lift-operator}
For a linear subspace $W \subseteq \R^n$, and a partition $I \cup
J = [n]$, then the function $L_I^W \colon \R^I \to \R^n$ as defined in \Cref{def:lifting-map} is a
linear operator. In particular, for $x \in \R^I$, $w = L_I^W(x)$ is the
unique solution to the following linear system:
\begin{equation}
\begin{aligned}
w &\in \Pi_W(\R^n_I) \, ,\\
\pi_I(w) &= \Pi_{\pi_I(W)}(x).
\end{aligned}
\end{equation}
Furthermore, $\pi_I(W) =
\pi_I(\Pi_W(\R^n_I))$. \end{lemma}
\begin{proof}
If $(I,J)$ is a trivial partition, the characterization follows directly by our convention, so we may assume that $(I,J)$ is non-trivial. 

By construction, the solution set of $L_I^W(x)$ is non-empty.  Furthermore,
the minimum norm solution exists and is unique by strict convexity of the
squared Euclidean norm. Thus, $w = L_I^W(x)$ is well-defined. The definition requires $w_I = \pi_I(w) = \Pi_{\pi_I(W)}(x)$; we show that subject to this, 
 $w=L_I^W(x)$ if and only if 
$w \in \Pi_W(\R^n_I)$.  By Lagrangian duality, using that the gradient of the squared norm objective is
$\nabla(\norm{w}^2)=2w$, it follows that $w$ is optimal to the program in the definition \eqref{eq:lift1} if and only if 
\begin{equation}\label{eq:lifting-system}
\begin{aligned}
w_I &= \Pi_{\pi_I(W)}(x)\, ,\\
w&\in W\, , \quad\mbox{and}\\
w &\perp (W \cap \R_J^n)\, ,
\end{aligned}
\end{equation}
where we have used that the difference of any two feasible solutions to~\eqref{eq:lift1} lives in $W \cap \R_J^n$. The last requirement can be written as  $w\in (W\cap \R_J^n)^\perp= W^\perp + \R^n_I$, where the equality follows by \Cref{prop:sum-complement-identity}.
By orthogonal decomposition, we note that $\Pi_W(\R^n_I) = W \cap (W^\perp + \R^n_I)$.

Thus, the last two requirements are equivalent to $w\in \Pi_W(\R^n_I)$, completing the proof. Further, since \eqref{eq:lifting-system} can be equivalently written as a linear system of equations, it follows that $L_I^W$ is a linear map.

We now prove the furthermore: $\pi_I(W) = \pi_I(\Pi_W(\R^n_I))$. Since
$\Pi_W(\R^n_I) \subseteq W$, we have $\pi_I (\Pi_W(\R^n_I)) \subseteq
\pi_I(W)$. To show the reverse containment $\pi_I(W)\subseteq \pi_I(\Pi_W(\R^n_I))$, take any
$x \in \pi_I(W)$. We have that $L_I^W(x) \in \Pi_W(\R^n_I)$ 
by the first part, and $\pi_I(L_I^W(x)) = x$. Therefore, $\pi_I(W) \subseteq
\pi_I(\Pi_W(\R^n_I))$, as needed.
\end{proof}

\begin{lemma}
\label{lem:lift-map}
For a linear subspace $W \subseteq \R^n$, a partition $I \cup J =
[n]$, $\lift^W_I \colon \pi_I(W) \rightarrow \pi_J(W^\perp)$ as in
\Cref{def:lifting-map} is a well-defined linear operator. Moreover,
$\im(\lift^W_I) = \pi_J(W) \cap \pi_J(W^\perp)$, $\ker(\lift^W_I) = \pi_I(W
\cap \R^n_I)$ and $\im(\adj(\lift^W_I)) = \pi_I(W) \cap \pi_I(W^\perp)$.
\end{lemma}

\begin{proof}
If $(I,J)$ is a trivial partition, the lemma follows trivially by the definition of $L_I^W$, so we may assume that $(I,J)$ is non-trivial. We first prove the moreover statement $\im(\lift^W_I) = \pi_J(W) \cap
\pi_J(W^\perp)$. Assuming this, the well-definedness of $\lift^W_I \colon \pi_I(W) \rightarrow
\pi_J(W^\perp)$ follows immediately from the inclusion $\im(\lift^W_I) \subseteq
\pi_J(W^\perp)$. Furthermore, linearity of $\lift^W_I$ follows from linearity of $L_I^W$, proved in \Cref{lem:lift-operator}, and the definition $\lift^W_I(x) = \pi_J(L_I^W(x))$, $\forall x \in \pi_I(W)$. 

We start by showing the inclusion $\im(\lift^W_I) \subseteq
\pi_J(W) \cap \pi_J(W^\perp)$.  For $x \in \pi_I(W)$, by \Cref{lem:lift-operator}
we have that $L_I^W(x) \in \Pi_W(\R^n_I)$ and $\lift^W_I(x)=\pi_J(L_I^W(x))$. Since $\Pi_W(\R^n_I) = W \cap (W^\perp + \R^n_I)$, we see that
$\lift^W_I(x) \in \pi_J(\Pi_W(\R^n_I)) \subseteq \pi_J(W) \cap
\pi_J(W^\perp + \R^n_I) = \pi_J(W) \cap \pi_J(W^\perp)$, as needed. We now
show the reverse inclusion $\pi_J(W) \cap \pi_J(W^\perp) \subseteq
\im(\lift^W_I)$. Take $y \in \pi_J(W) \cap \pi_J(W^\perp)$. By definition,
there exists $w \in W$ such that $w_J = y$.  It now suffices to show that
$L^W_I(w_I) = w$, since then $\lift^W_I(w_I) = y$. Clearly $\pi_I(L^W_I(w_I))
= w_I$ since $w \in W$. By \Cref{lem:lift-operator}, it suffices to show that $w
\in \Pi_W(\R^n_I) = W \cap (W^\perp+\R^n_I )$. From here, we see that $w \in
W^\perp + \R^n_I \Leftrightarrow \pi_J(w) \subseteq \pi_J(W^\perp+\R^n_I) =
\pi_J(W^\perp)$. The inclusion now follows since $w_J = y \in \pi_J(W^\perp)$
by assumption.

We now characterize the kernel of $\lift^W_I$ by
\[
\ker(\lift^W_I) = \{x \in \pi_I(W): \lift^W_I(x)=\0_J\}
= \{x \in \pi_I(W): (x,\0_J) \in W\} = \pi_I(W \cap \R^n_I),
\]
as needed. Finally, we characterize $\im(\adj(\lift^W_I))$. By
\Cref{prop:adjoint}, we have that
\begin{align*}
\im(\adj(\lift^W_I)) &= \pi_I(W) \cap \ker(\lift^W_I)^\perp = 
\pi_I(W) \cap \pi_I(W \cap \R^n_I)^\perp \\
&= \pi_I(W) \cap \pi_I((W \cap \R^n_I)^\perp \cap \R^n_I) \\
&= \pi_I(W) \cap \pi_I((W^\perp + \R^n_J) \cap \R^n_I) 
= \pi_I(W) \cap \pi_I(W^\perp),
\end{align*}
as needed.
\end{proof}

We now have the tools to prove \Cref{prop:lift-compute}.

\begin{proof}[Proof of \Cref{prop:lift-compute}]
To begin, recall that $\mA^\T(\mA \mA^\T)^{-1} \mA$ is the orthogonal
projection onto $\im(\mA^\T)$ $ = \ker(\mA)^\perp = W^\perp$, and hence $\mP =
\mI_n - \mA^\T(\mA \mA^\T)^{-1} \mA$ is the orthogonal projection onto $W$.
Furthermore, the inverse $(\mA \mA^\T)^{-1}$ is well-defined since $\mA$ has
full row rank. 

To justify the formula for $\cal M(L_I^W)$, it suffices to show that
$L_I^W(x) = \mP_{\mdot, I}(\mP_{I,I})^+ x$ for any $x \in \R^I$ since the
input space of $L_I^W$ is $\R^I$. By \Cref{lem:lift-operator}, we have that
$w = L_I^W(x)$ is the unique solution to $w_I = \Pi_{\pi_I(W)}(x) = x$ and $w
\in \Pi_W(\R^n_I) = \mP(\R^n_I)$. By \Cref{prop:pseudoinverse-projection} part~\eqref{eq:pseudo-proj}, we
have $\mP_{I,I}(\mP_{I,I})^+ x = \Pi_{\im(\mP_{I,I})} x = \Pi_{\pi_I(W)} x$,
where the last equality is $\pi_I(W) = \pi_I(\mP(\R^n_I)) = \im(\mP_{I,I})$ in
\Cref{lem:lift-operator}.  Since by construction $\mP_{\mdot,I}(\mP_{I,I})^+ x
\in \mP(\R^n_I)$, by uniqueness we have that $w = \mP_{\mdot,I}(\mP_{I,I})^+
x$ as needed. 

We now justify the formula for $\cal M(\lift_I^W)$. The identity
$\lift^W_I(x) = \mP_{J,I}(\mP_{I,I})^+ x$, for $x \in \pi_I(W)$, follows
directly from $\lift^W_I(x) \coloneqq \pi_J(L^W_I(x))$. To conclude the
proof, it suffices to show that $\im((\mP_{J,I}(\mP_{I,I})^+)^\T) \subseteq
\im(((\mP_{I,I})^+)^\T) \subseteq \pi_I(W)$, recalling that $\lift_I^W:
\pi_I(W) \rightarrow \pi_J(W^\perp)$. Since $\ker((\mP_{I,I})^+) =
\ker((\mP_{I,I})^\T)$ by definition of the pseudoinverse, we have that $\im(((\mP_{I,I})^+)^\T) =
\im(\mP_{I,I}) = \pi_I(W)$, where the first equality is by \Cref{prop:adjoint} part~\eqref{prop:adjoint-decomposition-u}. This proves the statement.

For the strongly polynomial computability, this follows directly from the
fact that computing matrix products, inverses and pseudoinverses
(\Cref{prop:pseudo-compute}) is strongly polynomial. 
\end{proof}

To conclude this section, we give the fundamental duality relation between
lifting operators, which will be crucial to the analysis of the IPM.

\begin{lemma}
\label{lem:lift-duality}
For a linear subspace $W \subseteq \R^n$, and a partition $I \cup
J = [n]$, $\lift^{W}_I = -\adj\left(\lift^{W^\perp}_J\right)$. In particular,
$\sigma^+(\lift^{W}_I) = \sigma^+(\lift^{W^\perp}_J)$. 
\end{lemma}
\begin{proof}
If $(I,J)$ is a trivial partition, the the statement follows by our definition of $\lift_I^W$, so may assume that $(I,J)$ is non-trivial. To prove the statement, it suffices to show that for all $x \in \pi_I(W)$ that
\[
L_I^{W}(x) = \left(x,\lift_I^W(x)\right) = \left(x, -\adj\left(\lift_J^{W^\perp}\right)(x)\right).
\]
Letting $z \coloneqq \left(x, -\adj\left(\lift_J^{W^\perp}\right)(x)\right)$, by \Cref{lem:lift-operator},
it suffices to show that $z \in \Pi_W(\R^n_I)$ \newline $= W \cap (\R^n_I + W^\perp)$. To show that $z \in W$, we must show that $\pr{v}{z} = 0$, $\forall v \in W^\perp$. For $v \in W^\perp$, we see that
\begin{align*}
\pr{z}{v} &= \pr{x}{v_I} - \pr{\adj\left(\lift_J^{W^\perp}\right)(x)}{v_J}
 \\
          &= \pr{x}{v_I} - \pr{x}{\lift_J^{W^\perp}(v_J)} \quad \left(\text{ since } v_J \in \pi_J(W^\perp)~ \right) \\
          &= \pr{x}{v_I - \lift_J^{W^\perp}(v_J)}  = \pr{L_I^W(x)}{(v_I - \lift_J^{W^\perp}(v_J), 0_J)} \\
          &= \pr{L_I^W(x)}{v-L_J^{W^\perp}(v_J)}=0\, , 
\end{align*}
where the last equality follows since $L_I^W(x) \in W$ and
$v-L_J^{W^\perp}(v_j) \in W^\perp$.

The inclusion $z \in \R^n_I + W^\perp$ follows directly from $z_J =
-\adj\left(\lift_J^{W^\perp}\right)(x) \in \pi_J(W^\perp)$ by definition of
the adjoint. Thus, $\lift^W_I = -\adj\left(\lift^{W^\perp}_J\right)$ as needed. 

The equality of the positive singular values now follows from
\Cref{prop:sing-adjoint}, noting that $\sigma^+(\lift^W_I) =
\sigma^+(-\adj(\lift^W_I)) = \sigma^+(\lift^{W^\perp}_J)$. \end{proof}

\subsection{Cheap Lift Subspaces}
\label{sec:cl}

To argue for the usefulness of the SLLS step direction, and to select
suitable subspaces $V$ and $U$ for a given partition $(B,N)$, we recall the
discussion of the trust region step from the Introduction (\Cref{sec:intro}).
As long as the step primal and dual directions $(\Delta x,\Delta s)$ are
feasible to the systems \eqref{eq:trust-region} for a suitably small
threshold, we are guaranteed to make progress as measured by the primal and
dual objective values as in \eqref{eq:trust-region-progress}.

Simply selecting $V=W$ and $U=W^\perp$ would attain the smallest possible
objective values; however, the constraints bounding the local norms of
$\Delta x^\slls_B$ and $\Delta s^\slls_N$ in  \eqref{eq:trust-region} could
be arbitrarily violated.  We will select the subspaces $V$ and $U$ so that
$\norm{\Xerr^{-1}\Delta x^\slls_B}$ and
$\norm{\Serr^{-1}\Delta s^\slls_N}$  are guaranteed
to be small.

To ensure that the constraints are satisfied, we will restrict the primal and dual movement directions to \emph{cheap lift subspaces} $V$ and $U$. These are formally defined below: 

\begin{definition}[Cheap Lift Subspace]\label{def:cheap-lift-subspace} Let $W \subseteq \R^n$ be a
subspace, $\scale \in \R_{++}^n$, $(B,N)$ be a partition of $[n]$. Then, $V
\subseteq W$ is a cheap lift subspace for $(W,\scale,B,N)$ with lifting cost
$\thresh \geq 0$ if
\[
\norm{\scale_B x_B} \leq \thresh \norm{\scale_N x_N}, \forall x \in V.
\]
\end{definition}
\begin{remark}
\label{rem:dim-req-cl}
Note that for the above inequality to hold, we must have that $\dim(\pi_N(V))
= \dim(V)$, since otherwise there exists a vector $x \in V$ with $x_B \neq
\0_B$ and $x_N = \0_N$. In particular, if $N = \emptyset$, we must have $V =
\{\0_n\}$, which has lifting cost $0$. Furthermore, if $N=[n]$, then $V = W$
is a cheap lift subspace for $(W,\scale,B,N)$ with lifting cost $0$ of maximum dimension (recall
that $(\scale x)_{\emptyset} \coloneqq \pi_{\emptyset}(\scale x) = 0$ by
convention). 
\end{remark}

In the context of solving the primal trust-region program, if $V \subseteq W$
is a cheap lift subspace with respect to $(W,\Xerr^{-1},B,N)$ with lifting
cost $\thresh = \frac{\nu}{\sqrt{n}}$, the primal subspace LLS direction
with respect to $(W,V,N,B)$ 
\[
\Delta x^\slls \coloneqq \argmin_{\Delta x \in V} \norm{\nx_N^{-1}(x_N+\Delta x_N)}
\]
automatically satisfies the trust-region constraint $\norm{\nx_B^{-1} \Delta x^\slls_B} \leq \nu$. This is because $\nx_N^{-1} \Delta x^\slls_N$ is the orthogonal
projection of $-\nx_N^{-1} x_N = -\error_N$ onto the subspace $\nx_N^{-1} \pi_N(V)$, and hence
$\norm{\nx_N^{-1} \Delta x^\slls_N} \leq \norm{\nx_N^{-1} x_N} = \sqrt{\sum_{i \in N} x_i s_i / \mu} \leq \sqrt{n}$. By the lifting
cost condition on $V$, we then have 
$
\norm{\nx_B^{-1} \Delta x^\slls_B} \leq \thresh \norm{\nx_N^{-1} \Delta x_N^\slls} \leq \sqrt{n} \thresh = \nu,
$
as needed. Analogously, any cheap lift subspace $U \subseteq W^\perp$ for
$(W^\perp,\ns^{-1},N,B)$ with lifting cost at most $\thresh$ will also suffice for
the dual subspace LLS direction $\Delta s^\slls$ as in
\Cref{def:subspace_lls} to be feasible.

While low lifting cost subspaces are sufficient to get feasible solutions to
the trust-region program, they do not necessarily yield useful
approximations.  In particular, one can always choose the subspace $V =
\{\0\}$, which is trivially cheap. To make significant progress along a
polarized segment, we will require that the use of cheap lift subspaces of
\emph{maximum dimension} subject to a lifting cost bound $\thresh$. These
will allow us to either quickly increase the dimension of the cheap lift
subspace, or to get past the end of the of the polarized segment. 

We now explain how to find these cheap lift subspaces and what their
achievable dimensions are. For this purpose, we will rely on approximate
singular subspaces of lifting operator, as defined in \Cref{def:lifting-map}.
The following lemma gives the precise relation between cheap lift subspaces
and approximate singular subspaces of the lifting operator, and will be the
main tool underlying the computation of cheap lift subspaces.

\begin{lemma}
Let $W \subseteq \R^n$ be a linear subspace, $\scale \in \R^n_{++}$, $B \cup N =
[n]$ be a partition. Let $L^{\scale W}_N: \R^N \rightarrow
\R^n$ and $\lift^{\scale W}_N: \pi_N(\scale W) \rightarrow \pi_B(\scale^{-1} W^\perp)$  be the lifting operator and map as in \Cref{def:lifting-map}. 

Then, the maximum dimension of a cheap lift subspace $V$ for $(W,\scale,B,N)$
with lifting cost $\thresh \geq 0$ is
\[
\cnt{\lift^{\scale W}_N}{\thresh} \coloneqq |\{i \in [\dim(\pi_N(\scale W))]: \sigma_i(\lift^{\scale W}_N) \leq \thresh\}|,
\]
the number of singular values of $\lift^{\scale W}_N$ of value at most $\thresh$
as in \Cref{def:svd}. Furthermore, if $S \subseteq \pi_N(\scale W)$ is a
$\SVDA$-approximate singular subspace for $\lift^{\scale W}_N$ of dimension
$\cnt{\lift^{\scale W}_N}{\tfrac{\thresh}{\SVDA}}$, then 
\[
V \coloneqq \scale^{-1} L^{\scale W}_N(S)
\]
is a cheap lift subspace for $(W,\scale,B,N)$ of dimension 
$\dim(S) = \cnt{\lift^{\scale W}_N}{\tfrac{\thresh}{\SVDA}}$ with lifting cost 
\[
\sigma_1\left(\restr{\lift^{\scale W}_N}{S}\right) \leq \SVDA \sigma_{\dim(\pi_N(\scale W))-\dim(S)+1}(\lift^{\scale W}_N) \leq \thresh.
\]
\label{lem:cheap-lift-dim}
\end{lemma}

\begin{proof}We first show the upper bound of $\cnt{\lift^{\scale W}_N}{\tau}$ on the
dimension of any cheap lift subspace. Let $V \subseteq W$ be a cheap lift
subpace as above with $\dim(V) = d$, and let $\wh{V} \coloneqq \scale V$ and
$\wh{W} \coloneqq \scale W$. For any $y \in \wh{V} \subseteq \wh{W}$, by
definition of the lifting map and operator, we have that 
\[
L_N^{\wh W}(y) \in \wh W, \, (L_N^{\wh W}(y))_N = y_N, \, \text{ and } \norm{y_B} \geq
\norm{(L_N^{\wh W}(y_N))_B} = \norm{\lift^{\wh W}_N(y_N)},
\]
where we note that this
still holds if $(B,N)$ is trivial as $y_\emptyset := \pi_{\emptyset}(y)=0$ by
convention. Therefore, without loss of generality, we may assume that 
\[
\wh V = L^{\wh W}_N(\pi_N(\wh V)),
\]
since this can only decrease lifting cost while maintaining the dimension,
recalling that $\dim(\pi_N(\wh V))= \dim(\wh V)$ for any cheap lift subspace.
If $d = 0$ then trivially $d \leq \cnt{\lift_N^{\wh W}}{\thresh}$, and there
is nothing left to prove. So assume $d \geq 1$. Then, since $\pi_N(\wh V)
\subseteq \pi_N(\wh W)$ and $\dim(\pi_N(\wh V)) = d$, the lifting cost
of $V$ satisfies  
\begin{align}
\tau &\geq \max_{x \in V, x_N \neq \0_N} \frac{\norm{\scale_B x_B}}{\norm{\scale_N x_N}} = \max_{y \in \wh V, y_N \neq \0_N} \frac{\norm{y_B}}{\norm{y_N}}
= \max_{z \in \pi_N(\wh V) \setminus \{\0_N\}} \frac{\norm{\lift^{\wh W}_N(z)}}{\norm{z}} \nonumber \\ &= \sigma_1(\restr{\lift^{\wh W}_N}{\pi_N(\wh V)}) \geq \sigma_{\dim(\pi_N(\wh W))-d+1}(\lift^{\wh W}_N), \label{eq:cheap-lift-reduction}
\end{align}
where the last inequality follows from~\eqref{eq:min-max-sing}. Since $\tau
\geq \sigma_{\dim(\pi_N(\wh W))-d+1}(\lift^{\wh W}_N)$ and the singular values
are in non-increasing order, we have that 
\begin{align*}
d &= |\{i \in [\dim(\pi_N(\wh W))]:
\dim(\pi_N(\wh W))-d+1 \leq i\}| \\ &\leq |\{i \in [\dim(\pi_N(\wh W))]:
\sigma_i(\lift^{\wh W}_N) \leq \thresh\}| = \cnt{\lift^{\wh W}_N}{\tau},
\end{align*}
as needed. 

For the second part, if $S \subseteq \pi_N(\wh W)$, then tracing the
reduction of the first part backwards, we immediately get
from~\eqref{eq:cheap-lift-reduction} that $\scale^{-1} L^{\wh W}_N(S) \subseteq
W$ is a cheap lift subspace for $(W,\scale,B,N)$ of dimension $\dim(S)$ and
lifting cost $\sigma_1(\restr{\lift^{\wh W}_N}{S})$. If $S$ is a
$\SVDA$-approximate singular subspace of $\lift^{\wh W}_N$ of dimension
$\cnt{\lift^{\wh W}_N}{\tfrac{\tau}{\SVDA}}$, then by definition
$\sigma_1(\restr{\lift^{\wh W}_N}{S}) \leq \SVDA \sigma_{\dim(\pi_N(\wh
W))-\dim(S)+1}(\lift^{\wh W}_N) \leq \SVDA(\tau/\SVDA) = \tau$, as needed. 
\end{proof}

We remark that so far we have been treating the tasks of finding primal and
dual cheap lift subspaces separately. Fortunately, it turns out the singular
values of the corresponding primal and dual lifting operators are identical.
Thus, the corresponding problems of finding cheap lift subspace on both sides
are intimately linked. In particular, we derive the following relation as a
direct corollary of \Cref{lem:lift-duality}.

\begin{corollary}
Let $\nx ,\ns \in \R^n_{++}$, satisfy $\nx \ns = \mu \1_n$ where $\mu = \gap(\nx,\ns)$.
Then, for a subspace $W \subseteq \R^n$ and non-trivial partition $B \cup N = [n]$, $\sigma^+(\lift^{\nx^{-1} W}_N) = \sigma^+(\lift^{\ns^{-1} W^\perp}_B)$.
Furthermore, if $\sum_{i=1}^r \sigma_i u_i v_i^\T$ is the SVD of
$\lift^{\nx^{-1} W}_N$, then $-\sum_{i=1}^r \sigma_i v_i u_i^\T$ is the SVD of
$\lift^{\ns^{-1}W^\perp}_B$.
\label{cor:lift-duality-sing}
\end{corollary}

Given the above, computing cheap lift subspaces for either the primal or dual
side is closely tied to the problem of computing an \emph{approximate} SVD of
the lifting operator. 

\subsection{Computing Cheap Lift Subspaces}
\label{sec:comp-cl}

In this section, we explain how to compute cheap lift subspaces. 

\subsubsection{Approximating the Singular Subspaces}
\label{sec:apx-svd}

To compute cheap lift subspaces, we will rely on a solver for the
$\SVDA$-\text{SVD} problem defined below, which provides the requisite
approximate singular value decomposition.

\begin{definition}[$\SVDA$-\textsc{SVD} problem]
\label{def:alpha-svd-problem}
For $\SVDA \geq 1$, the input to the $\SVDA$-\textsc{SVD} problem is a matrix
$\mM \in \R^{m \times n}$, and the output is an \emph{orthogonal} basis $\mB
\in \R^{n \times n}$ and $s \in \R^n_+$, $s$ sorted in non-increasing
order, such that
\begin{align}
\norm{\mM x} &\leq s_i \norm{x}, \forall x \in \im(\mB_{\geq i}), \forall i \in [n]\, ,\quad\mbox{and} \label{eq:svd-upper} \\
\SVDA^{-1} s &\leq \sigma(\mM) \leq s\,. \label{eq:ss-approx}
\end{align}
\end{definition}

\begin{remark}
The requirement that $s \geq \sigma(\mM)$ in~\eqref{eq:ss-approx} is
automatically satisfied given~\eqref{eq:svd-upper}
and~\eqref{eq:min-max-sing}. Furthermore, the requirement that $s$ be sorted
in non-increasing order is without loss of generality: if $s$ satisfies the
requirements~\eqref{eq:svd-upper} and~\eqref{eq:ss-approx}, then $\bar{s}$ defined by $\bar{s}_i = \min_{j \in [i]} s_j$, $\forall
i \in [n]$, also satisfies the requirements and is non-increasing. 
\end{remark}

\begin{remark}
The guarantees of $\SVDA$-\textsc{SVD} imply that for each $i \in [m]$, the
subspace $\im(\mB_{\geq i})$ is a $\SVDA$-approximate singular subspace of
$\mM$ (\ie, of the operator $\cal T(\mM)$) of dimension $n-i+1$. In
particular, given a threshold $\tau \geq 0$, letting $c_\tau = |\{i \in [n]:
s_i \leq \tau\}|$, the subspace $S_\tau = \im(\mB_{\geq n-c_\tau+1})$ is
the subspace of largest dimension in this collection that satisfies
$\sigma_1(\restr{\mM}{S_\tau}) \leq \tau$ under the guarantee of~\eqref{eq:svd-upper}.
\end{remark}

For the sake of being able to execute each iteration of our IPM in strongly
polynomial time, we will need a $\SVDA$-\textsc{SVD} solver that uses a
strongly polynomial number of basic arithmetic operations (i.e., depending
only on the dimension of the underlying matrix), and that is numerically
stable, i.e., that ensures iterates with polynomial bit complexity. In
particular, the algorithm cannot rely on computing square root computations.

In \Cref{sec:sing-val}, we show that for any fixed $\SVDA > 1$,
$\SVDA$-\textsc{SVD} can be solved in deterministic strongly polynomial time.
Our main theorem is as follows.

\begin{restatable}{theorem}{thmapproxsvd}
For $\SVDA > 1$, there is a deterministic algorithm for solving
$\SVDA$-\textsc{SVD} that on an input $m \times n$ matrix requires $O(n^2\max(m,n)^{3} \log(n + \frac{1}{\SVDA - 1}))$ operations and space polynomial in
$n,m,\max \{1,\frac{1}{\SVDA - 1}\}$ and the bit-encoding length of the input matrix.
\label{thm:approx-svd}
\end{restatable}

We defer the discussion of how this theorem is proved and how it relates to
earlier work on approximate SVDs to \Cref{sec:sing-val}. 

Our IPM can be implemented using any $\SVDA$-\textsc{SVD} solver (modulo
strongly polynomial considerations). While the iteration complexity of the IPM does
depend on the approximation factor $\SVDA$, this dependency is very mild: as
we show in \Cref{sec:running-time-analysis}, the iteration bound of our IPM will depend
only logarithmically on the approximation factor. The guarantees of our IPM in \Cref{th:main_upper_bound-slc} will be derived
by using the above for $\SVDA = 2$. 

\subsubsection{The Cheap Lift Subspace Algorithm}

\ifdefined\SIAMversion
\begin{algorithm2e}[h!]
\else
\begin{algorithm}[h!]
\fi
    \caption{\textsc{Cheap-Lift-Subspaces}}
    \label{alg:cheap_lift}
    \SetKwInOut{Input}{Input}
    \SetKwInOut{Output}{Output}
    \SetKw{And}{\textbf{and}}
    \Input{Matrix $\mA \in \R^{m \times n}$, $\rank(\mA) = m$, $\scale \in \R^n_{++}$, partition $B \cup N = [n]$ and SVD approximation factor $\SVDA \geq 1$.}
    \Output{Matrices $\MP \in \R^{n \times m_p}, \MD \in \R^{n \times m_d}$, $m_p, m_d \in [n]$. Letting $W=\ker(\mA)$ and $\lift=\lift^{\scale W}_N$, $\lift^\perp=\lift^{\scale^{-1} W^\perp}_B$, the subspaces $V=\im(\MP)$, $U=\im(\MD)$ satisfy: 
     \begin{enumerate}[rightmargin=1cm]
  \item \label{eq:cl-kernel} $W \cap \R^n_N \subseteq V$ and $W^\perp \cap \R^n_B \subseteq U$.
  \item \label{eq:cl-lc} $V$ is a cheap lift subspace for $(W,\scale,B,N)$ with lifting cost $\SVDA \sigma_{\dim(\pi_N(W))-\dim(V)+1}(\lift)$, and \newline
 $U$ is a cheap lift subspace for $(W^\perp,\scale^{-1},N,B)$ with lifting cost
$\SVDA \sigma_{\dim(\pi_B(W^\perp))-\dim(U)+1}(\lift^\perp)$.
     \item \label{eq:cl-dim} $\cnt{\lift}{\tfrac{1}{\SVDA}} \leq \dim(V) \leq \cnt{\lift}{1}$ and $\cnt{\lift^\perp}{\tfrac{1}{\SVDA}} \leq \dim(U) \leq \cnt{\lift^\perp}{1}$. 
    \end{enumerate}
    }
\tcp{Compute Primal and Dual Projection Maps}
$\mP^\perp \gets \diag(\scale^{-1}) \mA^\T (\mA \diag(\scale^{-2}) \mA^\T)^{-1} \mA \diag(\scale^{-1})$; $\mP \gets \mI_n-\mP^\perp$\label{line:proj-cl}\; 

\If{$B=\emptyset$}{ $(\MP,\MD) \gets (\mP,[\0_n])$\; }
\ElseIf{$N=\emptyset$}{ $(\MP,\MD) \gets ([\0_n],\mP^\perp)$\; } 
\Else{
    \tcp{Compute Primal and Dual Lifting Maps and Operators}
$(\bar{L},\bar{L}^\perp) \gets (\mP_{\mdot,N} (\mP_{N,N})^+,\mP^\perp_{\mdot,B} (\mP^\perp_{B,B})^+)$, associated matrices of lifting maps $L^{\scale W}_N$, $L^{\scale^{-1} W^\perp}_B$\label{line:liftmap-ass-cl}\; 
$(\bar \lift, \bar \lift^\perp) \gets
(\bar{L}_{B,\mdot},\bar{L}^\perp_{N,\mdot})$, associated matrices of
lifting operators $\lift^{\scale W}_N$, $\lift^{\scale^{-1} W^\perp}_B$\label{line:liftop-ass-cl}\;

    \tcp{Approximate SVDs of $\bar{\lift}$ and $\bar{\lift}^\perp$} 
    $(\bar \mC,\bar s)$ $\gets$ $\SVDA$-\textsc{SVD} on matrix $\bar{\lift} \in \R^{B \times N}$; $(\bar \mC^\perp,\bar s^\perp)$ $\gets$ $\SVDA$-\textsc{SVD} on matrix $\bar{\lift}^\perp \in \R^{N \times B}$\label{line:svd-cl}\;

    \tcp{Number of Small Primal and Dual Singular Values}
    $\bar c_p \gets |\{i \in [|N|]: \bar s_i \leq 1\}|$; $\bar c_d \gets |\{i \in [|B|]: \bar s_i^\perp \leq 1\}|$\label{line:small-sing-cl}\; 
    
    \tcp{Generators of Primal and Dual Cheap Lift Subspaces}
    $(\MP,\MD) \leftarrow (\scale^{-1} \bar{L}(\bar \mC_{\geq |N|-\bar c_p+1}), \scale \bar{L}^\perp(\bar \mC^\perp_{\geq |B|-\bar c_d+1}))$\label{line:gen-cl}\; 
}

    \Return ($\MP$, $\MD$)\label{line:cl-return}\; 
\ifdefined\SIAMversion
\end{algorithm2e}
\else
\end{algorithm}
\fi

We now overview the \nameref{alg:cheap_lift} algorithm. At a high level, the
algorithm is a direct algorithmic implementation of \Cref{lem:cheap-lift-dim}
applied to both the primal and dual separately. Specifically, it computes the
primal and dual lifting operators, followed by approximate singular subspaces
associated with the singular values smaller than $1$, and then lifts them to
build the cheap lift subspaces. One complication however is that the
approximate $\SVDA$-\textsc{SVD} solver directly computes approximate
singular subspaces for the associated matrices of the lifting operators and
not the lifting operators themselves. The difference here being that the
computed subspaces do not necessarily reside in the input spaces of the
lifting operators. Fortunately, the lifting maps automatically correct for
this discrepancy as they first project onto the input space. In
the proof of correctness, we justify that these projections yield the desired approximate singular subspaces using \Cref{lem:ass-op-mat}.  

\begin{remark}
\label{rem:cls-alg-out-of-bounds}
On line~\ref{line:gen-cl} of \nameref{alg:cheap_lift}, recall that by
convention, if $\bar c_p = 0$ then $\bar \mC_{\geq |N|-c_p+1} = \bar
\mC_{\emptyset} = [\0_B]$. In particular, $\MP = \scale^{-1} \bar{L}(\bar
\mC_{\geq |N|-\bar c_p+1}) = [\0_n]$. Similarly, if $\bar c_d = 0$, then $\MD
= \scale \bar{L}^\perp(\bar \mC^\perp_{\geq |B|-\bar c_d+1}) = [\0_n]$. In a similar vein, the checks on whether $B$ or $N$ are empty are not strictly speaking necessary. It can be verified that the output of the algorithm is the same without these checks, however we have decided to include them for the sake of clarity.    
\end{remark}

\begin{remark}
\label{rem:cls-cutoff}
One may wonder why \nameref{alg:cheap_lift} puts the singular value
cutoff at $1$ instead of $\nu/\sqrt{n}$, $\nu = O(\beta)$, as described in
\Cref{sec:slls}. As the analysis will show, we will only need the
cheap lift subspaces when there is a large multiplicative gap between the
singular values less than $1$ and greater than $1$. In particular, all the
singular values less than $1$ in this case will be significantly smaller than
the threshold $\nu/\sqrt{n}$. \end{remark}

\begin{remark}
\label{rem:cls-duality}
Given the duality lifting operators it is natural to wonder whether the cheap
lift subspace $V$ for $(W,\scale,B,N)$ outputted by
\nameref{alg:cheap_lift} can be be transformed into a cheap
lift subspace $U$ for $(W^\perp,\scale^{-1},N,B)$ satisfying the output
conditions (and vice versa). Assuming $\scale = \1_n$ for simplicity, one can
in fact show that setting $U = L_B^{W^\perp}((\lift_B^{W^\perp})^+(\pi_N(V)))+
(W^\perp \cap \R^n_B)$ indeed satisfies the desired conditions. For the sake of
keeping the analysis and description of the algorithm as simple as possible however, we have opted to avoid using and proving this relation here.  \end{remark}

We now prove correctness of \nameref{alg:cheap_lift} and give its running
time guarantee.

\begin{lemma} 
Algorithm \nameref{alg:cheap_lift} is correct. Furthermore, assuming a
strongly polynomial solver for $\SVDA$-\textsc{SVD}, \nameref{alg:cheap_lift}
runs in time strongly polynomial $n$.  
\label{lem:cheap_lift_correct}
\end{lemma}

\begin{proof}
The strong polynomiality claim is immediate from the assumption on the
$\SVDA$-\textsc{SVD} solver and the strong polynomiality of
computing pseudoinverses (\Cref{prop:pseudo-compute}). We thus focus on
correctness.

For correctness, we must show that the outputted matrices $\MP,\MD$ and
correspondings subspaces $V=\im(\MP)$, $U=\im(\MD)$ satisfy the output
conditions~\eqref{eq:cl-kernel},~\eqref{eq:cl-lc}, \eqref{eq:cl-dim} and that $m_p,m_d \in [n]$.

By \Cref{prop:pseudo-compute}, on line~\ref{line:proj-cl} the algorithm
computes the orthogonal projection matrices $\mP^\perp,\mP = \mI_n -
\mP^\perp$ onto $\scale^{-1} \im(\mA^\T) = \scale^{-1} W^\perp$ and
$(\scale^{-1} W^\perp)^\perp = \scale W = \scale \ker(\mA)$, respectively. 

After this line, the algorithm is completely symmetric for the primal and
dual, so we restrict to
proving~\eqref{eq:cl-kernel},~\eqref{eq:cl-lc},~\eqref{eq:cl-dim} for the
primal subspace $V$ and proving $m_p \in [n]$, recalling the shorthand $\lift
\coloneqq \lift_N^{\scale W}$. 

Let us first assume that $(B,N)$ is a trivial partition. If $B=\emptyset$,
then $N = [n]$, $V = \im(\mP) = W = W \cap \R^n_N$ and $m_p=\dim(W) \in [n]$, which
verifies~\eqref{eq:cl-kernel} and the requirement on $m_p$. Furthermore,
since $\lift \coloneqq \lift_{[n]}^W: W \rightarrow \R^\emptyset$, we have
$\dim(V) = \cnt{\lift}{0} = \cnt{\lift}{1/\rho} = \cnt{\lift}{1}$, which
verifies~\eqref{eq:cl-dim}. The lifting cost condition~\eqref{eq:cl-lc} also
holds trivially. If $N=\emptyset$, then $V = \im([\0_n]) = \{\0_n\} = W \cap
\R^n_\emptyset$ and $m_p=1$, which verifies~\eqref{eq:cl-kernel} and the
requirement on $m_p$. From here, since $\lift \coloneqq \lift_\emptyset^W:
\R^\emptyset \rightarrow W^\perp$, we have $\dim(V) = 0 = \cnt{\lift}{\R^+} =
\cnt{\lift}{1/\rho} = \cnt{\lift}{1}$, which verifies~\eqref{eq:cl-dim}.
Again, the lifting cost condition~\eqref{eq:cl-lc} is trivial.

Now assume that $(B,N)$ is non-trivial. Then by \Cref{prop:lift-compute},
$\bar L$ computed on line~\ref{line:liftmap-ass-cl} equals $\cal
M(L_N^{\scale W})$ and $\bar \lift$ computed on
line~\ref{line:liftop-ass-cl} equals $\cal M(\lift_N^{\scale W})$, the
associated matrix of the primal lifting map and operator.   

Let $(\bar \mC, \bar s)$ be the $\SVDA$-SVD of $\bar \lift \in \R^{B \times N}$ as
computed on line~\ref{line:svd-cl} and let $\bar c_p = |\{i \in [N]: \bar s_i \leq
1\}|$ be as computed on line~\ref{line:small-sing-cl}. By the guarantees of $\SVDA$-\textsc{SVD}, we have that $\sigma(\bar \lift) \leq \sqrt{ \bar s} \leq \SVDA \sigma(\bar \lift)$, and therefore 
\begin{equation}
\cnt{\bar \lift}{1/\SVDA} \leq \bar c_p \leq \cnt{\bar \lift}{1} \label{eq:cl-dim-fs}.
\end{equation}
Similarly, letting $\bar k_p \coloneqq \dim(\ker(\bar \lift)) = \cnt{\bar \lift}{0}$,
we have that 
\[
\sigma(\bar \lift)_{[|N|] \setminus [|N|- \bar k_p]} = \bar s_{[|N|] \setminus [|N|- \bar k_p]} = \0_{[|N|] \setminus [|N|- \bar k_p]}.
\]
By the guarantees of $\SVDA$-\textsc{SVD}, this implies that 
\begin{equation}
\im(\bar \mC_{\geq |N|-\bar k_p+1}) = \ker(\bar \lift), \label{eq:cl-svd-ker}
\end{equation}
since $\sigma_1(\restr{\bar \lift}{\im(\bar \mC_{\geq |N|-\bar k_p+1})})
= 0$ and $\dim(\im(\bar \mC_{\geq |N|-\bar k_p+1}) ) = \bar k_p =
\dim(\ker(\bar \lift))$. 

Since $\bar c_p \geq \bar k_p$ we get that $\bar{V}_N \coloneqq \im(\bar
\mC_{\geq |N|-\bar c_p+1}) \supseteq \im(\bar \mC_{\geq |N|-\bar k_p+1}) =
\ker(\bar \lift)$ is a $\SVDA$-approximate singular subspace for $\bar \lift$
of dimension $\bar c_p$ containing $\ker(\bar \lift)$. Let us use the shorthand $\src \coloneqq \pi_N(\scale W)$ and $\targ \coloneqq \pi_B(\scale^{-1} W^\perp)$, recalling that $\lift \colon \src \rightarrow \targ$. Defining $\wh V_N \coloneqq
\Pi_{\src}(\bar{V}_N)$, by \Cref{lem:ass-op-mat} we have
that $\wh V_N$ is a $\SVDA$-approximate
singular subspace for $\lift$ containing $\ker(\lift)$ of dimension $\bar c_p - \dim(\src^\perp)$. Again
by \Cref{lem:ass-op-mat}, for $\tau \geq 0$ we have $\cnt{\lift}{\tau} =
\cnt{\bar \lift}{\tau}-\dim(\src^\perp)$, and hence we conclude
that    
\begin{align}
\cnt{\lift}{1/\SVDA} &= \cnt{\bar \lift}{1/\SVDA} - \dim(\src^\perp) \leq  \bar c_p - \dim(\src^\perp) = \dim(\wh V_N) \nonumber \\ &\leq  \cnt{\bar \lift}{1} - \dim(\src^\perp) = \cnt{\lift}{1} \label{eq:cl-dim-cond}.
\end{align}

By line~\ref{line:gen-cl}, recall that $\MP = \scale^{-1} \bar L(\bar
\mC_{\geq n-\bar c_p+1})$. The requirement $m_p \in [n]$, then follows from
$m_p = \max \{1, \bar c_p\} \in [|N|] \subseteq [n]$. Recalling that
$L_N^{\scale W}(\Pi_{\src}(\cdot)) = L_N^{\scale W}(\cdot)$ by definition and
that $\bar L = \cal M(L_N^{\scale W})$, we have that
\begin{align}
V &\coloneqq \im(\MD) = \scale^{-1} L_N^{\scale W}(\im(\bar \mC_{\geq n-\bar c_p+1})) = \scale^{-1} L_N^{\scale W}(\Pi_{\src}\im(\bar \mC_{\geq n-\bar c_p+1})) \nonumber \\ &= \scale^{-1} L_N^{\scale W}(\wh V_N). \label{eq:V-formula}
\end{align}
By \Cref{lem:cheap-lift-dim}, we then get that $V$ is cheap lift subspace for
$(W,\scale,B,N)$ of dimension $\dim(V)=\dim(\wh V_N)$ with lifting cost at
most 
\[
\SVDA \sigma_{\dim(\src)-\dim(V)+1}(\lift) = \SVDA
\sigma_{\dim(\pi_N(W))-\dim(V)+1}(\lift).
\]
Together with~\eqref{eq:cl-dim-cond} this
establishes~\eqref{eq:cl-lc},~\eqref{eq:cl-dim} for $V$. By
\Cref{lem:lift-map}, we have that $\pi_N(\scale W \cap \R^n_N) = \ker(\lift)
\subseteq \wh V_N$ and therefore
\[
W \cap \R^n_N = \scale^{-1} L_N^{\scale W}(\pi_N(\scale W \cap \R^n_N)) \subseteq \scale^{-1} L_N^{\scale W}(\wh V_N) = V,
\]
which verifies~\eqref{eq:cl-kernel}.
\end{proof}
 
\section{The Subspace Layered Least Squares Algorithm}\label{sec:slls-alg}

In this section, we present our SLLS based IPM (\nameref{alg:subspace_ipm})
and give the proof of \Cref{th:main_upper_bound-curve} from the introduction.
The pseudocode for our IPM is provided in \Cref{sec:slls-ipm}, and its correctness is proved in \Cref{sec:ipm-correct}. In
\Cref{sec:running-time-analysis}, we bound the number of iterations needed to
traverse a polarized segment of the central path from which
\Cref{th:main_upper_bound-curve} follows directly. We note that the iteration
bound depends only logarithmically on the SVD approximation factor $\SVDA$. A
more fine-grained amortized iteration bound is deferred to
\Cref{sec:amortized}.

\ifdefined\SIAMversion
\begin{algorithm2e}[h!]
\else
\begin{algorithm}[h!]
\fi
    \caption{\textsc{SLLS-IPM}}
    \label{alg:subspace_ipm}
    \SetKwInOut{Input}{Input}
    \SetKwInOut{Output}{Output}
    \SetKw{And}{\textbf{and}}

    \Input{Instance of \eqref{LP_primal_dual} with primal constraint matrix $\mA \in \R^{m \times n}, \rank(\mA)=m$, and initial iterate $(x^0,s^0) \in \cal N^2(\beta)$, $\beta \in (0,1/6]$, and SVD approximation factor $\SVDA\ge 1$.}
    \Output{$(x^\star, s^\star,v^\star,w^\star)$ satisfying: 
           \begin{enumerate}[rightmargin=1cm]
           \item \label{ipm:optimal} $x^\star \in \Primal, s^\star \in \Dual$, $\pr{x^\star}{s^\star} = 0$.\item \label{ipm:partition} $B \coloneqq \{i \in [n]: x^\star_i > 0\}$, $N \coloneqq \{i \in [n]: s^\star_i > 0\}$ satisfy $B \cup N = [n]$.\item \label{ipm:certificates} $v^\star \in W^\perp \coloneqq \im(\mA^\T), v^\star_B > \0_{B}$, $w^\star \in W \coloneqq \ker(\mA)$, $w^\star_N > \0_{N}$. \item \label{ipm:centrality} $\norm{(x^\star_B v^\star_B, s^\star_N w^\star_N)-\1_n} \leq \beta$.\end{enumerate}
    }
    $(x,s) \gets (x^0,s^0)$\;
    \label{eq:line-gap-check}
    \While{$\gap(x,s) > 0$}{
        Compute corrector direction $(\Delta x^c, \Delta s^{c})$ at $(x,s)$\;
        $(x,s) \gets (x + \Delta x^c, s + \Delta s^c)$\; \label{eq:line-corrector}
        Compute affine scaling direction $(\Delta x^\as, \Delta s^\as)$ at $(x,s)$\;
        Set $\alpha^\as$ for $(\Delta x^\as, \Delta s^\as)$ according to \Cref{prop:compute-standard-step} with parameter $\beta$\;
        $\Balg \gets \left\{i \in [n]: \left|\frac{\Delta x^\as_i}{x_i}\right| < \left|\frac{\Delta s^\as_i}{s_i}\right|\right\}, \Nalg \gets [n] \setminus \Balg$\; 
	$(\MPalg,\MDalg) \gets \operatorname{Cheap-Lift-Subspaces}(\mA,\frac{1}{x},\Balg,\Nalg,\SVDA)$\; \label{line:cls-call} 
           $(\Valg,\Ualg) \gets (\im(\MPalg),\im(\MDalg))$\;
	   Compute the subspace LLS direction $(\Delta x^\slls, \Delta s^\slls)$ at $(x,s)$ with respect to $(\Balg,\Nalg,\Valg,\Ualg)$ using \Cref{prop:lls-step-formula} on input $(x,s,\Balg,\Nalg,\MPalg,\MDalg)$\;\label{line:slls-direction}
	   Set $\alpha^\slls$ for $(\Delta x^\slls,\Delta s^\slls)$ according to \Cref{prop:step-size} with parameters $\mu = \gap(x,s)$, $\beta$\;\label{eq:line-lls-step}
  \If{$\gap((x + \alpha^\as \Delta x^\as, s+\alpha^\as\Delta s^\as)) < \gap((x + \alpha^\slls \Delta x^\slls, s + \alpha^\slls \Delta s^\slls))$} {\label{eq:line-step-size-check}
            $(x, s) \gets (x + \alpha^\as \Delta x^\as, s+\alpha^\as\Delta s^\as)$\;
        } \Else{
            $(x, s) \gets (x + \alpha^\slls \Delta x^\slls, s + \alpha^\slls \Delta s^\slls)$\;
        }
    }

  $(x^\star,s^\star) \leftarrow (x,s)$\; \label{line:optimal}
        $\mu \gets \pr{x^\star-\Delta x^\slls}{s^\star-\Delta s^\slls}/n$\; \label{eq:line-mu-lls}  
        $(v^\star,w^\star) \leftarrow (-\Delta s^\slls/\mu,-\Delta x^\slls/\mu)$\;\label{eq:line-set-vu-lls}
    \Return{$(x^\star,s^\star,v^\star,w^\star)$}\;
\ifdefined\SIAMversion
    \end{algorithm2e}
\else
    \end{algorithm}
\fi

\subsection{Description of the Algorithm}
\label{sec:slls-ipm}

We are ready to describe the predictor-corrector algorithm
\nameref{alg:subspace_ipm}, shown in Algorithm~\ref{alg:subspace_ipm}.  We
are given a starting iterate $z = (x,s) \in \cal N^2(\beta)$.  In each
iteration, we compute first compute a corrector step to move $z$ into the
$\cal N^2(\beta/2)$ neighborhood. We then compute the affine scaling
direction $(\Delta x^\as,\Delta s^\as)$ at $z$ and identify the associated
partition $(\Balg_z,\Nalg_z)$. Using this partition, we compute cheap lift
subspaces $\Valg$ and $\Ualg$ using \nameref{alg:cheap_lift} (Algorithm~\ref{alg:cheap_lift}) from \Cref{sec:comp-cl}. 

We then compute the subspace LLS direction $(\Delta x^\slls,\Delta s^\slls)$
for $(\Balg,\Nalg,\Valg,\Ualg)$. For both directions, we then compute the
feasible step-lengths according to the bounds in
\Cref{prop:predictor-corrector} and \Cref{prop:step-size}, and use the better
of these two possible steps to obtain the next iterate.

Once the algorithm has found optimal solutions $(x^\star,s^\star) \in \clN(\beta)$, that is, satisfying $\gap(x^\star,s^\star) = 0$, it uses the direction of the last segment of the central
path to compute certificates $(v^\star,w^\star)$ that certify that the optimal solutions
are close to the analytic centers of the respective optimal faces.

\begin{remark}[Solving Linear Programs to Target Accuracy $\eps \geq 0$]
One can easily modify \nameref{alg:subspace_ipm} to terminate once it has
found a primal and dual solution $(x,s)$ with gap $\pr{x}{s} \leq \eps$,
where $\eps \geq 0$ is the target accuracy. This is achieved by changing the
while loop check on line~\ref{eq:line-gap-check} from $\gap((x,s)) > 0$ to
$\pr{x}{s}
> \eps$. Once the while loop terminates, if $\eps > 0$, we immediately return
$(x,s) \in \clN(\beta)$ and skip the computation of the additional
certificates $(v^\star,w^\star)$. If $\eps = 0$, the IPM remains unchanged.
We prove approximately optimal iteration bounds for any target gap $\eps \geq
0$ in \Cref{sec:running-time-analysis}, and give more refined amortized
bounds in \Cref{sec:amortized}.
\label{rem:general-accuracy}
\end{remark}

\begin{remark}[Choice of Cheap Lift Subspace Scaling]
\label{rem:cheap-lift-scaling}
One may wonder why we use the scaling $1/x$ instead of the
normalized scaling $1/\nx \coloneqq \sqrt{\frac{s}{x\gap(x,s)}}$ in the call
$\operatorname{Cheap-Lift-Subspaces}(\mA,\frac{1}{x},\Balg,\Nalg,\SVDA)$ on
line~\ref{line:cls-call}, which would have been more adapted to the
definition of the SLLS direction computed on
line~\ref{line:slls-direction}. The main reason is to avoid the need to
compute square roots within the algorithm, which is not a strongly
polynomial operation. Relying on the fact that $x \approx \nx$ (see
\Cref{prop:xi}), this choice of rescaling is safe and has little additional
impact on the analysis or the iteration bound. 
\end{remark}

\subsection{Correctness}
\label{sec:ipm-correct}
In this subsection, we prove that upon termination, the
\nameref{alg:subspace_ipm} satisfies its output requirements, namely it
outputs optimal primal and dual solutions that are close to the analytic
centers of the corresponding optimal faces. This will mainly depend on the
guarantees on the computed step lengths for subspace LLS and affine scaling,
which are given in \Cref{prop:step-size} and
\Cref{prop:compute-standard-step}. 

\begin{lemma} The output of \nameref{alg:subspace_ipm} is correct. Furthermore,
each iteration of the algorithm runs in strongly polynomial time. \label{lem:slls-correct} \end{lemma}

\begin{proof}
We first show that the output $(x^\star,s^\star,v^\star,w^\star)$ satisfies
the
requirements~\eqref{ipm:optimal}, \eqref{ipm:partition},\eqref{ipm:certificates},\eqref{ipm:centrality}
listed in the output description. 

To argue this, we first claim that during the last iteration of the while
loop, either the affine scaling step length $\alpha^\as$ or the subspace LLS step length $\alpha^\slls$ equals $1$. Firstly, by the assumption that
$\gap(x_0,s_0) > 0$, the algorithm enters the while loop. Secondly, for any
iterate $(x,s)$, by \Cref{prop:predictor-corrector} part \ref{i:corrector} and \ref{i:gap}, the corrector
step leaves $\gap(x,s)$ unchanged, whereas the affine scaling step satisfies
$\gap(x+ \alpha^\as \Delta x^\as, s + \alpha^\as \Delta s^\as) =
(1-\alpha^\as) \gap(x,s)$. Similarly, by \Cref{prop:step-size},
$9/8(1-\alpha^\slls) \gap(x,s) \geq \gap(x+\alpha^\slls \Delta x^\slls, s+\alpha^\slls \Delta s^\slls) \geq
7/8(1-\alpha^\slls)\gap(x,s)$. Thus, the only way to exit the loop,
corresponding to the condition $\gap(x,s)=0$, is if $\max \{\alpha^\as,
\alpha^\slls\} = 1$. 

We now further claim that upon exiting the while loop we have $\max \{\alpha^\as, \alpha^\slls\} = \alpha^\slls = 1$. For this purpose, let $(x,s) \in \cal N^2(\beta/2)$
be the iterate computed on line~\ref{eq:line-corrector} during the last
iteration of the while loop. From here, let $(\Delta x^\as, \Delta s^\as)$, $(\Delta x^\slls, \Delta s^\slls)$, be the affine scaling
and subspace LLS directions computed at $(x,s)$, and let $(\Balg,\Nalg)$,
$(\Valg,\Ualg)$ be the corresponding associated partition and cheap lift
subspaces at $(x,s)$. We now claim that if $\alpha^\as = 1$, then the affine
scaling and subspace LLS directions are equal. Assuming $\alpha^\as = 1$, we have that $(x+\Delta x^\as, s+\Delta
s^\as) \in \Primal \times \Dual$ is an optimal primal-dual pair. In
particular, $\0_n = (x+\Delta x^\as)(s+\Delta s^\as) = \Delta x^\as \Delta
s^\as$, where the first equality is by complementary slackness and the second
equality is by the defining equation for the affine scaling
direction~\eqref{aff:sum}.  Recalling that $\Balg = \{i \in [n]: |\Delta x^\as_i|/x_i <
|\Delta s^\as_i|/s_i\}$, $\Nalg = [n] \setminus \Balg$ and that $x,s > \0_n$, we conclude that
$\Delta x^\as = (\0_{\Balg},-x_{\Nalg}) \in W \cap \R^n_{\Nalg}$ and $\Delta
s = (-s_{\Balg}, \0_{\Nalg}) \in W^\perp \cap \R^n_{\Balg}$. Therefore by
the output guarantee~\eqref{eq:cl-kernel} of \nameref{alg:cheap_lift}, we have
that $(\Delta x^\as, \Delta s^\as) \in \Valg \times \Ualg$. Since
$\norm{(x_{\Nalg}+x_{\Nalg}^\as,s_{\Balg}+s^\as_{\Balg})} = 0$, we must have that
$(x^\as,s^\as) \in \Valg \times \Ualg$ are the (unique) optimal solutions to the subspace LLS programs~\eqref{eq:primal-slls} and~\eqref{eq:dual-slls} at
$(x,s)$ with respect to $(\Balg,\Nalg,\Valg,\Ualg)$. In particular, $(\Delta
x^\as, \Delta s^\as) = (\Delta x^\slls, \Delta s^\slls)$, as claimed. By \Cref{prop:step-size}, using that $0=\norm{(x+\Delta x^\slls)(s+\Delta s^\slls)}=\norm{\Delta x^\slls \Delta s^\slls}$, we get $\alpha^\slls=1$, as needed. 

Let $(x^\star,s^\star)$ be the iterate defined right after the while loop on
line~\ref{line:optimal}, and let $B = \supp(x^\star)$ and $N =
\supp(s^\star)$. Since $\alpha^\slls=1$ by the above, we have that
$(x^\star,s^\star) = (x+\Delta x^\slls, s+\Delta^\slls) \in \Primal \times \Dual$ and
$\gap(x^\star,s^\star) = 0$. Therefore, $(x^\star,s^\star)$ is an optimal primal-dual pair, which proves~\eqref{ipm:optimal}. 

Let $\mu$ be as defined on line~\ref{eq:line-mu-lls}. By the above, $\mu =
\gap(x,s)$. By the guarantees of \Cref{prop:step-size} used on
line~\ref{eq:line-lls-step}, we must have that $\norm{\Delta x^\slls \Delta
s^\slls} \leq \beta \mu/9$ and $\norm{(x+\Delta x^\slls)(s+\Delta
s^\slls)} = 0$. For $i \in [n]$, using that $|\Delta x^\slls_i \Delta
s^\slls_i| \leq \beta\mu/9 < (1-\beta/2)\mu \leq x_is_i$
(\Cref{prop:x_i-s_i}) and $(x_i + \Delta x^\slls_i)(s_i + \Delta
s^\slls_i) = 0$, we conclude that either $x^\star_i = x_i + \Delta
x^\slls_i > 0$ and $s^\star_i = s_i + \Delta s^\slls_i = 0$, and thus
$i \in B$, or that $x^\star_i = x_i + \Delta x^\slls_i = 0$ and
$s^\star_i = s_i + \Delta s^\slls_i > 0$, and thus $i \in N$. In
particular, we see that $(B,N)$ partitions $[n]$. This verifies the output
guarantee~\eqref{ipm:partition}. 

Given the above, for $(v^\star,w^\star) \coloneqq (-\Delta s^\slls/\mu,-\Delta x^\slls/\mu) \in W^\perp \times W$ as
set on line~\ref{eq:line-set-vu-lls}, we have that $v^\star_B = s_B/\mu > \0_B$ and
$u^\star_N = x_N/\mu > \0_N$ which verifies~\eqref{ipm:certificates}. We now establish~\eqref{ipm:centrality} as follows:
\begin{align}
\norm{(x^\star_B v^\star_B, s^\star_Nw^\star_N) - \1_n} &= \norm{\frac{((x_B+\Delta x^\slls_B) s_B,(s_N + \Delta s^\slls_N)x_N)}{\mu}-\1_n} \nonumber \\
 &= \norm{\frac{x s}{\mu} - \1_n - \frac{\Delta x^\slls \Delta s^\slls}{\mu}} \leq \norm{\frac{x s}{\mu} - \1_n} + \norm{\frac{\Delta x^\slls \Delta s^\slls}{\mu}} \nonumber \\
&\leq \beta/2 + \beta/9 \leq \beta. \label{eq:cent-guarantee}
\end{align}

For the strongly polynomial bound on each iteration, we must simply show that
all the computations performed within the while loop are strongly polynomial.
In particular, we must check the corrector, affine scaling and subspace LLS
steps can all be computed in strongly polynomial, and that the calls to
\nameref{alg:cheap_lift} run in strongly polynomial time (under the
assumption that the $\SVDA$-\textsc{SVD} solver is strongly polynomial).
These claims are verified in \Cref{prop:compute-standard-step} (corrector and
affine scaling), \Cref{prop:lls-step-formula} (subspace LLS
direction),~\Cref{prop:step-size} (subspace LLS step length) and
\Cref{lem:cheap_lift_correct} (\nameref{alg:cheap_lift}). Thus, each
iteration runs in strongly polynomial time as needed.
\end{proof}

\subsection{Bounding the Number of Iterations to Traverse a Polarized Segment}\label{sec:running-time-analysis}
In this section, we prove the following bound.

\begin{theorem}\label{thm:polar-ipm}
For $\mu_0>\mu_1\ge 0$, let $\CP[\mu_1, \mu_0]$ be $\polarized$-polarized for $\polarized \in (0,1]$. Then,
given an iterate $z \in \cal N^2(\beta)$, $\beta \in (0,1/6]$, with parameter $\gap(z)  \in
(\mu_1,\mu_0)$, the algorithm \nameref{alg:subspace_ipm} (Algorithm
\ref{alg:subspace_ipm}) takes $O\left(\tfrac{n^{1.5}}{\beta}\log(\tfrac{n\SVDA}{\beta \polarized})\right)$ many
iterations to find $z' \in \clN(\beta)$ such that $\gap(z') \le \mu_1$.
\end{theorem}
Together with \Cref{thm:wide-polar}, we obtain the proof of \Cref{th:main_upper_bound-curve}. 
\begin{proof}[Proof of \Cref{th:main_upper_bound-curve}]
As in the statement of the theorem, 
 let $\Gamma \colon (\mu_1,\mu_0) \to \clNw(\theta)$
  be any piecewise linear curve
satisfying $\gap\left(\Gamma(\mu)\right) = \mu$, $\forall \mu \in (\mu_1,\mu_0)$ with $T$ linear segments. According to \Cref{thm:wide-polar}, the segment $\CP[\mu_1,\mu_0]$ of the central path can be
decomposed into $T$ segments that are
$\frac{(1-\theta)^2}{16n^3}$-polarized. 
By \Cref{thm:polar-ipm}, \nameref{alg:subspace_ipm} instantiated with the
$2$-\textsc{SVD} solver (\ie, $\SVDA=2$) from \Cref{thm:approx-svd} traverses these segments in
$O\left(\tfrac{n^{1.5}}{\beta}\log\left(\tfrac{n}{\beta
(1-\theta)}\right)T\right)$ iterations.
\end{proof}
The stronger form \Cref{th:main_upper_bound-slc} will be shown in
Section~\ref{sec:amortized} by amortizing the running time estimates over
subsequent polarized segments.  We now introduce the main potential for the analysis.

Our main focus will be on the analyzing the evolution of the
singular values of the lifting operator in the normalized subspace at the
current iterate. We define $z \coloneqq (x,s)$ to be a \emph{basic iterate}
if it either corresponds to an iterate computed just after the corrector step
on line~\ref{eq:line-corrector} during the course of the while loop of
\nameref{alg:subspace_ipm}, in which case $z \in \cal N^2(\beta/2)$, or if it
corresponds to the final iterate $z \in \clN(\beta)$ after the end of the
while loop.      

For a basic iterate $z \in \cal N^2(\beta/2)$ with $\gap(z)\in [\mu_1,\mu_0]$,  we
use the shorthands
\[
\lift_z \coloneqq  \lift_{N}^{\xerr^{-1}W}\quad\mbox{and} \quad \quad {\lift}_z^\perp \coloneqq  \lift_{B}^{\serr^{-1}W^\perp}\, 
\]
to denote the primal and dual lifting operators with respect to $(B,N)$ in
the respective normalized subspaces, as well as $\sigma(\lift_z)$ to be the vector of
singular values of the lifting operator. If $\mu_1 = 0$ and $z \in \clN(\beta)$ is the basic iterate with $\gap(z) = 0$, let us by convention
additionally define $\lift_z$ to be the zero operator $\pi_N(W)$ to
$\pi_B(W^\perp)$ and $\lift_z^\perp$ to be the zero operator from
$\pi_B(W^\perp)$ to $\pi_N(W)$. Recall that $\sigma^+(\lift_z) =
\sigma^+(\lift_z^\perp)$ by \Cref{lem:lift-duality}, that is, the non-zero
singular values of $\lift_z$ and $\lift_z^\perp$ are identical. Define the
parameter
\[
\tau \coloneqq \frac{\beta}{256n}\, .
\]
The combinatorial potential we use to measure progress on the
polarized segment is 
\begin{equation}
    \label{eq:def-maxexpidx}
    \maxexpidx(z) \coloneqq \cnt{\lift^\perp_z}{(\tfrac{\thresh}{\SVDA},\infty)} = \cnt{\lift_z}{(\tfrac{\thresh}{\SVDA},\infty)}  \coloneqq \left|\left\{i \geq 1: \sigma_i(\lift_z) > \tfrac{\thresh}{\SVDA}\right\}\right|,\, 
\end{equation} 
the number of singular values of $\lift_z$ (and $\lift_z^\perp$) that are
larger than the threshold $\frac{\thresh}{\SVDA}$. Note that if
$\maxexpidx(z) > 0$, then 
\[
\sigma_{\maxexpidx(z)}(\lift_z) = \sigma_{\maxexpidx(z)}(\lift^\perp_z) = \min \left\{\sigma_i(\lift_z): \sigma_i(\lift_z) > \frac{\thresh}{\SVDA}, i \geq 1 \right\}\, .
\]
By convention, recall that if $B = \emptyset$, then $\lift_z$ is the unique
linear operator from $W$ to $\{0\}$, and if $N=\emptyset$, then $\lift_z$
if the unique linear operator from $\{0\}$ to $W^\perp$. In particular, $(B,N)$ is a trivial partition, then $\maxexpidx(z) = 0$. The key property of
\nameref{alg:subspace_ipm} is given by the following lemma. 

\begin{lemma}
Let $z \in \cal N^2(\beta)$, $\beta \in (0,1/6]$, $\gap(z) \in (\mu_1,\mu_0]$, be a basic iterate
on the polarized segment $\CP[\mu_1,\mu_0]$ with a polarization partition $B
\cup N = [n]$. Then, there exists a constant $\Cit \geq 1$, such that for any
basic iterate $z' \in \clN(\beta)$ computed after at least $\lfloor \Cit
\tfrac{\sqrt{n}}{\beta} \log (\tfrac{n \SVDA}{\beta \polarized}) \rfloor$
iterations from $z$, either $\gap(z') \leq \mu_1$ or $\mu_1 < \gap(z') \leq
\gap(z)$ and $\maxexpidx(z') < \maxexpidx(z)$.  
\label{lem:progress}
\end{lemma}

\begin{remark}
In the above lemma, if $\maxexpidx(z) = 0$, then $z'$ must pass the end of
segment $(\mu_1,\mu_0]$, since the potential is $\maxexpidx(z')$ a
non-negative integer. For convenience, if $z^*$ denotes the final optimal iterate, we define the
next basic iterate after $z^*$ to be $z^*$ itself. In this way,
\Cref{lem:progress} still directly applies if we reach $z^*$ before
$\lfloor \Cit \tfrac{\sqrt{n}}{\beta}\log(\tfrac{n \SVDA}{\beta \polarized}) \rfloor$
iterations after $z$.        
\end{remark}

\begin{proof}[Proof of \Cref{thm:polar-ipm}]
The potential $\maxexpidx(z) \leq \rank(\lift_z) \leq n$ at
the start of the segment and decreases by $1$ every
$\lceil \Cit \tfrac{\sqrt{n}}{\beta} \log(\tfrac{n \SVDA}{\beta \polarized})\rceil$
iterations while we remain in the segment by \Cref{lem:progress}. 
\end{proof}

The rest of the section is dedicated to the proof of \Cref{lem:progress}.  We
will require the following technical lemma which describes the evolution of
the singular values of the lifting operator along a polarized segment. The
proof is deferred to \Cref{sec:stab-proofs}.

\begin{restatable}[{Stability of singular values on polarized segments}]{siamlemma}{stabpolar}\label{lem:stability_singular_values}
    Let $\CP[\mu_1,\mu_0]$ be a $\polarized$-polarized segment of the central path with partition $B \cup N = [n]$. Let $z \in \cal N^2(\beta)$, $z' \in \clN(\beta)$, $\beta \in (0,1/6]$, such that $\mu \coloneqq \gap(z)$ and $\mu' \coloneqq \gap(z')$ satisfy $\mu_0 \geq \mu \geq \mu' \geq \mu_1$. Then we have: 
    \begin{equation}
        \frac{\polarized^2}{4n^2} \cdot \frac{\mu'}{\mu}\sigma(\lift_z) \le \sigma(\lift_{z'}) \le \frac{4n^2}{\polarized^2}\cdot \frac{\mu'}{\mu}\sigma(\lift_z).
    \end{equation}
\end{restatable}

\Cref{lem:stability_singular_values} asserts that, up to $\poly(n/\polarized)$ factors, the singular values \emph{scale
down} by a $\frac{\mu'}{\mu}$ factor. Recalling that the normalized
subspaces for $z,z'$ are $\nW = \nx^{-1} W$ and $\nW' = \nx^{'-1} W$, where $\nx \coloneqq \sqrt{x\mu/s}
\approx x$ and $\nx' \coloneqq \sqrt{x'\mu'/s'} \approx x'$, the above relation follows
straightforwardly from the polarization guarantee $(x_B,x_B') \approx
x\cp_B(\mu_0)$ and $(x_N, x'_N) \approx (\frac{\mu}{\mu_0} x\cp_N(\mu_0),
\frac{\mu'}{\mu_0} x\cp_N(\mu_0))$ (up to $\poly(n/\polarized)$
multiplicative factors) together with the variational characterization of
singular values. 

\medskip

The proof of \Cref{lem:progress} relies on the following concepts.
 For a segment $\CP[\mu',\mu]$, $0 \leq \mu' \leq \mu$ of the central
path, we  say that \nameref{alg:subspace_ipm} traverses this segment in
at most $T \geq 0$ iterations, if the number of iterations from the first
iterate $z^{\rm start} \in \cal N^2(\beta)$ with $\gap(z^{\rm start}) \leq
\mu$ to the first iterate $z^{\rm end} \in \cal N^2(\beta)$ with $\gap(z^{\rm
end}) \leq \mu'$ is at most $T$. Important to the analysis is the distinction
between \emph{long} and \emph{short} segments. We say that $\CP[\mu',\mu]$ is
\emph{long} if $\mu/\mu' \geq \poly(\tfrac{n\SVDA}{\beta \polarized})$ and \emph{short} otherwise, where $\polarized \in (0,1]$ is the polarization parameter. We recall from~\ref{i:gap-decrease} in \Cref{prop:predictor-corrector}, the standard predictor-corrector IPM traverses $\CP[\mu',\mu]$ in at most
$\left \lceil \tfrac{3\sqrt{n}}{\beta}\log(\tfrac{\mu}{\mu'}) \right \rceil$ iterations. In particular, short segments can be
traversed in $O(\tfrac{\sqrt{n}}{\beta}\log(\tfrac{n \SVDA}{\beta \polarized}))$ iterations. As
\nameref{alg:subspace_ipm} always takes predictor steps that are at least as
long as the affine scaling steps, the same guarantee holds for our IPM.

\begin{proof}[Proof of \Cref{lem:progress}]
The analysis is divided into two cases.

\paragraph{Case 1: $\mu_1 < \gap(z) \leq \mu_1 \tfrac{72n^{1.5}}{\beta\polarized}$ or $\cnt{\lift_z}{(\thresh/\SVDA,\SVDA/\thresh]} \neq 0$}

If the first condition holds, the segment
$\CP[\mu_1,\gap(z)]$ is short, and hence the IPM passes the end of the
polarized segment in at most $O(\tfrac{\sqrt{n}}{\beta} \log(\tfrac{n \SVDA}{\beta \polarized}))$
iterations, as needed. Let us now assume that
$\cnt{\lift_z}{(\thresh/\SVDA,\SVDA/\thresh]} \neq 0$ (note that $(B,N)$ is non-trivial in this case), or equivalently,
that $\maxexpidx(z) \neq 0$ and
$\sigma_{\maxexpidx(z)}(\lift_z) \in (\thresh/\SVDA,\SVDA/\thresh]$.  

Using the stability of singular values, we will show that the number of large
singular values quickly drops by one as long as we remain in the segment. Let $z'$ be any basic iterate satisfying $\gap(z') \le \frac{\thresh^2}{\SVDA^2}\cdot
\frac{\polarized^2}{4n^2}\cdot \gap(z)$. If $\gap(z') \geq \mu_1$ and $i \geq \maxexpidx(z)$, then by \Cref{lem:stability_singular_values},
\begin{equation}
\label{eq:sing-drop}
\sigma_i(\lift_{z'}) \le\frac{\thresh^2}{\SVDA^2}\cdot \sigma_i(\lift_z) \leq  \frac{\thresh^2}{\SVDA^2}\cdot \sigma_{\maxexpidx(z)}(\lift_{z})
                \leq \frac{\tau^2}{\SVDA^2} \cdot \frac{\SVDA}{\tau} = \frac{\thresh}{\SVDA}\, .
\end{equation}
In particular, if $\gap(z') \geq \mu_1$, then $\maxexpidx(z') < \maxexpidx(z)$.
Since the IPM traverses \ifdefined\SIAMversion \newline \else \fi $\CP\left[\frac{\thresh^2}{\SVDA^2}\cdot
\frac{\polarized^2}{4n^2}\cdot\gap(z),\gap(z)\right]$ in at most
$O(\tfrac{\sqrt{n}}{\beta} \log(\tfrac{n \SVDA}{\beta \polarized}))$
iterations, recalling that $\thresh = \frac{\beta}{256n}$, this
proves the lemma in this case.

\paragraph{Case 2: $\gap(z) > \mu_1\tfrac{72n^{1.5}}{\beta \polarized}$ and
$\cnt{\lift_z}{(\thresh/\SVDA,\SVDA/\thresh]} = 0$}

We first note that the condition on the singular values can be restated
equivalently as $\maxexpidx(z) = 0$ (that is,
$\cnt{\lift_z}{(\tfrac{\tau}{\SVDA},\infty)} = 0$) or $\maxexpidx(z)
> 0$ and $\sigma_{\maxexpidx(z)}(\lift_z) > \SVDA/\tau$. In this case,
there is a large gap in the singular values around the threshold, namely, for
$i > \maxexpidx(z)$, we have $\sigma_i(\lift_z) \leq \tau/\SVDA$, and for
$1 \leq i \leq \maxexpidx(z)$, we have $\sigma_i(\lift_z) >
\SVDA/\tau$.  

In this setting, an identical analysis as in~\eqref{eq:sing-drop} applies as
long as we can \emph{quickly} compute a basic iterate $z'$ such that
\begin{equation}\label{eq:gap-zp-desired}
\gap(z') \leq \hat \mu \coloneqq \begin{cases} \mu_1\, ,&\mbox{if } \maxexpidx(z) = 0\, , \\
\max \left\{\mu_1, \frac{1}{\sigma_{\maxexpidx(z)}(\lift_z)} \cdot \frac{\thresh}{\SVDA}\cdot \frac{\polarized^2}{4n^2}\cdot \gap(z)\right\}\, ,&\mbox{if } \zeta(z) \geq 1\, . \end{cases}
\end{equation}
That is, for any basic iterate $z'$ as above, either $\gap(z') \leq \mu_1$ or
$\mu_1 < \gap(z') \leq \gap(z)$ and $\maxexpidx(z') < \maxexpidx(z)$. Note that
$\gap(z)/\hat \mu$ may be \emph{arbitrarily large} in this setting, and thus, we
may need to take a \emph{huge step} down the
central path. It is precisely for this purpose that we require subspace LLS
steps.

Let $z^+$ denote the next basic iterate after $z$. Then, by
Lemma~\ref{lem:gap-estimate} below, we have that
\begin{align}
\frac{\gap(z^+)}{\hat \mu} &\leq \frac{11}{\beta} \begin{cases} \frac{4n^{1.5}\mu_1}{\polarized \hat \mu}\, ,&\mbox{if } \maxexpidx(z) = 0\, , \\
\frac{4n^{1.5}\mu_1}{\polarized \hat \mu} + \frac{12n}{\sigma_{\zeta(z)}(\lift_z)}\frac{\gap(z)}{\hat \mu}\, , &\mbox{if } \zeta(z) \geq 1\, . \\ \end{cases} \nonumber \\ 
&\leq \, \frac{11}{\beta} \begin{cases} \frac{4n^{1.5}}{\polarized}\, , &\mbox{if } \maxexpidx(z) = 0\, , \\
\frac{4n^{1.5}}{\polarized} + \frac{48n^3 \SVDA}{\tau \polarized^2}\, ,&\mbox{if } \zeta(z) \geq 1  \, .\end{cases} \label{eq:final-short-segment} 
\end{align}  
This guarantees that after the next basic iterate $z^+$ after $z$, 
at most $O(\tfrac{\sqrt{n}}{\beta} \log(\tfrac{n \SVDA}{\beta
\polarized}))$ additional iterations are needed to obtain $z'$ as in
\eqref{eq:gap-zp-desired}. 

Combining the two case analyses, the constant $\Cit \geq 1$ in the lemma
statement can be chosen so that $\Cit \tfrac{\sqrt{n}}{\beta} \log(\tfrac{n
\SVDA}{\beta \polarized})$ iterations  is an upper bound on $1$ plus the
number of affine scaling iterations needed to divide the normalized gap by a
factor 
\begin{equation}
\frac{11}{\beta} \max
\left\{\frac{\SVDA^2}{\thresh^2} \cdot \frac{4n^2}{\polarized^2},
\frac{4n^{1.5}}{\polarized} + \frac{48n^3 \SVDA}{\tau \polarized^2}\right\}
\leq \left(\frac{n \SVDA}{\beta \polarized}\right)^{12}\, , \label{eq:progress-gap-decrease} 
\end{equation}
where we have used $\tau \coloneqq \frac{\beta}{256 n}$ and $\beta \in
(0,1/6]$. Thus, by \ref{i:gap-decrease} in \Cref{prop:predictor-corrector},
the choice $\Cit = 3 \cdot 12+2 = 38$ is sufficient.
\end{proof}
\begin{lemma}
\label{lem:gap-estimate}
Let $\CP[\mu_1,\mu_0]$ be a $\polarized$-polarized segment with polarization partition $B \cup N = [n]$. Let $z = (x,s) \in \cal N(\beta/2)$, $\beta \in (0,1/6]$ be a basic iterate satisfying $\mu \coloneqq \gap(z) \in (\tfrac{72n}{\beta \polarized}\cdot \mu_1,\mu_0]$  and $\cnt{\lift_z}{(\tfrac{\tau}{\SVDA},\tfrac{\SVDA}{\tau}]}=0$. Then, the next basic iterate $z^+$ computed in \nameref{alg:subspace_ipm} after $z$ satisfies 
\begin{equation}
\label{eq:long-step}
\gap(z^+) \leq \frac{11}{\beta}\cdot \begin{cases} \frac{4n^{1.5}\mu_1}{\polarized }\, ,& \mbox{if }\zeta(z)=0\, , \\ \frac{4n^{1.5}\mu_1}{\polarized } + \frac{12n}{\sigma_{\zeta(z)}(\lift_z)}\gap(z)\, ,&\mbox{if } \zeta(z) \geq 1\, . \end{cases}
\end{equation}
\end{lemma}

It remains to prove \Cref{lem:gap-estimate}. We will show the subspace LLS step from
$z=(x,s) \in \cal N^2(\beta/2)$ achieves \eqref{eq:long-step}. We will show this by comparing
these steps to the \emph{ideal direction}, which goes straight to the end of
the polarized segment, as defined below:

\begin{definition}[Ideal Direction]
For an $z \in \cal N(\beta)$, $\beta \in (0,1/6]$, satisfying $\gap(z) >
\mu_1$, we define the \emph{ideal direction} from $z$ towards $\mu_1$ to be    
\[
\Delta z^{\ideal} \coloneqq z\cp(\mu_1)-z \coloneqq (x\cp(\mu_1)-x,s\cp(\mu_1)-s) \eqqcolon (\Delta x^{\ideal}, \Delta s^{\ideal}).
\]
That is, $\Delta z^\ideal$ is the direction from the current iterate to the
central path point $z\cp(\mu_1)$.
\label{def:ideal}
\end{definition}
We now recall the notation used for cheap lift subspaces in the algorithm.
We define
\begin{align}
(\MP,\MD) \gets& \operatorname{Cheap-Lift-Subspace}\left(\mA,\frac{1}{x},B,N,\SVDA \right), \nonumber \\
(V,U) \gets& (\im(\MP),\im(\MD)), \label{def:good-subspaces}
\end{align}
to be the output of the call to algorithm
$\nameref{alg:cheap_lift}$ with respect to the \emph{true
partition} $(B,N)$. By the guarantees of
$\nameref{alg:cheap_lift}$, recall that if $(B,N) =
(\emptyset,[n])$, then $(V,U) = (W,\{\0_n\})$, and if $(B,N) =
([n],\emptyset)$, then $(V,U) = (\{\0_n\},W^\perp)$. 

The crux of the argument will be to show that a suitable projection of
$\Delta x^\ideal$ onto $V$ and $\Delta s^\ideal$ onto $U$ induces a step
satisfying~\eqref{eq:long-step}. We recall the notation $\error
\coloneqq \sqrt{xs/\mu}$, $(\nx,\ns) \coloneqq (x\error^{-1},s\error^{-1})$, $\mu \coloneqq \gap(z)$, $\wh W \coloneqq \hat{x}^{-1}
W$, and $\wh W^\perp \coloneqq \hat{s}^{-1} W^\perp$ from \Cref{def:normalized}. Similarly, we let $\wh V \coloneqq
\hat{x}^{-1} V$, $ \wh U \coloneqq \hat{s}^{-1} U$ and
$(\ndx^{\ideal},\nds^{\ideal}) \coloneqq (\hat{x}^{-1} \Delta x^{\ideal},
\hat{s}^{-1} \Delta s^{\ideal})$.

We now define the projected ideal directions
\begin{equation*}
\Delta x^{\prj} \coloneqq \argmin_{v \in V} \norm{\nx^{-1}_N(v_N-\Delta x_N^\ideal)}, \quad \quad 
\Delta s^{\prj} \coloneqq \argmin_{u \in U} \norm{\ns^{-1}_B(u_B-\Delta s_B^{\ideal})},
\end{equation*}
where we recall that the minimizers are unique since $\dim(\pi_N(V))=\dim(V)$ and $\dim(\pi_B(U)) = \dim(U)$. Define the normalized projections by $(\ndx^\prj,\nds^\prj) \coloneqq (\nx^{-1} \Delta x^\prj, \ns^{-1} \Delta s^\prj)$. By construction of the projected ideal directions, note that
\begin{equation}
(\ndx^\prj_N,\nds^\prj_B) = (\Pi_{\pi_N(\wh V)}(\ndx^{\ideal}_N), \Pi_{\pi_B(\wh U)}(\nds^{\ideal}_B)). \label{eq:ideal-proj-interpret}
\end{equation}
Let us now define 
\begin{equation}
\wh V^\perp_N \coloneqq \pi_N(\wh W) \cap \pi_N(\wh V)^\perp \quad \quad \wh U^\perp_B \coloneqq \pi_B(\wh W^\perp) \cap \pi_B(\wh U)^\perp, \label{eq:cl-orth}
\end{equation}
to be the corresponding orthogonal complements inside $\pi_N(\wh W)$ and
$\pi_B(\wh W^\perp)$, where we note that
\[
\ndx^\ideal_N-\ndx^\prj_N \in \wh V^\perp_N \quad \quad \nds^\ideal_B-\nds^\prj_B \in \wh U^\perp_B.
\]

Our goal will be to show that $(\Delta x^\prj, \Delta s^\prj)$ induces a
trust-region step achieving the guarantees of~\eqref{eq:long-step}. This will
then imply that the subspace LLS step $(\Delta x^\slls, \Delta s^\slls)$
also achieves the guarantee in~\eqref{eq:long-step}. 

The analysis relies on two auxiliarly lemmas. The first one exposes basic
properties of the ideal direction, while the second one gives basic
properties of the cheap lift subspaces. The proofs are deferred to
subsections~\ref{sec:ideal-properties} and~\ref{sec:cheap-lift-properties}.

\begin{restatable}{siamproposition}{idealbounds} \label{prop:ideal} Let $\CP[\mu_1,\mu_0]$ be a $\polarized$-polarized
segment with partition $B \cup N = [n]$. Let $z = (x,s) \in \cal N(\beta)$,
$\beta \in (0,1/6]$, with $\mu \coloneqq \gap(z) \in (\tfrac{72 n^{1.5}}{\beta
\polarized}\cdot\mu_1 , \mu_0]$. Then, the ideal direction $\Delta z^\ideal$ towards $\mu_1$ satisfies
\begin{align}
\norm{(\ndx^{\ideal},\nds^\ideal)} &\leq 4n, \label{eq:ideal-dx-ds} \\
\norm{(\error_N+\ndx^{\ideal}_N,\error_B+\nds^\ideal_B)} &\leq \frac{2n^{1.5}\mu_1}{\polarized \mu} \leq \frac{\beta}{36}. \label{eq:ideal-rx-rs} \end{align}
\end{restatable}

\begin{restatable}{siamlemma}{subspaceprops}\label{lem:subspace-props} Let $\CP[\mu_1,\mu_0]$ be a
$\polarized$-polarized segment with partition $B \cup N = [n]$. Let $z =
(x,s) \in \cal N(\beta)$, $\beta \in (0,1/6]$. Let $V \coloneqq \im(\MP)$, $U \coloneqq \im(\MD)$ be the subspaces returned by
\nameref{alg:cheap_lift} on $(\mA,x^{-1},B,N,\SVDA)$ as
in~\eqref{def:good-subspaces}, and $\lift_z \coloneqq \lift^{\nx^{-1}
W}_N$, $\lift^\perp_z \coloneqq \lift^{\ns^{-1}W^\perp}_B$. Then, if
$\cnt{\lift_z}{(\tau/\SVDA,\SVDA/\tau]} = 0$, we have
\begin{enumerate}
\item \label{eq:sub-dim} $\dim(V) = \dim(\pi_N(W))-\zeta(z)$, $\dim(U) = \dim(\pi_B(W^\perp))-\zeta(z)$.

\item \label{eq:sub-lift-cost} $\norm{\hat{v}_B} \leq 2\SVDA \sigma_{\zeta(z)+1}(\lift_z) \norm{\hat{v}_N} \leq 2\tau \norm{\hat{v}_N}$, $\forall \hat{v} \in \wh V \coloneqq \nx^{-1} V$, and

$\norm{\hat{u}_N} \leq 2\SVDA \sigma_{\zeta(z)+1}(\lift^\perp_z) \norm{\hat{u}_B} \leq 2\tau \norm{\hat{u}_B}$, $\forall \hat{u} \in \wh U \coloneqq \ns^{-1} U$.
\end{enumerate}
\end{restatable}

We are ready to prove \Cref{lem:gap-estimate}.
\begin{proof}[Proof of \Cref{lem:gap-estimate}]
We show that the subspace LLS step $(\Delta x^\slls, \Delta s^\slls)$ with respect to $(B,N)$ is sufficient to achieve the guarantees of~\eqref{eq:long-step}. As we show below, this will correspond to a long trust-region step and hence \Cref{lem:as-vs-tr} will guarantee that $(\Balg_z,\Nalg_z) = (B,N)$. In particular, the subspace LLS direction computed by the algorithm will be the same as $(\Delta x^\slls, \Delta s^\slls)$.    

We now analyze the norms and residuals of the subspace LLS direction $(\Delta x^\slls, \Delta s^\slls)$ with respect to $(B,N)$. We first provide sufficient conditions to ensure that $z^\slls \coloneqq
(x+\alpha^\slls \Delta x^\slls,s+\alpha^\slls \Delta s^\slls)$, corresponding
to the subspace LLS step applied to $z$, satisfies~\eqref{eq:long-step}. Letting $\mu \coloneq \gap(z)$, we will show below that
\begin{align}
\norm{(\ndx^\slls_B,\nds_N^\slls)} &\leq \frac{\beta}{18}\, , \label{eq:ipm-dx-ds} \\
\norm{(\error_N+\ndx^\slls_N,\error_B+\nds^\slls_B)} &\leq \begin{cases} \frac{2n^{1.5}\mu_1}{\polarized \mu}\, , &\mbox{if } \maxexpidx(z) = 0\, , \\ \frac{2n^{1.5}\mu_1}{\polarized \mu}+ \frac{6n}{\sigma_{\maxexpidx(z)}(\lift_z)}\, , &\mbox{if } \zeta(z) \geq 1 \, .\end{cases} \label{eq:ipm-rx-rs}
\end{align}
By our assumption that $\mu >
\tfrac{72 n^{1.5}}{\beta \polarized}\cdot \mu_1 $ and, if $\zeta(z) \geq 1$, that $\sigma_{\zeta(z)}(\lift_z) >
\frac{\SVDA}{\tau}\geq \tfrac{256 n}{\beta}$, the upper bound in \eqref{eq:ipm-rx-rs} is always at most $\beta/18$.
From \Cref{prop:trust-to-step-size}, we now get
\begin{equation*}
\begin{aligned}
\delta \coloneqq \norm{\frac{\Delta x^\slls \Delta s^\slls}{\mu}} &\leq \frac\beta9\quad\mbox{ and }\\
\eps \coloneqq \norm{\frac{(x+\Delta x^\slls)(s+\Delta
s^\slls)}{\mu}} &\leq \begin{cases} \frac{4n^{1.5}\mu_1}{\polarized \mu}\, , &\mbox{if } \maxexpidx(z) = 0\, , \\ \frac{4n^{1.5}\mu_1}{\polarized \mu}+ \frac{12n}{\sigma_{\maxexpidx(z)}(\lift_z)}\, , &\mbox{if } \zeta(z) \geq 1  \, .\end{cases}
\end{aligned}\, ,
\end{equation*}
and $\eps \le \beta/9$.  Consequently, 
\Cref{prop:step-size} is applicable. Choosing  $1-\frac{9\eps}{\beta} \leq \alpha^\slls \leq 1-\frac{8\eps}{\beta}$, we have
\[
\gap(z^\slls) \leq (1+\tfrac{1}{8})\tfrac{9\eps \mu}{\beta} \leq \frac{11}{\beta}\eps \mu \leq \frac{11}{\beta}\cdot \begin{cases} \frac{4n^{1.5}\mu_1}{\polarized }\, ,& \mbox{if }\zeta(z)=0\, , \\ \frac{4n^{1.5}\mu_1}{\polarized } + \frac{12n}{\sigma_{\zeta(z)}(\lift_z)}\mu\, ,&\mbox{if } \zeta(z) \geq 1\, .\end{cases}\, ,
\]
which is precisely the guarantee required for~\eqref{eq:long-step}.
Since~\eqref{eq:ipm-dx-ds} and~\eqref{eq:ipm-rx-rs} are at most $\beta/18 \leq 1/30$, then assuming these bounds hold, \Cref{lem:as-vs-tr}
implies that the associated partition $(\Balg_z,\Nalg_z) = (B,N)$. In particular,
if the bounds hold, \nameref{alg:subspace_ipm} correctly computes the
subspace LLS direction $(\Delta x^\slls, \Delta s^\slls)$. Since $\gap(z^+) \leq \gap(z^\slls)$, as we take the better of the AS and subspace LLS step, this will yield the claimed bound \eqref{eq:long-step}.

We now prove~\eqref{eq:ipm-dx-ds}. By \Cref{lem:subspace-props}
part~\eqref{eq:sub-lift-cost}, we have that the subspace LLS direction
$(\Delta x^\slls,\Delta s^\slls)$ satisfies
\begin{align*}
\norm{(\ndx^\slls_B,\nds^{\slls}_N)}
&\leq 2 \SVDA \sigma_{\zeta(z)+1}(\lift_z) \norm{(\ndx^\slls_N,\nds^\slls_B)}
\leq 2\tau \norm{(\error_N,\error_B)}
= 2 \tau \sqrt{n} \leq \frac{\beta}{18},
\end{align*}
as needed, where the second inequality uses $(\ndx^\slls_N,\nds^\slls_B) = (\Pi_{\pi_N(\wh V)}(\error_N),\Pi_{\pi_B(\wh U)}(\error_B))$ as per \Cref{rem:subspace_lls}.

We now prove~\eqref{eq:ipm-rx-rs}. By definition of subspace LLS (\Cref{def:subspace_lls}) and the
triangle inequality, we have that
\begin{align}
\norm{(\error_N+\ndx^\slls_N,\error_B+\nds^\slls_B)} &\leq \norm{(\error_N + \ndx^\prj_N, \error_B + \nds^\prj_B)} \nonumber \\ &\leq \norm{(\error_N+\ndx^\ideal_N,\error_B+\nds^\ideal_B)} + \nonumber \\ &\quad \,
\norm{(\ndx^\prj_N-\ndx^\ideal_N,\nds^\prj_B-\nds^\ideal_B)} \nonumber \\
&\leq \frac{2n^{1.5}\mu_1}{\polarized \mu} + 
\norm{(\ndx^\prj_N-\ndx^\ideal_N,\nds^\prj_B-\nds^\ideal_B)},
\label{eq:error-decomp}
\end{align}
where the last inequality follows by~\eqref{eq:ideal-rx-rs} in
\Cref{prop:ideal}.

We now bound the second term of~\eqref{eq:error-decomp}. Note that if
$\zeta(z)=0$, then by \Cref{lem:subspace-props} part~\eqref{eq:sub-dim}, we
have that $\pi_N(\wh V)=\pi_N(\wh W)$ and $\pi_B(\wh U)=\pi_B(\wh W^\perp)$
and hence $\wh V^\perp_N=\{\0_n\}$ and $\wh U^\perp_B=\{\0_n\}$. In particular,
\[
\norm{\ndx^\prj_N-\ndx^\ideal_N} = 0 \quad \quad
\norm{\nds^\prj_B-\nds^\ideal_B} = 0. 
\]   
Thus~\eqref{eq:ipm-rx-rs} follows directly from~\eqref{eq:error-decomp}.
Now assume that $\zeta(z) \geq 1$. Note that this implies that $(B,N)$ is a non-trivial partition. By \Cref{lem:subspace-props}, we have that
$\pi_N(\wh V)$ is a $(2\SVDA)$-approximate singular subspace for $\lift_z$ with
dimension $\dim(\pi_N(\wh W))-\zeta(z) < \dim(\pi_N(\wh W))$. By our assumption that $\cnt{\lift_z}{(\tau/\SVDA,\SVDA/\tau])}=0$, we have that
\[
\sigma_{\zeta(z)}(\lift_z) > \frac{\SVDA}{\tau} = \frac{\SVDA^2}{\tau^2}
\cdot\frac{\tau}{\SVDA} \geq \frac{\SVDA^2}{\tau^2} 
\sigma_{\zeta(z)+1}(\lift_z) \geq 6\SVDA \sigma_{\zeta(z)+1}(\lift_z).
\]
By \Cref{lem:approx-complement}, recalling that $\wh V^\perp_N \coloneqq \pi_N(\wh W) \cap \pi_N(\wh V)^\perp$, we therefore have that
\begin{equation}
\label{eq:min-sing-lb-primal}
\sigma_{\min}(\restr{\lift_z}{\wh V^\perp_N})^2 \geq \sigma_{\zeta(z)}(\lift_z)^2 - (2\SVDA \sigma_{\zeta(z)+1}(\lift_z))^2 \geq \frac{8}{9} \sigma_{\zeta(z)}(\lift_z)^2.
\end{equation}
By a symmetric argument, we have that $\wh U^\perp_B \coloneqq \pi_B(\wh W^\perp) \cap \pi_B(\wh U_B)^\perp$ is a $\SVDA$-approximate singular subspace for $\lift^\perp$
of dimension $\dim(\pi_B(\wh W^\perp))-\zeta(z)$ and satisfies
\begin{equation}
\label{eq:min-sing-lb-dual}
\sigma_{\min}(\restr{\lift_z^\perp}{\wh U^\perp_B})^2 \geq \frac{8}{9} \sigma_{\zeta(z)}(\lift_z^\perp)^2 = \frac{8}{9} \sigma_{\zeta(z)}(\lift_z)^2.
\end{equation}

We now bound the second term in~\eqref{eq:error-decomp} as follows:
\begin{align}
\norm{\ndx^\prj_N-\ndx^\ideal_N}^2 &\leq \frac{\norm{\lift_z(\ndx^\prj_N-\ndx^\ideal_N)}^2}{\sigma_{\rm min}(\restr{\lift_z}{\wh V^\perp_N})^2} \leq \frac{\norm{\ndx^\prj_B-\ndx^\ideal_B}^2}{\sigma_{\rm min}(\restr{\lift_z}{\wh V^\perp_N})^2} \nonumber \\
&\leq \frac{2\norm{\ndx^\prj_B}^2 + 2\norm{\ndx^\ideal_B}^2}{\sigma_{\rm min}(\restr{\lift_z}{\wh V^\perp_N})^2} \nonumber \\ 
&\leq \frac{2(2\tau)^2\norm{\ndx^\prj_N}^2 + 2\norm{\ndx^\ideal_B}^2}{\sigma_{\rm min}(\restr{\lift_z}{\wh V^\perp_N})^2} \nonumber \\
&\leq \frac{2\norm{\ndx^\ideal_N}^2+2\norm{\ndx^\ideal_B}^2}{\sigma_{\rm min}(\restr{\lift_z}{\wh V^\perp_N})^2} \nonumber \\
&= \frac{2\norm{\ndx^\ideal}^2}{\sigma_{\min}(\restr{\lift_z}{\wh V^\perp_N})^2} \leq \frac{9\cdot 2\norm{\ndx^\ideal}^2}{8\sigma_{\zeta(z)}(\lift_z)^2} = \frac{9\norm{\ndx^\ideal}^2}{4\sigma_{\zeta(z)}(\lift_z)^2}. \label{eq:second-term-part-1}
\end{align}
By a symmetric argument,
\begin{equation}
\norm{\nds^\prj_B-\nds^\ideal_B}^2 \leq \frac{2\norm{\nds^\ideal}^2}{\sigma_{\min}(\restr{\lift^\perp_z}{\wh U^\perp_B})^2} \leq \frac{9\norm{\nds^\ideal}^2}{4\sigma_{\zeta(z)}(\lift_z)^2}. \label{eq:second-term-part-2}
\end{equation}
Combining~\eqref{eq:second-term-part-1} and~\eqref{eq:second-term-part-2} together with~\eqref{eq:ideal-dx-ds} in \Cref{prop:ideal}, we get that
\begin{align}
\norm{(\ndx^\prj_N-\ndx^\ideal_N,\nds^\prj_B-\nds^\ideal_B)} &\leq \frac{3\norm{(\ndx^\ideal,\nds^\ideal)}}{2\sigma_{\zeta(z)}(\lift_z)} \leq \frac{3 \cdot 4n}{2\sigma_{\zeta(z)}(\lift_z)} \nonumber \\ &= \frac{6n}{\sigma_{\zeta(z)}(\lift_z)}. \label{eq:second-term}
\end{align}
The desired bound \eqref{eq:ipm-rx-rs} now follows by
combining~\eqref{eq:error-decomp} and~\eqref{eq:second-term}.
\end{proof}

\subsubsection{Stability of Singular Values on Polarized Segments}\label{sec:stab-proofs}

We now present the proof of \Cref{lem:stability_singular_values} on the evolution of singular values of the map $\lift_z$ on polarized segments of the central path. This will rely on the next lemma that  bounds the change in the singular values under a rescaling of the space. 

\begin{lemma}[{Stability of singular values for multiplicative perturbation}]\label{lem:general_rescaling_of_singular_values}
Let $W \subseteq \R^n$ be a subspace, and let $B\cup N=[n]$ be a non-trivial
partition. Let $y \in \R_{++}^n$, and let $\lift_N^{y^{-1}W} \colon
\pi_N(y^{-1}W) \to  \pi_B(y W^\perp)$ and $\lift_N^{W} \colon \pi_N(W) \to
\pi_B(W^\perp)$ be defined according to \Cref{def:lifting-map}. Let $\sigma
\coloneqq \sigma(\lift_N^{y^{-1}W})$ and $\hat \sigma \coloneqq
\sigma(\lift_N^{W})$ denote their respective singular values. Then, we have that
\begin{equation}
\frac{1}{\norm{y_B^{-1}}_\infty \norm{y_N}_\infty} \sigma \le \hat \sigma \le \norm{y_B}_\infty \norm{y_N^{-1}}_\infty \sigma \, .
\end{equation}
\end{lemma}
\begin{proof}
We only prove the second inequality. The first inequality then follows by
swapping $y$ with $y^{-1}$ and $\sigma$ with $\hat{\sigma}$. We denote $\lift
\coloneqq \lift_N^{y^{-1}W}$ and $\hat \lift \coloneqq \lift_N^{W}$. Since $y
\in \R^n_{++}$, we clearly have that $p \coloneqq \dim(\pi_N(W)) =
\dim(\pi_N(y^{-1}W))$. Therefore $\sigma$ and $\hat{\sigma}$ are both vectors
in $\R^p_+$. By \eqref{eq:min-max-sing}, recall that for $i \in [p]$, we have
that
\[
\sigma_k = \min_{\substack{S \subseteq \pi_N(y^{-1} W) \\ \dim(S) \geq p-k+1}} \sigma_1(\restr{\lift}{S}), \quad \quad
\hat{\sigma}_k = \min_{\substack{\hat{S} \subseteq \pi_N(W) \\ \dim(\hat{S}) \geq p-k+1}} \sigma_1(\restr{\hat{\lift}}{\hat{S}}).
\]
To prove the inequality, we will show that for any $S \subseteq
\pi_N(y^{-1}W)$, there exists $\hat{S} \subseteq \pi_N(W)$ with $\dim(S) =
\dim(\hat{S})$ satisfying $\sigma_1(\restr{\hat{\lift}}{\hat{S}}) \leq
\norm{y_B}_\infty \norm{y_N^{-1}}_\infty \sigma_1(\restr{\lift}{S})$. For this
purpose, define $\hat{S} = y_N S$, where clearly $\dim(\hat{S})=\dim(S)$ and $\hat{S} \subseteq y_N \pi_N(y^{-1} W) = \pi_N(W)$. From here, we have that
\begin{align*}
\sigma_1(\restr{\hat{\lift}}{\hat{S}}) &= 
\max_{\hat{x} \in \hat{S} \setminus \{\0\}} 
\frac{\norm{\hat{\lift}(\hat{x})}}{\norm{\hat{x}}} 
= \max_{x \in S \setminus \{\0\}} \frac{\norm{\hat{\lift}(y_N x)}}{\norm{y_N x}} 
\leq \max_{x \in S \setminus \{\0\}} \frac{\norm{y_B \lift(x)}}{\norm{y_N x}} \\
& \leq \max_{x \in S \setminus \{\0\}} \norm{y_B}_\infty \norm{y_N^{-1}}_\infty \frac{\norm{\lift(x)}}{\norm{x}} = \norm{y_B}_\infty \norm{y_N^{-1}}_\infty \sigma_1(\restr{\lift}{S}),
\end{align*}
where $\norm{\hat{\lift}(y_N x)} \leq \norm{y_B \lift(x)}$ follows since $(y_B \lift(x), y_N x) \in W$ and by the optimality of the lift provided by $\hat{\lift}(y_N x)$. The statement thus follows.
\end{proof}

   \stabpolar*
    \begin{proof}
Let $z \coloneqq (x,s)$ and $z' \coloneqq (x', s')$. Since $z \in \cal
N^2(\beta)$, note that $\mu = \gap(z) > 0$. If $(B,N)$ is a trivial partition, then
$\lift_z, \lift_{z'}$ are both zero operators on input spaces of the same dimension and hence the statement trivially holds. If $\mu' =
0$, then $\lift_{z'}$ is a zero operator on $\pi_N(W)$, and hence again the
statement holds trivially. Thus, we may assume that $\mu' > 0$ and that $(B,N)$ is a non-trivial partition of $[n]$. Consequently,
$z' \in \cal N^2(\beta)$.  

Let $\nx = \sqrt{\frac{x\mu}{s}}$ and $\nx' = \sqrt{\frac{x'\mu}{s'}}$ denote the normalized iterates. To prove the bounds, we apply \Cref{lem:general_rescaling_of_singular_values}
to the subspace $\frac{1}{\nx'} W$ and $\xy \coloneqq \frac{\nx}{\nx'}$, noting that $\xy \nx' = \nx$. Since $z,z' \in \cal N^2(\beta)$, these subspaces are well-defined as $\nx,\nx' \in \R^n_{++}$. By \Cref{prop:xi}, we have that
    \begin{equation}
    \frac{\sqrt{1-\beta}}{\sqrt{1+\beta}} \frac{x}{x'} \le  \frac{\nx}{\nx'} \leq \frac{\sqrt{1+\beta}}{\sqrt{1-\beta}} \frac{x}{x'}.
    \end{equation} 
    Combining the above with \Cref{prop:near-central} and \Cref{cor:polarization}, we get that
    \begin{equation}
    \begin{aligned}
       \norm{\xy_B}_\infty &= \norm{\frac{\nx_B}{\nx'_B}}_\infty 
           \leq \frac{\sqrt{1+\beta}}{\sqrt{1-\beta}} \norm{\frac{x_B}{x'_B}}_\infty 
           \leq \frac{\sqrt{1+\beta}(1+2\beta)}{(1-\beta)^{3/2}} \norm{\frac{x\cp(\mu)_B}{x\cp(\mu')_B}}_\infty  \le \frac{2n}{\polarized}, \\
        \norm{\xy_N^{-1}}_\infty &= \norm{\frac{\nx'_N}{\nx_N}}_\infty \le \frac{\sqrt{1+\beta}}{\sqrt{1-\beta}} \norm{\frac{x'_N}{x_N}}_\infty \leq \frac{\sqrt{1+\beta}(1+2\beta)}{(1-\beta)^{3/2}} \norm{\frac{x\cp(\mu')_N}{x\cp(\mu)_N}}_\infty \le \frac{2n}{\polarized}\cdot\frac{\mu'}{\mu},
    \end{aligned}
    \end{equation}
    and 
    \begin{equation}
    \begin{aligned}
       \norm{\xy_B^{-1}}_\infty &= \norm{\frac{\nx'_B}{\nx_B}}_\infty 
           \leq \frac{\sqrt{1+\beta}}{\sqrt{1-\beta}} \norm{\frac{x'_B}{x_B}}_\infty 
           \leq \frac{\sqrt{1+\beta}(1+2\beta)}{(1-\beta)^{3/2}} \norm{\frac{x\cp(\mu')_B}{x\cp(\mu)_B}}_\infty  \le \frac{2n}{\polarized}, \\
        \norm{\xy_N}_\infty &= \norm{\frac{\nx_N}{\nx'_N}}_\infty \le \frac{\sqrt{1+\beta}}{\sqrt{1-\beta}} \norm{\frac{x_N}{x'_N}}_\infty \leq \frac{\sqrt{1+\beta}(1+2\beta)}{(1-\beta)^{3/2}} \norm{\frac{x\cp(\mu)_N}{x\cp(\mu')_N}}_\infty \le \frac{2n}{\polarized}\cdot\frac{\mu}{\mu'},
    \end{aligned}
    \end{equation}
    using $\frac{(1+2\beta)(1+\beta)^{1/2}}{(1 - \beta)^{3/2}} \leq 2$ for $\beta \in (0,1/6]$, $\norm{\frac{x\cp(\mu)_B}{x\cp(\mu')_B}}_\infty = \max_{i \in B} \frac{x\cp(\mu)_i}{x\cp(\mu_0)_i}\frac{x\cp(\mu_0)_i}{x\cp(\mu')_i} \leq n \cdot \frac{1}{\polarized}$ and $\norm{\frac{x\cp(\mu')_N}{x\cp(\mu)_N}}_\infty$ $= \max_{i \in N} \frac{x\cp(\mu')_i}{x\cp(\mu_0)_i}\frac{x\cp(\mu_0)_i}{x\cp(\mu)_i} \leq \frac{\mu'}{\polarized \mu_0} \cdot \frac{n \mu_0}{\mu}$, together with the the analogous bounds for the remaining inequalities.
    
    Plugging these estimates into \Cref{lem:general_rescaling_of_singular_values} yields the result.
\end{proof}

\subsubsection{Properties of the Ideal Direction}
\label{sec:ideal-properties}

\idealbounds*
\begin{proof}
We first prove~\eqref{eq:ideal-dx-ds}
\begin{align*}
\norm{(\ndx^{\ideal},\nds^{\ideal})} &= \norm{(\nx^{-1}x\cp(\mu_1)-\error,\ns^{-1}s\cp_N(\mu_1)-\error)} \\
&\leq \norm{(\nx^{-1} x\cp(\mu_1),\ns^{-1}s\cp(\mu_1))} + \norm{(\error,\error)} \\
&\le \norm{(\nx^{-1} x \cp(\mu_1), \ns^{-1} s\cp(\mu_1))}_1 + \sqrt{2n} \\
&\leq \frac{n}{\sqrt{1-\beta}}(1+\mu_1/\mu) + \sqrt{2n} \quad \left(\text{ by \Cref{lem:local-norm-monotone} }\right) \\
&\leq 4n,  
\end{align*}
where the last inequality follows from $\beta \in(0,1/6]$ and $\mu \geq 6
\mu_1$.

We now continue with~\eqref{eq:ideal-rx-rs},
\ifdefined\SIAMversion
\begin{align*}
\norm{(\error_N+\ndx^{\ideal}_N,\error_B+\nds^\ideal_B)}
&= \norm{(\nx^{-1}_N x\cp_N(\mu_1),\ns^{-1}_B s\cp_B(\mu_1))} \\
&\leq \norm{\left(\frac{x\cp_N(\mu_1)}{x_N},\frac{s\cp_B(\mu_1)}{s_B}\right)}_\infty \norm{(\nx^{-1}_N x_N,\ns^{-1}_B s_B)} \\
&= \norm{\left(\frac{x\cp_N(\mu_1)}{x_N},\frac{s\cp_B(\mu_1)}{s_B}\right)}_\infty \norm{\error} \\
&\leq \frac{1}{1-\beta} \norm{\left(\frac{x\cp_N(\mu_1)}{x_N\cp(\mu)},\frac{s\cp_B(\mu_1)}{s\cp_B(\mu)}\right)}_\infty \sqrt{n} \\ & \hspace{12.8em} \left(\text{ by \Cref{prop:near-central} } \right) \\
&\leq \frac{n\mu_1}{(1-\beta)\mu\polarized} \cdot \sqrt{n} \leq \frac{2n^{1.5}\mu_1}{\polarized \mu}, \quad \left(\text{ by \Cref{cor:polarization} }\right) \\
&\leq \frac{\beta}{36}. \quad \left(\text{ since } \mu \geq \mu_1 \tfrac{72 n^{1.5}}{\beta \polarized}~\right)
\end{align*}
\else
\begin{align*}
\norm{(\error_N+\ndx^{\ideal}_N,\error_B+\nds^\ideal_B)}
&= \norm{(\nx^{-1}_N x\cp_N(\mu_1),\ns^{-1}_B s\cp_B(\mu_1))} \\
&\leq \norm{\left(\frac{x\cp_N(\mu_1)}{x_N},\frac{s\cp_B(\mu_1)}{s_B}\right)}_\infty \norm{(\nx^{-1}_N x_N,\ns^{-1}_B s_B)} \\
&= \norm{\left(\frac{x\cp_N(\mu_1)}{x_N},\frac{s\cp_B(\mu_1)}{s_B}\right)}_\infty \norm{\error} \\
&\leq \frac{1}{1-\beta} \norm{\left(\frac{x\cp_N(\mu_1)}{x_N\cp(\mu)},\frac{s\cp_B(\mu_1)}{s\cp_B(\mu)}\right)}_\infty \sqrt{n} \quad \left(\text{ by \Cref{prop:near-central} } \right) \\
&\leq \frac{n\mu_1}{(1-\beta)\mu\polarized} \cdot \sqrt{n} \leq \frac{2n^{1.5}\mu_1}{\polarized \mu}, \quad \left(\text{ by \Cref{cor:polarization} }\right) \\
&\leq \frac{\beta}{36}. \quad \left(\text{ since } \mu \geq \mu_1 \tfrac{72 n^{1.5}}{\beta \polarized}~\right)
\end{align*}
\fi
\end{proof}

\subsubsection{Dimension and Lifting Cost of Cheap Lift Subspaces}
\label{sec:cheap-lift-properties}

\subspaceprops*
\begin{proof}
Assume first that $(B,N)$ is a trivial partition. Recall that
$\maxexpidx(\lift_z) = \maxexpidx(\lift^\perp_z) = 0$. By the guarantees of
\nameref{alg:cheap_lift}, we have $(V,U) = (W,\{\0_n\})$ if $B = \emptyset$
and $(V,U) = (\{\0_n\},W^\perp)$ if $N = \emptyset$, and
thus~\eqref{eq:sub-dim} and~\eqref{eq:sub-lift-cost} follow directly. 

Now assume that $(B,N)$ is a non-trivial partition. We first prove the dimension guarantees for $V$ and $U$. By \Cref{lem:general_rescaling_of_singular_values}, using that $x^{-1} W = \error^{-1} \nx^{-1} W$, we have that
\[
\left(\norm{\error^{-1}_B}_\infty \norm{\error_N}_\infty\right)^{-1} \sigma(\lift_N^{x^{-1}W}) \leq \sigma(\lift_N^{\nx^{-1} W}) = \sigma(\lift_z) \leq \norm{\error_B}_\infty \norm{\error^{-1}_N}_\infty \sigma(\lift_N^{x^{-1}W}).
\]
By \Cref{prop:xi}, $\sqrt{1-\beta} \1_n \leq \error \leq \sqrt{1+\beta}
\1_n$, and since $\beta \in (0,1/6]$ we have that 
\[
\max \{\norm{\error^{-1}_B}_\infty \norm{\error_N}_\infty, \norm{\error_B}_\infty \norm{\error^{-1}_N}_\infty\} \leq \sqrt{\frac{1+\beta}{1-\beta}} \leq \sqrt{2}.
\]
Applying the same argument to $\lift_B^{x W^\perp}$ and $\lift_B^{\ns^{-1}
W^\perp} = \lift_z^\perp$, using that $x W^\perp = \tfrac{x}{\mu} W^\perp =
\error \ns^{-1} W$, we get that
\begin{equation}
\label{eq:cl-sing-comp}
\frac{1}{\sqrt{2}} \sigma(\lift_N^{x^{-1} W}) \leq \sigma(\lift_z) \leq \sqrt{2} \sigma(\lift_N^{x^{-1} W}), \quad \quad \frac{1}{\sqrt{2}} \sigma(\lift_B^{x W^\perp}) \leq \sigma(\lift_z^\perp) \leq \sqrt{2} \sigma(\lift_B^{x W^\perp}),
\end{equation}
where we recall that $\sigma^+(\lift_z) = \sigma^+(\lift_z^\perp)$ by
\Cref{lem:lift-duality}. By the guarantees of \nameref{alg:cheap_lift} and~\eqref{eq:cl-sing-comp}, we have that
\[
\cnt{\lift_z}{\tfrac{1}{\sqrt{2}\SVDA}} \leq \cnt{\lift_N^{x^{-1}W}}{1/\SVDA} \leq
\dim(V) \leq \cnt{\lift_N^{x^{-1}W}}{1} \leq \cnt{\lift_z}{\sqrt{2}}.
\]
Since $\cnt{\lift_z}{(\tfrac{\tau}{\SVDA},\tfrac{\SVDA}{\tau}]}=0$ and $\tfrac{\tau}{\SVDA} \leq \tfrac{1}{\sqrt{2}\SVDA} \leq \sqrt{2} \leq \tfrac{\SVDA}{\tau}$, we have $\dim(V) = \cnt{\lift_z}{\tfrac{\tau}{\SVDA}} = \cnt{\lift_z}{\tfrac{\SVDA}{\tau}}$. In particular, 
\begin{align*}
\dim(V) &= \cnt{\lift_z}{\tfrac{\tau}{\SVDA}} = \cnt{\lift_z}{[0,\infty)}-\cnt{\lift_z}{(\tfrac{\tau}{\SVDA},\infty)} \coloneqq \dim(\pi_N(\nx^{-1} W)) - \zeta(z) \\ &= \dim(\pi_N(W))-\zeta(z).
\end{align*}
By a completely symmetric argument on the dual, we get that
\[
\dim(U) = \cnt{\lift_z^\perp}{\tfrac{\tau}{\SVDA}} = \dim(\pi_B(W^\perp))-\zeta(z).
\]
We now prove the lifting cost guarantees. For $v \in V$ and the guarantees of \nameref{alg:cheap_lift}, we have that
\begin{align*}
\norm{\nx^{-1}_B v_B} 
&\leq \norm{\error_B}_\infty \norm{x^{-1}_B v_B}
\leq \norm{\error_B}_\infty \SVDA \sigma_{\dim(\pi_N(W))-\dim(V)+1}(\lift^{x^{-1}W}_N) \norm{x^{-1}_N v_N} \\
&\leq \norm{\error_B}_\infty \norm{\error_N^{-1}}_\infty \sqrt{2} \SVDA \sigma_{\zeta(z)+1}(\lift_z) \norm{\nx^{-1}_N v_N} 
\leq 2 \SVDA \sigma_{\zeta(z)+1}(\lift_z) \norm{\nx^{-1}_N v_N} \\
&\leq 2 \tau \norm{\nx^{-1}_N v_N}.
\end{align*}
By a symmetric algorithm for the dual, for $u \in U$, we have that
\[
\norm{\ns^{-1}_N u_N} \leq 2 \SVDA \sigma_{\zeta(z)+1}(\lift_z^\perp) \norm{\ns^{-1}_B u_B} \leq 2 \tau \norm{\ns^{-1}_B u_B},
\]
as needed.
\end{proof}
  
\section{Amortized Iteration Bound for \nameref{alg:subspace_ipm}}\label{sec:amortized}

In this section we prove \Cref{th:main_upper_bound-slc}. For this purpose we
will rely on \Cref{thm:delta-slc}, which decomposes the central path into
polarized segments where the sum of partition changes is bounded by the
straight line complexity. The proof of \Cref{th:main_upper_bound-slc} will
then follow from the following theorem, which shows that the number of
iterations of our IPM can be upper bounded in terms of a sum of partition
changes.

\begin{theorem}\label{thm:amortized_main}
For any $T \in \mathbb{N}$ let $0 \le \mu^{(T)} < \mu^{(T-1)} < \ldots < \mu^{(1)} < \mu^{(0)}$ such that for all $1 \le i \le T$ we have that $\CP[\mu^{(i)}, \mu^{(i-1)}]$ is $\polarized$-polarized with polarizing partition $B^{(i)} \cup N^{(i)}$, where $N^{(0)}\coloneqq\emptyset$. 
Then, \nameref{alg:subspace_ipm} equipped with a $\SVDA$-\textsc{SVD} solver, starting from any point $z^{(0)} \in \mathcal N^2(\beta)$ such that $\gap(z^{(0)}) \le \mu^{(0)}$, finds a point $z \in \clN(\beta)$ such that $\gap(z) \le \mu^{(T)}$ in 
\begin{equation}
    O\left(\tfrac{\sqrt{n}}{\beta} \log(\tfrac{n \SVDA }{\beta \polarized})\sum_{i = 1}^T \left(\left|N^{(i)} \Delta N^{(i-1)}\right|+1\right)\right)
\end{equation}
many iterations.
\end{theorem}

As explained in the introduction, the above theorem strengthens
\Cref{th:main_upper_bound-curve} by reducing the $n^{1.5}$ factor in the
iteration bound for traversing the $i$th polarized segment down to
$\sqrt{n}\left|N^{(i)} \Delta N^{(i-i)}\right|$. This will yield an overall
amortized improvement as long as the polarizing partitions do not change too
rapidly.   

We are now ready to give the proof of \Cref{th:main_upper_bound-slc}.

\begin{proof}[Proof of \Cref{th:main_upper_bound-slc}]
By \Cref{lem:slls-correct}, the output $(x^*,s^*,v^*,w^*)$ of \nameref{alg:subspace_ipm} is correct when $\mu_1 = 0$.
The (trivial) modification to \nameref{alg:subspace_ipm} for $\mu_1 > 0$,
where it only outputs $z^1 = (x^1,s^1)$ with $\gap(z^1) \leq \mu_1$, is
explained in \Cref{rem:general-accuracy}. The desired upper bound on the
number of iterations to traverse the segment $\CP[\mu_1,\mu_0]$ follows
directly by combining \Cref{thm:delta-slc} and \Cref{thm:amortized_main} with
$\SVDA = 2$, where we use \Cref{thm:approx-svd} to instantiate the
$2$-\textsc{SVD} solver.    
\end{proof}

It remains to prove \Cref{thm:amortized_main}. In \Cref{sec:delta}, we show
that the singular values cannot change too quickly under a change in
polarization partition, and finally we prove \Cref{thm:amortized_main} in
\Cref{sec:amortized-proof}.

\subsection{Stability of Singular Values under a Partition
Change}\label{sec:delta}
The crucial statement towards proving \Cref{thm:amortized_main} is the following lemma on lifting operators for different partitions. This will ensure that we do not lose too much progress, in terms gaining large singular values, when moving from one partition to another.

\begin{lemma}\label{lem:singular_values_for_different_partitions}
Given a subspace $W \subseteq \R^n$ and two partitions $B \cup N = [n]$ and $\hat B \cup \hat N = [n]$, the lifting operators $\lift \coloneqq \lift_N^{W}$ and $\hat \lift \coloneqq \lift_{\hat N}^{W}$ satisfy
\begin{equation}
\label{eq:sing-diff-part}
    \sigma_i(\lift) \ge \sigma_{i + |N \Delta \hat N|}(\hat \lift),~ \forall i \geq 1\, .  
\end{equation}
\end{lemma}

\begin{proof}
If either $\hat{B}$ or $\hat{N}$ is empty, $\sigma_i(\hat \lift) = 0$,
$\forall i \geq 1$, by convention, and hence the statement is trivial. If $N \cap \hat N = \emptyset$, then $\sigma_{i+|N \Delta \hat{N}|}(\hat \lift) \leq 
\sigma_{i+|\hat{N}|}(\hat \lift) = 0$ since $\rank(\hat \lift) \leq |\hat
N|$, and hence again the statement is trivial. Thus, we may assume that
$(\hat B, \hat N)$ is non-trivial and that $N \cap \hat N \neq \emptyset$. 

Let $\check N \coloneqq N \cap \hat N \neq \emptyset$, $\check B\coloneqq [n]\setminus \check
N=B\cup \hat B \neq \emptyset$, and define $\check \lift \coloneqq \lift_{\check N}^{W}$. By our
assumptions, note that $(\check B, \check N)$ is a non-trivial partition of
$[n]$.  
  
First, we are going to prove that for all $i \geq 1$ we have $\sigma_i(\lift) \ge \sigma_{i + |N \setminus \check N|}(\check \lift)$. Note that $\dim(\pi_N(W)) \ge \dim(\pi_{\check N}(W))$. 
  For $i > \dim(\pi_N(W))$, we have  $\sigma_i(\lift) = \sigma_{i + |N \setminus \check N|}(\check \lift)=0$.
  Assume now $i \le \dim(\pi_N(W))$. By \eqref{eq:min-max-sing}, there exists a subspace $S^{(i)}\subseteq \pi_N(W)$, $\dim(S) = \dim(\pi_N(W)) - i + 1$ with
  \begin{equation} 
    \sigma_i(\lift) = \sigma_{\max}\left(\restr{\lift}{S^{(i)}}\right)\, .
  \end{equation}
  Consider $T^{(i)} \coloneqq S^{(i)} \cap \R_{\check N}^N$. 
  Note that $\dim(\pi_{\check N}(T^{(i)})) = \dim(T^{(i)})$ as $\pi_{N\setminus \check N}(T^{(i)}) = \{\0_{N \setminus \check N}\}$ by definition. Further, $\pi_{\check N}(T^{(i)}) \subseteq \pi_{\check N}(W)$ by construction. 
  Therefore, we have that 
  \begin{equation}
  \begin{aligned}\label{eq:dim-bound}
    \dim(\pi_{\check N}(T^{(i)})) &= \dim(T^{(i)}) \ge \dim(S^{(i)}) - |N \setminus \check N| \\ 
    &= \dim(\pi_N(W)) - i + 1 - |N \setminus \check N|\, , 
  \end{aligned}
  \end{equation}
  where the inequality follows as $T^{(i)}$ arises from $S^{(i)}$ by adding $N\setminus \check N$ homogeneous linear equations to its defining system. 
  Finally, note that for any $v \in \pi_{\check N}(T^{(i)})$ we have that 
  \begin{equation}
    \norm{\check \lift(v)} = \min_{w \in W, w_{\check N} = v} \norm{w_{[n]\setminus \check N}} \le \min_{\substack{w \in W, w_{\check N} = v, \\ w_{N \setminus \check N} = \0_{N \setminus \check N}}} \norm{w_{[n]\setminus \check N}} = \norm{\lift(\0_{N \setminus \check N}, v)}\, ,
  \end{equation}
  and so in particular
  \begin{equation}\label{eq:dom-lift-bar-lift}
    \sigma_{\max}\left(\restr{\check \lift}{\pi_{\check N}(T^{(i)})}\right)
    \le \sigma_{\max}\left(\restr{\lift}{T^{(i)}}\right) \, .
  \end{equation}
  
  From here, since $\dim(\pi_{\check N}(W)) \le \dim(\pi_{N}(W))$ we get that 
  \begin{equation}\label{eq:main-singular-value-shift-first-part}
    \begin{aligned}
    \sigma_{i + |N \setminus \check N|}(\check\lift) &\leftstackrel{\eqref{eq:dim-bound}}{\le} \sigma_{\dim(\pi_N(W)) - \dim(\pi_{\check N}(T^{(i)})) + 1}(\check \lift) 
    \stackrel{\eqref{eq:min-max-sing}}{\le} \sigma_{\max}\left(\restr{\check \lift}{\pi_{\check N}(T^{(i)})}\right)
    \\
    & \leftstackrel{\eqref{eq:dom-lift-bar-lift}}{\le} \sigma_{\max}\left(\restr{\lift}{T^{(i)}}\right) \le \sigma_{\max}\left(\restr{\lift}{S^{(i)}}\right) = \sigma_i(\lift)\, .  
  \end{aligned}
  \end{equation}

It remains to show that we can use similar shift argumentation between the
operators $\check \lift$ and $\hat \lift$, that is, that $\sigma_i(\check \lift)
\geq \sigma_{i+|\hat N \setminus \check N|}(\hat \lift)$, $\forall i \geq 1$.
Using \Cref{lem:lift-duality}, we may equivalently consider the dual
operator $\lift_{\check B}^{W^\perp} = -\adj(\lift_{\check N}^W)$, where $\check B=
[n] \setminus \check N$, and show that  
\begin{equation}
\label{eq:singular-value-shift-adjoint}
\sigma_i(\lift_{\check B}^{W^\perp}) \ge \sigma_{i + |\check B \setminus \hat B|}(\lift_{\hat B}^{W^\perp})\, ,\, \forall i \geq 1.  
\end{equation}
Since $\emptyset \neq \hat B \subseteq \check{B}$,
the inequality above follows directly from same arguments we used to obtain \eqref{eq:main-singular-value-shift-first-part}. Therefore, for all $i \ge 1$, by \Cref{lem:lift-duality} we get that
\begin{equation}
\label{eq:main-singular-value-shift-second-part}
\sigma_i(\check \lift) = \sigma_i(\lift^{W^\perp}_{\check B}) \ge \sigma_{i + |\check B \setminus \hat B|}(\lift^{W^\perp}_{\hat B})
 = \sigma_{i + |\hat N \setminus \check N|}(\hat \lift)\, 
\end{equation}
holds, where we used that $|\hat N \setminus \check N| = |\check B \setminus \hat B|$.
Combining \eqref{eq:main-singular-value-shift-first-part} and \eqref{eq:main-singular-value-shift-second-part} gives the result as now 
\begin{equation}
    \begin{aligned}
    \sigma_i(\lift) &\ge \sigma_{i + |N \setminus \check N|}(\check \lift) = \sigma_{i + |N \setminus \hat N|}(\check \lift) \ge \sigma_{i + |N \setminus \hat N| + |\hat N \setminus \check N|}(\hat \lift) \\
    &= \sigma_{i + |N \setminus \hat N| + |\hat N \setminus N|}(\hat \lift) = \sigma_{i+ |N \Delta \hat N|}(\hat \lift),~ \forall i \geq 1.  
    \end{aligned}
\end{equation}
\end{proof}

\subsection{Proof of the Amortized Bound}
\label{sec:amortized-proof}

We need few more preparations to prove \Cref{thm:amortized_main}.  First,
consider the basic iterates $z^{(0)}, \ldots, z^{(K)}$ of the algorithm,
where basic is as defined in \Cref{sec:running-time-analysis}, with $z^{(K)}$
being the first basic iterate with $\gap(z^{(K)}) \leq \mu^{(T)}$. Without
loss of generality, we may assume that $\mu^{(1)} < \gap(z^{(0)}) \leq
\mu^{(0)}$. 

For $i \in \{0,\dots,T\}$, define $t_i$ to be the smallest index of an
iterate with $\mu(z^{(t_i)}) \leq \mu^{(i)}$. With this definition, note that
$t_0 = 0$, $t_T = K$, and for $i \in [T]$, that $t_i-t_{i-1} = |\{j \in [K]:
\mu^{(i)} < \gap(z^{(j)}) \leq \mu^{(i-1)}\}|$. Further, define $y^{(i)}
\coloneqq z\cp(\mu^{(i)})$, $i \in \{0,\dots,T\}$, to be the central path
point with parameter $\mu^{(i)}$. 

For $z = (x,s) \in \cal N^2(\beta)$, $\beta \in (0,1/6)$, let us extend the
notation \eqref{eq:def-maxexpidx} to $\lift_{z,N^{(i)}} \coloneqq
\lift^{\nx^{-1} W}_{N^{(i)}}$ and
\begin{equation}
\label{eq:def-maxexpidx-gen}
    \maxexpidx^{(i)}(z) \coloneqq \left|\left\{j \geq 1: \sigma_j(\lift_{z,N^{(i)}}) > \tfrac{\thresh}{\SVDA}\right\}\right| = \cnt{\lift_{z,N^{(i)}}}{(\tfrac{\thresh}{\SVDA},\infty)}\, .
\end{equation} 

\paragraph{Proof Overview} We start by giving an outline of the proof of \Cref{thm:amortized_main}.
Assume for simplicity of presentation that $\gap(z^{(t_i)}))=\mu^{(i)}$
(\ie, $y^{(i)} = z^{(t_i)}$), $i \in \{0,\dots,T\}$, which can be achieved
by adding artificial iterates in each segment. Consider the
behaviour of the algorithm on the polarized segment
$\CP[\mu^{(i)},\mu^{(i-1)}]$ with partition $(B^{(i)},N^{(i)})$, which starts
with $z^{(t_{i-1})}$ and ends with $z^{(t_i)}$. By
\Cref{lem:progress}, we have that after every $D =
\Theta(\tfrac{\sqrt{n}}{\beta} \ln(\tfrac{n \SVDA}{\beta \polarized}))$
iterations, either we decrease the potential $\maxexpidx^{(i)}(\cdot)$ by one
or we pass the end of the segment. Specifically, for $k \geq 1$, either $t_i + k D>t_{i+1}$, or $\maxexpidx^{(i)}(z^{i+k D}) \leq \maxexpidx^{(i)}(z^{(t_i)}) - k$.  

By \Cref{lem:stability_singular_values}, for $z^{(t')},z^{(t)}$ with $t_i
\leq t' < t \leq t_{i-1}$ (\ie, $\mu^{(i)} \leq \gap(z^{(t')}) < \gap(z^{(t)})
\leq \mu^{(i-1)}$), recall that $\sigma(\lift_{z^{(t')},N^{(i)}}) \leq \Gamma
\frac{\gap(z^{(t')})}{\gap(z^{(t)})} \sigma(\lift_{z^{(t)},N^{(i)}})$ where
$\Gamma = \frac{4n^2}{\polarized^2}$. Let us assume for now that the inequality
holds for $\Gamma = 1$ instead of $\frac{4n^2}{\polarized^2}$. Then, under
this assumption, the singular values would be non-increasing over the
segment, and hence the potential $\maxexpidx^{(i)}(z^{(t)})$ would be monotonic in $t$.
Therefore, we would have that 
\begin{equation}
\label{eq:potential-drop}
\maxexpidx^{(i)}(z^{(t_i)}) \leq \maxexpidx^{(i)}(z^{(t_{i-1})})-\lfloor (t_i-t_{i-1})/D \rfloor \Rightarrow t_i-
t_{i-1} \leq D(\maxexpidx^{(i)}(z^{(t_{i-1})})-\maxexpidx^{(i)}(z^{(t_i)})+1).
\end{equation}
That is, the drop in potential from the start to the end of the segment pays
for the number of iterations. From here, the handover to the next segment is
controlled by \Cref{lem:singular_values_for_different_partitions}, which
implies that $\maxexpidx^{(i)}(z^{(t_{i-1})}) \leq |N^{(i)} \Delta N^{(i-1)}|
+ \maxexpidx^{(i-1)}(z^{(t_{i-1})})$ for $i \in [T]$. Putting these bounds
together, we would get
\begin{align}
K &= t_T-t_0 = \sum_{i=1}^T t_i-t_{i-1} \leq \sum_{i=1}^T D(\maxexpidx^{(i)}(z^{(t_{i-1})})-\maxexpidx^{(i)}(z^{(t_i)})+1) \nonumber \\
&\leq \sum_{i=1}^T D(|N^{(i)} \Delta N^{(i-1)}|+1)+D(\maxexpidx^{(i-1)}(z^{
(t_{i-1})}) - \maxexpidx^{(i)}(z^{(t_i)})) \nonumber \\
&= \sum_{i=1}^T D(|N^{(i)} \Delta N^{(i-1)}|+1) + D(\maxexpidx^{(0)}(z^{(0)})-\maxexpidx^{(T)}(z^{(t_T)})) \nonumber \\
&\leq D \sum_{i=1}^T (|N^{(i)} \Delta N^{(i-1)}|+1), \label{eq:it-bd-ideal}
\end{align}
where the last inequality follows since $\maxexpidx^{(0)}(z^{(0)}) = 0$ as 
$N^{(0)} = \emptyset$. This is precisely the desired bound in
\Cref{thm:amortized_main}.

Unfortunately, this argument breaks due to the $(4n^2/\polarized^2)$ factor in \Cref{lem:stability_singular_values}. The singular values may in fact increase on the short term. This would not be a problem on sufficiently `long' segments. However, on `short' segments, this could lead to an increase in singular values, where there may not be sufficiently many---$D$---iterations to compensate. Such increments could even aggregate over a sequence of short segments.

The proof below works with a more robust version of the potential in
\eqref{eq:def-maxexpidx-gen}. The robust potential is defined as a minimum
over parametrized potentials: for $i \geq 0$ and $z \in
\cal N^2(\beta)$,
\begin{equation}
\label{eq:def-maxexpidx-robust}
\genexpidx^{(i)}(z) \coloneqq \min_{q \geq 0} \maxexpidx^{(i)}(z,q), \quad \quad \maxexpidx^{(i)}(z,q) \coloneqq \cnt{\lift_{z,N^{(i)}}}{(\tfrac{\thresh}{\SVDA} \Gamma^q,\infty)} + q,\, \,\mbox{for}\, q \geq 0\, ,
\end{equation} 
where $\Gamma \coloneqq \tfrac{4n^2}{\polarized^2}$ as above. 

Each parametrized potential keep track of the singular values above a
threshold $\Gamma^{q}\tfrac{\tau}{\SVDA}$ for some $q \geq 0$. The additional
factor represents possible increases accumulated on previous short segments.
However, the contributions $|N^{(j)}\Delta N^{(j-1)}|+1$ of these short
segments will `pay' for bringing down large singular values later, which
corresponds to the additive $q$ factor in these potentials. Note that the
original potential in~\eqref{eq:def-maxexpidx-gen} corresponds to the
parametrized potential with $q=0$.

The following technical lemma shows that the drop in robust potential
$\genexpidx^{(i)}(\cdot)$ over a polarized segment indeed pays for the number
of iterations, in analogy to~\eqref{eq:potential-drop} in the ideal setting
(where we assumed monotonicity of the potential). With this lemma in hand,
the iteration bound in \Cref{thm:amortized_main} follows along identical
lines to the analysis in \eqref{eq:it-bd-ideal}.

\begin{lemma}\label{lem:Gamma-D-change}
Define $\Dit \coloneqq 3 \lfloor
\Cit \tfrac{\sqrt{n}}{\beta} \log(\tfrac{n\SVDA}{\beta \polarized}) \rfloor$,
where $\Cit$ is as in \Cref{lem:progress}. Let $t_i,N^{(i)},y^{(i)}$, $i \in
\{0,\dots,T\}$ be as defined above, let $k_i \coloneqq \max \{0, \lfloor
(t_i-t_{i-1}-1)/\Dit \rfloor\}$, $i \in [T]$. Then, for $i \in [T]$, the
following holds  
\begin{enumerate}[label=(\roman*)]
\item \label{it:zeta-next} $\genexpidx^{(i)}(y^{(i-1)}) \leq \genexpidx^{(i-1)}(y^{(i-1)})+|N^{(i)} \Delta N^{(i-1)}|$.
\item \label{it:zeta-dec} $\genexpidx^{(i)}(y^{(i)})+k_i \leq \genexpidx^{(i)}(y^{(i-1)})+1$.
\end{enumerate}
\end{lemma}
\begin{proof} We prove \ref{it:zeta-next} and \ref{it:zeta-dec} separately below. 

\paragraph{Proof of \ref{it:zeta-next}} By
\Cref{lem:singular_values_for_different_partitions}, $\sigma_{i
+ |N^{(i)} \Delta N^{(i-1)}|}(\lift_{y^{(i-1)},N^{(i)}}) \leq
\sigma_i(\lift_{y^{(i-1)},N^{(i-1)}})$, $\forall i \geq 1$. This implies
\[
\cnt{\lift_{y^{(i-1)},N^{(i-1)}}}{(\frac{\thresh}{\SVDA}
\Gamma^q,\infty)}+|N^{(i)} \Delta N^{(i-1)}| \geq
\cnt{\lift_{y^{(i-1)},N^{(i)}}}{(\frac{\thresh}{\SVDA} \Gamma^q,\infty)},
\]
and hence $\maxexpidx^{(i)}(y^{(i-1)},q) \leq \maxexpidx^{(i-1)}(y^{(i-1)},q)
+ |N^{(i)} \Delta N^{(i-1)}|$, for $q \geq 0$. Therefore,
$\genexpidx^{(i)}(y^{(i-1)}) \leq$ \ifdefined\SIAMversion \else \newline \fi $\genexpidx^{(i-1)}(y^{(i-1)}) + |N^{(i)}
\Delta N^{(i-1)}|$. 

\paragraph{Proof of \ref{it:zeta-dec}} For $q \geq 0$, we claim that 
\begin{equation}
\label{eq:gen-drop}
\maxexpidx^{(i)}(y^{(i)},\max \{0,q+1-k_i\})+k_i \leq \maxexpidx^{(i)}(y^{(i-1)},q)+1.  
\end{equation}
Assuming the claim, \ref{it:zeta-dec} follows noting that
\begin{align*}
\genexpidx^{(i)}(y^{(i)})+k_i &\leq \min_{q \geq 0}
\maxexpidx^{(i)}(y^{(i)},\max \{0,q+1-k_i\})+k_i \leq
\min_{q \geq 0} \maxexpidx^{(i)}(y^{(i-1)},q)+1 \\ &= \genexpidx^{(i)}(y^{(i-1)})+1\, .
\end{align*}

To prove \eqref{eq:gen-drop}, we need the following intermediate inequality:
\begin{equation}
\sigma(\lift_{y^{(i)},N^{(i)}}) \leq \Gamma^{1-k_i} \sigma(\lift_{y^{(i-1)},N^{(i)}})\, . \label{eq:sing-drop-1}
\end{equation}

We prove this inequality first. Recall that $\CP[\mu^{(i)},\mu^{(i-1)}]$ is a $\polarized$-polarized segment with partition $(B^{(i)},N^{(i)})$. Therefore, for $z' \in \clN(\beta)$, $z \in
\cal N^2(\beta)$, $\beta \in (0,1/6)$, with $\mu^{(i)} \leq \gap(z') \leq
\gap(z) \leq \mu^{(i-1)}$, by \Cref{lem:stability_singular_values} we have
that
\begin{equation}
\sigma(\lift_{z',N^{(i)}}) \leq \Gamma \frac{\gap(z')}{\gap(z)} \sigma(\lift_{z,N^{(i)}}). \label{eq:sing-stab-robust}
\end{equation}

Assume that $k_i = 0$. Then,~\eqref{eq:sing-drop-1} follows
from~\eqref{eq:sing-stab-robust} with $z'=y^{(i)}$ and $z=y^{(i-1)}$, noting
that $\gap(y^{(i)})=\mu^{(i)} < \mu^{(i-1)} = \gap(y^{(i-1)})$ and $\Gamma =
\Gamma^{1-k_i}$. 

Now assume that $k_i \geq 1$. Let $z^{(i)}$, $i \in \{0,\dots,K\}$, be the
basic iterates as defined in the overview, and let $t = t_{i-1}+\Dit/3$ and
$t' = t_i-1-\Dit/3$. Since $k_i \geq 1$, we have that 
\begin{equation}
t'-t \geq k_i \Dit - 2\Dit/3 \geq k_i \Dit/3. \label{eq:it-dif-lb}
\end{equation}
By the choice of $\Cit$ in~\eqref{eq:progress-gap-decrease}, $\Dit/3 = \lfloor \Cit \tfrac{\sqrt{n}}{\beta} \log(\tfrac{n\SVDA}{\beta \polarized}) \rfloor$
iterations are sufficient to divide the normalized gap by $\Gamma$. That is,
$\gap(z^{(i+\Dit/3)}) \leq \gap(z^{(i)})/\Gamma$, for $i \in
\{0,\dots,K-\Dit/3\}$. In particular, 
\[
\gap(z^{(t')}) \leq \gap(z^{(t)}) /\Gamma^{\lfloor \tfrac{3(t'-t)}{\Dit} \rfloor} \leq \gap(z^{(t)}) / \Gamma^{k_i}.
\]
Similarly, $\mu^{(i)} \leq \gap(z^{(t_i-1)}) \leq \gap(z^{(t')})/\Gamma$ and
$\gap(z^{(t)}) \leq \gap(z^{(t_{i-1})})/\Gamma \leq \mu^{(i-1)}/\Gamma$. Therefore, by~\eqref{eq:sing-stab-robust} we have that 
\begin{equation}
\sigma(\lift_{y^{(i)},N^{(i)}}) \leq \sigma(\lift_{z^{(t')},N^{(i)}}) \leq \sigma(\lift_{z^{(t)},N^{(i)}}) / \Gamma^{k_i-1} \leq \sigma(\lift_{y^{(i-1)},N^{(i)}}) / \Gamma^{k_i-1}\, , \label{eq:robust-sing-rel}
\end{equation}
completing the proof of \eqref{eq:sing-drop-1}.

We now prove claim~\eqref{eq:gen-drop} by dividing the analysis into two
additional cases.

\begin{mycases} 
\item $\boldsymbol{k_i \leq q+1:}$ 
\begin{align*}
\maxexpidx^{(i)}(y^{(i)},q+1-k_i)+k_i &= \cnt{\lift_{y^{(i)},N^{(i)}}}{(\tfrac{\tau}{\SVDA} \Gamma^{q+1-k},\infty)} + (q+1-k_i) + k_i \\ &\stackrel{\eqref{eq:sing-drop-1}}{\leq} \cnt{\lift_{y^{(i-1)},N^{(i)}}}{(\tfrac{\tau}{\SVDA}
\Gamma^{q},\infty)}+q+1 \\ &= \maxexpidx^{(i)}(y^{(i-1)},q)+1, \text{ as needed.}
\end{align*}

\item $\boldsymbol{k_i \geq q+2:}$ Let $\hat{t} = t + (q+1) \Dit/3$. By our assumption that $k_i \geq q+2$, using~\eqref{eq:it-dif-lb} we have that
that $t'-\hat{t} \geq (k_i-q-1) \Dit/3 \geq \Dit/3$. Via an identical
calculation to~\eqref{eq:robust-sing-rel}, we have that
\[
\sigma(\lift_{z^{(\hat{t})},N^{(i)}}) \leq
\sigma(\lift_{z^{(t)},N^{(i)}})/\Gamma^{q} \leq
\sigma(\lift_{y^{(i-1)},N^{(i)}})/\Gamma^{q}.
\]  
Using the above, we get that 
\begin{align}
\maxexpidx^{(i)}(z^{(\hat{t})}) &= \cnt{\lift_{z^{(\hat{t})},N^{(i)}}}{(\tfrac{\tau}{\SVDA},\infty)} \leq \cnt{\lift_{y^{(i-1)},N^{(i)}}}{(\tfrac{\tau}{\SVDA} \Gamma^q,\infty)} \nonumber \\ &= \maxexpidx^{(i)}(y^{(i-1)},q)-q.
\label{eq:zeta-dec-1}
\end{align}
By \Cref{lem:progress}, since $t'-\hat{t} \geq (k_i-q-1) \Dit/3 = (k_i-q-1)
\lfloor \Cit \tfrac{\sqrt{n}}{\beta} \log(\tfrac{n\SVDA}{\beta \polarized})
\rfloor$, we have that
\begin{align}
\maxexpidx^{(i)}(y^{(i)},0) &\coloneqq \maxexpidx^{(i)}(y^{(i)}) \stackrel{\eqref{eq:robust-sing-rel}}{\leq} \maxexpidx^{(i)}(z^{(t')}) \stackrel{\ref{lem:progress}}{\leq}
\maxexpidx^{(i)}(z^{(\hat{t})})+q+1-k_i \nonumber \\
&\stackrel{\eqref{eq:zeta-dec-1}}{\leq} \maxexpidx^{(i)}(y^{(i-1)},q)+1-k_i.
\end{align}
The claimed bound now follows by rearranging.
\end{mycases}

\end{proof}

\begin{proof}[Proof of \Cref{thm:amortized_main}]
Let $t_i,y^{(i)}$,$N^{(i)}$, $i \in \{0,\dots,T\}$ be as defined in the overview. Recall that $N^{(0)}=\emptyset$, $t_0=0$ and $t_T = K$. Let $\Dit$ and $k_i \coloneqq \max \{0,
\lfloor (t_i-t_{i-1}-1)/\Dit \rfloor\}$, $i \in [T]$, be as in \Cref{lem:Gamma-D-change}. Proceeding in a similar manner to~\eqref{eq:it-bd-ideal}, we bound the number of iteration as
follows: 
\ifdefined\SIAMversion
\begin{align*}
K &= t_T-t_0 = \sum_{i=1}^T (t_i-t_{i-1}-1)+1 \leq \sum_{i=1}^T \Dit(k_i+1) \\ &\hspace{22em} \left(~b+1 \leq a(\lfloor \tfrac{b}{a} \rfloor+1), a,b \in \mathbb{N}~\right) \\
&\leq \sum_{i=1}^T \Dit \left(\genexpidx^{(i)}(y^{(i-1)}) - \genexpidx^{(i)}(y^{(i)})+2 \right) \hspace{9.2em} \left(\text{ by \Cref{lem:Gamma-D-change}\ref{it:zeta-dec} }\right) \\
&\leq \sum_{i=1}^T \Dit \left( \genexpidx^{(i-1)}(y^{(i-1)})-\genexpidx^{(i)}(y^{(i)})+|N^{(i)} \Delta N^{(i-1)}|+2\right)  \quad \left(\text{ by \Cref{lem:Gamma-D-change}\ref{it:zeta-next} }\right) \\
&= \left(\sum_{i=1}^T \Dit \left(|N^{(i)} \Delta N^{(i-1)}|+2\right) \right) + \Dit(\genexpidx^{(0)}(y^{(0)}) - \genexpidx^{(T)}(y^{(T)})) \\ &\leq  \sum_{i=1}^T \Dit \left(|N^{(i)} \Delta N^{(i-1)}|+2\right)\, ,
\end{align*}
\else
\begin{align*}
K &= t_T-t_0 = \sum_{i=1}^T (t_i-t_{i-1}-1)+1 \leq \sum_{i=1}^T \Dit(k_i+1) \quad \left(~b+1 \leq a(\lfloor \tfrac{b}{a} \rfloor+1), a,b \in \mathbb{N}~\right) \\
&\leq \sum_{i=1}^T \Dit \left(\genexpidx^{(i)}(y^{(i-1)}) - \genexpidx^{(i)}(y^{(i)})+2 \right) \hspace{9.2em} \left(\text{ by \Cref{lem:Gamma-D-change}\ref{it:zeta-dec} }\right) \\
&\leq \sum_{i=1}^T \Dit \left( \genexpidx^{(i-1)}(y^{(i-1)})-\genexpidx^{(i)}(y^{(i)})+|N^{(i)} \Delta N^{(i-1)}|+2\right)  \quad \left(\text{ by \Cref{lem:Gamma-D-change}\ref{it:zeta-next} }\right) \\
&= \left(\sum_{i=1}^T \Dit \left(|N^{(i)} \Delta N^{(i-1)}|+2\right) \right) + \Dit(\genexpidx^{(0)}(y^{(0)}) - \genexpidx^{(T)}(y^{(T)})) \\ &\leq  \sum_{i=1}^T \Dit \left(|N^{(i)} \Delta N^{(i-1)}|+2\right)\, ,
\end{align*}
\fi
where the last inequality uses that $\genexpidx^{(0)}(y^{(0)}) \leq \maxexpidx^{(0)}(y^{(0)}) = 0$ since $N^{(0)} = \emptyset$.
\end{proof}

 \section{Computing Approximate Singular Value Decompositions}
\label{sec:sing-val}

In this section, we give our algorithm for computing approximate SVDs and prove \Cref{thm:approx-svd}.

We begin by presenting a deterministic strongly polynomial method for obtaining a $(n2^{n})$-approximate SVD for matrices $\mM
    \in \R^{m \times n}$ using QR decomposition with greedy column permutations. In the next
subsection, we demonstrate how to enhance the approximation factor to $(1+\eps)$ by leveraging
the classical power iteration.

Using QR decompositions to approximate the SVD is a well-established technique that dates back to
the mid-1960s \cite{Businger1965,Golub1965}. Many existing approaches focus on finding
rank-revealing QR factorizations instead of a full $\SVDA$-approximate SVD. These methods aim to
identify a number $k \le n$ such that $\sigma_k(\mM) \gg \sigma_{k+1}(\mM)$, along with an
approximate subspace for the eigenvalues $\sigma_1(\mM), \ldots, \sigma_k(\mM)$. This is achieved by
greedily selecting column permutations, followed by a standard QR factorization. The top $k \times
    k$ block of the resulting upper triangular matrix provides an approximation of the subspace
corresponding to the top $k$ singular values.

One of the earliest algorithms that utilizes column pivoting and subsequent QR decomposition was
proposed by Chan \cite{Chan1987}, who attributes the procedure to \cite{Golub1976RankDA}. He
achieves approximation guarantees similar to ours in \Cref{lem:approx-svd}. Algorithms with tighter
bounds (even polynomial instead of exponential) for rank-revealing QR factorizations can be found in
\cite{Chandrasekaran1994,Gu1996,Pan2000}.

However, these algorithms are designed to work only for a fixed $k$, not as the $\SVDA$-approximate
SVD requires for all $k$ simultaneously. Further, the initial algorithms
\cite{Businger1965,Golub1965} critically rely on Householder reflection, which intrinsically make
use of unit vectors in the computations. Unit vectors however, can not in general be used in the
strongly polynomial model as they require the computation of square roots. The mentioned approach by
Chan \cite{Chan1987} requires the computation of the smallest singular vector of the matrix. This
again can not be done in strongly polynomial time. More precisely, given a matrix $\mM \in \R^{m
        \times n}$, the algorithms requires solving $\min_{x \in \R^n, \norm{x} = 1} \norm{\mM x}$. The
algorithm we present in this section, instead solves the problem $\min_{x \in \R^n, \norm{x}_\infty =
        1} \norm{\mM x}$. As we will see below, this problem can be solved in strongly polynomial time. Also
note, that the objective values of the two minimization problems above, differ by at most a factor
of $\sqrt{m}$. This will explain why our algorithm achieves similar approximation guarantees as the
algorithm of Chan \cite{Chan1987}.

In general, most works in the literature have not focused on strong polynomiality. Beyond the use
of square roots, many of these results require the computation of orthonormal matrices, which in
general is not achievable in strongly polynomial space. Additionally, it is important to avoid the
sequential computation of $\Omega(n)$ Gram-Schmidt orthogonalizations (GSO). While a single GSO can
be performed in strongly polynomial time \cite{gls}, the size of the numbers may increase by a
polynomial factor. It remains unclear whether a sequence of super-logarithmically many GSO
computations can be performed in strongly polynomial time.

A good overview of QR-type algorithms with column pivoting can be found in
\cite{Chandrasekaran1994}. They also present new algorithms for rank-revealing QR factorizations
with improved approximation guarantees. Although the authors do not emphasize this fact, some
algorithms in \cite{Chandrasekaran1994}, in particular \textsc{Greedy-I.1}, \textsc{Greedy-I.2} and
\textsc{Greedy-I.3}, can be implemented in strongly polynomial time. Furthermore, they give similar
approximation guarantees as us in \Cref{lem:approx-svd} for the algorithm \textsc{Greedy-I.3}.

Nonetheless, we present a self-contained new algorithm here for the sake of completeness and to
focus on achieving strong polynomiality. While the framework of our algorithm fits into the regime
of performing column pivoting and subsequent QR decomposition, we believe that the exact rule for
the column pivoting is novel.

Very recently, Diakonikolas, Tzamos, and Kane~\cite{Diakonikolas22} also
provided a strongly polynomial $(1+\eps)$-singular value decomposition
algorithm based on randomized power iteration (the randomness corresponds to
a random choice of initial basis of the input space). At a high level, our
algorithm removes the need for randomness by using a suitable greedy QR
decomposition (see the description below). In constrast to our algorithm
however, the approximate singular value decomposition
of~\cite{Diakonikolas22} outputs a decomposition of
a matrix $\widetilde{\mM}$ (as in \Cref{def:svd}) that is $\eps$-``spectrally close'' to $\mM$. This provides a
somewhat stronger guarantee than what is needed for our formalization of the
$(1+\eps)$-\textsc{SVD} problem, which only requires an orthogonal basis inducing a chain of $(1+\eps)$-approximate singular subspaces. While we expect that one can extract such a
decomposition in a blackbox manner from any $(1+\eps')$-\textsc{SVD}
solution, for a suitably chosen $\eps' \ll \eps$, for the sake of simplicity
we do not pursue this direction here.     

\medskip

We proceed by describing our main algorithm on a high level. For a matrix $\mM \in \R^{m \times n}$,
we will use the following procedure (Algorithm~\ref{alg:greedy-svd}): For all columns $j$ of $\mM$,
consider the projection of the column $\mM_{\mdot,j}$ onto the orthogonal complement of the span of
all the other columns $\im(\mM_{\mdot,[n] \setminus \{j\}})$. Then, remove the column $j_{\min}$
from the matrix for which the norm of this projection is the smallest and recurse on the remaining
matrix $\mM_{\mdot, [n]\setminus \{j_{\min}\}}$. When this process finishes, we obtain a permutation
$\mM \mP$ of the columns, given by the order in which they were removed from the matrix. It turns
out, that the norms of the columns of the orthogonal matrix $\mQ$ obtained from a Gram-Schmidt
process on the permuted matrix $\mM \mP$ (the first column removed from $\mM$ is the last column in
its reordering) provide an exponential approximation of the singular values of the original matrix
$\mM$. The main observation for the proof is that the matrix $\mR$ in the Gram-Schmidt process,
uniquely defined by $\mM \mP \mR= \mQ$, has only exponential condition number.

\ifdefined\SIAMversion
\begin{algorithm2e}[h!]
\else
\begin{algorithm}[h!]
\fi
    \caption{\textsc{Greedy-SVD}}
    \label{alg:greedy-svd}
    \SetKwInOut{Input}{Input} \SetKwInOut{Output}{Output} \SetKw{And}{\textbf{and}}

    \Input{Matrix $\mM \in \R^{m \times n}$.} \Output{Matrix $\mB \in \R^{n
                \times n}$ and a vector $s \in \R^n$ such that $(\mB, s)$ is a $(n2^n)$-approximate SVD
        (\Cref{def:alpha-svd-problem}).}

    $J_n \gets [n]$\;
\For{$t = n$ down to $1$ \label{line:for-loop-2} \label{line:outer-for-loop-in-svd}} {
        \For{$i \in J_t$}{
            $w^{(t,i)} \gets \0_n$\; $w^{(t,i)}_{i} \gets 1$\; 
            $w^{(t,i)}_{J_t\setminus \{i\}} \gets - [\mM_{\mdot, J_{t} \setminus \{i\}}]^+ \mM_{\mdot,i}$ \label{line:compute-with-pseudo-inverse}\;
        }
        $K_t \gets \argmin_{j \in J_t} \norm{\mM w^{(t,j)}}$ \label{line:pick-min-norm-set}\;
        $\pi(t) \gets \text{any element in } \argmin_{j \in K_t} \norm{w^{(t,j)}}$ \label{line:pick-min-norm-vector}\;

        $v^{(t)} \gets w^{(t, \pi(t))}$\; $q^{(t)} \gets \mM v^{(t)}$\;
        $J_{t - 1} \gets J_t \setminus \{\pi(t)\}$\;
    }

    $\mP \gets (1_{i = \pi(j)})_{i,j \in [n]}$ \label{line:define_P}\;
    $\mR \gets \mP^\T  \begin{bmatrix} v^{(1)} & \dots & v^{(n)} \end{bmatrix}$\;
    $\mQ \gets \begin{bmatrix} q^{(1)} & \dots & q^{(n)} \end{bmatrix}$\;
    $\mB \gets \mP \mR$\;
    \For{$i = 1, \ldots, n-1$}{
        $\mB_{\mdot, i} \gets \mB_{\mdot,i} - \proj_{\im(\mB_{\mdot, > i})}(\mB_{\mdot,i})$ \label{line:projection-computation-in-svd}
        \tcp*{Reverse Gram-Schmidt Orthogonalization}
    }
    $s \gets \0_n$\; \For{$i = 1, \ldots, n$}{ $s_i \gets 2^{n} \norm{\mQ_{\mdot, i}}_1$\; } \Return{$(\mB, s)$}
\ifdefined\SIAMversion
\end{algorithm2e}
\else
\end{algorithm}
\fi

It is not hard to see that the output of \nameref{alg:greedy-svd} satisfies that $\mQ$ is precisely
the result of Gram Schmidt orthogonalization (GSO) on the matrix $\mM \mP$, i.e., the matrix $\mM$
after its columns have been permuted according to $\mP$.

Before we begin the analysis, we introduce some additional notation and
recall some fundamental matrix inequalities. For a matrix $\mM \in \R^{m
\times n}$, we use $\norm{\mM}_2 \coloneqq \sigma_1(\mM)$ to denote the
operator norm, $\norm{\mM}_{\rm F} \coloneqq \sqrt{\sum_{i \in [m], j \in [n]}
\mM_{ij}^2}$ to denote the Frobenius norm, and $\norm{\mM}_\infty \coloneqq
\max_{i \in [m], j \in [n]} |\mM_{ij}|$ to denote the maximum absolute value
of any entry. We will use the inequality $\norm{\mM}_2 \leq \norm{\mM}_{\rm F}$,
and if $\mM$ is invertible, the relation $\sigma_{\min}(\mM) =
\frac{1}{\norm{\mM^{-1}}_2} \geq \frac{1}{\norm{\mM^{-1}}_{\rm F}}$. 

\begin{lemma}\label{lem:approx-svd-properties} 

 The matrices $\mP$, $\mR$, and $\mQ$
    constructed in \nameref{alg:greedy-svd}(Algorithm~\ref{alg:greedy-svd}) have the following properties:
    \begin{enumerate}[label=(\roman*)]
        \item \label{it:P-permutation} $\mP \in \R^{n \times n}$ is a permutation matrix,
        \item \label{it:R_upper-triangular} $\mR \in \R^{n \times n}$ is upper triangular,
        \item \label{it:Q-orthogonal} $\mQ \in \R^{n \times n}$ has orthogonal columns,
        \item \label{it:mat-product} $\mM \mP \mR = \mQ$,
        \item \label{it:Q-norm-monotonicity} $\norm{\mQ_{\mdot,s}} \ge \norm{\mQ_{\mdot,t}}$ for all $s, t
                  \in [n]$ such that $s \le t$,
        \item \label{it:Q-regression} $\norm{\mQ_{\mdot,t}} = \min_{j \in [t]} \min_{x \in \R^t,
                      x_j=1} \norm{\mM \mP_{\mdot,[t]} x}$,
        \item \label{it:R-properties} $\norm{\mR}_\infty = 1$ and $\mR_{t,t} = 1$, $\forall t \in [n]$ .
    \end{enumerate}

\end{lemma}
\begin{proof}

    It is easy to see that $\pi$ is a
    bijection from $[n]$ to $[n]$. Therefore, by definition of $\mP$ in line~\ref{line:define_P} we
    obtain \ref{it:P-permutation}.

    Let $t \in [n]$ and $i \in J_t$, Then, noting that $w^{(t,i)}_i = 1$, we have that
    \begin{equation}
        \label{eq:projection-expression-for-w-t-i}
        \begin{aligned}
            \mM w^{(t,i)} & = \mM_{\mdot, J_t \setminus \{i\}} w^{(t,i)}_{J_t \setminus \{i\}} + \mM_{\mdot,i}                      \\
                          & =  -\mM_{\mdot, J_t \setminus \{i\}} [\mM_{\mdot, J_t \setminus \{i\}}]^+ \mM_{\mdot,i} + \mM_{\mdot,i} \\
                          & =  -\proj_{\im(\mM_{\mdot,J_t \setminus \{i\}})}\mM_{\mdot,i} + \mM_{\mdot,i}                           \\
                          & = \proj_{\im(\mM_{\mdot,J_t \setminus \{i\}})^\perp}\mM_{\mdot,i}\, ,
        \end{aligned}
    \end{equation}
    where the third equality follows by \Cref{prop:pseudoinverse-projection} part~\eqref{eq:pseudo-proj}.
    This furthermore implies, that
    \begin{align}
        \norm{\mM w^{(t,i)}} &= \norm{\proj_{\im(\mM_{\mdot,J_t \setminus \{i\}})^\perp}\mM_{\mdot,i}} = \min_{z \in \im(\mM_{\mdot,J_t \setminus \{i\}})} \norm{\mM_{\mdot,i} + z} \nonumber \\ &= \min_{x \in \R^{J_t},  x_i = 1} \norm{\mM_{\mdot, J_t} x} \, .
        \label{eq:min-characterization}
    \end{align}
    Now, \ref{it:Q-regression} follows by \eqref{eq:min-characterization} and the definition
    of $\mP$, which gives $\mM_{\mdot, J_t} = \mM \mP_{\mdot, [t]}$ as $\pi([t]) = J_t$.
    From here, we also conclude that for any $j \in J_t \setminus \{i\}$ with $w^{(t,i)}_j \neq 0$, we
    have that $[w^{(t,i)}_j]^{-1}w^{(t,i)}_{J_t}$ is a feasible vector for the minimization problem
    on the right most side of \eqref{eq:min-characterization} for index $j$. Therefore, we have that

    \begin{equation}
        \label{eq:bound-norm-Mv}
        \norm{\mM w^{(t,j)}} \le \frac{1}{|w^{(t,i)}_{j}|} \norm{\mM w^{(t,i)}} \, .
    \end{equation}

    We distinguish two cases for the norm of $\mM w^{(t,\pi(t))}$. If $\norm{\mM
        w^{(t,\pi(t))}} > 0$, then for all $j \in J_t$, the inequality \eqref{eq:bound-norm-Mv} gives $|w_j^{(t,\pi(t))}| \le \norm{\mM
        w^{(t,\pi(t))}}/\norm{\mM w^{(t,i)}} \le 1$ by $\pi(t) \in K_t$ and the definition of $K_t$.

    In the other case we have that $\norm{\mM w^{(t,\pi(t))}} = 0$. Then, notice that for any $j \in
        J_t$ such that $w^{(t,\pi(t))}_j \neq 0$, we have by \eqref{eq:bound-norm-Mv} that $\norm{\mM
        w^{(t,j)}} = 0$ and therefore $j \in K_t$. In particular, we have that $\mM_{\mdot, j} \in
        \im(\mM_{\mdot, J_t \setminus \{j\}})$. Therefore, by \Cref{prop:pseudoinverse-projection} part~\eqref{eq:pseudoinverse-regression}, for $j \in \supp(w^{(t,\pi(t)})$ we obtain

    \begin{align}
        \norm{w^{(t,j)}_{J_t \setminus \{j\}}} & = \norm{(\mM_{\mdot, J_t \setminus \{j\}})^+ \mM_{\mdot,j}}                                                                                          \nonumber \\
                                              & = \min\big\{\norm{u} : u \in \R^{|J_t| - 1}, \mM_{\mdot, J_t \setminus \{j\}}u = \proj_{\im(\mM_{\mdot, J_t \setminus \{j\}})}(\mM_{\mdot, j})\big\} \nonumber \\
                                              & = \min\big\{\norm{u} : u \in \R^{|J_t| - 1}, \mM_{\mdot, J_t \setminus \{j\}}u = \mM_{\mdot, j}\big\} \label{eq:min-program-for-u}                   \\
                                              & \le \frac{\norm{w^{(t,\pi(t))}_{J_t \setminus \{j\}}}}{|w^{(t,\pi(t))}_j|}\, , \nonumber
    \end{align}
    where the inequality follows by noting that the vector $u \in \R^{J_t \setminus \{j\}}$ defined as $u = -[w^{(t,\pi(t))}_j]^{-1} w^{(t,\pi(t))}_{J_t \setminus
                \{j\}}$ is a feasible solution to the minimization problem in \eqref{eq:min-program-for-u} as 
                \[
                    \mM_{\mdot, J_t \setminus \{j\}}u 
                    = -[w^{(t,\pi(t))}_j]^{-1} \mM_{\mdot, J_t \setminus \{j\}}  w^{(t,\pi(t))}_{J_t \setminus \{j\}} 
                    = [w^{(t,\pi(t))}_j]^{-1} \mM_{\mdot, j} w^{(t,\pi(t))}_j 
                    = \mM_{\mdot, j}\, .   
                \]

    This gives for all $j \in \supp(w^{(t,\pi(t))}) \subseteq K_t$ that 
    \[|w_j^{(t,\pi(t))}| \le \frac{\norm{w_{J_t \setminus \{\pi(t)\}}^{(t,\pi(t))}}}{\norm{w_{J_t \setminus \{i\}}^{(t,i)}}}  = \frac{
        \sqrt{\norm{w^{(t,\pi(t))}}^2 - 1}}{\sqrt{\norm{w^{(t,i)}}^2 - 1}}   \le 1 \, ,
        \] 
    where the last inequality follows by definition of $\pi(t)$. We can therefore conclude that $\norm{v^{(t)}}_\infty = \norm{w^{(t,\pi(t))}}_\infty = 1$ for all $t
        \in [n]$, recalling that $v^{(t)}_{\pi(t)}=w^{(t,\pi(t))}_{\pi(t)}=1$. 

        Let us now show the statements for $\mR = \mP^\top
        \begin{bmatrix} v^{(1)} & \dots & v^{(n)} \end{bmatrix}$. As $\mP$ is a permutation matrix by \ref{it:P-permutation}, we have
    shown that $\norm{\mR}_\infty = 1$ and furthermore $\mR_{t,t} = 1$, $\forall t \in [n]$, as $\mR_{t,t} = (\mP^\top)_{t, \mdot} v^{(t)} = v^{(t)}_{\pi(t)} = 1$ for all $t \in [n]$. Therefore, \ref{it:R-properties} holds. To show that $\mR$ is upper triangular,
    recall that $\supp(v_t) \subseteq J_t$ and $\pi([t]) = J_t$. Therefore, letting $\pi^{-1}$ denote the inverse permutation, we get $\supp(\mP^\T v^{(t)})
        \subseteq \pi^{-1}(J_t) = [t]$ as needed for \ref{it:R_upper-triangular}.

    Further, note that by construction $\mM v^{(t)}
        = q^{(t)}$, $t \in [n]$. Therefore,
\[
        \mM\mP\mR = \mM\mP\mP^\T \begin{bmatrix} v^{(1)} & \dots & v^{(n)}\end{bmatrix} = \mM \begin{bmatrix} v^{(1)} & \dots & v^{(n)}\end{bmatrix} = \begin{bmatrix}  q^{(1)} & \dots & q^{(n)} \end{bmatrix} = \mQ,
\]
    as needed for \ref{it:mat-product}.  To see
    that $\mQ$ has orthogonal columns, note that for by definition of $q^{(t)}$ we have with
    \eqref{eq:projection-expression-for-w-t-i} that
    \begin{equation}
        \label{eq:projection-expression-for-q}
        q^{(t)} = \proj_{\im(\mM_{\mdot, J_t \setminus \{\pi(t)\}})^\perp}\mM_{\mdot, \pi(t)}\, ,
    \end{equation}
    and so in particular $q^{(t)} \in \im(\mM_{\mdot, J_t \setminus \{\pi(t)\}})^\perp$.
    However, since for $1 \le s < t$ we have that $q^{(s)} \in \im(\mM_{\mdot, J_t \setminus \{\pi(t)\}})$. We conclude that $q^{(t)}$ is
    orthogonal to $q^{(s)}$, proving \ref{it:Q-orthogonal}.

    To show \ref{it:Q-norm-monotonicity}, note that for $1 \le s < t \le n$, the inclusion 
    $\im(\mM_{\mdot, J_s \setminus \{\pi(s)\}}) \subseteq \im(\mM_{\mdot, J_t \setminus
                \{\pi(t)\}})$ holds, and hence
    \begin{align*}
        \norm{\mQ_{\mdot, s}} & = \norm{q^{(s)}}                                                                                                                                                                                  \\
                             & = \norm{\proj_{\im(\mM_{\mdot, J_s \setminus \{\pi(s)\}})^\perp}\mM_{\mdot, \pi(s)}} \quad \left(\text{by \eqref{eq:projection-expression-for-q}}\right)                                                                           \\
                             & \ge \norm{\proj_{\im(\mM_{\mdot, J_t \setminus \{\pi(t)\}})^\perp}\mM_{\mdot, \pi(s)}} \quad \left(\text{$\im(\mM_{\mdot, J_t \setminus \{\pi(t)\}})^\perp \subseteq \im(\mM_{\mdot, J_s \setminus \{\pi(s)\}})^\perp $}\right) \\
                             & = \norm{\mM w^{(t,\pi(s))}} \quad \left(\text{by \eqref{eq:projection-expression-for-w-t-i}}\right)                                                                                                                                \\
                             & \ge \norm{\mM w^{(t,\pi(t))}} \quad \left(\text{by choice of $\pi(t)$ }\right)                                                                                                                                                          \\
                             & = \norm{\proj_{\im(\mM_{\mdot, J_t \setminus \{\pi(t)\}})^\perp}\mM_{\mdot, \pi(t)}} \quad \left(\text{by \eqref{eq:projection-expression-for-w-t-i}}\right)                                                                     \\
                             & = \norm{q^{(t)}} \quad \left(\text{by \eqref{eq:projection-expression-for-q}}\right)                                                                                                                                             \\
                             & = \norm{\mQ_{\mdot,t}}  \, .
    \end{align*}
    This proves the lemma.
\end{proof}

The following proposition shows that the matrix $\mR$ computed by
\nameref{alg:greedy-svd} is ``well-conditioned''.

\begin{proposition} Let $\mR \in \R^{n \times n}$ be an upper triangular matrix with diagonal $\1_n$
    and with entries of absolute value at most $1$. Then, $\mR^{-1}$ is upper triangular with
    diagonal $\1_n$ and $|\mR^{-1}_{ij}| \leq \max \{1, 2^{j-i-1}\}$, $i,j \in [n]$, $i \leq j$. In
    particular, $\norm{\mR^{-1}}_2 \leq 2^n$.
    \label{prop:well-conditioned}
\end{proposition}
\begin{proof}

    The proof goes by induction on $n$. The base case $n=1$ is trivial, so assume $n > 1$. Then, it
    is directly verifiable that
    \[
        \mR^{-1} = \begin{bmatrix} \mR_{[n-1],[n-1]}^{-1} & -\mR_{[n-1],[n-1]}^{-1} \mR_{[n-1],n} \\ \0_{n-1}^\T & 1 \end{bmatrix}.
    \]
    By the induction hypothesis, it holds that $|\mR^{-1}_{ij}| \le \max\{1, 2^{j-i-1}\}$, $i,j \in
        [n-1]$, $i \leq j$. Note that $\mR^{-1}$ is upper triangular as claimed. We now prove the
    coefficient bound for $\mR^{-1}_{i,n}$, for $i \leq n$. If $i=n$, $\mR^{-1}_{i,i} = 1$, as
    needed. For $i < n$, using that $\mR$ has entries at most $1$ we get that
    \begin{align*}
        |\mR^{-1}_{in}| & = \big|(\mR^{-1}_{[n-1],[n-1]} \mR_{[n-1],n})_i\big| = \big|\sum_{j=i}^{n-1} \mR^{-1}_{i,j} \mR_{j,n}\big| \leq \sum_{j=i}^{n-1} |\mR^{-1}_{i,j}| \\
                        & \leq \sum_{j=i}^{n-1} \max \{1,2^{j-i-1}\} = 1 + \sum_{l=0}^{n-i-2} 2^l = \max \{1, 2^{n-i-1}\},
    \end{align*}
    as needed. For the last statement, by a direct calculation
    \[
        \norm{\mR^{-1}}_2^2 \leq \norm{\mR^{-1}}_{\rm F}^2 \leq \sum_{i=1}^n \sum_{j=i}^n \max \{1, 4^{j-i-1}\} = \frac{2}{3} n - \frac{1}{9} + \frac{4^n}{9} \leq 4^n,
    \]
    as needed.
\end{proof}

We now show how \nameref{alg:greedy-svd} produces a
$(n2^n)$-approximation of the singular values and corresponding subspaces of $\mM$.

\begin{lemma}
    \label{lem:approx-svd}
  Let $\mM \in \R^{m \times n}$ be a matrix and $\mP$, $\mR$ and $\mQ$ be the matrices constructed in
    \nameref{alg:greedy-svd} applied to $\mM$. Let $V_{\geq i} \coloneqq \im(\mP\mR_{\mdot, \ge i})$ and
    similarly $V_{\leq i} \coloneqq \im(\mP\mR_{\mdot, \le i})$ Then, we have that
    \[
        2^{-n}\max_{v \in V_{\geq i} \setminus \{\0_n\}} \frac{\norm{\mM v}}{\norm{v}} \le \norm{\mQ_{\mdot, i}} \le n  \min_{v \in V_{\leq i} \setminus \{\0_n\}} \frac{\norm{\mM v}}{\norm{v}}\,.
    \]
    Furthermore, $\norm{\mQ_{\mdot, i}} \le \sqrt{n} \sigma_{i}(\mM)$.
\end{lemma}
\begin{proof}
    For the lower bound on $\norm{\mQ_{\mdot ,i}}$, we have that
    \begin{align*}
        \max_{v \in V_{\geq i}} \frac{\norm{\mM v}}{\norm{v}} & = \max_{x \in \R^{n-i+1} \setminus \{\0\}} \frac{\norm{\mM \mP \mR_{\mdot,\geq i}x}}{\norm{\mP \mR_{\mdot,\geq i}x}}                                                                                      \\
                                                            & = \max_{x \in \R^{n-i+1} \setminus \{\0\}} \frac{\norm{\mQ_{\mdot,\geq i} x}}{\norm{\mR_{\mdot,\geq i}x}}                                                                                        \\
                                                            & \leq \max_{x \in \R^{n-i+1} \setminus \{\0\}} \norm{\mQ_{\mdot,i}} \frac{\norm{x}}{\norm{\mR_{\mdot, \geq i}x}} \quad \left(\text{\Cref{lem:approx-svd-properties} \ref{it:Q-orthogonal}, \ref{it:Q-norm-monotonicity}}\right) \\
                                                            & \leq \norm{\mQ_{\mdot,i}} \norm{\mR^{-1}}                                                                                                                                                        \\
                                                            & \leq 2^n  \norm{\mQ_{\mdot,i}}\, , \quad \left(\text{\Cref{prop:well-conditioned}}\right)
    \end{align*}
    proving the inequality.

    For the upper bound on $\norm{\mQ_{\mdot, i}}$, note that
     \begin{align*}
         \min_{v \in V_{\leq i}} \frac{\norm{\mM v}}{\norm{v}} & = \min_{x \in \R^i \setminus \{\0\}} \frac{\norm{\mM\mP\mR_{\mdot,\leq i}x}}{\norm{\mP\mR_{\mdot,\leq i}x}}                                                                                    \\
                                                             & = \min_{x \in \R^i \setminus \{\0\}} \frac{\norm{\mQ_{\mdot,\leq i} x}}{\norm{\mR_{\mdot,\leq i}x}}                                                                                      \\
                                                             & \ge \min_{x \in \R^i \setminus \{\0\}} \norm{\mQ_{\mdot,i}} \frac{\norm{x}}{\norm{\mR_{\mdot,\leq i}x}} \quad \left(\text{\Cref{lem:approx-svd-properties} \ref{it:Q-orthogonal}, \ref{it:Q-norm-monotonicity}}\right) \\
                                                             & \ge \norm{\mQ_{\mdot,i}} \norm{\mR}_2^{-1}                                                                                                                                                 \\
                                                             & \geq \frac{\norm{\mQ_{\mdot,i}}}{n}\, . \quad \left(\text{$\norm{\mR} \leq n$ by \Cref{lem:approx-svd-properties} \ref{it:R-properties}}\right)
     \end{align*}
    It remains to prove the furthermore statement. By~\Cref{prop:min-max-sing}, there exists a subspace $U_i \subseteq \R^n$ with $\dim(U_i) =
        n-i+1$ such that $\max_{x \in U_i \setminus \{\0\}} \frac{\norm{\mM \mP x}}{\norm{x}} = \sigma_i(\mM \mP) = \sigma_i(\mM)$. By dimension counting, there exists $\bar{x} \in U_i \setminus \{\0_n\}$ such that $\supp(\bar{x})
        \subseteq [i]$. Therefore,
    \[
        \sigma_i(\mM) \geq \frac{\norm{\mM \mP \bar{x}}}{\norm{\mP \bar{x}}} \geq \frac{1}{\sqrt{n}}
        \frac{\norm{\mM \mP \bar{x}}}{\norm{\bar{x}}_\infty} \ge \frac{1}{\sqrt{n}} \min_{j \in [i]} \min_{x \in \R^{i}, x_j = 1} \norm{\mM \mP_{\mdot,[i]} x} = \frac{\norm{\mQ_{\mdot,i}}}{\sqrt{n}},
    \]
    where the last equality follows form \Cref{lem:approx-svd-properties} \ref{it:Q-regression}.
    This concludes the proof.
\end{proof}

\begin{theorem}
    \label{thm:approx-svd-alg-correct}
    Let $\mM \in \R^{m \times n}$ be a matrix. The output of
    \nameref{alg:greedy-svd} on $\mM$ is a $(n2^n)$-\textsc{SVD} approximation of
    $\mM$ in the sense of \Cref{def:alpha-svd-problem}. Furthermore, the algorithm runs in strongly
    polynomial time $O(n^2 \max(m,n)^3)$.
\end{theorem}
\begin{proof}
    First, note that $\mB$ is an orthogonal basis of $\R^n$. Furthermore, $\im(\mB_{\ge i}) =
        V_{\ge i}$. Using \Cref{lem:approx-svd}, we therefore have that
    \begin{equation}
        \max_{v \in \im(\mB_{\geq i}) \setminus \{\0_n\}} \frac{\norm{\mM v}}{\norm{v}} = \max_{v \in V_{\geq i} \setminus \{\0_n\}} \frac{\norm{\mM v}}{\norm{v}} \le 2^n \norm{\mQ_{\mdot, i}} \le 2^n\sqrt{n} \sigma_i(\mM)\, .
    \end{equation}
    Using the variational characterization of singular
    values~\eqref{eq:min-max-sing} we obtain that
    \begin{equation}
        \sigma_{i}(\mM) \le \max_{v \in V_{\geq i} \setminus \{\0_n\}} \frac{\norm{\mM v}}{\norm{v}} \le 2^n  \norm{\mQ_{\mdot, i}} \\
    \end{equation}
    Together with \Cref{lem:approx-svd}, this proves that
    \begin{equation}
        \frac{1}{\sqrt{n}}\norm{\mQ_{\mdot, i}} \leq \sigma_{i}(\mM) \le 2^n \norm{\mQ_{\mdot, i}}\, .
    \end{equation}
    In particular, by choice of $s_i = 2^{n} \norm{\mQ_{\mdot, i}}_1$ for $i \in [n]$, we get that
    \begin{equation}
        \frac{1}{2^{n}n}s \leq \sigma(\mM) \le s\, ,
    \end{equation}
    as needed for \Cref{def:alpha-svd-problem}. The choice of $\norm{\cdot}_1$ instead of $\norm{\cdot}_2$ in definition of $s$ is necessary to ensure that the output of the algorithm has polynomial bit complexity.

    For the strongly polynomial guarantees, it remains to check that all the intermediate iterates have polynomial bit complexity. To see this, note first that the columns of $\mR$ correspond to solutions of a well-described linear system in the original matrix $\mM$, and thus has bit-complexity polynomially related to that of $\mM$. From here, $\mB$ corresponds to a ``reverse order'' GSO applied to $\mP\mR$ (or again, the columns are solutions to well-described linear systems in $\mP\mR$), and hence $\mB$ also has bit-complexity polynomially related to that of $\mM$ (see, e.g.,~\cite[Section 1.4]{gls} for a thorough overview of the bit-complexity of GSO). 
    
    Let us now prove the statment about the running time. Using the characterization in \Cref{prop:pseudo-compute} and further noting that the required rank factorization can be computed in $O(\max(m,n)^3)$, we are able to compute the pseudoinverse of an $m \times n$ matrix as well as the projection onto the image of an $m \times n$ matrix in time $O(\max(m,n)^3)$.  Therefore, a single execution of both Lines~\ref{line:compute-with-pseudo-inverse} and \ref{line:projection-computation-in-svd} of Algorithm~\ref{alg:greedy-svd} takes $O(\max(m,n)^3)$. The repeated execution of these lines dominates the overall running time of the algorithm, hence the overall running time of the algorithm can be bounded by $O(n \cdot n \cdot \max(m,n)^3) = O(n^2 \max(m,n)^3)$. 
    
    This completes the proof of the theorem. \end{proof}

\subsection{Boosting via the Power Method}
\label{sec:boosting}

In this section, we demonstrate a method to enhance the approximation ratio of any algorithm for the
$\SVDA$-\textsc{SVD} problem (as in \Cref{def:alpha-svd-problem}) to a $(1 + \eps)$-\textsc{SVD}
approximation. This is possible provided $\SVDA= \SVDA(\dim(\mM))$, i.e., the approximation ratio $\SVDA$ is a function of the dimension of the matrix only and does not depend on the
matrix's conditioning. The strategy involves using a low accuracy
$\SVDA$-\textsc{SVD} approximation to calculate an approximate SVD of the matrix $\mS \coloneqq
    (\mM^\top \mM)^p$. Here, a large power $p = O(n/\eps \cdot \log(\SVDA/\eps))$ is used to achieve the
desired error $\eps$. We will then show that the approximate SVD of $\mS$ is in fact a $(1+\eps)$-approximation of the SVD of $\mM$.

\begin{lemma}
\label{lem:boosting-alpha-svd}
Let $0 < \eps \leq 1/2$ and $\SVDA \geq \eps$. Then, there is a reduction
from $(1+\eps)$-\textsc{SVD} on an $m \times n$ matrix to
$\SVDA$-\textsc{SVD} on one $n \times n$ and one $m \times n$ matrix that
runs in time $O(m n^2+ n^3 \log(\log(2 + \SVDA)/\eps))$, and
requires space polynomial in $m,n,\log(2+\SVDA)/\eps$ and the bit-encoding length of
$\eps,\SVDA$ and the input matrix. 
\end{lemma}
\begin{proof}
Let $\mM \in \R^{m \times n}$ be a $m \times n$ matrix. The reduction proceeds as follows. Compute $\mS \coloneqq (\mM^\top \mM)^p$ for $p \coloneqq \lceil
\log_{1+\eps/3}(\SVDA)/2 \rceil$. Let $(\mB,s)$ be
the output of $\SVDA$-\textsc{SVD} on $\mS$ and $(\cdot,s')$ be the output of
$\SVDA$-\textsc{SVD} on $\mM$. For each $i \in [n]$, apply binary search to
compute $\hat{s}_i = s'_i(1+\eps/3)^{-2k_i}$, $-1 \leq k_i \leq 2p$, satisfying $s_i^{1/(2p)} \leq \hat{s}_i \leq
(1+\eps/3) s_i^{1/(2p)}$. Finally, return $(\mB,\hat{s})$.  

We first proof correctness of the reduction. Consider a singular value
decomposition $\mM = \sum_{i = 1}^{r} \sigma_i(\mM)$ $u_i v_i^\top$ of $\mM$,
where $r = \rank(\mM)$. Fix $i \in [n]$. By definition of $\mB$, we have for all $x \in \im(\mB_{\ge i}) \setminus \{\0_n\}$ that $\norm{\mS x}_2/\norm{x}_2 \leq s_i \leq \SVDA \sigma_i(\mS)$, which implies that
    \begin{align}
        \sum_{j = 1}^r \sigma_j(\mM)^{4p}\langle v_j,x \rangle^2 &= 
        \norm{\sum_{j = 1}^r \sigma_j(\mM)^{2p} v_j \langle v_j,x \rangle}^2 =
        \norm{\mS x}^2  \nonumber \\ &\le s_i^2 \norm{x}^2 \leq \SVDA^2 \sigma_i(\mS)^2 \norm{x}^2 = \SVDA^2 \sigma_i(\mM)^{4p} \norm{x}^2.
        \label{eq:boosting}
    \end{align}
Therefore, $\forall x \in \im(\mB_{\geq i}) \setminus \{\0_n\}$, we have that 
    \begin{align}
        \norm{\mM x}^{4p} & = \left(\sum_{j=1}^r \sigma_j(\mM)^2 \langle v_j, x\rangle^2 \right)^{2p}
        \nonumber                                                                                                                                    \\
                         & = \left(\norm{x}^2 \cdot \sum_{j = 1}^r \sigma_j(\mM)^2 \frac{\langle v_j, x\rangle^2}{\norm{x}^2}\right)^{2p} \nonumber    \\
                         & \le \norm{x}^{4p} \sum_{j = 1}^r \sigma_j(\mM)^{4p} \frac{\langle v_j, x\rangle^{2}}{\norm{x}^2}  \quad \left(\text{Jensen's inequality}\right) \nonumber \\
                         &\leq \norm{x}^{4p} s_i^2 \nonumber \\
                         &\leq \norm{x}^{4p} \cdot \SVDA^2 \sigma_i(\mM)^{4p} \left(\text{by \eqref{eq:boosting}}\right) \, ,
    \end{align}
    and so
    \begin{equation}
        \label{eq:boosting-2}
        \max_{x \in \im(\mB_{\ge i}) \setminus \{\0_n\}} \frac{\norm{\mM x}}{\norm{x}} \le s_i^{1/(2p)} \leq \SVDA^{1/(2p)} \sigma_i(\mM) \le (1+\eps/3) \sigma_i(\mM) \, .
    \end{equation}

From the above, note that $(\mB,s^{1/(2p)})$ is already a solution to
$(1+\eps)$-\textsc{SVD} on $\mM$. Unfortunately, we cannot compute the vector
$s^{1/(2p)}$ in strongly polynomial time, which motivates the binary search
used to construct the approximation $\hat{s}$. By the guarantees of $s$ and
$s'$, we have that 
\[
s'(1+\eps/3)^{-2p} \leq \frac{s'}{\SVDA} \leq \sigma(\mM) \leq s^{1/(2p)} \leq (1+\eps/3) \sigma(\mM) \leq s'(1+\eps/3),  
\]
and thus the claimed range for $k_i$, $i \in [n]$, in the binary search is
correct. From here, by the guarantees on $\hat{s}$, since $\eps \in (0,1/2]$,
we have that 
\[
s^{1/(2p)} \leq \hat{s} \leq (1+\eps/3) s^{1/(2p)} 
\leq (1+\eps/3)^2 \sigma(\mM) \leq (1+\eps) \sigma(\mM).
\]
Therefore, $(\mB,\hat{s})$ is a valid solution to $(1+\eps)$-\textsc{SVD}.

We now justify the running time of the reduction. Since $p =
\Theta(\log(2+\SVDA)/\eps)$, we can compute $\mS = (\mM^\top \mM)^p$ in time
$O(mn^2+ n^{3}\log(p)) = O(m n^2+ n^3 \log(\log(2 + \SVDA)/\eps))$ by repeated
squaring of $\mM^\top \mM$. For each $i \in [n]$, the corresponding binary
search requires $O(\log p)$ comparisons, where each comparison are of the
form $s_i \leq (s')^{2p}(1+\eps/3)^{-2k}$ for $-1 \leq k \leq 2p$.
Lastly, the reduction makes two $\SVDA$-\textsc{SVD} calls, one to $\mM \in
\R^{m \times n}$ and one to $\mS \in \R^{n \times n}$.     

It remains to show the claim about the space needed by the algorithm. To this
end, it suffices to show that the matrix $\mS$ has a bit-encoding length
polynomial in $\log(2+\rho)/\eps^{-1}$ and the bit-encoding of $\mM$. To this end, let the
input matrix $\mM$ be given as $\mM_{i,j} = p_{ij}/q_{ij}$ for $i \in [m], j
\in [n]$ with $p_{ij} \in \mathbb{Z}$ and $q_{ij} \in \mathbb{Z}_{> 0}$.  We
can write $\mM = k^{-1}\bar \mM$, where $k =
\operatorname{lcm}_{i,j}{q_{ij}}$ and $\bar \mM_{ij} = |p_{ij}/q_{ij}| \cdot
\operatorname{lcm}_{i,j} q_{ij} \in \mathbb{Z}$, where $\operatorname{lcm}$
denotes the least common multiple. In particular, $\bar \mM \in \mathbb{Z}^{m
\times n}$ and $\norm{\bar \mM}_{\rm F} \le \sum_{ij} |p_{ij}| \cdot \prod_{ij}
q_{ij}$.  Therefore, we can write $\mS = (\mM^\top \mM)^p  = k^{-{2p}} (\bar
\mM^\top \bar \mM)^p$, where $\norm{(\bar \mM^\top \bar \mM)^p}_{\rm F} \le
\norm{\bar \mM}^{2p}_{\rm F} \le (\sum_{ij} |p_{ij}| \prod_{ij} |q_{ij}|)^{2p}$. In
particular, the entries of $\mS$ can be written in a number of bits that is
polynomial in $p$ and the number of bits of $\mM$. Recalling that $p =
O(\log(2+\SVDA)/\eps)$ finishes the proof.
\end{proof}

\Cref{thm:approx-svd} is a now a simple consequence of \Cref{lem:boosting-alpha-svd} and \Cref{thm:approx-svd-alg-correct}.
\thmapproxsvd*
\begin{proof}
Let $\eps = \min \{\SVDA-1, 1/2\}$, and apply the reduction in
\Cref{lem:boosting-alpha-svd} with parameters $(\eps,n2^n)$, where we use
\Cref{thm:approx-svd-alg-correct} to implement the $(n2^n)$-\textsc{SVD} algorithm. The
space guarantees now follow from the guarantees of
\Cref{thm:approx-svd-alg-correct} and \Cref{lem:boosting-alpha-svd}. For the running time, observe that \Cref{thm:approx-svd-alg-correct} requires $O(n^2 \max(m,n)^3)$ operations, while the application of  \Cref{lem:boosting-alpha-svd} requires $O(mn^2 + n^3 \log(\log (2 + n2^n)/\min(\SVDA - 1, \frac{1}{2}))) = O(mn^2 + n^3\log(n + \frac{1}{\SVDA - 1}))$ operations. The sum of these running times is  $O(n^2\max(m,n)^{3} \log(n + \frac{1}{\SVDA - 1}))$, which completes the proof.
\end{proof}
 \section{Self-Concordant Barrier Central Paths}\label{sec:self-concordant}

The theory of self-concordant barriers, introduced by Nesterov and Nemirovski~\cite{Nesterov1994}, is used to describe interior point methods in a general setting. In this section, we show that the central paths obtained from self-concordant barriers relate to the max central path, as in the case of the log-barrier. From this, we deduce a proof of \Cref{th:coro-lower-bound}.

We follow the presentation of self-concordant barriers made by Renegar~\cite{Renegar01}. We consider a $\cal C^2$ function $f \colon D \to \R$ with open convex domain $D \subseteq E$ where $E = x^0 + L$ is an affine subspace of $\R^n$. We recall that, for $x \in D$, the gradient $\nabla f(x)$ of $f$ at $x$ is the element of $L$ such that for all $h \in L$, $\pr{\nabla f(x)}{h} = \lim_{t \to 0^+} \frac{f(x + t h) - f(x)}{t}$. Moreover, the Hessian $\nabla^2 f(x)$ of $f$ at $x$ is the self-adjoint linear operator from $L$ to itself such that for all $h, k \in L$, $\pr{\nabla^2 f(x) h}{k} = \lim_{t \to 0^+} \frac{\pr{\nabla f(x + t k) - \nabla f(x)}{h}}{t}$. We assume that $\nabla^2 f(x)$ is positive definite, \ie, $\pr{\nabla^2 f(x) h}{h} > 0$ for all $h \in L$, $h \neq 0$. In particular, $f$ is a strictly convex function on $D$. Given $x \in D$, we introduce $\norm{\cdot}_x$, the norm over $L$ associated with the Hessian $\nabla^2 f(x)$, \ie, $\norm{z}_x \coloneqq \big(\pr{\nabla^2 f(x) z}{z}\big)^{1/2}$ for all $z \in L$. We define $\cal B_x(y; r) \coloneqq \{z \in D : \norm{y - z}_x < r\}$, the open ball with center $y \in D$ and radius $r \geq 0$ with respect to the norm $\norm{\cdot}_x$. 

The function $f$ is  \emph{(strongly nondegenerate) self-concordant} if, for all $x \in D$, we have $\cal B_x(x; 1) \subseteq D$, and for all $y \in \cal B_x(x; 1)$ and nonzero vector $v \in L$,
\[
1 - \norm{y - x}_x \leq \frac{\norm{v}_x}{\norm{v}_y} \leq \frac{1}{1 - \norm{y - x}_x} \, .
\]
We define the \emph{complexity value} as
\[
\vartheta_f \coloneqq \sup_{x \in D} \sup_{\norm{h}_x \leq 1} \pr{\nabla f(x)}{h}^2\, .
\]
If $\vartheta_f< \infty$, then $f$ is called a \emph{self-concordant barrier}. 

Let $f$ be self-concordant barrier over the (relative) interior of the polyhedron $\Primal$. The \emph{central path} associated with $f$ is the function that maps $\mu > 0$ to the unique minimizer of the function $x \mapsto \pr{c}{x} + \mu f(x)$. When $f$ is (the restriction to $E$ of) the log-barrier $x \mapsto -\sum_{i = 1}^n \log (x_i)$, the scalar $\mu$ is used as the parameter of the central path. In the case of general barriers, it is more convenient to parametrize the central path by the optimality gap $g$. This is made possible by the fact that, if any two points of the central path have the same optimality gap $g$, they both minimize the (strictly convex) function $f$ over $\{ x \in \Pcal \colon \pr{c}{x} = g \}$, thus they are equal. Following this, we denote by $x\sccp(g)$ the unique point of the central path induced by the barrier $f$ with optimality gap $g > 0$. 

In the next statement, we generalize \Cref{lem:cp-max-cp} to the case of the central path associated with the self-concordant barrier $f$. 
\begin{proposition}\label{prop:mcp-sccp}
For all $g > 0$, we have $\frac{1}{2(2\vartheta_f + 1)} \mx(g) < x\sccp(g) \leq \mx(g)$. 
\end{proposition}

\begin{proof}
The right hand side inequality follows from $x\sccp(g) \in \Pcal_g$. 
Using~\cite[Theorem~2.3.4]{Renegar01} and the fact that $\mu \nabla f(x\sccp(g)) + c \in L^\perp$ for some $\mu > 0$, we know that $\Pcal_g \subseteq \cal B_{x\sccp(g)} (x\sccp(g); 4 \vartheta_f + 1)$. Since $\Pcal_g \subseteq \R^n_{\geq 0}$, \cite[Lemma~4]{AGV22} shows that $y < 2 (2 \vartheta_f + 1) x\sccp(g)$ for all $y \in \cal B_{x\sccp(g)} (x\sccp(g); 4 \vartheta_f + 1)$. We deduce that $\mx(g) < 2 (2 \vartheta_f + 1) x\sccp(g)$.
\end{proof}

We are now ready to prove \Cref{th:coro-lower-bound}.

\begin{proof}[Proof (\Cref{th:coro-lower-bound})]
Let $x^{(0)}, \dots, x^{(T)}$ be the successive iterates of the IPM. We denote by $g^{(0)}, \dots, g^{(T)}$ their respective optimality gap, where  $g^{(0)} \ge g_0$ and $g^{(T)} \le g_1$.

Let $i \in [n]$. We claim that, $\forall k \in [T]$, the segment $[(g^{(k-1)}, x^{(k-1)}_i), (g^{(k)}, x^{(k)}_i)]$ is included in the neighborhood $\mnbp_{i}(\theta')$ for $\theta' = 1 - \frac{1 - \theta}{2(2\vartheta_f + 1)}$. Indeed, for all $\lambda \in [0,1]$, the point $x = (1-\lambda) x^{(k-1)}+ \lambda x^{(k)}$ has optimality gap $g = (1-\lambda) g^{(k-1)} + \lambda g^{(k)}$. As $x \in \Pcal$, we have $x_i \leq \mx_i(g)$. Moreover, since $x \geq (1-\theta) x\sccp(g)$, \Cref{prop:mcp-sccp} ensures that $x_i > \frac{1-\theta}{2(2\vartheta_f + 1)} \mx_i(g)$. Hence, $(g,x) \in \mnbp_{i}(\theta')$. We deduce by \Cref{lemma:slc2} that $T \geq \SLC^{\mathrm{p}}_{\theta',i}(g_1,g_0)$. 

Consequently, we have $\sum_{i = 1}^n \SLC^{\mathrm{p}}_{\theta',i}(g_1,g_0) \leq n T$. By \Cref{th:main_upper_bound-slc}, the number $T'$ of iterations performed by the algorithm~\nameref{alg:subspace_ipm} is in 
\begin{equation}
O\left(\sqrt{n}\log\left(\frac{n}{1-\theta}\right)\min\left\{\sum_{i=1}^n \SLC^{\mathrm{p}}_{\theta,i}(g_1,g_0),\sum_{i=1}^n \SLC^{\mathrm{d}}_{\theta,i}(g_1,g_0)\right\}\right)\,. 
\end{equation}
This implies that
\[
T' = O\left(n^{1.5} \log \left(\frac{n\vartheta_f}{1-\theta}\right) T \right)\, .
\]
\end{proof}

\Cref{th:coro-lower-bound} relates the iteration complexity of general
barrier IPMs traversing the wide neighborhood $\cal N\sccp(\theta)$ as
defined in~\eqref{eq:self_concordant_neighborhood} to the iteration
complexity of our IPM. Our wide neighborhood definition is however not
entirely standard, and in particular, it does not map directly to the
more standard neighbhorhoods such as those based on the Newton decrement. Furthermore, it does not obviously capture IPMs using primal-dual neighborhoods.

In~\cite[Section~4.1]{AGV22} however, it has been shown that all the neighborhoods of self-concordant central paths used in the literature (including that of the log-barrier central path like the projections of $\cal N^2(\beta)$, $0 < \beta < 1/4$, and $\cal N^{-\infty}(\theta)$, $0 < \theta < 1$, to primal variables) are actually contained in \emph{multiplicative neighborhoods}, \ie, sets of form 
\[
\cal M\sccp(\underline m, \overline m) = \big\{ x \in \Pcal : \exists g\,  \enspace  \underline m x\sccp(g) \leq x \leq \overline m x\sccp(g) \big\} \qquad (0 < \underline m \leq 1 \leq \overline m)  \, .
\]
The latter turn out to be essentially equivalent to our definition of wide
neighborhoods in~\eqref{eq:self_concordant_neighborhood}. This is justified
by the following lemma:
\begin{lemma}
We have $\cal N\sccp(\theta) \subseteq \cal M\sccp\big(1-\theta, 2(2\vartheta_f + 1)\big)$ for all $\theta \in (0,1)$, and $\cal M\sccp(\underline m, \overline m) \subseteq \cal N\sccp\big(\frac{\underline m}{2(2\vartheta_f + 1) \overline m}\big)$ for all $0 < \underline m \leq 1 \leq \overline m$.
\end{lemma}

\begin{proof}
For the first statement,
let $\theta \in (0,1)$, and $x \in \cal N\sccp(\theta)$, and let $g=\pr{c}{x}-v^*$ denote  the optimality gap of $x$. The inequality $x \geq (1 - \theta) x\sccp(g)$ follows by the definition of $\cal N\sccp(\theta)$. We have $x \leq \mx(g)$ since $x \in \Pcal$, and so $x \leq 2(2\vartheta_f + 1) x\sccp(g)$ by \Cref{prop:mcp-sccp}. 

For the second statement,
let $0 < \underline m \leq 1 \leq \overline m$, and let $(\tilde g, x)$ be such that $\underline m x\sccp(\tilde g)$ $\leq x \leq \overline m x\sccp(\tilde g)$. Let $s^\star\ge0$ denote a dual optimal solution to \eqref{LP_primal_dual}. The optimality gap $g = \pr{c}{x}-v^*=\pr{s^\star}{x}$ of $x$ satisfies $g \leq \overline m \pr{s^\star}{x\sccp(\tilde g)} = \overline m \tilde g$ (since $s^\star \geq 0$). Then, we have 
\begin{align*}
x & \geq \frac{\underline m}{2(2\vartheta_f + 1)} \mx(\tilde g) && \left(\text{by \Cref{prop:mcp-sccp}}\right)\\
& \geq \frac{\underline m}{2(2\vartheta_f + 1) \overline m} \mx(\overline m \tilde g) && \left(\text{by \Cref{lemma:mcp_subhomogeneous}}\right)\\
& \geq \frac{\underline m}{2(2\vartheta_f + 1) \overline m} \mx(g) && \left(\text{as $g \leq \overline m \tilde g$}\right)\\
& \geq \frac{\underline m}{2(2\vartheta_f + 1) \overline m} x\sccp(g). && \left(\text{by \Cref{prop:mcp-sccp}}\right)
\end{align*}
\end{proof}

As a consequence, the neighborhoods $\cal N\sccp(\theta)$ are flexible enough
to capture all known
neighborhoods of central paths. Therefore, \Cref{th:coro-lower-bound}
indeed shows that our IPM is approximately optimal when compared against
essentially any straight-line following IPM.

\bibliographystyle{abbrv}
\bibliography{curvature}

\appendix 

\section{Missing Proofs in Section~\ref{sec:lin-alg-prelim}}
\label{sec:appendix-la}

\begin{proof}[Proof of \Cref{prop:adjoint}]
We show the first part. In particular, we show that $\ker(\adj(T)) = V \cap
\im(T)^\perp$:  
\begin{align*}
\ker(\adj(T)) &= \{y \in V: \pr{\adj(T)y}{\adj(T)y} = 0\} = \{y \in V: \pr{\adj(T)y}{u} = 0, \forall u \in U\} \\ &= \{y \in V: \pr{y}{Tu} = 0, \forall u \in U\} = V \cap \im(T)^\perp.
\end{align*}
The second part now follows from the first part, applying the first part with
$T$ replaced by $\adj(T)$ and using that $\adj(\adj(T)) = T$.

For the third part, we must show $\im(T) = T(\im(\adj(T)))$ and $\im(\adj(T))
= \adj(T)(\im(T))$. Since $\adj(\adj(T))=T$, it suffices to prove the first
equality. For this purpose, by the first part, we have that
\[
\im(T) = T(U) = T(\im(\adj(T))+\ker(T)) = T(\im(\adj(T))),
\]
as needed. For the in particular, note that the above equality directly
implies that $\rank(T)=\dim(\im(T))$ $= \dim(T(\im(\adj(T)))) \leq
\dim(\adj(T))$. Using again that \ifdefined\SIAMversion \newline \else \fi$\adj(\adj(T)) = T$, we also have
$\rank(\adj(T)) \leq \rank(T)$, which proves the desired equality $\rank(T) =
\rank(\adj(T))$. 
\end{proof}

\begin{proof}[Proof of \Cref{prop:associated-matrix}]
Let $\mM_{ij} = (e^i)^\T T(\Pi_U e^j)$, $\forall i \in [m], j \in [n]$. With
this definition, clearly $\mM x = T(\Pi_U x)$, $\forall x \in \R^n$, and thus
$\mM u = T(\Pi_U u) = T(u)$, $\forall u \in U$. We now check that
$\im(\mM^\T) \subseteq U$. Since $U^\perp = \ker(\Pi_U) \subseteq \ker(\mM)$
by construction, we see that $\im(\mM^\T) = \ker(\mM)^\perp \subseteq
(U^\perp)^\perp = U$ as needed. 

We now show uniqueness: if $\bar{\mM} u = T(u)$, $\forall u \in U$
and $\im(\bar{\mM}^\T) \subseteq U$, then $\mM = \bar{\mM}$. Let $\mD = \mM -
\bar{\mM}$ and examine $\mD \mD^\T e^i$ for $i \in [m]$. Since $\mD^\T e^i =
\mM^\T e^i - \bar{\mM}^\T e^i \in U$, we have $\mD \mD^\T e^i = \mM \mD^\T
e^i - \bar{\mM} \mD^\T e^i = T(\mD^\T e^i)-T(\mD^\T e^i) = \0_m$. In
particular, $0 = (e^i)^\T \mD \mD^\T e^i = \norm{\mD^\T e_i}^2$, and thus
$\mD^\T e^i = \0_n$, $\forall i \in [m]$. In particular, $\mD = \0_{m \times
n}$ and uniqueness follows. Therefore, $\mM = \cal M(T)$, as needed. 

The equality $\cal T(\cal M(T)) = \restro{T}{U}{\R^m} \circ
\restro{\Pi_U}{\R^n}{U}$ follows directly from the above construction, recalling that $\cal M(T)(x) = T(\Pi_U(x))$, $\forall x \in \R^n$.

We now prove the furthermore. Since $\im(\cal M(T)^\T) \subseteq U$ by
definition, we have $U^\perp \subseteq \ker(\cal M(T))$, and hence $\im(\cal
M(T)) = \cal M(T)(U + U^\perp) = \cal M(T)(U) = T(U) = \im(T)$, where the
second to last equality follows by definition of $\cal M(T)$. By
\Cref{prop:adjoint}, we have $\R^m = \ker(\cal M(T)^\T) + \im(\cal M(T)) =
\ker(\cal M(T)^\T) + \im(T)$. In particular, $\im(\cal M(T)^\T) = \cal
M(T)^\T(\ker(\cal M(T)^\T) + \im(T)) = \cal M(T)^\T(\im(T)) =
\adj(T)(\im(T))$, where the last equality follows by \Cref{rem:adjoint} and
$\im(T) \subseteq V$. Since $\ker(\adj(T)) + \im(T) = V$ by
\Cref{prop:adjoint}, we similarly get that $\adj(T)(\im(T)) = \im(\adj(T))$,
as needed.  
\end{proof}

\subsection{Approximate Singular Subspaces}
\label{sec:appendix-apx-ss}

We now give the proof of \Cref{lem:ass-op-mat}, which relates the approximate
singular subspaces of an operator to those of its associated matrix.

\begin{proof}[Proof of \Cref{lem:ass-op-mat}]
By \Cref{prop:associated-matrix}, we have that $\bar{T}(x) = T(\Pi_\src(x))$,
$\forall x \in \R^n$. Therefore, for all $x \in \src^\perp$ we have $\bar{T}(x) = T(\proj_\src(x)) = T(\0_n) = \0_m$, and hence $\src^\perp \subseteq \ker(\bar{T})$. 

We begin by proving~\eqref{lem:ass-op-mat-1}. Firstly, $\ker(T) =
\ker(\bar{T}) \cap \src \subseteq S \cap \src$ by our assumption that $\ker(\bar{T}) \subseteq S$, which proves the first
inclusion. Using that $\src^\perp \subseteq \ker(\bar{T}) \subseteq S$ and
orthogonal decomposition, we get the desired first equality
\[
\Pi_\src(S) = \src \cap (S + \src^\perp) = \src \cap S.
\]   
We now claim that $(S \cap \src) + \src^\perp = S$. Using orthogonal decomposition
again, this follows by 
\begin{equation}
S \subseteq \proj_\src(S) + \src^\perp = (S \cap \src) + \src^\perp \subseteq S + S = S \label{eq:section-projection}.  
\end{equation}
Using the above and $\src^\perp \subseteq \ker(\bar{T})$, we get
that $\sigma_1(\restr{T}{S \cap \src})= \sigma_1(\restr{\bar{T}}{(S \cap \src
) + \src^\perp})$ via the same argument as in \Cref{rem:canonical-subspaces}.
This yields $\sigma_1(\restr{T}{S \cap \src}) = \sigma_1(\restr{\bar{T}}{S})$
by~\eqref{eq:section-projection}.

We now prove~\eqref{lem:ass-op-mat-2}. Firstly, using that $\R^n =
\src+\src^\perp$ and $S = (S \cap \src) + \src^\perp$, we have $n = \dim(\src)
+ \dim(\src^\perp)$ and $\dim(S) = \dim(S \cap \src) + \dim(\src^\perp)$. In
particular,
\begin{equation}
n - \dim(S) + 1 = \dim(\src) - \dim(S \cap \src) + 1. \label{eq:align-dimension}
\end{equation}
Therefore, using $\sigma_1(\restr{T}{S \cap \src}) =
\sigma_1(\restr{\bar{T}}{S})$ from part~\eqref{lem:ass-op-mat-1}, we get
that $S$ is a $\SVDA$-approximate singular subspace for $\bar{T}$ of
dimension $d \geq 0$ if and only if $S \cap \src$ is a $\SVDA$-approximate
singular subspace for $T$ of dimension $d-\dim(\src^\perp)$.

We now prove the moreover. By \Cref{prop:restrict-sing}, recall that
$\sigma^+(T) = \sigma^+(\cal M(T))$ $= \sigma^+(\bar T)$. Thus,
$\sigma_{[\dim(X)]}(T) = \sigma_{[\dim(X)]}(\bar{T})$ and $\sigma_{[n] \setminus [\dim(X)]}(\bar{T}) = \0_{[n] \setminus [\dim(X)]}$. Therefore, for $\tau \geq 0$, 
\[
\cnt{\bar T}{\tau} = \cnt{T}{\tau} + (n-\dim(\src)) = \cnt{T}{\tau}+\dim(\src^\perp).
\]  
\end{proof}

We now move on to the proof of \Cref{lem:approx-complement}, which shows that
the orthogonal complement of an approximate singular subspace is an
approximate maximizer of~\eqref{eq:max-min-sing} whenever there is a large
enough gap in the singular values. The proof will require the following two
helper lemmas. 

\begin{lemma}[Subspace Projection Adjoints]
\label{lem:subspace-projection-adjoint}
Let $S,U \subseteq \R^n$ be subspaces. Then, $\restro{\proj_U}{S}{U} = \adj(\restro{\proj_S}{U}{S})$. In particular, $\sigma^+(\restr{\proj_U}{S}) = \sigma^+(\restr{\proj_S}{U})$ and if $\dim(U) = \dim(S)$, then $\sigma(\restr{\proj_U}{S}) = \sigma(\restr{\proj_S}{U})$.
\end{lemma}
\begin{proof}
Let $s \in S$ and $u \in U$. Then,
\begin{align*}
\pr{s}{\restr{\proj_S}{U}(u)} &= \pr{s}{\proj_S(u)} =  \pr{\proj_S(s)}{u} =  \pr{s}{u} \\ 
&= \pr{s}{\proj_U(u)} = \pr{\proj_U(s)}{u} =  \pr{\restr{\proj_U}{S}(s)}{u}\, .
\end{align*}
This proves that $\restro{\proj_U}{S}{U}$ and  $\restro{\proj_S}{U}{S}$ are adjoints of one another. The other statements follow from this and \Cref{prop:sing-adjoint}.
\end{proof}

\begin{lemma}[Singular Subspace Proximity]
\label{lem:subspace-proximity}
Let $T \colon \src \rightarrow \targ$, $\src \subseteq \R^n, \targ \subseteq \R^m$ be a linear operator with singular value decomposition $\cal M(T) = \sum_{i=1}^{\rank(T)} \sigma_i(T) y_i x_i^\T$ as in \Cref{def:svd}. Let $S \subseteq \src$ be a
    $\SVDA$-approximate singular subspace for $T$ with $\dim(\ker(T)) \leq \dim(S) < \dim(\src)$ and let $U = \lspan(x_1,\dots,x_{\dim(X)-\dim(S)})^\perp \cap X$ . Then, for all $s \in S \setminus \{\0_n\}$ we have that 
    \[
        1 \ge \frac{\|\proj_{U}(s)\|_2^2}{\|s\|_2^2} \ge 1 - \SVDA^2 \frac{\sigma_{\dim(X) - \dim(S) + 1}(T)^2}{\sigma_{\dim(X) - \dim(S)}(T)^2}\, . 
    \]
\end{lemma}
\begin{proof}
The first inequality follows from properties of projections. We now prove the
second inequality. Recalling that $\rank(T)+\dim(\ker(T))=\dim(X)$ and
using that $\dim(\ker(T)) \leq \dim(S) < \dim(X)$, we have $1 \leq \dim(X)-\dim(S) \leq
\rank(T)$. Let $\bar U = \lspan(x_1,\dots,x_{\dim(X)-\dim(S)}) = U^\perp \cap X$. Note that $T(\bar{U}) = \lspan(y_1,\ldots,y_{\dim(X)-\dim(S)}) \perp \lspan(y_{\dim(X)-\dim(S)+1},\ldots, y_{\rank(T)}) = T(U)$. Therefore, for any $s \in S \setminus \{\0_n\}$ we have that
\begin{equation}
\label{eq:subspace-proximity}
\begin{aligned}
\SVDA^2 \sigma_{\dim(X) - \dim(S) + 1}(T)^2 &\ge \frac{\|T(s)\|_2^2}{\|s\|_2^2} = \frac{\|T(\proj_{U}(s))\|_2^2 + \|T(\proj_{\bar U}(s))\|_2^2}{\|s\|_2^2} \\ &\ge \sigma_{\min}(\restr{T}{\bar U})^2\frac{\|\proj_{\bar U}(s)\|_2^2}{\|s\|_2^2}  
= \sigma_{\dim(X) - \dim(S)}(T)^2 \frac{\|\proj_{\bar U}(s)\|_2^2}{\|s\|_2^2} \, . 
\end{aligned}
\end{equation}
Noting that $\|s\|_2^2 = \|\proj_U(s)\|_2^2 + \|\proj_{\bar U}(s)\|_2^2$ and
reordering the terms gives the result, where the condition $\dim(X) > \dim(S) \geq \dim(\ker(T))$ ensures that $\sigma_{\dim(X)-\dim(S)}(T)$ $\geq
\sigma_{\rank(T)}(T) > 0$.
\end{proof}

We are now ready to prove \Cref{lem:approx-complement}. 

\begin{proof}[Proof of \Cref{lem:approx-complement}]
If $\dim(S) < \dim(\ker(T))$, then $\sigma_{\dim(X) - \dim(S)}(T) \leq$ \newline $
\sigma_{\rank(T)+1}(T)=0$ and there is nothing to prove. So assume $\dim(S)
\geq \dim(\ker(T))$. Let $U$ be as in \Cref{lem:subspace-proximity}. Recall that $\bar{U} = U^\perp \cap \src$ satisfies $T(U) \perp T(\bar{U})$ and
$\sigma_{\min}(\restr{T}{\bar{U}}) = \sigma_{\dim(X)-\dim(S)}(T)$ by construction. Furthermore, $\dim(S)=\dim(U)$ and $\dim(\bar{S})=\dim(\bar{U})=\dim(X)-\dim(S)$. From here, we get that
\ifdefined\SIAMversion
\begin{align*}
    \sigma_{\min}(\restr{T}{\bar S})^2 &= \min_{\bar s \in \bar S \setminus \{\0\}} \frac{\|T(\bar s)\|_2^2}{\|\bar s\|_2^2}  \\
    &\ge \min_{\bar s \in \bar S \setminus \{\0\}} \frac{\|T(\proj_{\bar U}(\bar s))\|_2^2}{\|\bar s\|_2^2} \hspace{13em} \left(\text{by $T(U) \perp T(\bar U)$}\right) \\
    &\ge \sigma_{\dim(X) - \dim(S)}(T)^2 \min_{\bar s \in \bar S \setminus \{\0\}} \frac{\|\proj_{\bar U}(\bar s)\|_2^2}{\|\bar s\|_2^2} \\ & \hspace{16em} \left(\text{as $\sigma_{\min}(\restr{T}{\bar U}) = \sigma_{\dim(X) - \dim(S)}(T)$}\right) \\
    &= \sigma_{\dim(X) - \dim(S)}(T)^2 \left(1 - \max_{\bar s \in \bar S \setminus \{\0\}} \frac{\|\proj_{U}(\bar s)\|_2^2}{\|\bar s\|_2^2}\right) \\ & \hspace{19.7em} \left(\text{as $S \subseteq U + \bar U$ and $U \perp \bar U$}\right) \\
    &= \sigma_{\dim(X) - \dim(S)}(T)^2 \left(1 - \max_{u \in U \setminus \{\0\}} \frac{\|\proj_{\bar S}(u)\|_2^2}{\|u\|_2^2}\right) \hspace{3.5em} \left(\text{by \Cref{lem:subspace-projection-adjoint}}\right) \\
    &=  \sigma_{\dim(X) - \dim(S)}(T)^2 \min_{u \in U \setminus \{\0\}} \frac{\|\proj_{S}(u)\|_2^2}{\|u\|_2^2} \hspace{2.6em} \left(\text{as $U \subseteq S + \bar S$ and $S \perp \bar S$}\right) \\
    &=  \sigma_{\dim(X) - \dim(S)}(T)^2 \min_{s \in S \setminus \{\0\}} \frac{\|\proj_{U}(s)\|_2^2}{\|s\|_2^2} \hspace{7em} \left(\text{by \Cref{lem:subspace-projection-adjoint}}\right) \\
    &\ge  \sigma_{\dim(X) - \dim(S)}(T)^2 \left( 1 - \SVDA^2 \frac{\sigma_{\dim(X) - \dim(S) + 1}(T)^2}{\sigma_{\dim(X) - \dim(S)}(T)^2} \right)\, . \, \left(\text{by \Cref{lem:subspace-proximity}}\right)
\end{align*}  
\else
\begin{align*}
    \sigma_{\min}(\restr{T}{\bar S})^2 &= \min_{\bar s \in \bar S \setminus \{\0\}} \frac{\|T(\bar s)\|_2^2}{\|\bar s\|_2^2}  \\
    &\ge \min_{\bar s \in \bar S \setminus \{\0\}} \frac{\|T(\proj_{\bar U}(\bar s))\|_2^2}{\|\bar s\|_2^2} \hspace{13em} \left(\text{by $T(U) \perp T(\bar U)$}\right) \\
    &\ge \sigma_{\dim(X) - \dim(S)}(T)^2 \min_{\bar s \in \bar S \setminus \{\0\}} \frac{\|\proj_{\bar U}(\bar s)\|_2^2}{\|\bar s\|_2^2} \quad \left(\text{as $\sigma_{\min}(\restr{T}{\bar U}) = \sigma_{\dim(X) - \dim(S)}(T)$}\right) \\
    &= \sigma_{\dim(X) - \dim(S)}(T)^2 \left(1 - \max_{\bar s \in \bar S \setminus \{\0\}} \frac{\|\proj_{U}(\bar s)\|_2^2}{\|\bar s\|_2^2}\right) \quad \left(\text{as $S \subseteq U + \bar U$ and $U \perp \bar U$}\right) \\
    &= \sigma_{\dim(X) - \dim(S)}(T)^2 \left(1 - \max_{u \in U \setminus \{\0\}} \frac{\|\proj_{\bar S}(u)\|_2^2}{\|u\|_2^2}\right) \hspace{3.5em} \left(\text{by \Cref{lem:subspace-projection-adjoint}}\right) \\
    &=  \sigma_{\dim(X) - \dim(S)}(T)^2 \min_{u \in U \setminus \{\0\}} \frac{\|\proj_{S}(u)\|_2^2}{\|u\|_2^2} \hspace{2.6em} \left(\text{as $U \subseteq S + \bar S$ and $S \perp \bar S$}\right) \\
    &=  \sigma_{\dim(X) - \dim(S)}(T)^2 \min_{s \in S \setminus \{\0\}} \frac{\|\proj_{U}(s)\|_2^2}{\|s\|_2^2} \hspace{7em} \left(\text{by \Cref{lem:subspace-projection-adjoint}}\right) \\
    &\ge  \sigma_{\dim(X) - \dim(S)}(T)^2 \left( 1 - \SVDA^2 \frac{\sigma_{\dim(X) - \dim(S) + 1}(T)^2}{\sigma_{\dim(X) - \dim(S)}(T)^2} \right)\, . \, \left(\text{by \Cref{lem:subspace-proximity}}\right)
\end{align*}  
\fi
This proves the lemma.
\end{proof}
 
\end{document}